\documentclass[10pt]{article}
\usepackage{amssymb}
\usepackage{amssymb}
\usepackage{amsmath}
\usepackage{amsfonts}
\usepackage{graphicx}
\usepackage{soul, color}
\usepackage{tikz}
\usepackage{hyperref}
\usepackage{multirow}

\newtheorem{theorem}{Theorem}

\newtheorem{corollary}[theorem]{Corollary}

\newtheorem{definition}[theorem]{Definition}

\newtheorem{lemma}[theorem]{Lemma}

\newtheorem{proposition}[theorem]{Proposition}
\newtheorem{remark}[theorem]{Remark}

\newenvironment{proof}[1][Proof]{\noindent\textbf{#1.} }{\ \rule{0.5em}{0.5em}}
\input{tcilatex}

\newcommand{\ud}{\,\mathrm{d}}
\newcommand{\p}{\ensuremath{\partial}}
\newcommand{\n}{\ensuremath{\nonumber}}
\newcommand{\eps}{\ensuremath{\varepsilon}}
\newcommand{\bigO}{\mathcal{O}}

\DeclareSymbolFont{rmlargesymbols}{OMX}{mdbch}{m}{n}
\DeclareMathSymbol{\rmintop}{\mathop}{rmlargesymbols}{82}

 \allowdisplaybreaks

\title{\vspace{-50pt} Validity of Steady Prandtl Layer Expansions}
\author{ \Large Yan Guo\footnote{\url{yan_guo@brown.edu}. Division of Applied Mathematics, Brown University, 182 George Street, Providence, RI 02912, USA.} \hspace{10 mm} Sameer Iyer \footnote{\url{ssiyer@math.princeton.edu}. Department of Mathematics, Princeton University, Fine Hall, Washington Road, Princeton, NJ 08540, USA.}}
\date{October 12, 2018}

\begin{document}

\maketitle

\begin{abstract}
Let the viscosity $\varepsilon \rightarrow 0$ for the 2D steady
Navier-Stokes equations in the region $0\leq x\leq L$ and $0\leq y<\infty $
with no slip boundary conditions at $y=0$. For $L<<1$, we justify the validity of the steady Prandtl layer expansion for scaled Prandtl layers, including the celebrated Blasius boundary layer. Our uniform
estimates in $\varepsilon $ are achieved through a fixed-point scheme: 
\begin{equation*}
[u^{0}, v^0] \overset{\text{DNS}^{-1}}{\longrightarrow }v\overset{\mathcal{L}^{-1}}{%
\longrightarrow }[u^{0}, v^0]  \label{fixedpoint}
\end{equation*}%
for solving the Navier-Stokes equations, where $[u^{0}, v^0]$ are the tangential
 and normal velocities at $x=0,$ DNS stands for $\partial _{x}$ of the vorticity
equation for the normal velocity $v$, and $\mathcal{L}$ the compatibility ODE for $[u^{0}, v^0]$ at $x=0.$
\end{abstract}

\section{Introduction and Notation}

We consider the steady, incompressible Navier-Stokes equations on the two-dimensional domain, $(x,Y) \in \Omega = (0,L) \times (0, \infty)$. Denoting the velocity $\bold{U}^{NS} := (U^{NS}, V^{NS})$, the equations read: 
\begin{align}  \label{intro.NS}
\left.
\begin{aligned}
&\bold{U}^{NS} \cdot \nabla \bold{U}^{NS} + \nabla P^{NS} = \eps \Delta \bold{U}^{NS}  \\
&\nabla \cdot \bold{U}^{NS} = 0 
\end{aligned}
\right\} \text{ in } \Omega
\end{align}

The system above is taken with the no-slip boundary condition on $\{Y = 0\}$: 
\begin{align} \label{noslip.BC}
[U^{NS}, V^{NS}]|_{Y= 0} = [0, 0]. 
\end{align}

In this article, we fix an outer Euler shear flow of the form $[u^0_e(Y), 0, 0]$, (satisfying generic smoothness and decay assumptions). A fundamental question is to describe the asymptotic behavior of solutions to (\ref{intro.NS}) as the viscosity vanishes, that is as $\eps \rightarrow 0$. Generically, there is a mismatch of the tangential velocity at the boundary $\{Y = 0\}$ of the viscous flows, (\ref{noslip.BC}), and inviscid flows. Thus, one cannot expect $[U^{NS}, V^{NS}] \rightarrow [u^0_e, 0]$ in a sufficiently strong norm (for instance, $L^\infty$).  

To rectify this mismatch, it was proposed in 1904 by Ludwig Prandtl that there exists a thin fluid layer of size $\sqrt{\eps}$ near the boundary $Y = 0$ that bridges the velocity of $U^{NS}|_{Y = 0} = 0$ with the nonzero Eulerian velocity. This layer is known as the Prandtl boundary layer. 

In this article, we address Prandtl's classical setup. We work with the scaled boundary layer variable:
\begin{align} \label{BL.variable}
y = \frac{Y}{\sqrt{\eps}},
\end{align}

Consider the scaled Navier-Stokes velocities: 
\begin{align}
U^\eps(x,y) = U^{NS}(x,Y), \hspace{3 mm} V^\eps = \frac{V^{NS}(x,Y)}{\sqrt{\eps}}, \hspace{3 mm} P^\eps(x,y) = P^{NS}(x,Y). 
\end{align}

Equation (\ref{intro.NS}) now becomes: 
\begin{align} \label{scaledNSint}
\begin{aligned}
&U^\eps  U^\eps_x + V^\eps U^\eps_y + P^\eps_x = \Delta_\eps U^\eps  \\
&U^\eps V^\eps_x + V^\eps V^\eps_y + \frac{P^\eps_y}{\eps} = \Delta_\eps V^\eps\\
&U^\eps_x + V^\eps_y = 0
\end{aligned}
\end{align}

We expand the solution in $\eps$ as:
\begin{align}
\begin{aligned} \label{exp.u}
&U^\eps = u^0_e + u^0_p + \sum_{i = 1}^n \sqrt{\eps}^i (u^i_e + u^i_p) + \eps^{N_0} u^{(\eps)} := u_s + \eps^{N_0} u^{(\eps)}, \\
&V^\eps = v^0_p + v^1_e + \sum_{i = 1}^{n-1} \sqrt{\eps}^i (v^i_p + v^{i+1}_e) + \sqrt{\eps}^n v^n_p + \eps^{N_0} v^{(\eps)} := v_s + \eps^{N_0} v^{(\eps)}, \\
&P^\eps = P^0_e + P^0_p + \sum_{i = 1}^n \sqrt{\eps}^i (P^i_e + P^i_p) + \eps^{N_0} P^{(\eps)} := P_s + \eps^{N_0} P^{(\eps)},
\end{aligned}
\end{align}

\noindent where the coefficients are independent of $\eps$. Here $[u^i_e, v^i_e]$ are Euler correctors, and $[u^i_p, v^i_p]$ are Prandtl correctors. These are constructed in the Appendix, culminating in Theorem \ref{thm.construct}. For our analysis, we will take $n = 4$ and $N_0 = 1+$. Let us also introduce the following notation: 
\begin{align}
\begin{aligned} \label{intro.bar.prof}
 \bar{u}^i_p := u^i_p - u^i_p|_{y = 0}, \hspace{3 mm} \bar{v}^i_p := v^i_p - v^1_p|_{y = 0}, \hspace{3 mm} \bar{v}^i_e := v^i_e - v^i_e|_{Y = 0}.
\end{aligned}
\end{align}

The profile $\bar{u}^0_p, \bar{v}^0_p$ from (\ref{intro.bar.prof}) is classically known as the ``boundary layer"; one sees from (\ref{exp.u}) that it is the leading order approximation to the Navier-Stokes velocity, $U^\eps$. We will sometimes use the notation $u_{\parallel} := \bar{u}^0_p$, and $v_{\parallel} := \bar{v}^0_p$. The final layer, 
\begin{align*}
[u^{(\eps)}, v^{(\eps)}, P^{(\eps)}] = [\bold{u}^{(\eps)}, P^{(\eps)}].
\end{align*}

\noindent are called the ``remainders" and importantly, they depend on $\eps$. Controlling the remainders uniformly in $\eps$ is the fundamental challenge in order to establish the validity of (\ref{exp.u}), and the centerpiece of our article. 

Thanks to the elliptic feature of the steady Navier-Stokes equations, the set-up of our program is to assume the remainders $[u^{(\eps)}, v^{(\eps)}]$ are bounded in a suitable sense at the boundaries $\{x = 0\}$ and $\{x = L \}$ and to prove that they remain bounded for $x \in [0, L]$. It is important to note that there are no natural boundary conditions for the Navier-Stokes equations in a channel at $\{x = 0\}, \{x = L\}$, and thus part of the mathematical challenge is to impose boundary conditions for $[u^{(\eps)}, v^{(\eps)}]$ which ensure its solvability for $x \in [0, L]$. 

We begin by briefly discussing the approximations, $[u_s, v_s]$. The particular equations satisfied by each term in $[u_s, v_s]$ is derived and analyzed in Appendices \ref{appendix.derive} - \ref{appendix.Euler}, culminating in Theorem \ref{thm.construct}. We are prescribed the shear Euler flow, $u^0_e$. The profiles $[u^i_p, v^i_p]$ are Prandtl boundary layers. Importantly, these layers are rapidly decaying functions of the boundary layer variable, $y$. At the leading order, $[u^0_p, v^0_p]$ solve the nonlinear Prandtl equation: 
\begin{align}
\begin{aligned} \label{Pr.leading.intro}
&\bar{u}^0_p u^0_{px} + \bar{v}^0_p u^0_{py} - u^0_{pyy} + P^0_{px} = 0, \\
&u^0_{px} + v^0_{py} = 0, \hspace{3 mm} P^0_{py} = 0, \hspace{3 mm} u^0_p|_{x = 0} = U^0_P, \hspace{3 mm} u^0_p|_{y = 0} = - u^0_e|_{Y = 0}.
\end{aligned}
\end{align}

Soon after Prandtl's seminal 1904 paper, Blasius discovered the celebrated self-similar solution to (\ref{Pr.leading.intro}) (with zero pressure). This solution reads 
\begin{align} \label{blasius}
[\bar{u}^0_p, \bar{v}^0_p] = \Big[f'(\eta), \frac{1}{\sqrt{x + x_0}}\{ \eta f'(\eta) - f(\eta) \} \Big], \text{ where } \eta = \frac{y}{\sqrt{x+x_0}},
\end{align}

\noindent where $f$ satisfies
\begin{align} \label{blasius.ODE}
ff'' + f''' = 0, \hspace{3 mm} f'(0) = 0, \hspace{2 mm} f'(\infty) = 1, \hspace{2 mm} \frac{f(\eta)}{\eta} \xrightarrow{n \rightarrow \infty} 1.
\end{align}

\noindent Here, $x_0 > 0$ is a free parameter. It is well known that $f''(\eta)$ has a Gaussian tail, and that the following hold: 
\begin{align*}
0 \le f' \le 1, \hspace{3 mm} f''(\eta) \ge 0, \hspace{3 mm} f''(0) > 0, \hspace{3 mm} f'''(\eta) < 0 . 
\end{align*}

Such a Blasius profile has been confirmed by experiments with remarkable accuracy as the main validation of the Prandtl theory (see \cite{Schlicting} for instance). These profiles are also canonical from a mathematical standpoint in the following sense: the work, \cite{Serrin}, has proven that when $x$ gets large (downstream), solutions to the Prandtl equation, (\ref{Pr.leading.intro}), converge to an appropriately renormalized Blasius profile. Therefore, validating the expansions (\ref{exp.u}) for the Blasius profile is the main objective and motivation in our study.  

It is well known that the Prandtl equations (\ref{Pr.leading.intro}) admit the two parameter scaling invariance: 
\begin{align} \label{scaling.intro}
&[\bar{u}^{\lambda, \sigma}, \bar{v}^{\lambda, \sigma}] =  [\frac{\lambda^2}{\sigma} \bar{u}^0_p(\sigma x, \lambda y), \lambda \bar{v}^0_p(\sigma x, \lambda y)],
\end{align}

\noindent meaning that if $[\bar{u}^0_p, \bar{v}^0_p]$ solve (\ref{Pr.leading.intro}), then so do $[\bar{u}^{\lambda, \sigma}, \bar{v}^{\lambda, \sigma}]$ (with appropriately modified initial data). 

Typically in boundary layer analyses, the central mathematical analysis concerns the linearized Navier-Stokes operator. Such an operator has coefficients $[u_s, v_s]$, which are the approximate Navier-Stokes solutions defined as in (\ref{exp.u}). The unknown that this operator acts on is the ``remainders", $[u^{(\eps)},v^{(\eps)},P^{(\eps)}]$. In vorticity formulation, the operator reads 
\begin{align} 
\begin{aligned} \label{eqn.vort.intro}
-R[q^{(\eps)}] - u^{(\eps)}_{yyy} &+ 2\eps v^{(\eps)}_{xyy} + \eps^2 v^{(\eps)}_{xxx} + v_s \Delta_\eps u^{(\eps)} - u^{(\eps)} \Delta_\eps v_s \\
& = \eps^{N_0} \{u^{(\eps)} \Delta_\eps v^{(\eps)} - v^{(\eps)} \Delta_\eps v^{(\eps)}\} + F_R,
\end{aligned}
\end{align}

\noindent Here, $\Delta_\eps := \p_{yy} + \eps \p_{xx}$, $F_R$ is a forcing term defined in (\ref{forcingdefn}), and where we have defined the Rayleigh operator
\begin{align} \label{rayleigh.quotient}
R[q^{(\eps)}] = \p_y\{ u_s^2 \p_y q^{(\eps)}\} + \eps \p_x \{ u_s^2 q^{(\eps)}_x \}, \hspace{3 mm} q^{(\eps)} := \frac{v^{(\eps)}}{u_s}.
\end{align}

\noindent The boundary condition we take are the following
\begin{align}
\begin{aligned} \label{rene}
&v^{(\eps)}_x|_{x = L} = a^\eps_1(y), v^{(\eps)}_{xx}|_{x = 0} = a^\eps_2(y), v^{(\eps)}_{xxx} = a^\eps_3(y) \\
& v^{(\eps)}|_{x = 0} = v^0(y), \hspace{3 mm} v^0_y + u^0 = h(y) \in C^\infty(e^y), \hspace{3 mm} h(0) = 0, \\
&v^{(\eps)}|_{y = 0} = v^{(\eps)}_y|_{y = 0} = u^{(\eps)}|_{y = 0} = 0, v^{(\eps)}_y|_{y \uparrow \infty} = 0, 
\end{aligned}
\end{align}

\noindent Here, the $a^\eps_i(y)$ are prescribed boundary data which we assume satisfy 
\begin{align} \label{assume.bq.intro}
\| \p_y^{j} a^\eps_i \{ \frac{1}{\eps^{\frac{1}{2}}} \langle y \rangle \langle Y \rangle^m \} \| \le o(1) \text{ for } j = 0,...,4, \text{ and } m \text{ large,}
\end{align}

\noindent which is a quantitative statement that the expansion (\ref{exp.u}) is valid at $\{x = 0\}$ and $\{x = L\}$. 

We are now able to state our main result, so long as we remain vague regarding the norm $\| \cdot \|_{\mathcal{X}}$ that appears below. A discussion of this norm will be in Subsection \ref{subsection.Main}.

\begin{theorem}[Main Theorem] \label{thm.main}  Assume boundary data are prescribed as in  Theorem \ref{thm.construct}, (\ref{rene}), and satisfying (\ref{assume.bq.intro}). Assume $0 < \sigma << 1$ in (\ref{scaling.intro}). Then let $0 < \eps << L << 1$. Take $N_0 = 1+$ and $n = 4$ in (\ref{exp.u}). Then all terms in the expansion (\ref{exp.u}) exist and are regular, $\| u_s, v_s \|_\infty \lesssim 1$. The remainders, $[u^{(\eps)},v^{(\eps})]$ exists uniquely in the space $\mathcal{X}$ and satisfy
\begin{align} \label{unif.x}
\| \bold{u}^{(\eps)} \|_{\mathcal{X}} \lesssim 1.
\end{align}
\noindent The Navier-Stokes solutions satisfy
\begin{align} \label{main.inviscid}  
\| U^{NS} - u^0_e - u^0_p \|_\infty \lesssim \sqrt{\eps} \text{ and } \| V^{NS} - \sqrt{\eps} v^0_p - \sqrt{\eps}v^1_e \|_\infty \lesssim \eps.
\end{align}
\end{theorem}

Upon establishing the uniform bound (\ref{unif.x}), the result (\ref{main.inviscid}) follows from the following inequalities: $\| v^{(\eps)} \|_\infty \lesssim  \eps^{-\frac{1}{2}} \| \bold{u}^{(\eps)} \|_{\mathcal{X}}$, and $\|u^{(\eps)} \|_{\infty} \lesssim  \|\bold{u}^{(\eps)} \|_{\mathcal{X}}$. These are established in Lemmas \ref{sister.lemma}, \ref{lemma.nonlinear} together with the definitions in (\ref{defn.norms.ult}).
 
Our main result thus ensures a local in space ($L << 1$) validity for the Prandtl expansion, (\ref{exp.u}). This marks an important first step to study the optimal bound for $\sup L$. Such a study (in progress, \cite{Iyer-Global}) would address the phenomenon of ``boundary layer detachment" (which would correspond to $\sup L < \infty$) versus global in $x$ validity (in the sense of \cite{Iyer2}). 

Regarding our scaling, (\ref{scaling.intro}), it is important to note that $\lambda$ can be arbitrary. This covers rich structures in the Prandtl equation. In particular, when $\lambda^2 = \sigma$, the scaling of $\lambda \rightarrow 0$ is equivalent to letting $x_0 \rightarrow \infty$ in (\ref{blasius}). Letting $\eta_\lambda$ denote the rescaled self-similar variable, one has by definition
\begin{align*}
\eta_\lambda := \frac{\lambda y}{\sqrt{\lambda^2 x + x_0}} = \frac{y}{\sqrt{x + \lambda^{-2} x_0}}.
\end{align*}
For this reason, we interpret our main theorem as being \textit{asymptotic}, that is for large values of $x_0$: in the particular case of $\lambda = \sigma^2$, setting $\sigma$ small is equivalent to taking $x_0$ large. Moreover, in light of \cite{Serrin}, general solutions to the Prandtl equation converge to the Blasius profile as $x_0 \rightarrow \infty$. We thus expect that the validity of (\ref{exp.u}) holds for generic Prandtl data without rescaling, for $x_0 >> 1$. Furthermore, we remark the $L$ may not necessarily need to be small in this case. 

\subsection{Notation}

Before we state the main ideas of the proof, we will discuss our notation. Since we use the $L^2$ norm extensively in the analysis, we use $\| \cdot \|$ to denote the $L^2$ norm. It will be clear from context whether we mean $L^2(\mathbb{R}_+)$ or $L^2(\Omega)$. When there is a potential confusion (for example, when changing coordinates), we will take care to specify with respect to which variable the $L^2$ norm is being taken (for instance, $L^2_y$ means with respect to $\ud y$, whereas $L^2_Y$ will mean with respect to $\ud Y$). Similarly, when there is potential confusion, we will distinguish $L^2$ norms along a one-dimensional surface (say $\{x = 0\}$) by $\| \cdot \|_{x = 0}$. Analogously, we will often use inner products $(\cdot, \cdot)$ to denote the $L^2$ inner product. When unspecified, it will be clear from context if we mean $L^2(\mathbb{R}_+)$ or $L^2(\Omega)$. When there is potential confusion, we will distinguish inner products on a one-dimensional surface (say $\{x = 0\}$) by writing $(\cdot, \cdot)_{x = 0}$. Given a weight function $w$, we use the notation $\| \cdot \|_{L^2(w)} := \| \cdot w \|$, and $L^2(w)$ to refer to the corresponding weighted $L^2$ space. 

We will often use scaled differential operators
\begin{align*}
\nabla_\eps := (\p_x, \sqrt{\eps}\p_y), \hspace{5 mm} \Delta_\eps := \p_{yy} + \eps \p_{xx}. 
\end{align*}

\noindent Define also the integration operator, $I_x[g] := \int_0^x g(x') \ud x'$. For functions $w: \mathbb{R}_+ \rightarrow \mathbb{R}$, we distinguish between $w'$ which means differentiation with respect to its argument versus $w_y$ which refers to differentiation with respect to $y$. 

Regarding unknowns, the central object of study in our paper are the remainders, $[u^{(\eps)}, v^{(\eps)}]$. By a standard homogenization argument (see subsection \ref{subsection.rem}), we may move the inhomogeneous boundary terms $a_i^{\eps}$ to the forcing and consider the homogeneous problem. Specifically, we homogenize $v^{(\eps)}$ to $v$ using the following:
\begin{align}
\begin{aligned} \label{homfin}
&\tilde{v} :=  v^0 + x\{a^\eps_1 - L a^\eps_2 - \frac{L^2}{2}a^\eps_3  \} + x^2 \frac{a^\eps_2}{2} + x^3 \frac{a^\eps_3}{6} =:  v^0 + a^\eps(x,y), \\
&v := v^\eps - \tilde{v} = v^\eps - v^0 - a^\eps, \hspace{3 mm} u := u^\eps + \int_0^x \tilde{v}_y = u^\eps + x v^0 + I_x[a^\eps].
\end{aligned} 
\end{align}

\noindent We call the new unknowns $[u, v]$ (= $\bold{u}$), and these are actually the objects we will analyze throughout the paper.   

When we write $a \lesssim b$, we mean there exists a number $C < \infty$ such that $a \le C b$, where $C$ is independent of small $L, \eps$ but could depend on $[u_s, v_s]$. We write $o_L(1)$ to refer to a constant that is bounded by some unspecified, perhaps small, power of $L$: that is, $a = o_L(1)$ if $|a| \le C L^\delta$ for some $\delta > 0$. 

We will, at various times, require localizations. All such localizations will be defined in terms of the following fixed $C^\infty$ cutoff function: 
\begin{align} \label{basic.cutoff}
\chi(y) := \begin{cases}1 \text{ on } y \in [0,1) \\ 0 \text{ on } y \in (2,\infty) \end{cases} \hspace{3 mm} \chi'(y) \le 0 \text{ for all } y > 0.  
\end{align}

We will use $\| \cdot \|_{loc}$ to mean localized $L^2$ norms. More specifically we take for concreteness $\| \cdot \|_{loc} := \| \cdot \chi(\frac{y}{10}) \|$. We adopt the notation that $\langle a \rangle = 1 + a$. Define the weight 
\begin{align} \label{w0}
w_0 := \langle y \rangle \langle Y \rangle^{m}, \text{ for } m \text{ sufficiently large, universal number.}
\end{align}

We will define now the key norms that appear throughout our analysis:
\begin{definition} \label{defn.norms.intro}   Given a weight function $w = w(y)$, define:  
\begin{align} 
\begin{aligned} \label{defn.norms.ult}
&\| v \|_{X_w} := \eps^{-\frac{3}{16}}||||v||||_w + |||q|||_w, \\
&\| v \|_{Y_w} := ||||v||||_w + \sqrt{\eps}|||q|||_w, \\
& [u^0, v^0 ]_B := \| u^0 \| + \| u^0_y \| + \| u^0_{yy}w \| + \|u^0_{yyy} w \|  + \| q^0_y \| + \| \frac{q^0}{y} \| \\
& \hspace{20 mm} + \| \sqrt{u_s} q^0_{yy} w_0 \| + \| q^0_y \|_{y = 0} + \| v^0_{yyy} w_0 \| + \| v^0_{yyyy} w_0 \|, \\
&\| \bold{u} \|_{\mathcal{X}} (= \| v, u^0, v^0 \|_{\mathcal{X}}) := [u^0, v^0]_B + \eps^{\frac{1}{4}} \| v \|_{X_1} + \eps^{\frac{1}{4}}  \|v \|_{Y_{w_0}}, \\
&|||q|||_w := \| \nabla_\eps q_x \cdot u_s w \| + \| \sqrt{u_s} \{ q_{yyy}, q_{xyy}, \sqrt{\eps} q_{xxy}, \eps q_{xxx} \} w \| + |q|_{\p, 2, w}\\
&||||v||||_w := \| \{v_{yyyy}, \sqrt{\eps} v_{xyyy}, \eps v_{xxyy}, \eps^{\frac{3}{2}}v_{xxxy}, \eps^2 v_{xxxx} \} w \| + |v|_{\p, 3, w}\\
&| q |_{\p,2,w} :=  \|u_s q_{xy}w\|_{x = 0} + \|q_{xy} w\|_{y = 0} + \|\sqrt{\eps}u_s q_{xx}w\|_{x = L} + \|q_{yy}w\|_{y= 0} \\
&| v |_{\p, 3,w} := \|\eps^{\frac{3}{2}}\sqrt{u_s} v_{xxx} w\|_{x = 0} + \|\sqrt{\eps} u_s v_{xyy}w\|_{x = 0} + \|\eps u_s v_{xxy}w\|_{x = L}.
\end{aligned}
\end{align}
\end{definition}

\noindent Note above that we identify the vector $\bold{u}$ with the triple $(v, u^0, v^0)$. We will use the above set of norms with either the choice $w = 1$ or $w = w_0$ (see (\ref{w0}).  We also define now the space $\mathcal{X}$:

\begin{definition} \label{space.X} The space $\mathcal{X}$ is defined via 
\begin{align}
\begin{aligned} \label{sp.X}
\mathcal{X} := \Big\{ &(v, u^0, v^0)  \in L^2(\Omega) \times L^2(\mathbb{R}_+) \times L^2(\mathbb{R}_+): \|v, u^0, v^0 \|_{\mathcal{X}} < \infty,  \\
& v|_{y = 0} = v_y|_{y = 0} = v_{x}|_{x = L} = v_{xx}|_{x = 0} = v_{xxx}|_{x = L} = v|_{y = \infty} = 0,  \\
&v^0(0) = v^0_y(0) = \p_y^k v(\infty) = 0 \text{ for } k \ge 1, \hspace{3 mm} u^0 + v^0_y = h(y), \hspace{3 mm} u^0(0) = 0. \Big\} 
\end{aligned}
\end{align}
\end{definition}

\subsection{Overview of Proof} \label{subsection.Main}

Let us first recap the ideas introduced in \cite{GN}, which treated the case when the boundary $\{y = 0\}$ was moving with velocity $u_b > 0$. First, let us extract: 
\begin{align} \label{leading.1}
\text{Leading order operators in (\ref{eqn.vort.intro})} = - R[q] - u_{yyy}.
\end{align}

Due to the nonzero velocity at the $\{y = 0\}$ boundary, the quantity $\bar{u}|_{y = 0} > 0$. A central idea introduced by \cite{GN} is the coercivity of $R[q]$ over $\| \nabla_\eps q \|$.This coercivity relied on the fact that $q = \frac{v}{u_s} = 1 \notin \text{Ker}(R)$, thanks to the non-zero boundary velocity of $\bar{u}|_{y = 0}$. Extensive efforts without success have been made to extracting coercivity from $R[q]$ in the present, motionless boundary, case. However, it appears that this procedure interacts poorly with the operator $\p_{yyy} u$, producing singularities too severe to handle. In fact, the natural multiplier for the Rayleigh operator is $q$ itself, which produces $(R[q], q) = \| u_s q_y \|^2$. However, due to the degeneracy of $u_s$ at $y = 0$ (which is notably absent when $u_s|_{y = 0} > 0$ as in \cite{GN}) this is too weak of a contribution to control the interaction term $(u_{yy}, q_y)$.

Our main idea is based on the observation that the $x$ derivative of (\ref{leading.1}) produces, at leading order:
\begin{align} \label{leading.2}
- \p_x R[q] + v_{yyyy}.
\end{align}

\noindent Unlike (\ref{leading.1}), these two operators enjoy better interaction properties. To see this on a preliminary level, consider the interaction between $v_{yyyy}$ and the multiplier $-q_{xx}$ (ignoring boundary contributions at $x = 0, x = L$): 
\begin{align} \n
(v_{yyyy}, -q_{xx}) \sim & -(v_{yyyx}, q_{xy}) \\ \n
\sim &  -(u_s q_{yyyx} + 3 u_{sy}q_{yyx} + 3 u_{syy}q_{yx} + u_{syyy}q_x, q_{xy}) \\ \label{90}
\sim & \| \sqrt{u_s} q_{yyx} \|^2 + \frac{3}{2} \| \sqrt{u_{sy}} q_{xy} \|_{y = 0}^2,
\end{align}

\noindent which is a crucial favorable boundary contribution at $\{y = 0 \}$ as $u_{sy}|_{y = 0} \sim \bar{u}^0_{py} > 0$.  

To this end, we split the equation (\ref{eqn.vort.intro}) into two pieces that are linked together. First, we study the boundary trace, $[u^0, v^0] = [u, v]|_{x = 0}$. By evaluating the vorticity equation (\ref{eqn.vort.intro}) at $\{x = 0\}$ and using the relation (\ref{rene}), we obtain the following system for $v^0$:
\begin{align}
\begin{aligned} \label{sys.u0.intro}
&\mathcal{L} v^0 = F_{(v)} + F^a_R + \mathcal{Q} + \mathcal{H}, \\
&\mathcal{L} v^0 :=  v^0_{yyyy} - \p_{y}\{ u_s^2 q^0_y \} - \{ v_s v^0_{yyy} - v^0_y v_{syy} \} + \eps u_{sxx} v^0 + \eps v_{sxx} v^0_y, \\
&F_{(v)} := -2\eps u_s u_{sx}q_x|_{x = 0} - 2\eps v_{xyy}|_{x = 0} - \eps^2 v_{xxx}|_{x = 0} - \eps v_s u_{xx}|_{x = 0}, \\
& v^0(0) = 0, \hspace{3 mm} v^0_y(0) = 0, \hspace{3 mm}  \p_y^k v^0(\infty) = 0 \text{ for } k \ge 1, \\
&v^0_y + u^0 = h(y). 
\end{aligned}
\end{align}

\noindent The $F^a_R + \mathcal{Q} + \mathcal{H}$ terms above contain contributions of $h(y)$, quadratic nonlinearities in $v^0$, and pure forcing terms. We refrain from discussing these terms further in the introduction; the full equations are specified in (\ref{sys.u0.app.unh}). The important point is that the forcing term $F_{(v)}$ in (\ref{sys.u0.intro}) depends on (derivatives of $v)|_{x = 0}$.

Second, we take $\p_x$ of (\ref{eqn.vort.intro}) (call this ``DNS" for Derivative Navier-Stokes) to obtain: 
\begin{align}
\begin{aligned} \label{intro.v.sys}
&\text{DNS}(v) := - \p_x R[q] + \Delta_\eps^2 v + J(v) =  - B_{v^0} + \eps^{N_0} \mathcal{N} + F_{(q)} \\
&v|_{x = 0} = v_x|_{x = L} =  v_{xx}|_{x = 0} = v_{xxx} = 0. \\
&v|_{y = 0} = v|_{y = 0} = 0.  
\end{aligned}
\end{align}

\noindent Here, the $\eps^{N_0}\mathcal{N} + F_{(q)}$ terms are quadratic and forcing terms which shall remain unspecified for the moment. Note the change in notation as we have dropped the superscript $\eps$, and homogenized the boundary conditions on the sides $\{x = 0\}, \{x = L\}$. Above, $B_{v^0}$ is the result of homogenizing the boundary condition $v|_{x = 0} = v^0$ as well as using $u = u^0 - I_x[v_y]$. The operators $J, B_{v^0}$ are defined: 
\begin{align} \label{defn.J.conc}
&J(v) := - v_{sx}I_x[v_{yyy}] - v_s v_{yyy} - \eps v_{sx}v_{xy} \\ \n
& \hspace{20 mm} - \eps v_s v_{xxy} + v_y \Delta_\eps v_s + I_x[v_y] \Delta_\eps v_{sx}, \\ \n
&B_{v^0}:=  v^0_{yyyy} - 2 \p_y \{ u_s u_{sx} q^0_y \} + [v^0_{yyy} \p_x \{(x + 1) v_{s} \} \\  \label{portrait} 
& \hspace{20 mm} - v^0_y \p_x \{(x+1) v_{syy} \}]  - \eps v^0_y \p_{x} \{ (x+1) v_{sxx} \}.
\end{align}

Thus, the approach we take is to analyze (\ref{sys.u0.intro}) in order to control the boundary trace $[u^0, v^0]$ in terms of $v$, and subsequently analyze \ref{intro.v.sys}) to control $v$ in terms of the boundary trace, $[u^0, v^0]$. We may schematize this procedure via: 
\begin{align} \label{diag1}
[u^0, v^0] \xrightarrow{\text{DNS}^{-1}} v \xrightarrow{\mathcal{L}^{-1}} [u^0, v^0]. 
\end{align}

\noindent We then recover a solution to the original Navier-Stokes equation (NS) via a fixed point of (\ref{diag1}). This structure of analysis gives rise to a linked set of inequalities which we summarize here:
\begin{align}
\left.
\begin{aligned} \label{scheme.1}
&[ u^0, v^0 ]_{B}^2 \lesssim  \eps \| v \|_{Y_{w_0}}^2 + \eps^{\frac{1}{2}+\frac{3-}{16} } \| v \|_{X_1}^2 + \text{Data} \\
&\| v \|_{X_1}^2 \lesssim \eps^{-\frac{1}{2}} [u^0, v^0 ]_{B}^2 + \text{Data}  \\
&\| v \|_{Y_{w_0}}^2 \lesssim  \| v \|_{X_1}^2 + [ u^0, v^0]_{B}^2 + \text{Data}.
\end{aligned}
\right\}.
\end{align}

\noindent Above $w_0$ is the specific weight given in (\ref{w0}). It is clear that the above scheme of estimates closes to yield control over $\| \bold{u} \|_{\mathcal{X}}$.

As shown in Section \ref{Subsection.nonlin} that $B_{(v^0)}$ (Lemma \ref{lemma.bb1}), $F_{(v)}$ (Lemma \ref{lemma.bb2}) and the nonlinearity (Lemma \ref{lemma.nonlinear}) can be controlled with a small constant. We therefore turn our attention to the following two linear problems. 

\subsubsection*{ \normalfont \textit{Section 2: Study of linear problem $\mathcal{L} v^0 = F$}}

Let us turn now to the system, (\ref{sys.u0.intro}). The main estimate we prove is:
\begin{align} \label{prop.intro.1.est}
[ u^0, v^0 ]_B & \lesssim \| F_{(v)} w_0\| + \text{ Data}. 
\end{align}

\noindent Upon recalling the terms in $F_{(v)}$ shown in (\ref{sys.u0.intro}) and analyzing the resulting expressions, such an estimate produces the first bound shown in (\ref{scheme.1}).

By evaluating the vorticity equation, (\ref{eqn.vort.intro}) at $\{x = 0 \}$, one obtains a compatibility equation that must be satisfied by the tuple, $[u^0, v^0]$. However, it is important to realize that we have the freedom to prescribe the relationship between these two boundary data. We do so by selecting $u^0 + v^0_y = h(y)$ as shown in (\ref{rene}). This boundary condition is natural from the setup of our program, since both $u^0$ and $v^0_y$ should be individually bounded in Sobolev norms. The selection of this boundary condition results in a fourth order equation $v^0_{yyyy} - \p_y \{ u_s^2 q^0_y \}$, which enjoys similar favorable properties to DNS and similar corresponding quotient estimates as in (\ref{90}).

Estimate (\ref{prop.intro.1.est}) is obtained in two steps. The first step is to apply the multiplier $q^0 = \frac{v}{u_s}$, and the second is to apply the multiplier $v_{yyyy}w_0^2$. The multiplier $q^0$ is the delicate step in which the interaction between the $\p_y^4$ operator and the Rayleigh term $-\p_y \{u_s^2 \p_y  q^0 \}$ must be understood. The key estimate we prove in this direction is the positivity 
\begin{align*}
(\p_y^4 v^0 - \p_y \{ u_s^2 q^0_y \}, q^0) \gtrsim \| \sqrt{u_s} q^0_{yy} \|^2 + \| u_s q^0_y \|^2 + \| q^0_y \|_{y = 0}^2. 
\end{align*}

\noindent It is for this lower bound that we require the assumption that $\sigma << 1$ in (\ref{scaling.intro}). Once this is established, the remaining terms may be treated perturbatively. Overall, the upshot of the selection of boundary condition (\ref{rene}) is to capitalize on similar favorable structures to the DNS analysis. 

\subsubsection*{\normalfont \textit{Section 3 and 4: Study of linear problem $\text{DNS}(v) = F$}}

We now turn our attention to (\ref{intro.v.sys}). The goal is to establish control over the norms $\| \cdot \|_{Y_{w_0}}, \| \cdot \|_{Y_1}, \| \cdot \|_{X_1}$. Consulting (\ref{defn.norms.ult}), the basic building blocks of these norms are the fourth and third order norms, $||||\cdot||||_w, |||\cdot|||_w$. Hence, our discussion will be centered on the control of $|||| \cdot||||_w, |||\cdot|||_w$. Let us also emphasize that we require $L << 1$ to establish these controls and ultimately solve the DNS equation.

Based on the crucial quotient estimate, (\ref{90}), we perform a cascade of five estimates which culminate in the following: 
\begin{align}
\begin{aligned} \label{feelsjai}
&||||v||||_1^2 \lesssim \eps^{\frac{3}{8}}|||q|||_1^2 +\text{ Data,} \\
&|||q|||_1^2 \lesssim ||||v||||_1^2 +\text{ Data}. 
\end{aligned}
\end{align}

Let us discuss the important features of the above scheme. The top (fourth order) bound in (\ref{feelsjai}) consists of two estimates, first using the multiplier $\eps^2 v_{xxxx}$ and second using the multiplier $\eps u_s v_{xxyy}$. These estimates are possible due to carefully designed boundary conditions at $x = 0$ and $x = L$ for $v$ (see (\ref{intro.v.sys})). Our central observation at this level is that the $\eps u_s v_{xxyy}$ estimate is essentially \textit{standalone at the top order}, up to $|||q|||$, thanks to the crucial weight $u_s$. 

The bottom (third order) bound in (\ref{feelsjai}) consists of three delicate estimates, using the multipliers successively $q_x, q_{xx}, q_{yy}$. First, we emphasize that the multipliers at this stage are derivatives of the quotient, $q$. This is because the main coercivity is extracted from the Rayleigh operator, $R[q]$. The key feature we capitalize on in this scheme is that the estimates using multipler $q_x, q_{yy}$ are \textit{standalone up to $o_L(1)$ contributions}. It is important to note that since $q = \frac{v}{u_s}$, despite the presence of $\sqrt{u_s}$ weight in $|||q|||_w$ (see (\ref{defn.norms.ult})), this is still significantly stronger at $\{y = 0\}$ to $|||q|||_w$. In turn, to facilitate estimates near $\{y = 0\}$, we establish careful embedding estimates in (\ref{subsection.careful}).

The weighted analog of the scheme (\ref{feelsjai}) is, for any given $w(y)$ (satisfying reasonable hypotheses): 
\begin{align}
\begin{aligned} \label{feelsjaiw}
&||||v||||_w^2 \lesssim \eps^{\frac{3}{8}} ||| q |||_1^2 + \eps |||q|||_{w}^2 + \sqrt{\eps} |||q|||_{w} |||q|||_{w_y} + \text{ Data}, \\
&|||q|||_w^2 \lesssim o_L(1) ||||v||||_w^2 + o_L(1) \| q_{xx} \|_{w_y}^2 +  \text{ Data}.
\end{aligned}
\end{align}

Apart from the key elements discussed above in the unweighted case, the new features here is a \textit{gain of $\eps$} when going from $||||v||||_w$ to $|||q|||_w$. This crucial gain of $\eps$ is what ultimately enables us to relate the weighted estimate for $\| v \|_{Y_w}$ back to the $\| v \|_{X_1}$ unweighted norm. 

As a final remark, we note that the Appendix of this paper contains the construction of the profiles, $[u_s, v_s]$, culminating in Theorem \ref{thm.construct}. We wish to emphasize that our techniques provide regularity and decay estimates which, to our knowledge, were unknown even for the classical Prandtl equation, (\ref{Pr.leading.intro}).

\subsection{Other Works}

Let us now place this result in the context of the existing literature. To organize the discussion, we will focus on the setting of stationary flows in dimension 2. This setting in particular occupies a fundamental role in the theory, as it was the setting in which Prandtl first formulated and introduced the idea of boundary layers for Navier-Stokes flows in his seminal 1904 paper, \cite{Prandtl}. 

In this context, one fundamental problem is to establish the validity of the expansions (\ref{exp.u}). This was first achieved under the assumption of a moving boundary in \cite{GN} for $x \in [0,L]$, for $L$ sufficiently small. The method of \cite{GN} is to establish a positivity estimate to control $||\nabla_\eps v||_{L^2}$, which crucially used the assumed motion of the boundary. Several generalizations were obtained in \cite{Iyer}, \cite{Iyer2}, \cite{Iyer3}. First, \cite{Iyer} considered flows over a rotating disk, in which geometric effects were seen, \cite{Iyer2} considered flows globally in the tangential variable, and \cite{Iyer3} considered outer Euler flows that are non-shear. All of these works are under the crucial assumption of a moving boundary.

The classical setting of a motionless boundary with the no-slip condition is treated by the present work, the recent work \cite{GI1}, as well as the exciting result of \cite{Varet-Maekawa}. First, let us emphasize that the work \cite{GI1} requires an external forcing. Second, it is our understanding that our present work is mutually exclusive with the work of \cite{Varet-Maekawa}. Our work here, and main concern, treats the classical self-similar Blasius solution which appears to not be covered by \cite{Varet-Maekawa}. On the other hand, our result does not cover a pure shear boundary layer of the form $(U_0(y), 0)$ since such shears are not a solution to the homogeneous Prandtl equation. 

A related question is that of wellposedness of the Prandtl equation (the equation for $\bar{u}^0_p$, as defined in (\ref{intro.bar.prof})). This investigation was by Oleinik in \cite{Oleinik}, \cite{Oleinik1}. In the 2D, stationary setting, it is shown that under local monotonicity assumptions, solutions exist in $[0,L]$. In the case where the pressure gradient is favorable, it is shown that $L$ can be taken arbitrarily large. The recent work of \cite{MD} addresses the related issue of blowup of the Prandtl equation under the assumption of an unfavorable pressure gradient. The regularity results obtained in the present paper can be viewed as an extension of Oleinik's local-in-$x$ result: assuming strong decay at $y \rightarrow \infty$, we can obtain enhanced regularity of Oleinik's solutions. 

For unsteady flows, expansions of the form (\ref{exp.u}) have been verified in the analyticity framework: \cite{Caflisch1}, \cite{Caflisch2}, in the Gevrey setting: \cite{DMM}, for the initial vorticity distribution assumed away from the boundary: \cite{Mae}. The reader should also see \cite{Asano}, \cite{BardosTiti}, \cite{Kelliher}, \cite{HLop}, \cite{LXY}, \cite{Taylor}, \cite{TWang} for related results. There have also been several works (\cite{GGN1}, \cite{GGN2}, \cite{GGN3}, \cite{GN2}, \cite{GrNg1}, \cite{GrNg2}, \cite{GrNg3}) establishing generic invalidity of expansions of the type (\ref{exp.u}) in Sobolev spaces in the unsteady setting. 

In the unsteady setting, there is again the related matter of wellposedness of the Prandtl equation. This was also initiated by Oleinik, who under the monotonicity assumption, $\p_y U(t = 0) > 0$, obtained global-in-time regular solutions on $[0,L] \times \mathbb{R}_+$ for $L$ small, and local-in-time solutions on $[0,L] \times \mathbb{R}_+$ for arbitrarily large by finite $L$. Global-in-time weak solutions were obtained by \cite{Xin} for arbitrary $L$ under the monotonicity assumption and a favorable pressure gradient of the Euler flow: $\p_x P^E(t,x) \le 0$ for $t \ge 0$.

The works mentioned above use the Crocco transform, which is available in the monotonic setting. Still assuming monotonicity, local wellposedness was proven in \cite{AL} and \cite{MW} without using the Crocco transform, and in \cite{KMVW} for multiple monotonicity regions. \cite{MW} introduced a good unknown which enjoys a crucial cancellation, whereas \cite{AL} performed energy estimates on a transformed quantity together with a Nash-Moser iteration.

When the assumption of monotonicity is removed, the wellposedness results are largely in the analytic or  Gevrey setting. The reader should consult \cite{Caflisch1} - \cite{Caflisch2},  \cite{Kuka}, \cite{Lom}, \cite{Vicol}, and \cite{GVM} for some results in this direction. In the Sobolev setting without monotonicity, the equations are linearly and nonlinearly ill-posed (see \cite{GVD} and \cite{GVN}). A finite-time blowup result was obtained in \cite{EE} when the outer Euler flow is taken to be zero, in \cite{KVW} for a particular, periodic outer Euler flow, and in \cite{Hunter} for both the inviscid and viscous Prandtl equations. 

The above discussion is not comprehensive; we refer the reader to the review articles, \cite{E}, \cite{Temam} and references therein for a more thorough review of the wellposedness theory.

\section{$\mathcal{L}^{-1}$ and Boundary Estimates for $[u^0, v^0]$} \label{Section.1}

\subsection{Setup and Basic Inequalities}

\noindent The system we analyze in this section is that for $[u^0, v^0]$. Recall (\ref{sys.u0.app.unh}) and the definition of $w_0$, (\ref{w0}). We thus consider
\begin{align}
\begin{aligned} \label{nawa}
&\mathcal{L} v^0 = F \in L^2(w_0), \\
&\mathcal{L} v^0 := v^0_{yyyy} - \{ u_s v^0_{yy} - u_{syy}v^0 \} - \{ v_s v^0_{yyy} - v^0_y v_{syy} \} \\
& \hspace{20 mm} + \eps u_{sxx} v^0 + \eps v_{sxx} v^0_y, \\
&v^0(0) = v^0_y(0) = 0, \hspace{3 mm} \p_y^k v^0(\infty) = 0 \text{ for } k \ge 1.  
\end{aligned}
\end{align}

\noindent Above, we take $F$ as an abstract forcing term. We also write $\mathcal{L}$ as shown in (\ref{sys.u0.intro}). Define the unknown $q^0 = \frac{v^0}{u_s}$, which satisfies the boundary condition $q^0(0) = 0$. 

We now introduce norms in which we control $[u^0, v^0]$ (recall the definition (\ref{w0})): 
\begin{align} \label{doub.norm}
&[[q^0]] := \| \sqrt{u_s} q^0_{yy} \| + \| u_s q^0_y \| + \| \sqrt{u_{sy}} q^0_{y}\|_{y = 0}, \\ \label{trip.norm}
&[[[v^0]]] := \| u_s v^0_{yy} w_0 \| + \| v^0_{yyyy} w_0 \| + \| \sqrt{u_s} v^0_{yyy} w_0 \|. 
\end{align}

We also now define the $[\cdot]_B$ norms in which we control the solution:
\begin{align} \label{d.norm.B}
[u^0, v^0 ]_B &:= \| u^0 \| + \| u^0_y \| + \| u^0_{yy}w_0 \| + \|u^0_{yyy} w_0 \| + \| q^0_y \| + \| \frac{q^0}{y} \| \\ \n
& + \| \sqrt{u_s} q^0_{yy} w_0 \| + \| q^0_y \|_{y = 0} + \| v^0_{yyy} w_0 \| + \| v^0_{yyyy} w_0 \|.
\end{align}

\noindent We also define the space $B$ via 
\begin{align}
B = \{ [u^0, v^0] \in L^2 \times L^2(\frac{1}{\langle y \rangle}) : u^0 + v^0_y = h(y), \hspace{3 mm} [u^0, v^0]_B < \infty \}
\end{align}

The main result of Section \ref{Section.1} is 
\begin{proposition} \label{prop.L.exist} There exists a unique solution $v^0$ (and thus $u^0$ according to (\ref{rene})) to (\ref{nawa}) such that $[u^0, v^0] \in B$, and the following estimate holds 
\begin{align}
\begin{aligned} \label{mainu0est}
&[ u^0, v^0 ]_B^2 \lesssim  |(F, q^0)| + \| F w_0 \|^2. 
\end{aligned}
\end{align}
\end{proposition}

The first task is to generate inequalities relating the norms (\ref{doub.norm}), (\ref{trip.norm}), and (\ref{d.norm.B}) to various other quantities that will arise in the analysis. 

\begin{lemma} For any $0 < \sigma << 1$ in (\ref{scaling.intro}), 
\begin{align} \label{leh.2}
&\| q^0_y \| \lesssim \sigma^{2/3} \lambda^{-2} [[q^0]], \\ \label{leh.3}
&|q^0| \le \sigma^{2/3} \lambda^{-2}  \langle y \rangle^{1/2} [[q^0]].
\end{align}
\end{lemma}
\begin{proof} Fix a $\delta << 1$ to be selected later. We split at scale $\lambda y = \delta$ via 
\begin{align*}
\| q^0_y \| \le & \| q^0_y \chi(\frac{\lambda y}{\delta}) \| + \| q^0_y \{1 - \chi(\frac{\lambda y}{\delta}) \} \| \\
= & \| q^0_y \chi(\frac{\lambda y}{\delta}) \| + \| \frac{1}{u_s} u_s q^0_y \{1 - \chi(\frac{\lambda y}{\delta}) \} \| \\
\lesssim & \| q^0_y \chi(\frac{\lambda y}{\delta}) \| + \frac{\sigma}{\lambda^2 \delta} \| u_s q^0_y \|.
\end{align*}

\noindent Above, we have used that $u_s \gtrsim \frac{\lambda^2 \delta}{\sigma}$ when $\lambda y \ge \delta$ by (\ref{scaling.intro}).

It thus remains to examine the localized contribution, for which we integrate by parts: 
\begin{align} \n
\| q^0_y \chi(\frac{\lambda y}{\delta}) \|^2 = & ( \p_y \{y \} q^0_y, q^0_y \chi(\frac{\lambda y}{\delta})^2) \\  \label{jay}
= & - ( 2 y q^0_y, q^0_{yy} \chi(\frac{\lambda}{\delta} y)^2) - ( 2y q^0_y, q^0_y \frac{\lambda}{\delta} \chi'(\frac{\lambda}{\delta} y) \chi(\frac{\lambda}{\delta}y)) \\ \n
\lesssim & \frac{\sqrt{\sigma \delta}}{\lambda^2} \| q^0_y \chi(\frac{\lambda y}{\delta}) \| \| \sqrt{u_s} q^0_{yy} \| + \frac{\sigma^2}{\lambda^4 \delta^2} \| u_s q^0_y \|^2 \\  \n
\le & \frac{1}{2} \| q^0_y \chi(\frac{\lambda y}{\delta}) \|^2 + \bigO(1) \Big\{ \frac{\sigma \delta}{\lambda^4} \| \sqrt{u_s} q^0_{yy} \|^2 + \frac{\sigma^2}{ \lambda^4 \delta^2} \| u_s q^0_y \|^2 \Big\}.
\end{align}

\noindent Above, for the first term from (\ref{jay}), we used that in the support of the cut-off $\chi(\frac{\lambda}{ \delta} y )$, one has $y \le \frac{\lambda}{\delta}$, so recalling (\ref{scaling.intro}) we obtain
\begin{align*}
y \chi(\frac{\lambda}{ \delta} y ) \le \sqrt{y} \sqrt{\frac{\delta}{\lambda}} \chi(\frac{\lambda}{ \delta} y ) \le \frac{\sqrt{\sigma}}{\lambda^{\frac{3}{2}}} \sqrt{u_s} \frac{\sqrt{\delta}}{\sqrt{\lambda}}. 
\end{align*}

\noindent For the second term from (\ref{jay}), we have used 
\begin{align*}
y \frac{\lambda}{\delta} \chi'(\frac{\lambda}{\delta}y) \sim 1 \text{ and } u_s^2 \chi'(\frac{\lambda}{\delta}y) \gtrsim \frac{\lambda^4 \delta^2}{\sigma^2} \chi'(\frac{\lambda}{\delta}y).
\end{align*}

In summary, we have 
\begin{align*}
\| q^0_y \| \lesssim o(1) \| q^0_y \| + \frac{\sqrt{\sigma \delta}}{\lambda^2} \| \sqrt{u_s} q^0_{yy} \| + \frac{\sigma}{\lambda^2 \delta} \| u_s q^0_y \|. 
\end{align*}

\noindent We optimize above using $\delta = \sigma^{1/3}$ which gives:
\begin{align} \label{ult.qy}
\| q^0_y \| \lesssim \sigma^{2/3} \lambda^{-2}[[q^0]]. 
\end{align}

To conclude the proof, the $q^0$ bound, (\ref{leh.3}), follows via integration 
\begin{align*}
q^0 = \int_0^y q^0_y \le \sqrt{y} \| q^0_y \| \le \sqrt{y} \sigma^{2/3} \lambda^{-2} [[q^0]]. 
\end{align*}
\end{proof}

\begin{lemma} The following estimates hold
\begin{align} \label{paired}
&\| v^0_{yyy} w_0 \| \lesssim \sigma^{1/3} \lambda^{-1} [[[q^0]]], \\ \label{paired2}
&\| v^0_y \| \lesssim(\sigma^{-1/3} \lambda^{1/2} + 1)[[q^0]] \\ \label{paired3}
& \| v^0_{yy} \| \lesssim( \lambda^{\frac{3}{2}} \sigma^{-\frac{1}{3}} + \lambda \sigma^{-\frac{1}{3}} + 1) [[q^0]]. 
\end{align}
\end{lemma}
\begin{proof}  \textit{Proof of (\ref{paired}):} We again fix scale $\lambda y \ge \delta$ and $\lambda y \le \delta$ by introducing the cutoff $\chi(\frac{\lambda}{\delta}y)$: 
\begin{align*}
\| v^0_{yyy} w_0 \| \le & \| v^0_{yyy} w_0 \chi(\frac{\lambda}{\delta}y) \| + \| v^0_{yyy} w_0 \{1 - \chi(\frac{\lambda}{\delta}y) \} \| \\
= & \| v^0_{yyy } w_0 \chi(\frac{\lambda}{\delta}y) \| + \| \frac{1}{\sqrt{u_s}} \sqrt{u_s} v^0_{yyy} w_0 \{1 - \chi(\frac{\lambda}{\delta}y) \} \| \\
\le & \| v^0_{yyy } w_0 \chi(\frac{\lambda}{\delta}y) \| + \frac{\sqrt{\sigma}}{\lambda \sqrt{\delta}} \|  \sqrt{u_s} v^0_{yyy} w_0 \{1 - \chi(\frac{\lambda}{\delta}y) \} \| \\
\lesssim & \|v^0_{yyy} \chi(\frac{\lambda}{\delta}y) \| + \frac{\sqrt{\sigma}}{\lambda \sqrt{\delta}} \|  \sqrt{u_s} v^0_{yyy} w_0 \{1 - \chi(\frac{\lambda}{\delta}y) \} \|.
\end{align*}

\noindent Above, we have used that $u_s \gtrsim \frac{\lambda^2}{\sigma} \lambda y \gtrsim \frac{\lambda^2}{\sigma} \delta$ on the region where $\lambda y \ge \delta$. We have also used for the localized term that $w_0 \lesssim 1$ on the support of $\chi(\frac{\lambda}{\delta}y)$. 

For the first integral above, we integrate by parts 
\begin{align} \n
&( \p_y \{ y \} |v^0_{yyy}|^2, \chi(\frac{\lambda}{\delta}y)^2) \\ \label{firestone.1}
= & - ( 2 y v^0_{yyy}, v^0_{yyyy} \chi(\frac{\lambda}{\delta}y)^2)  - ( y |v^0_{yyy}|^2, \frac{\lambda}{\delta} \chi'(\frac{\lambda}{\delta}y) \chi(\frac{\lambda}{\delta}y))  \\ \n
\lesssim & \frac{\delta}{\lambda} \| v^0_{yyy}  \chi(\frac{\lambda}{\delta}y) \| \| v^0_{yyyy}  \| + \frac{\sigma}{\lambda^2 \delta} \| \sqrt{u_s} v^0_{yyy}  \|^2.
\end{align}

\noindent In the first term of (\ref{firestone.1}), we have used that $y \le \frac{\delta}{\lambda}$ on the support of the cut-off function. For the second term, we have used that $|y \frac{\lambda}{\delta}| \lesssim 1$ on the support of $\chi'(\frac{\lambda}{\delta}y)$. Moreover, we have used by (\ref{scaling.intro}) that $u_s \gtrsim \frac{\lambda^2 \delta}{\sigma}$ when $\lambda y \ge \delta$. We thus take $\delta = \sqrt{\lambda}$. We now optimize the constant $\frac{\delta}{\lambda} + \frac{\sqrt{\sigma}}{\lambda \sqrt{\delta}}$ with a choice of $\delta = \sigma^{1/3}$

\textit{Proof of (\ref{paired2}):} We have, upon recalling (\ref{leh.3}) and (\ref{scaling.intro}),
\begin{align*}
\|v^0_y\| = & \| \{ u_s q^0 \}_y\| \le \|u_{sy} q^0\| + \|u_s q^0_y\| \\
\le & \| u_{sy} \sqrt{y} \|_\infty \sigma^{2/3} \lambda^{-2} [[q^0]] + [[q^0]] \\
\lesssim &( \frac{\lambda^3}{\sigma} \lambda^{-1/2} \sigma^{2/3} \lambda^{-2}   +1)[[q^0]] \\
\lesssim & (\sigma^{-1/3} \lambda^{1/2} + 1) [[q^0]]. 
\end{align*}

\textit{Proof of (\ref{paired3})} We have, upon recalling (\ref{leh.3}) and (\ref{scaling.intro}),
\begin{align*}
\| v^0_{yy} \| \le & \| u_{syy} q^0 \| + 2 \| u_{sy} q^0_y \| + \| u_s q^0_{yy} \| \\
\lesssim & \| u_{syy} \sqrt{y} \|_\infty \sigma^{2/3} \lambda^{-2}[[q^0]] + \| u_{sy} \|_\infty \sigma^{2/3} \lambda^{-2}[[q^0]] + [[q^0]] \\
\lesssim & \Big(\frac{\lambda^4}{\sigma} \lambda^{-1/2} \sigma^{2/3} \lambda^{-2} + \frac{\lambda^3}{\sigma} \sigma^{2/3} \lambda^{-2} + 1 \Big) [[q^0]] \\
= & ( \lambda^{\frac{3}{2}} \sigma^{-\frac{1}{3}} + \lambda \sigma^{-\frac{1}{3}} + 1 ) [[q^0]]. 
\end{align*}
\end{proof}

We will also need the following embedding results for later use: 
\begin{lemma} \label{sister.lemma} The following inequality is valid
\begin{align} \label{sister}
\eps^{\frac{1}{4}} \| v^0 \|_\infty \le C_{\lambda, \sigma} [u^0, v^0]_B. 
\end{align}
\end{lemma}
\begin{proof} We compute by Sobolev interpolation and Hardy's inequality (as $v^0(0) = 0$),  
\begin{align*}
\| v^0 \|_\infty \le & \| \frac{v^0}{y} \|^{\frac{1}{2}} \| y v^0_y \|^{\frac{1}{2}}  \lesssim  \| v^0_y \|^{\frac{1}{2}} \| y v^0_y \|^{\frac{1}{2}}  \lesssim  \| v^0_y \|^{\frac{1}{2}} \| y^2 v^0_{yy} \|^{\frac{1}{2}} \\
\lesssim & \eps^{-\frac{1}{4}} \| v^0_y \|^{\frac{1}{2}} \| \langle y \rangle \langle Y \rangle v^0_{yy} \|^{\frac{1}{2}}  \lesssim  \eps^{-\frac{1}{4}}[u^0, v^0]_B. 
\end{align*}
\end{proof}

For later use, we shall record the following: 
\begin{corollary} For a constant $C = C_{\lambda, \sigma}$ depending on the parameters $(\lambda, \sigma)$, 
\begin{align} \label{B.dep}
[u^0, v^0]_B \le C_{\lambda, \sigma} \Big( [[q^0]] + [[[q^0]]]  \Big). 
\end{align}
\end{corollary}

\subsection{Estimates for $[[q^0]]$ and $[[[q^0]]]$}

Define the following: 
\begin{align} \label{def.a0}
a_0 := &\frac{3}{2}u_{sy}v_s - \frac{3}{2} u_s v_{sy},  \hspace{3 mm} a_1 := \frac{1}{2} u_{sy}v_{syy} - \frac{1}{2} u_{syyy} v_s. 
\end{align}

\begin{lemma} The following estimates are valid
\begin{align} \label{inspectT1}
&\| u_{syyyy} \langle y \rangle\|_1 \lesssim \lambda^4 \sigma^{-1}, \\ \label{inspectionT2}
&\|a_0\|_\infty + \| a_1 \langle y \rangle \|_1 \lesssim \lambda^4 \sigma^{-1}. 
\end{align}
\end{lemma}
\begin{proof} We decompose the profiles 
\begin{align}
\begin{aligned} \label{expand}
&u_s = u_\parallel + \bar{u}^0_e + \sqrt{\eps}u_e + \sqrt{\eps} u_p, \\
&v_s = v_\parallel + \bar{v}^1_e + \sqrt{\eps} v_p + \sqrt{\eps} v_e.
\end{aligned}
\end{align}

\noindent The chief properties are that $u_e, v_e$ and $u_p$ decay rapidly in their arguments, whereas $v_p$ is bounded. 

Using the decompositions (\ref{expand}), we have 
\begin{align*}
\| u_{syyyy} \langle y \rangle \|_2 \le &\| u_{\parallel yyyy} \langle y \rangle \|_1 + \| \eps^2 \bar{u}^0_{eYYYY} \langle y \rangle \|_1 \\
& + \| \eps^{5/2} u_{eYYYY} \langle y \rangle \|_1 + \| \sqrt{\eps} u_{pyyyy} \langle y \rangle \|_1 \\
\lesssim & \lambda^4 \sigma^{-1} + \sqrt{\eps}.
\end{align*}

\noindent Above, we have used the scaling 
\begin{align*}
\| u_{\parallel yyyy} \langle y \rangle \|_1 = & \| \p_y^4 \{ \frac{\lambda^2}{\sigma} u_p(\sigma x , \lambda y) \} \langle y \rangle \|_1 \\
= \frac{\lambda^6}{\sigma} \lambda^{-2} = \frac{\lambda^4}{\sigma}.
\end{align*}

Recall the definition of $a_1$ in (\ref{def.a0}). Recall further the expansions given in (\ref{expand}). 
\begin{align*}
\| v_s u_{syyy} \langle y \rangle \| =& \| [v_\parallel + \bar{v}^1_e + \sqrt{\eps}v_p + \sqrt{\eps}v_e] \\
& \times[u_{\parallel yyy} + \eps^{3/2}\bar{u}^0_{eYYY} + \eps^2 u_{eYYY} + \sqrt{\eps} u_{pyyy}] \langle y \rangle \|_1 \\
\le &  \| v_{\parallel} u_{\parallel yyy} \langle y \rangle \|_1 + \|  \eps^{3/2} \bar{v}^1_e \bar{u}^0_{eYYY} \langle y \rangle \|_1 + \sqrt{\eps} \\
\lesssim & \lambda^4 \sigma^{-1} + \sqrt{\eps}.
\end{align*}

\noindent Note above that 
\begin{align*}
\| \bar{v}^1_e u_{\parallel yyy} \langle y \rangle \| \le & \|  \bar{v}^1_e u_{\parallel yyy} \langle y \rangle \chi(Y) \|_1 + \| \bar{v}^1_e u_{\parallel yyy} \langle y \rangle \{1 - \chi(Y) \}  \|_1 \\
\lesssim & \sqrt{\eps} \| u_{\parallel yyy} \langle y \rangle^2 \|_1 + \eps^\infty,
\end{align*}

\noindent since $\bar{v}^1_e \lesssim \sqrt{\eps}$ for $Y \lesssim 1$ while $u_{\parallel yy} \lesssim \eps^\infty$ for $Y \gtrsim 1$. 

Next, 
\begin{align*}
\| u_{sy} v_{syy} \langle y \rangle\|_1 =& \| [u_{\parallel y} + \sqrt{\eps} \bar{u}^0_{eY} + \eps u_{eY} + \sqrt{\eps} u_{py}] \\
& \times [v_{\parallel yy} + \eps \bar{v}^1_{eYY} + \sqrt{\eps}v_{pyyy} + \eps^{3/2}v_{eYY}] \langle y \rangle\|_1 \\
\lesssim & \| u_{\parallel y} v_{\parallel yy} \langle y \rangle\|_1  + \| \eps^{3/2} \bar{u}^0_{eY} \bar{v}^1_{eYY} \langle y \rangle \|_1 + \sqrt{\eps} \\
\lesssim & \lambda^4 \sigma^{-1} + \sqrt{\eps}.
\end{align*}

\noindent The above computations account for all of the terms in $a_1$. 

We move now to the pointwise bound of $a_0$, from whose definition we obtain  
\begin{align*}
|a_0| \lesssim & |u_{sy}v_s| + |u_s v_{sy}| \\
\lesssim & |[ u_{\parallel y} + \sqrt{\eps} \bar{u}^0_{eY} +  \eps u_{eY} + \sqrt{\eps}u_{py}]| \times |[ v_{\parallel} + \bar{v}^1_e + \sqrt{\eps} v_p + \sqrt{\eps}v_e ] | \\
\lesssim & |u_{\parallel y}| |v_{\parallel} + \bar{v}^1_e| + \sqrt{\eps} \\
\lesssim &  \lambda^4 \sigma^{-1} + |u_{\parallel y} \bar{v}^1_e| \chi(Y) + |u_{\parallel y} \bar{v}^1_e| \{1 - \chi(Y) \} \\
\lesssim & \lambda^4 \sigma^{-1} +  \sqrt{\eps} + \eps^\infty. 
\end{align*}
\end{proof}

We will use these estimates to prove the following lemma. 

\begin{lemma} \label{Lem.double} Let $v^0$ be a solution to (\ref{nawa}). Let $\sigma << 1$ in (\ref{scaling.intro}). Then the following estimate holds 
\begin{align} \label{estimate.lemma.q0.1}
[[q^0]]^2 \lesssim  |(F, q^0)|. 
\end{align}

\end{lemma}
\begin{proof} We use the expression in (\ref{sys.u0.intro}). First, 
\begin{align} \n
(v^0_{yyyy} - \{ u_s^2 q^0_y \}_y , q^0) =& (u_s q^0_{yy}, q^0_{yy}) + (u_s^2 q^0_y, q^0_y) + (u_{sy} q^0_y, q^0_y)_{y = 0} \\ \label{id1}
& - 2(u_{syy} q^0_y, q^0_y) + \frac{1}{2}(u_{syyyy}q^0, q^0) \\ \n
\gtrsim & [[q^0]]^2. 
\end{align}

\noindent Above, we have used (\ref{scaling.intro}), (\ref{leh.2}) and (\ref{leh.3}) paired with (\ref{inspectT1}) and (\ref{inspectionT2}) to estimate the last two terms by
\begin{align*}
|(\ref{id1}.4)| + |(\ref{id1}.5)| \lesssim \frac{\lambda^4}{\sigma} (\sigma^{\frac{2}{3}} \lambda^{-2}[[q^0]])^2 = \sigma^{\frac{1}{3}}[[q^0]]^2 = o(1) [[q^0]]^2,
\end{align*} 

\noindent upon invoking the assumption that $\sigma << 1$. 

To prove the identity (\ref{id1}) we record 
\begin{align*}
(v^0_{yyyy}, q^0) = &- (v^0_{yyy}, q^0_y) \\
 = & (v^0_{yy}, q^0_{yy}) + (v^0_{yy}, q^0_y)_{y = 0} \\
= & (\p_{yy} \{ u_s q^0 \}, q^0_{yy}) + (2 u_{sy} q^0_y, q^0_y)_{y = 0} \\
= & (u_s q^0_{yy} + 2 u_{sy}q^0_y + u_{syy}q^0, q^0_{yy}) + (2 u_{sy}q^0_y, q^0_y)_{y = 0}\\
= & (u_s q^0_{yy}, q^0_{yy}) - (u_{syy} q^0_y, q^0_y) - (u_{sy} q^0_y, q^0_y)_{y = 0} \\
& - (u_{syy} q^0_y, q^0_y) - (u_{syyy} q^0, q^0_y) + (2u_{sy} q^0_y, q^0_y)_{y = 0} \\
= & (u_s q^0_{yy}, q^0_{yy}) + (u_{sy}q^0_y, q^0_y)_{y = 0} - (2 u_{syy} q^0_y, q^0_y) + \frac{1}{2} (u_{syyyy} q^0, q^0). 
\end{align*}

For the next term from (\ref{sys.u0.intro}), we record the integration by parts identity and estimate due to (\ref{leh.2}), (\ref{leh.3}), and (\ref{inspectionT2})
\begin{align} \label{eqohboy}
|-(\{ v_s v^0_{yyy} - v^0_y v_{syy} \}, q^0)| = & |(a_0 q^0_y, q^0_y) + (a_1 q^0, q^0)| \\ \n
\lesssim & \lambda^4 \sigma^{-1} ( \lambda^{-2} \sigma^{\frac{2}{3}} [[q^0]])^2 \\ \n
= & \sigma^{\frac{1}{3}}[[q^0]]^2 = o(1) [[q^0]]^2, 
\end{align}

\noindent upon invoking the assumption that $\sigma << 1$. 

 To prove the equality in (\ref{eqohboy}), we record the following integrations by parts: 
\begin{align} \n
- (v_s v^0_{yyy}, q^0) = & (v_{sy} v^0_{yy}, q^0) + (v_s v^0_{yy}, q^0_y) \\ \n
= & ( v_{sy} [u_s q^0_{yy} + 2 u_{sy} q^0_y + u_{syy}q^0], q^0) \\ \n
& + (v_s [u_s q^0_{yy} + 2 u_{sy}q^0_y + u_{syy}q^0], q^0_y) \\ \n
= & -((v_{sy}u_s)_yq^0_y, q^0) - (u_s v_{sy} q^0_y, q^0_y) - ((u_{sy}v_{sy})_y q^0, q^0)\\ \n
& + (v_{sy}u_{syy} q^0, q^0)- \frac{1}{2}((u_s v_s)_y q^0_y, q^0_y) + 2 (u_{sy}v_s q^0_y, q^0_y) \\ \n
& - \frac{1}{2} ((u_{syy}v_s)_y q^0, q^0) \\ \n
= & \frac{1}{2} ( (v_{sy}u_s)_{yy} q^0, q^0) - (u_s v_{sy}q^0_y, q^0_y) - ((u_{sy}v_{sy})_y q^0, q^0) \\ \n
& + (v_{sy}u_{syy}q^0, q^0) - \frac{1}{2}((u_s v_s)_y q^0_y, q^0_y) + 2 (u_{sy}v_s q^0_y, q^0_y) \\ \n
& - \frac{1}{2}((u_{syy}v_s)_y q^0, q^0) \\ \label{combine1pnb}
= & \frac{1}{2} (\{ u_s v_{syyy} - v_s u_{syyy} \}q^0, q^0) + \frac{3}{2} ( \{ u_{sy}v_s - v_{sy}u_s \} q^0_y, q^0_y)
\end{align}

\noindent We record the second integration by parts: 
\begin{align} \n
(v_{syy}v^0_y, q^0) = & (v_{syy} u_{sy}q^0, q^0) + (v_{syy} u_s q^0_y, q^0) \\ \label{combine2pnb}
= & (v_{syy} u_{sy} q^0, q^0) - \frac{1}{2} ( \p_y \{ u_s v_{syy} \} q^0, q^0). 
\end{align}

\noindent Combining (\ref{combine1pnb}) and (\ref{combine2pnb}) with the definition of $a_0, a_1$ given in (\ref{def.a0}) proves the equality in (\ref{eqohboy}).

We now treat the final two terms in (\ref{sys.u0.intro}). First, we insert (\ref{leh.3}) to obtain
\begin{align*}
|(\eps u_{sxx} u_s q^0, q^0)| \le & (\eps u^i_{pxx} u_s q^0, q^0) + (\eps^{3/2} u^i_{exx} u_s q^0, q^0) \\
\lesssim & (\| \eps u^i_{pxx} \langle y \rangle \|_1 + \eps^{3/2} \| u^i_{exx} \langle y \rangle \|_1) [[q^0]]^2 \\
\lesssim & \sqrt{\eps}[[q^0]]^2. 
\end{align*}

A similar estimate is available for the final term upon integrating by parts: 
\begin{align*}
(\eps v_{sxx} v^0_y, q^0) =& (\eps v_{sxx} u_{sy}q^0, q^0) + (\eps v_{sxx} u_s q^0_y, q^0) \\
= & (\eps v_{sxx} u_{sy} q^0, q^0) - (\frac{\eps}{2} \p_y \{v_{sxx} u_s \} q^0, q^0). 
\end{align*}

The right-hand side clearly contributes $|(F, q^0)|$. This completes the proof. 
\end{proof}

\begin{lemma} \label{Lem.triple} Let $v^0$ be a solution to (\ref{nawa}). Then the following estimate holds
\begin{align} \label{estimate.triple}
 [[[q^0]]]^2 \lesssim [[q^0]]^2 + \|F w_0 \|^2.  
\end{align}
\end{lemma}
\begin{proof}  We take the inner product of $v^0_{yyyy} w_0^2$ with (\ref{nawa}). Clearly, the $v^0_{yyyy}$ term in $\mathcal{L}$ produces coercivity over $\| v^0_{yyyy} w_0 \|^2$. 

According to (\ref{nawa}), the next term from $\mathcal{L}$ is 
\begin{align*}
-(u_s v^0_{yy}, v^0_{yyyy} w_0^2) = & (u_s w_0^2 v^0_{yyy}, v^0_{yyy}) + (  \{ u_s w_0^2 \}_y v^0_{yy}, v^0_{yyy})  \\
= & (u_s w_0^2 v^0_{yyy}, v^0_{yyy}) + (u_{sy} w_0^2 v^0_{yy}, v^0_{yyy}) + (u_s  \{ w_0^2 \}_y v^0_{yy}, v^0_{yyy}) \\
= & (u_s w_0^2 v^0_{yyy}, v^0_{yyy}) + (u_{sy} w_0^2 v^0_{yy}, v^0_{yyy}) - \frac{1}{2} (u_{sy} \{ w_0^2 \}_y v^0_{yy}, v^0_{yy}) \\
& - \frac{1}{2} (u_s  \{w_0^2 \}_{yy} v^0_{yy}, v^0_{yy}) \\
\gtrsim & \| \sqrt{u_s} v^0_{yyy} w_0 \|^2 - \| u_{sy} w_0 \|_\infty \| v^0_{yy} \| \| v^0_{yyy} w_0 \| \\
& - \| u_{sy}  \{ w_0^2 \}_y \|_\infty \| v^0_{yy} \|^2 - \| u_s  \|_\infty \| \{ w_0^2 \}_{yy} v^0_{yy} \|^2 \\ 
\gtrsim & \| \sqrt{u_s} v^0_{yyy} w_0 \|^2 - \| u_{sy} w_0 \|_\infty \| v^0_{yy} \| \| v^0_{yyy} w_0 \| \\
& - \| u_{sy}  \{ w_0^2 \}_y \|_\infty \| v^0_{yy} \|^2 - o(1) \| v^0_{yy} w_0 \|^2 - [[q^0]]^2. 
\end{align*}

The next term from $\mathcal{L}$ in (\ref{nawa}) is $u_{syy}v^0$, which we combine with $\eps u_{sxx}v^0$ (the sixth term in $\mathcal{L}$ in (\ref{nawa})) to produce $v^0 \Delta_\eps u_s$. We treat this via: 
\begin{align*}
(v^0 \Delta_\eps u_s, v^0_{yyyy} w_0^2) = & (\Delta_\eps u_s u_s q^0, v^0_{yyyy}w_0^2) \\
= & (u_{syy} u_s q^0, v^0_{yyyy} w_0^2) + (\eps u_{sxx} u_s q^0, v^0_{yyyy}w_0^2) \\
\le & \| \Delta_\eps u_s w_0 \sqrt{\langle y \rangle} \|_1 [[q^0]]  \| v^0_{yyyy} w_0 \| \\
\lesssim & [[q^0]] \| v^0_{yyyy} w_0 \|. 
\end{align*}

Next, we integrate by parts, using that $v_{sy}|_{y = 0} = 0$, to obtain 
\begin{align*}
-(v_s v^0_{yyy}, v^0_{yyyy} w_0^2) = & \frac{1}{2}(\p_y \{ v_s w_0^2 \} v^0_{yyy}, v^0_{yyy}) \\
= & - \frac{1}{2} (\p_{yy} \{ v_s w_0^2 \} v^0_{yyy}, v^0_{yy}) - \frac{1}{2}(\p_y \{ v_s w_0^2 \} v^0_{yyyy}, v^0_{yy}) \\
\lesssim & \| \p_{yy} \{ v_s w_0^2 \} \|_\infty \| v^0_{yyy} \| \| v^0_{yy} \| + \| v^0_{yyyy}w_0 \| \| v^0_{yy} \| \|v_{sy}w_0 + v_s w_{0y} \|_\infty \\
\lesssim & o(1) [[[q^0]]]^2 + [[q^0]]^2. 
\end{align*}

Next, we arrive at 
\begin{align*}
|(v_{syy}v^0_y, v^0_{yyyy}w_0^2)| \le & ( \| v_{syy} u_s q^0_y w_0 \| + \| v_{syy} u_{sy} q^0 w_0 \| ) \| v^0_{yyyy} w_0 \| \\
\le & ( \|v_{syy} w_0 \|_\infty \| u_s q^0_y \| + \| q^0 \langle y \rangle^{-\frac{1}{2}} \|_\infty \| v_{syy} u_{sy} w_0 \|) \| v^0_{yyyy} w \| \\
\lesssim & [[q^0]] [[[q^0]]]. 
\end{align*}

Next, we arrive at 
\begin{align*}
|(\eps v_{sxx} v^0_y, v^0_{yyy} w_0^2)| = & |(\eps v_{sxx} \{ u_s q^0_y + u_{sy} q^0 \}, v^0_{yyy} w_0^2 )| \\
\le & ( \sqrt{\eps} \| \sqrt{\eps} w_0 v_{sxx} \|_\infty \| u_s q^0_y \| + \eps^{\frac{1-}{2}}  \| v_{sxx} u_{sy} w_0 y \|_\infty [[q^0]]  ) \| v^0_{yyy} w_0 \|
\end{align*}

The remaining step is to estimate $\| u_s v^0_{yy} w_0\|$ by using (\ref{nawa}). Indeed, upon rearranging, 
\begin{align*}
\| u_s v^0_{yy} w_0 \| \le & \| v^0_{yyyy} w_0 \| + \| F w_0 \| + \| u_{syy} w_0 \langle y \rangle \| \| \frac{v^0}{\langle y \rangle} \| + \| v_s v^0_{yyy} w_0 \| \\
& + \| v_{syy} w_0 \|_\infty \| v^0_y \| + \| \eps u_{sxx} w_0 \langle y \rangle \|_\infty \| \frac{v^0}{\langle y \rangle} \| + \| \eps v_{sxx} w_0 \|_\infty \| v^0_y \| \\
\lesssim & \| v^0_{yyyy} w_0 \| + \| v^0_{yyy} w_0 \| + \| F w_0 \| + [[q^0]] \\
\lesssim & o(1) [[[q^0]]] + \| F w_0 \| + [[q^0]]. 
\end{align*}

To conclude the proof, the right-hand side clearly contributes $|(F, v^0_{yyyy}w_0^2)| \lesssim \| Fw_0 \| \| v^0_{yyyy}w_0\|$.  
\end{proof}

\subsection{Existence and Uniqueness}

We now establish existence and uniqueness for the system (\ref{nawa}). First, consider the operator: 
\begin{align} 
\begin{aligned} \label{first}
&\mathcal{L}_0 v^0 = F, \hspace{3 mm} v^0(0) = v^0_y(0) = \p_y^k v^0(\infty) = 0 \text{ for } k \ge 1, \\
&\mathcal{L}_0 := v^0_{yyyy} - u_\parallel^\infty v^0_{yy}.
\end{aligned}
\end{align}

\begin{lemma} Assume $F \in C^\infty_0$. There exists a unique solution $v^0$ to the problem (\ref{first}). Moreover, $v^0$ is given by the expression $v^0 = C_1 + C_2 e^{-\sqrt{u_\parallel^\infty}y} + u_p[F]$, where $u_p[F]$ is the particular solution defined below. 
\end{lemma}
\begin{proof} The characteristic equation is $r^4 - u_\parallel^\infty r^2 = 0$. The roots thus correspond to the basis solutions $\{v^0_1, v^0_2, v^0_3, v^0_4\} := \{1, y, e^{ry}, e^{-r y} \}$ where $r = \sqrt{u_\parallel^\infty}$. 

\begin{align*}
&\bold{W}(y) = 
\begin{bmatrix}
    1 & y & e^{r y} & e^{-ry} \\
    0 & 1 & r e^{r y} & - r e^{-ry}  \\
    0 & 0 & r^2 e^{r y} & r^2 e^{-ry} \\
    0 & 0 & r^3 e^{ry} & -r^3 e^{-ry}
  \end{bmatrix} \hspace{10 mm} 
\bold{W}_1(y) = 
\begin{bmatrix}
    0 & y & e^{r y} & e^{-ry} \\
    0 & 1 & r e^{r y} & - r e^{-ry}  \\
    0 & 0 & r^2 e^{r y} & r^2 e^{-ry} \\
    F & 0 & r^3 e^{ry} & -r^3 e^{-ry}
\end{bmatrix} \\
&\bold{W}_2(y) = 
\begin{bmatrix}
    1 & 0 & e^{r y} & e^{-ry} \\
    0 & 0 & r e^{r y} & - r e^{-ry}  \\
    0 & 0 & r^2 e^{r y} & r^2 e^{-ry} \\
    0 & F & r^3 e^{ry} & -r^3 e^{-ry}
\end{bmatrix} \hspace{8 mm} 
\bold{W}_3(y) = 
\begin{bmatrix}
    1 & y & 0 & e^{-ry} \\
    0 & 1 & 0 & - r e^{-ry}  \\
    0 & 0 & 0 & r^2 e^{-ry} \\
    0 & 0 & F & -r^3 e^{-ry}
\end{bmatrix} \\
&\bold{W}_4(y) = 
\begin{bmatrix}
    1 & y & e^{r y} & 0 \\
    0 & 1 & r e^{r y} & 0  \\
    0 & 0 & r^2 e^{r y} & 0 \\
    0 & 0 & r^3 e^{ry} & F
\end{bmatrix}
\end{align*}

\noindent Let $W(y) = |\bold{W}|$ and $W_i(y) = |\bold{W}_i|$. Define 
\begin{align*}
c_i[F](y) = - \int_y^\infty \frac{W_i(z)}{W(z)} \ud z
\end{align*}

\noindent As $F$ has compact support, it is clear that $c_i$ and its derivatives decay rapidly at $y = \infty$. The full solution is thus given by $v^0 = C_1 + C_2 e^{-ry} + u_p[F]$, where $u_p[F]$ is the particular solution $u_p[F] := \sum c_i[F] v^0_i$. We achieve the boundary conditions by solving $C_1 + C_2 + u_p[F](0) = 0$ and $C_1 - r C_2 + \p_y u_p[F](0) = 0$. 
\end{proof}

We now quantify the space in which $v^0$ lives. To do so, define 
\begin{align*}
&\| v^0 \|_{T} := \| v^0_{yyyy}  \| + \| v^0_{yyy} \| + \| v^0_{yy}  \| + \| v^0_y  \| + \| \frac{v^0}{y} \|, \\
&\| v^0 \|_{T_s} := \| v^0_{yyyy} e^{sy} \| + \| v^0_{yyy} e^{sy} \| + \| v^0_{yy} e^{sy} \| + \| v^0_y \|, \\
&\|v^0 \|_{\tilde{T}} := \| v^0_{yyy} \| + \| v^0_{yy}  \| + \| v^0_y  \| + \| \frac{v^0}{y} \|
\end{align*}

\begin{lemma} Let $F \in C^\infty_0$. Then $v^0 \in T, T_s$ and the following estimate is valid 
\begin{align*}
\| v^0 \|_{T} \lesssim \| F w_0 \|, \text{ and } \| v^0 \|_{T_s} \lesssim \| F e^{sy} \|,
\end{align*}

\noindent for $0 < s < r$. 

\end{lemma}
\begin{proof} We square and integrate the equation $\| \mathcal{L}_0 v^0\|^2 = \| F \|^2$. It is immediate to see that 
\begin{align*}
\| \mathcal{L}_0 v^0 \|^2 = \| v^0_{yyyy} \|^2 + 2 u_\parallel^\infty \| v^0_{yyy} \|^2 + |u_\parallel^\infty|^2 \| v^0_{yy}\|^2. 
\end{align*}

\noindent Next, one takes inner product with $v^0$ to obtain control over $\|v^0_y\|^2$, whereas on the right hand side one uses Hardy inequality via $|(F, v^0)| \lesssim \| F \langle y \rangle \| \| v^0_y \|$. We may repeat the first step with weights $e^{sy}$, and all integrations by parts are justified since $s < r$. 
\end{proof}

We now remove the compact support assumption on $F$. 

\begin{lemma} Let $F \in L^2(w_0)$. Then there exists a solution $v^0 \in T$ satisfying $\| v^0 \|_T \lesssim \|F w_0 \|$. Let $F \in L^2(e^{sy})$. Then $v^0 \in T_s$ satisfying the estimate $\| v^0 \|_{T_s} \lesssim \| F e^{sy} \|$.
\end{lemma}
\begin{proof} This follows from a straightforward density argument. 
\end{proof}

The final step is to add on the perturbations from $\mathcal{L}$ to $\mathcal{L}_0$. To do so, write $\mathcal{L} = \mathcal{L}_0 + K$, where 
\begin{align*}
K = (u_s - u_s^\infty) v^0_{yy} - u_{syy}v^0 - v_s v^0_{yyy} - v_{syy} v^0_y - \eps u_{sxx}v^0 + \eps v_{sxx} v^0_y
\end{align*}

\begin{lemma} \label{lemma.fin.exist} Let $F \in L^2(w_0)$. Assume the operator $\mathcal{L}$ satisfies the \textit{a-priori} bound $\| \mathcal{L} v^0 \| \gtrsim \| v^0 \|_{T}$. Then there exists a unique solution $v^0 \in T$ which satisfies the bound $\| v^0 \|_T \lesssim \| F w_0 \|$.
\end{lemma}
\begin{proof} We note first that $\mathcal{L}_0^{-1}K$ is a compact operator on $\tilde{T}$. Indeed, letting $v^0 \in \tilde{T}$, we see that $Kv^0 \in L^2 e^{sy}$ for some $0 < s$. Thus, we may apply $\mathcal{L}_0^{-1}$ which brings $\mathcal{L}_0^{-1}Kv^0$ into $T_s$, which is compactly embedded in $\tilde{T}$. We thus apply the Fredholm alternative so that we must rule out nontrivial solutions to the homogeneous problem $\mathcal{L}_0 v^0 = - K v^0$. Since $v^0 \in \tilde{T}$, we bootstrap to conclude $v^0 \in T_s$. We may subsequently apply the assumed \textit{a-priori} bound on $\mathcal{L}$ to conclude that $v^0 = 0$ is the only solution. 
\end{proof}

\begin{proof}[Proof of Proposition \ref{prop.L.exist}]
Estimate (\ref{mainu0est}) is obtained by combining (\ref{B.dep}) with (\ref{estimate.lemma.q0.1}) and (\ref{estimate.triple}). Together with Lemma \ref{lemma.fin.exist} (whose hypotheses are verified by estimate (\ref{mainu0est})), this concludes the proof of the proposition. 
\end{proof}

\section{Formulation of DNS} \label{Section.2}

\subsection{Solvability of DNS}

The main object of study in this section, motivated by (\ref{spec.nl}), will be the following system: 
\begin{align}  
\begin{aligned} \label{eqn.dif.1}
&-\p_x R[q] + \Delta_\eps^2 v + J(v)= F  \\ 
&v_{xxx}|_{x = L} = v_x|_{x = L} = 0 \text{ and } v|_{x = 0} = v_{xx}|_{x = 0} = 0, \\
&v|_{y = 0} = v_y|_{y = 0} = 0. 
\end{aligned}
\end{align} 

\noindent The above $F$ serves as an abstract forcing for this section. 

Recall that $q = \frac{v}{u_s}$ from (\ref{rayleigh.quotient}). Define: 
\begin{align} \label{Ix.defn}
\tilde{u} := u - u^0 = \int_0^x -v_y(x', y) \ud x' := I_x[-v_y]. 
\end{align}

We will record now identities regarding the boundary conditions for $q$: 
\begin{align} \label{BTBC}
\begin{aligned}
&q_x|_{x = L} = -\frac{u_{sx}}{u_s}q|_{x = L}, \hspace{3 mm} q_{xx}|_{x = 0} = -2 \frac{u_{sx}}{u_s} q_x|_{x = 0}.
\end{aligned}
\end{align}

Define our ambient function space via: 
\begin{align*}
H^4_0 := \{v \in H^4: (\ref{eqn.dif.1}) \text{ is satisfied.} \} 
\end{align*}

We want to establish existence for $v$ as a solution to the system (\ref{eqn.dif.1}). We will define now several function spaces which will enable us to state the existence theorem. 
\begin{definition}[Function Spaces] Fix any weight, $w(y) \in C^\infty(\mathbb{R}_+)$. 
\begin{align*}
&\| v \|_{L^2(w)} := \|v \cdot w \|, \| v \|_{\dot{H}^k_\eps(w)} := \| \nabla_\eps^k v \|_{L^2(w)}, \| v \|_{H^k_\eps(w)} := \sup_{0 \le j \le k} \| v \|_{\dot{H}^j_\eps(w)}, \\
&\| v\|_{\dot{H}^4_{\eps, d}(w)} = \|\{v_{yyyy}, \eps v_{xxyy}, \eps^{\frac{3}{2}} v_{xxxy}, \eps^2 v_{xxxx} \} \cdot w\| + \| \sqrt{u_s} \sqrt{\eps} v_{xyyy} \cdot w\|, \\
&\|v \|_{H^4_{\eps, d}(w)} = \| v \|_{\dot{H}^4_{\eps, d}(w)} + \| v \|_{H_\eps^3(w)}.
\end{align*}

We adopt the convention that $\| v \|_{H^0_\eps(w)} := \| \sqrt{\eps} v \cdot w \|$, and that when $w$ is left unspecified, $w = 1$.  The relevant class of test functions is $C^\infty_V := \{ \phi \in C^\infty : \phi(0) = 0$  and  $\p_x \phi = 0$ in a neighborhood of $x = 0$, and are compactly supported in $y \}$. The following spaces are defined: $H^2_\eps(w) := \overline{C^\infty_V}^{\| \cdot \|_{H^2_\eps(w)}}$, and $X_w := \{ v \in H^4_\eps: \| v \|_{X_w} < \infty \}$.

\end{definition}

We now define notation for several operators: 
\begin{align}\n
&J^0(v) :=  \p_{x} ( -[u_s - u_s(\infty)] v_{yy} + u_{syy}v ) + \eps \p_x ( - [u_s - u_s(\infty)] v_{xx} \\ \label{defn.J0} & \hspace{15 mm} + u_{sxx}v ) + v_s ( I_x[-v_{yyy}] - \eps v_{xy}  ) - I_x[-v_y] \Delta_\eps v_s \\ \label{defn.Dv}
&\mathcal{D}_N(v) := \Delta_\eps^2 v - u_s(\infty) \chi(\frac{y}{N}) \Delta_\eps v_{x}, \hspace{3 mm} \mathcal{D}(v) := \mathcal{D}_\infty(v). 
\end{align}

We now prove the following result: 
\begin{proposition} \label{result.v.exist} Let $F \in L^2(w_0)$. Assume the \textit{a-priori estimate}: 
\begin{align} \label{aprior.assume}
\| v \|_{X_1} \lesssim C_\eps \| F \| \text{ for solutions } v \in X_1 \text{ to } (\ref{eqn.v.2}) 
\end{align}
Then there exists a unique solution $v \in X_{1}$  to the problem: 
\begin{align}  \label{eqn.v.2}
\begin{aligned}
&\Delta_\eps^2 v -\p_x R[q] +J(v)= F, \\
&v_x|_{x = L} = v_{xxx}|_{x = L} = 0, \hspace{3 mm} v|_{x = 0} = v_{xx}|_{x = 0} = 0, \\
&v|_{y = 0} = v_y|_{y = 0} = 0, \hspace{3 mm} v|_{y \rightarrow \infty} = 0. 
\end{aligned}
\end{align}

\noindent  Moreover, for any $w$ satisfying $|\p_y^k w| \lesssim |w|$, this $v$ satisfies the following estimates: 
\begin{align} \label{sum.vxxxx.lap}
&\| \{ \eps v_{xxyy}, \eps^{\frac{3}{2}}v_{xxxy}, \eps^2 v_{xxxx} \} w \|^2 - \eps |||q|||_w^2 \lesssim |(F, \eps^2 v_{xxxx}w)|, \\ \label{sum.vxxyy.lap}
&\| \{ \sqrt{\eps}v_{xyyy}, \eps v_{xxyy}, \eps^{\frac{3}{2}} v_{xxxy} \} \sqrt{u_s}w  \|^2 - \eps |||q|||_1^2 \\ \n
& \hspace{10 mm} - |||q |||_{\sqrt{\eps}w}^2 \lesssim |(f, \eps u_s v_{xxyy}w )|. 
\end{align}

\end{proposition}

The first step is to invert the highest-order operator, $\Delta_\eps^2$. In so doing, the first point is the existence of a finite-energy solution:
\begin{lemma} \label{alread}  Given $F \in L^2$, there exists a unique $H^4_\eps$ solution to $\Delta_\eps^2 v = F$ with boundary conditions from (\ref{eqn.v.2}).
\end{lemma}
\begin{proof} Fix $f^m \in C^\infty_C$ such that $\| F - f_m \| \xrightarrow{L^2} 0$. Let $ \tilde{f}^m$ denote the even extension over $x = L$, which satisfies $\tilde{f}^m(0) = \tilde{f}^m(2L) = 0$. We may now expand $\tilde{f}^m$ periodically in a Fourier sine series: $\bar{f}^m = \sum_n \sin( n \frac{\pi}{2L} x)$. Since $\bar{f}$ is even across $x = L$, only the $n$-odd coefficients remain. We now solve the equation $\Delta_\eps^2 \bar{v}^m = \bar{f}^m$ on $\mathbb{H}$. Thus: 
\begin{align*}
\bar{f}^m = \sum_{n \text{ odd}} f^m_n(y) \sin( n \frac{\pi}{2L} x), \hspace{3 mm} \bar{v}^m = \sum_{n \text{ odd}} v^m_n(y) \sin(n \frac{\pi}{2L}x).
\end{align*}

\noindent  We thus obtain the following ODEs: 
\begin{align} \label{Fourier.1}
(v^m_n)'''' + 2\eps n^2 (v^m_n)'' + \eps^2 n^4 (v^m_n) = f^m_n \text{ for } n \neq 0 \text{ and } n \text{ odd. }
\end{align}

\noindent  Note that $f^m_{n = 0} = 0$ since $\bar{f}^m$ is odd. For each fixed $n$, we solve the above ODE using Lax Milgram. Precisely, define the bilinear form: 
\begin{align*}
B_n[v, \phi] := (v'', \phi'') + 2\eps ( n v', n \phi') + \eps^2 ( n^2 v, n^2 \phi): H^2_y \times H^2_y \rightarrow \mathbb{R}.
\end{align*}

First, for $n \neq 0$, $B_n$ is coercive over $H^2_y$ since $B_n[v, v] = |v''| + 2\eps n^2 |v'|^2 + \eps^2 n^2 |v|^2$. Similarly, $B_n$ is bounded on $H^2_y \times H^2_y$. Summing in $n$ yields the estimate $\| \bar{v}^m \|_{H^2_\eps} \lesssim \| \bar{f}^m \|$. 

We now estimate $\| v^m_{xyyy} \|$. Integration by parts in $y$ and appealing to the trace theorem in $\mathbb{R}_+$ produces: 
\begin{align*}
n^2 \|(v^m_{n})_{yyy}\|^2 =& (n^2 v^n_{yy}, v^n_{yyyy}) + (n^{\frac{1}{2}} v^n_{yyy}(0), n^{\frac{3}{2}} v^n_{yy}(0)) \\
\le &  \|n^2 v^n_{yy}\| \|v^n_{yyyy}\| + |n^{\frac{1}{2}} v^n_{yyy}(0)| |n^{\frac{3}{2}}v^n_{yy}(0)| \\
\le &  \|n^2 v^n_{yy}\| \|v^n_{yyyy}\| + \|n v^n_{yyy}\|^{\frac{1}{2}}\|v^n_{yyyy}\|^{\frac{1}{2}} \|n^2 v^n_{yy}\|^{\frac{1}{2}}\|n v^n_{yyy}\|^{\frac{1}{2}}.
\end{align*}

\noindent  Taking summation over $n$ gives and applying Young's inequality for products with exponents $\frac{1}{4} + \frac{1}{(4/3)} = 1$: 
\begin{align*}
\| v^m_{xyyy} \|^2 \lesssim& \| v^m_{xxyy} \| \| v^m_{yyyy} \| + \| v^m_{xyyy}\|^{\frac{1}{2}} \| v^m_{yyyy}\|^{\frac{1}{2}} \| v^m_{xxyy}\|^{\frac{1}{2}} \| v^m_{xyyy}\|^{\frac{1}{2}} \\
\lesssim & \| v^m_{xxyy} \| \| v^m_{yyyy} \| +  \kappa (\| v^m_{xyyy}\|^{\frac{1}{2}})^4 + N_\kappa ( \| v^m_{yyyy}\|^{\frac{1}{2}} \| v^m_{xxyy}\|^{\frac{1}{2}} \| v^m_{xyyy}\|^{\frac{1}{2}})^{\frac{4}{3}}.
\end{align*}

Multiplying by $v^m_{xxxx}$ produces the bound: $\| \eps v^m_{xxyy}, \eps^{\frac{3}{2}}v^m_{xxxy}, \eps^2 v^m_{xxxx} \|^2 \lesssim \| \bar{f}^m \|^2$. We use the equation to estimate $\| v^m_{yyyy} \|$. This then concludes the full $H^4_\eps$ bound. 

That $\bar{v}^m$ is in $C^\infty(\mathbb{H})$ follows by multiplying (\ref{Fourier.1}) by factors of $n^{j}$, summing in $n$, and using that $f^m$ is smooth to ensure summability of the right-hand side $\sum_n n^{2j} \|f^m_n\|^2 < \infty$.

That $v^m(0) = v^m_{xx}(0) = 0$ is guaranteed by the fact that $v^m$ is a Fourier sine series and $v^m_x(L) = v^m_{xxx}(L) = 0$ is guaranteed by the fact that only odd $n$ coefficients are nonzero.

We turn now to the estimate (\ref{sum.vxxxx.lap}). Integrating by parts produces: 
\begin{align} \n
&( \Delta_\eps^2 v^m, \eps^2 v^m_{xxxx} w^2) = \| \{ \eps v^m_{xxyy}, 2\eps^{\frac{3}{2}} v^m_{xxxy}, \eps^2 v^m_{xxxx}\} w \|^2 \\ \n
& -  4 \|\eps v^m_{xxy} \sqrt{ ( |w_{y}|^2 + w w_{yy} )}\|^2  +  \| \eps v^m_{xx} \sqrt{( \p_{yy} \{ w w_{yy} \} + \p_{yy} \{ |w_{y}|^2 \} )} \|^2 \\ \n
& -  2 \| \eps^{\frac{3}{2}} v^m_{xxx} \sqrt{ ( |w_{y}|^2 + w w_{yy} )} \|^2.
\end{align}

On the right-hand side, we have $(f^m, \eps^2 v^m_{xxxx}w^2)$. As $f^m \rightarrow f$ in $L^2$, and $v^m_{xxxx} \rightharpoonup v_{xxxx}$ weakly in $L^2$, we may pass to the limit in the inner product. From here, (\ref{sum.vxxxx.lap}) follows immediately. 

We turn now to (\ref{sum.vxxyy.lap}): 
We integrate by parts the $\Delta_\eps^2$ terms: 
\begin{align*}
( \Delta_\eps^2 v^m, \eps v^m_{xxyy} u_s w^2) = & - ( \eps v^m_{xyy} w^2, \p_x \{ u_s v^m_{yyyy} \}) + 2 \| \eps v^m_{xxyy} \sqrt{u_s} w \|^2 \\
& - ( \eps^3 v^m_{xxx} w^2, \p_x \{ u_s v^m_{xxyy} \}) \\
= & - ( \eps u_s v^m_{xyy}, v^m_{xyyyy} w^2) - ( \eps u_{sx} v^m_{xyy},v^m_{yyyy}w^2) \\
& + 2 \| \eps v^m_{xxyy} \sqrt{u_s} w \|^2 -  (\eps^3 u_{sx} v^m_{xxx}, v^m_{xxyy} w^2) \\
& - ( \eps^3 u_s v^m_{xxx}, v^m_{xxxyy} w^2) \\
= & \| \sqrt{\eps} v^m_{xyyy} w \sqrt{u_s}\|^2 + (\eps v^m_{xyyy}, v^m_{xyy} \p_y \{w^2 u_s \}) \\
& + ( \eps v^m_{yyy}, v^m_{xyyy} w^2 u_{sx}) + (\eps v^m_{yyy},v^m_{xyy} \p_y \{ w^2 u_{sx} \}) \\
& + 2\|\eps v^m_{xxyy} \sqrt{u_s} w\|^2 + (\eps^3 u_{sx} v^m_{xxy}, v^m_{xxxy} w^2) \\
& + ( 2 \eps^3 u_{sx} v^m_{xxy},v^m_{xxx} ww_{y}) + (\eps^3 u_{sxy} v^m_{xxy}, v^m_{xxx} w^2) \\
& + \| \eps^{\frac{3}{2}} \sqrt{u_s} v^m_{xxxy} w\|^2 + ( \eps^3 u_{sy} v^m_{xxx},v^m_{xxxy} w^2)  \\
& + (2\eps^3 u_s v^m_{xxx},v^m_{xxxy} ww_{y}) \\
\gtrsim & | \sqrt{u_s} \{ \sqrt{\eps} v^m_{xyyy}, 2\eps v^m_{xxyy}, \eps^{\frac{3}{2}}v^m_{xxxy} \} w \|^2 -  \eps |||q^m|||_1^2 \\
& - |||q^m |||_{\sqrt{\eps}w}^2
\end{align*}

\noindent  We have used the bound $|w_y| \lesssim |w|$ and Young's inequality for products to perform the above estimate. We again pass to the limit as $m \rightarrow \infty$ in the same manner as for (\ref{sum.vxxxx.lap}).
\end{proof}


\begin{lemma}   Let $F \in L^2(\langle y \rangle^m)$ for some $1 \le m < \infty$. Then $\| v \|_{H^4_\bold{d}(\langle y \rangle^m)} \le C_\eps \| F \|_{L^2(\langle y \rangle^m)}$. In particular, in the case when $F \in L^2 \cap L^2(w_0)$, $v = (\Delta_\eps^2)^{-1}F \in X_{1} \cap Y_{w_0}$.
\end{lemma}
\begin{proof} This follows from standard polynomial-type weighted estimates, and we omit the proof. 
\end{proof}

We will now study the perturbation in two steps. 

\begin{lemma}   The map $\mathcal{D}^{-1}: L^2 \rightarrow H^4$ is well-defined. 
\end{lemma}
\begin{proof}
Consider the map $\mathcal{D}_N(v) = F \in L^2$. By calling $v_0 = \Delta_\eps^2 v$, we may rewrite the equation as $v_0 + \chi_N(y) u_s(\infty) \Delta_\eps \Delta_\eps^{-2} v_0 = F$. We will study this as an equality in $L^2$, and it is clear that $\chi_N(y) u_s(\infty) \Delta_\eps \Delta_\eps^{-2}$ is a compact operator on $L^2$ due to the cutoff function. Therefore, by the Fredholm alternative, to establish solvability of $\mathcal{D}_N$, we must prove uniqueness of the homogeneous solution. This follows by performing an energy estimate:
\begin{align*}
(\Delta_\eps^2 v, v_{xx}) - u_s(\infty)( \chi_N(y) \Delta_\eps v_x, v_{xx}) = (F, v_{xx})
\end{align*}

\noindent  The Bilaplacian term produces the quantities $-\| v_{xyy}, 2\sqrt{\eps} v_{xxy}, \eps v_{xxx} \|^2$.

 Next, assuming $N = \eps^\infty$, we have: 
\begin{align*}
- u_s(\infty)( \chi_N \Delta_\eps v_x, v_{xx}) = &- \frac{u_s(\infty)}{2} [ |v_{xy}(0) \sqrt{\chi_N}|^2 + |\sqrt{\eps}v_{xx}(L) \sqrt{\chi_N}|^2 \\
& + \frac{1}{N}(v_{xy}, v_{xx} \chi_N').
\end{align*}

\noindent  Thus, the operator is coercive over the quantities $\| v_{xyy}, 2\sqrt{\eps} v_{xxy}, \eps v_{xxx} \|^2 +  [ |v_{xy}(0) \sqrt{\chi_N}|^2 + |\sqrt{\eps}v_{xx}(L) \sqrt{\chi_N}|^2$. By Poincare inequalities this implies that $v = 0$ if $F = 0$. Passing to the limit as $N \uparrow \infty$, we find that $\mathcal{D}$ is invertible from $H^4 \rightarrow L^2$. 
\end{proof}

\vspace{1 mm}

\begin{proof}[Proof of Proposition \ref{result.v.exist}] We will now consider the full equation (\ref{eqn.v.2}), which may be written as $\mathcal{D}(v) + J^0(v) = F \in L^2$. Again, standard arguments show that $\chi_N J^0 \circ \mathcal{D}^{-1}$ is a compact operator on $L^2$ or $L^2(e^Y)$. By the Fredholm alternative, it suffices to show uniqueness for the homogeneous solution to (\ref{eqn.v.2}). For this, we apply the assumed \textit{a-priori} estimate, (\ref{aprior.assume}) to conclude. 
\end{proof}

\subsection{Basic Estimates} \label{subsection.careful}

First, we urge the reader to recall the definitions in (\ref{defn.norms.ult}). For the weight, $w$, we will take
\begin{align} \label{weight}
&w = \text{ either } 1 \text{ or } w_0,
\end{align}

\noindent where $w_0$ is defined in (\ref{w0}). For both of these choices, the following elementary inequalities hold: 
\begin{align} \label{weight.prop.1}
|w_y| \lesssim \sqrt{\eps}|w| + 1, \hspace{5 mm} |w_y| \lesssim |w|.
\end{align}

\begin{lemma}[Hardy-type inequalities]   Let $f$ satisfy $f|_{y = 0} = 0$ and $f|_{y \rightarrow \infty} = 0$.  Then: 
\begin{align} \label{hardy.orig.w}
\| \frac{f}{ y } w \| \lesssim \|f_y w \| + \| \sqrt{\eps} f w \|. 
\end{align}
\end{lemma}
\begin{proof} The case of $w = 1$ follows from the standard Hardy inequality. We thus consider $w =  w_0$. We integrate by parts in $y$ in the following manner:  
\begin{align*}
\| \frac{f}{ y } w_0 \|^2 = &  (\p_y \{ y  \} f, f \langle Y \rangle^{2m}) \\
= & 2m(f  y , f \langle Y \rangle^{2m-1} \sqrt{\eps} ) - 2(f  y  , f_y \langle Y \rangle^{2m}) \\
\lesssim & \| f  \langle Y \rangle^m\| \Big[ \| \sqrt{\eps}f  y  \langle Y \rangle^{m-1}\| + \| f_y  y \langle Y \rangle^m\| \Big].
\end{align*}
\end{proof}

\begin{lemma}   Let $v \in H^4_0$, let $q = \frac{v}{u_s}$, and let $w$ satisfy $|w_y| \lesssim |w|$. 
\begin{itemize}

\item[1.] The following Poincare type inequalities hold: 
\begin{align} \label{Poincare}
&\| \p_y^{j} \p_x^{j_2 }v w\| \lesssim L \| \p_y^j \p_y^{j_2 + 1} v w \| \text{ for } j_2 = 0,1,2,3, \\ \label{Poincare.2}
&\| u_s^k \p_{y}^j q w \| \lesssim L \| u_s^k \p_y^j q_x w \| \text{ for } k = 0, 1, \\ \label{Poincare.3}
&\| \sqrt{\eps} q_x w \| \lesssim L \| \sqrt{\eps} q_{xx}w \|. 
\end{align}

\item[2.] The following quantities are controlled by the triple norm: \begin{align} \label{systemat1}
\begin{aligned}
&\| \{q_{yy}, q_{xy}, \sqrt{\eps} q_{xx} \} \cdot w \| + \| \nabla_\eps q \cdot w \|  + \| \frac{q_x}{y} w \|  \\
&\hspace{5 mm}  + \| \{v_{yyy}, v_{xyy}, \sqrt{\eps}v_{xxy}, \eps v_{xxx} \} \cdot w \|  \\
& \hspace{5 mm} + \| \{v_{yy}, v_{xy}, \sqrt{\eps} v_{xx} \} \cdot w \|  + \| \nabla_\eps v \cdot w \|  \\
&\lesssim |||q|||_w. 
\end{aligned} 
\end{align}

\item[3.] The following quantities can be controlled with a pre-factor of $o_L(1)$:
\begin{align} 
\begin{aligned} \label{small.guys}
&\| v_{yy}  \cdot w \| + \| \sqrt{\eps} \{ q_x, v_x\} \cdot w \| \le L |||q|||_{w}, \\
&\| q_{yy} w \| \lesssim \sqrt{L} |||q|||_{w}. 
\end{aligned}
\end{align}

\item[4.] Fix any $\delta > 0$. The following interpolation estimate holds: 
\begin{align} \label{interp.1}
&\| \nabla_\eps q_x \cdot w \|  \le \delta |||q|||_w + N_\delta \| u_s \nabla_\eps q_x \cdot  w \|.
\end{align}

We will often (for the sake of concreteness) apply the above interpolation with the following choices of $\delta$:
\begin{align} \label{interp.L}
\| \nabla_\eps q_x w \| \lesssim L^{\frac{\alpha}{2}} |||q|||_w + L^{-\alpha} \| u_s \nabla_\eps q_x w \|.
\end{align}

\item[5.] The following boundary estimates are valid: 
\begin{align} \label{pff}
&\|q_x w\|_{x = L} + \|\sqrt{\eps} q_x \cdot w\|_{x = 0} + \|v_{yy} \cdot w\|_{x = L}  \\ \n
& \hspace{2 mm} + \|\sqrt{\eps} q_{xx} w\|_{x = 0} + \|\eps \sqrt{u_s} q_{xx} w\|_{x = L} + \|q_{xy}w\|_{x = L} \lesssim \sqrt{L} ||| q |||_w  \\ \label{gp.2.5}
&\|\eps^{\frac{1}{4}} q_{xx} \cdot w\|_{x = 0} + \| \eps^{\frac{1}{4}}\frac{q_x}{\langle y \rangle} w \|_{x = 0} \lesssim |||q|||_1 + \frac{\eps^{\frac{1}{4}}}{L} |||q|||_w  \\ \label{gp.3}
&\|\sqrt{u_s} q_{xyy} \cdot w\|_{x = L} \lesssim  \sqrt{L} |||q|||_{w}.
\end{align}

\end{itemize}

\end{lemma}

\begin{proof} \textit{Step 1: Proof of (\ref{Poincare}) - (\ref{Poincare.3}):} Fix a function $\tilde{u}_s$ that is a function of $y$ only, and such that $C_0 \tilde{u}_s \le u_s \le C_1 \tilde{u}_s$ for all $(x,y) \in \Omega$. For any function $g$ satisfying $g|_{x = 0} = 0$ or $g|_{x = L}$, a Poincare inequality gives: 
\begin{align*}
\|u_s g w \| \lesssim \| \tilde{u}_s g w \| \lesssim L \| \tilde{u}_s g_x w \| \lesssim L \| u_s g_x w \|. 
\end{align*}

\noindent  We will apply the above with $g = \p_y^j q, \p_y^{j_1} \p_y^{j_2} v$ for $j_2 = 0,1,2,3$. We turn now to the following Poincare-type inequality in the $x$-direction: 
\begin{align} 
\begin{aligned} \label{pc.x}
\| \sqrt{\eps} q_x \cdot w \| = & \| \sqrt{\eps} \Big(q_x(L) + \int_L^x q_{xx} \Big) \cdot w \| \\
= & \| \sqrt{\eps} \frac{u_{sx}}{u_s} q|_{L}(y) \cdot w \| + \| \sqrt{\eps} \int_L^x q_{xx} \ud x' \cdot w \| \\ 
= & \| \sqrt{\eps} \frac{u_{sx}}{u_s} \int_0^L q_x \ud x' \cdot w \| + \| \sqrt{\eps} \int_L^x q_{xx} \ud x' \cdot w \| \\
\lesssim & o_L(1) \| \sqrt{\eps} q_{x} \cdot w \| + L \| \sqrt{\eps} q_{xx} \cdot w \|. 
\end{aligned}
\end{align}

\noindent  By absorbing the $\| \sqrt{\eps} q_x \cdot w \|$ to the left-hand side, we obtain the desired estimate.

\textit{Step 2: Proof of (\ref{systemat1}):} We will work systematically through (\ref{systemat1}). Let us start with the $\nabla_\eps^2 q$ terms. For this, let $\xi > 0$ a free parameter, and we will compute the localized quantity: 
\begin{align*}
\| q_{yy} w \chi(\frac{y}{\xi}) \|^2 = & ( \p_y \{ y \}, q_{yy}^2 w^2 \chi(\frac{y}{\xi})^2) \\
= & - ( 2y q_{yy}, q_{yyy} w^2 \chi(\frac{y}{\xi})^2) - ( y q_{yy}^2, 2 w w_{y} \chi(\frac{y}{\xi})^2) \\
& - (y q_{yy}^2, w^2 \frac{1}{\xi} \chi'(\frac{y}{\xi})) \\
\le & L \| \sqrt{u_s} q_{xyy} w \chi(\frac{y}{\xi}) \| \| \sqrt{u_s} q_{yyy} w \| \\
& + L^2\| \sqrt{u_s} q_{xyy} \|^2 \sup_{y \le \xi} |w w_y| + \frac{L^2}{\xi} \| \sqrt{u_s} q_{xyy} w \|^2 \\
\le & (L + \frac{L^2}{\xi}) |||q|||_{w}^2.
\end{align*}

\noindent  We have used (\ref{Poincare.2}). Inserting this below gives:
\begin{align} \n
\| q_{yy} \cdot w \| & \le \| q_{yy} \cdot w [1- \chi(\frac{y}{\xi})] \| + \| q_{yy} \cdot w \chi(\frac{y}{\xi}) \| \\ \n
& \le \frac{L}{\sqrt{\xi}}  \| \sqrt{u_s} q_{xyy} w \| + (\frac{L}{\sqrt{\xi}} + \sqrt{L}) |||q|||_{w}^2 \\ \label{chi}
& \le \sqrt{L} |||q|||_{w}^2 \text{ for } \xi = L. 
\end{align}

A similar bound can be performed for the remaining components of $\nabla_\eps^2 q$. However, we must forego the pre-factor of $o_L(1)$ for these terms. Let $g$ be generic for now. For the far-field component, estimate $\| g \cdot w [1-\chi(\frac{y}{\xi})] \| \le \frac{1}{\xi} \| u_s g w \|$. For the localized component: 
\begin{align*}
\| g \cdot w \chi(\frac{y}{\xi}) \|^2 = & - ( y, \p_y \{ g^2 w^2 \chi(\frac{y}{\xi}) \}) \\
= & - ( 2 y g, g_y w^2 \chi(\frac{y}{\xi}) ) - ( 2 y g^2, w w_y \chi(\frac{y}{\xi}) \\
& - ( y g^2, w^2 \chi'(\frac{y}{\xi}) \xi^{-1}) \\
\lesssim & \sqrt{\xi} \bigO(\sqrt{\text{LHS}})\| u_s g_y w \| + \sup_{y \le \xi} |w w_y| \sqrt{\xi} \| g \|^2 \\
& + \xi^{-1}  \| u_s g w \|^2.  
\end{align*}

\noindent  Accumulating these estimates gives: 
\begin{align} \label{chi2}
\| g w\| \le \xi \| u_s g_y w\|^2 + \xi^{-1} \| u_s g w \|^2 + \sup_{y \le \xi} |w w_y| \sqrt{\xi} \| g \|^2.
\end{align}

\noindent We will apply the above computation to $g = q_{xy}$ and $g = \sqrt{\eps} q_{xx}$ and take $\xi = 1$. Next, applying (\ref{hardy.orig.w}) with $f = q_x$ gives: 
\begin{align} \label{w.Hardy}
\| \frac{q_x}{y} w \| \lesssim  \| q_{xy}w \| + \| \sqrt{\eps} q_x w \|. 
\end{align}

\noindent  Upon using (\ref{Poincare.3}), this concludes all of the $q$ terms from (\ref{systemat1}). 

We now move to $v$ terms from (\ref{systemat1}), for which we expand: 
\begin{align*}
&v_x = u_s q_x + u_{sx} q, \hspace{3 mm} v_y = u_s q_y + u_{sy} q, \\
&v_{xy} = u_{sxy} q + u_{sx} q_y + u_{sy} q_x + u_{s} q_{xy}, \\
&v_{yy} = u_{syy} q + 2u_{sy} q_y + u_s q_{yy}, \\
&v_{xx} = u_{sxx} q + 2 u_{sx} q_x + u_s q_{xx}, \\
&v_{yyy} = u_s q_{yyy} + u_{syyy} q + 3 u_{syy} q_y + 3 u_{sy} q_{yy} \\
&v_{xyy} = u_s q_{xyy} + u_{sxyy} q + u_{syy} q_x + u_{sx} q_{yy} + 2u_{sxy} q_y + 2 u_{sy} q_{xy} \\
&v_{xxy} = u_s q_{xxy} + u_{sxxy} q + u_{sxx} q_y + u_{sy} q_{xx} + 2u_{sxy} q_x + 2u_{sx} q_{xy} \\
&v_{xxx} = u_{sxxx} q + u_s q_{xxx} + 3 u_{sxx} q_x + 3 u_{sx} q_{xx}.
\end{align*}

We turn to the third order terms for $v$, starting with $v_{yyy}$. We have already established the required estimates for $u_s q_{yyy}, q_y, q_{yy}$, and so we must estimate using Hardy's inequality: 
\begin{align*}
\| u_{syyy} q \| \le & \| u^i_{pyyy} q \| + \| \eps^{2} u^i_{eYYY} q \| \\
\le & \| u^i_{pyyy} y \|_\infty \| q \langle y \rangle^{-1} \| + \eps^{\frac{3}{2}} \| u^i_{eYYY} \|_\infty \|\sqrt{\eps} q_x\| \\
\lesssim & \| q_y \| + \eps^{\frac{3}{2}} \| \sqrt{\eps} q_x \|. 
\end{align*}

The same argument is performed for the remaining quantities from $\nabla^3 v$. The quantities in $\nabla^2 v$ and $\nabla v$ follow immediately upon using (\ref{eqn.dif.1}) and Poincare's inequality. This concludes the proof of (\ref{systemat1}).

\textit{Step 3: Proof of (\ref{small.guys}):} The $q_{yy}$ estimate follows from taking $\xi = 1$ in (\ref{chi}). For $v_{yy}$, we use (\ref{Poincare}) and (\ref{systemat1}) which shows that $\| v_{xyy} w \| \lesssim |||q|||_{w}$. Both $q_x$ and $v_x$ follow from (\ref{Poincare}) - (\ref{Poincare.3}).

\textit{Step 4: Proof of (\ref{interp.1}), (\ref{interp.L})}: This follows immediately from (\ref{chi2}) upon selecting $g = q_{xy}$ or $g = \sqrt{\eps} q_{xx}$.

\textit{Step 5: Proof of (\ref{pff})} The estimate for $q_x|_{x = L}$ is obtained by appealing to the boundary condition, (\ref{BTBC}): 
\begin{align*}
\|q_x w\|_{x = L} = & \|\frac{u_{sx}}{u_s} q w\|_{x= L} \le \sqrt{L} \| [\p_x \{ \frac{u_{sx}}{u_s} \} q + \frac{u_{sx}}{u_s} q_x ] w\| \\
\lesssim & \sqrt{L} \| (\p_x \{ \frac{u_{sx}}{u_s} \} + \frac{u_{sx}}{u_s}) \langle y \rangle \|_\infty \| \frac{q_x}{\langle y \rangle} w \|_. 
\end{align*}

For $q_x|_{x = 0}$, we use Fundamental Theorem of Calculus: 
\begin{align*}
\|q_x w\|_{x = 0} = & \| q_x(L,\cdot) w + \int_L^0 q_{xx}w \| \le \| q_x w\|_{x=  L} + \sqrt{L} \| q_{xx} w \|. 
\end{align*}

Next, $|v_{yy} w(L,\cdot) | \le \sqrt{L} \| v_{xyy} w \|$ by using $v|_{x = 0} = 0$. We now move to the $q_{xx}$ terms from (\ref{pff}) for which we recall (\ref{BTBC}). From here, $|\sqrt{\eps} q_{xx} w(0,\cdot)| = 2|\sqrt{\eps} \frac{u_{sx}}{u_s} q_x(0,\cdot)|$. From here, the result follows from the $q_x$ estimate. At $x = L$, we use Fundamental theorem of calculus to write: 
\begin{align*}
\|\eps \sqrt{u_s}q_{xx} \|_{x = L} \le \|\eps \sqrt{u_s}q_{xx} w \|_{x = 0} + \sqrt{L} \| \eps \sqrt{u_s}q_{xxx}w \|. 
\end{align*}

We now compute using (\ref{BTBC}): 
\begin{align} \n
\|q_{xy} \cdot w\|_{x = L} = &| \p_y \{ \frac{u_{sx}}{u_s} q \} \cdot w |_{x = L} \\ \n
\le & \| \p_y \{ \frac{u_{sx}}{u_s}\} q \cdot w \|_{x = L} + \| \frac{u_{sx}}{u_s} q_y \cdot w \|_{x = L}.
\end{align}

The latter term is estimated using $q|_{x = 0} = 0$ so by Fundamental Theorem of Calculus is majorized by $\sqrt{L}\|q_{xy} w \|$. The former term requires a decomposition, upon which we use that $q|_{x = 0} = 0$ and Hardy's inequality for the localized and Prandtl component, and the extra $\sqrt{\eps}$ for the Euler component coupled with the Poincare inequality in (\ref{Poincare.3}) for the $q_x$ term: 
\begin{align*}
&\|\p_y\{ \frac{u_{sx}}{u_s} \} q w \chi \|_{x = L} + \|\p_y \{ \frac{u^0_{px}}{u_s} \} q w [1-\chi] \|_{x = L} + \|\p_y \{ \frac{\sqrt{\eps} u^1_{ex}}{u_s} \} q w [1-\chi]\|_{x = L} \\
&\lesssim \sqrt{L} \| \frac{q_x}{y}\| + \sqrt{L} \| \p_y \{ \frac{u^0_{px}}{u_s} \}[1-\chi] y \|_\infty \|\frac{q_x}{y}\| + \| \p_y \{ u^1_{ex} \} \|_\infty \sqrt{L} \| \sqrt{\eps} q_x \|. 
\end{align*}

This concludes the treatment of (\ref{pff}).

\textit{Step 6: Proof of (\ref{gp.2.5})}

Using (\ref{BTBC}):
\begin{align*}
\|\eps^{\frac{1}{4}} q_{xx}  w \|_{x = 0} = & \|2\eps^{\frac{1}{4}} \frac{u_{sx}}{u_s} q_{x} w [ \chi + (1-\chi) ] \|_{x = 0} \\
\lesssim & \|u_{sx} \langle y \rangle \|_\infty \| \eps^{\frac{1}{4}} \frac{q_x}{\langle y \rangle} w \|_{x = 0}.
\end{align*}

We use the cutoff function $\chi(\frac{x}{10 L})$, which satisfies $|\p_x \chi(\frac{x}{10 L})| \lesssim \frac{1}{L}$, and use the standard Trace inequality to estimate: 
\begin{align*}
&\|\eps^{\frac{1}{4}} \frac{q_x}{\langle y \rangle} w\|_{x = 0} = \|\eps^{\frac{1}{4}} \frac{q_x}{\langle y \rangle} \chi(\frac{x}{L}) w \|_{x = 0} \lesssim \|\frac{q_x}{y} w \|^{\frac{1}{2}} \| \sqrt{\eps} q_{xx} w \|^{\frac{1}{2}} + \frac{\eps^{\frac{1}{4}}}{L} \| \frac{q_x}{\langle y \rangle} w \|.
\end{align*}

To conclude, we apply the Hardy inequality in (\ref{hardy.orig.w}).

\textit{Step 7: Proof of (\ref{gp.3})} Again using (\ref{BTBC}), the fact that $q|_{x = 0} = 0$, and the Fundamental Theorem of Calculus: 
\begin{align*}
\|\sqrt{u_s} \p_{yy} \{ \frac{u_{sx}}{u_s} q \} w\|_{x = L} =& \|\sqrt{u_s} [\frac{u_{sx}}{u_s} q_{yy} + 2 \p_y \{ \frac{u_{sx}}{u_s} \}q_y + (\frac{u_{sx}}{u_s})_{yy} q] w\|_{x = L} \\
\lesssim & \sqrt{L} \| \sqrt{u_s} q_{xyy} w \| + \sqrt{L} \| q_{xy} \| + \sqrt{L} \| (\frac{u_{sx}}{u_s})_{yy} y \|_\infty \|\frac{q}{y}\|.
\end{align*}
\end{proof}

We must now collect some blow-up rates near $y = 0$ of various quantities according to the $H^4_0$ norm. We emphasize that these are \textit{qualitative} estimates (and thus, any $\eps$ dependence on the right-hand side is acceptable): 
\begin{lemma} \label{lemma.ratio}   Let $v \in H^4_0$. Then the following are valid for $j = 0,1,2$ and $k = 0,1,2,3$: 
\begin{align} \label{ratio.estimate}
\sup_{y_0 \le 1}\Big[ \|\nabla^k v\|_{y = y_0} + \|\nabla^{j} q\|_{y = y_0} + \sqrt{y_0} \|\nabla^3 q\|_{y = y_0} \Big]\le C_\eps,  
\end{align}

for some constant $C_\eps < \infty$ that may depend poorly on small $\eps$. 
\end{lemma}
\begin{proof} First, that $\sup_y |\nabla^k v|_{L^2_x} < \infty$, for $k = 0,1,2,3$, follows immediately from $\| v \|_{H^4} < \infty$. We now use the elementary formula $\frac{1}{a+b} = \frac{1}{a} - \frac{b}{a+b}$ to write: 
\begin{align*}
q = \frac{v}{u_s} = \frac{v}{u_{sy}(0) y + [u_s - u_{sy}(0)y]} = \frac{1}{u_{sy}(0)} \frac{v}{y} - v \frac{u_s - u_{sy}(0)y}{y u_{sy}(0) u_s}.
\end{align*}

Using the estimates $u_s \gtrsim y$ as $y \downarrow 0$ and $|u_s - u_{sy}(0)y| \lesssim y^2$ as $y \downarrow 0$, it is easy to see that the second quotient above is bounded and in fact $\mathcal{C}^\infty$. We may thus limit our study to $q_0 := \frac{v}{y}$. We let $k_1 + k_2 = 3$ and differentiate the formula:
\begin{align} \n
q_0(x,y) = \frac{1}{y} \int_0^y v_y(x,y') \ud y'  = \int_0^1 v_y(x,ty) \ud t, 
\end{align}

where we changed variables via $ty = y'$, to obtain: 
\begin{align*}
\sqrt{y}_0 \p_x^{k_1} \p_x^{k_2} q_0(x,y_0) =& \int_0^1 \p_x^{k_1} \p_y^{k_2} v_y(x,ty_0) t^{k_2} \sqrt{y_0} \ud t.\end{align*}

We take $L^2_x$ and use Cauchy-Schwartz in $y$ to majorize: 
\begin{align*}
\sqrt{y_0}\|\p_x^{k_1} \p_y^{k_2} q_0\|_{y = y_0} \le & (\int_0^1 \| \p_x^{k_1} \p_y^{k_2 + 1} v\|_{y = t y_0}^2 y_0 t^{2k_2} \ud t )^{\frac{1}{2}} \\
\le & (\int_0^1 \| \p_x^{k_1} \p_y^{k_2 + 1} v\|_{y = t y_0}^2 y_0 \ud t )^{\frac{1}{2}} \\
 \le & (\int_0^{y_0} \|\nabla^4 v\|^2 )^{\frac{1}{2}}
\end{align*}

This establishes the $\nabla^3 q$ estimate. For $\nabla^2 q$, a similar calculation produces: 
\begin{align*}
\|\p_x^{j_1} \p_y^{j_2} q_0 \|_{y = y_0} = &\|\int_0^1 \p_x^{j_1} \p_y^{j_2 + 1}v(x,ty_0) t^{j_2} \ud t\|_{y = y_0} \\
\le &  \int_0^1 \|\nabla^3 v \|_{y = t y_0} t^{j_2} \ud t \\
 \le &  \Big( \int_0^y \|\nabla^3 v \|_{y = s}^2 \frac{s^{j_2}}{y^{j_2}} \frac{\ud s}{y} \Big)^{\frac{1}{2}} \\\le & \Big( \frac{1}{y} \int_0^y \|\nabla^3 v \|_{y = s}^2 \ud s \Big)^{\frac{1}{2}} \\
 \le & \Big(\frac{1}{y} y \sup_{s \lesssim 1} \| \nabla^3 v \|_{y = s}\Big)^{\frac{1}{2}}. 
\end{align*}

Squaring yields: 
\begin{align*}
|\nabla^2 q_0(y)|_{L^2_x}^2 \le \frac{1}{y} \int_0^y |\nabla^3 v (\cdot, s)|_{L^2_x}^2 \ud s := \frac{1}{y} \int_0^y F(s) \ud s, 
\end{align*}

where $F(s) := |\nabla^3 v (\cdot, s)|_{L^2_x}^2 \in L^1$. Thus, we may conclude by the Lebesgue Differentiation Theorem. 
\end{proof}

\begin{corollary}   Let $v \in H^4_0$. The trace $\nabla^2 q|_{y = 0}$ is well defined as an element of $L^2_x$, and moreover the following continuity is satisfied: $\nabla^2 q(\cdot, y) \xrightarrow{y \downarrow 0} \nabla^2 q(\cdot, 0)$ in $L^2(0,L)$.  
\end{corollary}
\begin{proof} $(\nabla^2 q|_{y = 0})^2$ is realized as the boundary trace of a $W^{1,1}$ function $|\nabla^2 q|^2$. Indeed, this follows from estimating the product $\nabla^2 q \cdot \p_y \nabla^2 q \in L^1$:  
\begin{align*}
\| \nabla^2 q \cdot \p_y \nabla^2 q \|_{1} \le  \| \nabla^2 q\|_{L^2_x L^\infty_y} \| \nabla^3 q \|_{L^2_x L^1_y} < \infty, 
\end{align*}

\noindent  The continuity statement in the lemma is a consequence of the above estimate and the Lebesgue Differentiation Theorem. 
\end{proof}

\begin{corollary}  Let $v \in H^4_0$. Then all quantities appearing in $\| \cdot \|_{X_1}$ are finite. 
\end{corollary}
\begin{proof} All $\nabla^3 q$ terms, upon taking $|\cdot|_{L^2_x}$ scale like $y^{-1/2}$, and so clearly $\| \sqrt{u_s} \nabla^3 q \| < \infty$. The second derivatives, upon taking $|\cdot|_{L^2_x}$ are bounded, and so clearly $\| \nabla^2 q \| < \infty$. The boundary terms are well-defined from the above corollary. 
\end{proof}

\section{\textit{a-priori} Estimates for DNS}

\subsection{Quotient Estimates}

\begin{lemma} Let $v$ be a solution to (\ref{eqn.dif.1}), let $w$ satisfy $|\p_y^k w| \lesssim w$. 
\begin{align} \label{B3.estimate}
\begin{aligned}
& \| \sqrt{u_s} \{ q_{xyy}, \sqrt{\eps} q_{xxy}, \eps q_{xxx} \} \cdot w \|^2 + \|u_s q_{xy} \cdot w\|_{x = 0}^2 + \|q_{xy} \cdot w\|_{y = 0}^2 \\
& + \|\sqrt{\eps} u_s q_{xx} \cdot w\|_{x = L}^2  \\
 \le & o_L(1) \Big[ |||q|||_{w}^2  + |||| v ||||_w^2 \Big] + L^{-\frac{1}{8}} \| \nabla_\eps q_x \cdot u_s w \|^2   + L^{\frac{1}{8}} \| q_{xx} w_y \|^2 \\
& +|(F,q_{xx}w^2)|
\end{aligned}
\end{align}
\end{lemma}
\begin{proof} We will compute $( \text{Equation } (\ref{eqn.dif.1}), q_{xx} w^2)$. 

\subsubsection*{\normalfont \textit{Step 1: Rayleigh Terms}}
\begin{align} \n
(-\p_x R[q], q_{xx}w^2) \gtrsim & \|u_s q_{xy} w\|_{x = 0}^2 + \|u_s \sqrt{\eps} q_{xx}w\|_{x = 0}^2 -L ||| v |||_{w}^2 \\
& + L^{\frac{1}{8}} \| q_{xx} w_{y} \|^2 + L^{-\frac{1}{8}} \| u_s \nabla_\eps q_x w\|^2. 
\end{align}

We first integrate by parts in $y$, distribute the $\p_x$, and then integrate by parts in $x$:
\begin{align} \n
( - \p_{xy} \{ u_s^2 q_{y} \}, q_{xx} w^2) = & ( \p_x \{ u_s^2 q_y \}, q_{xxy} w^2) + ( \p_x \{ u_s^2 q_y \}, q_{xx} 2ww_{y} )\\ \n
= & ( 2 u_s u_{sx} q_y, q_{xxy} w^2) + ( u_s^2 q_{xy}, q_{xxy} w^2) \\ \n
& + ( 4 u_s u_{sx} q_y, q_{xx} ww_{y}) + ( 2u_s^2 q_{xy}, q_{xx} ww_{y})\\ \n
= & - ( 2 u_s u_{sx}, q_{xy}^2 w^2) - ( 2 \p_x \{ u_s u_{sx} \} q_{y}, q_{xy} w^2) \\ \n
& - (  u_s u_{sx}, q_{xy}^2 w^2)  +( 4 u_s u_{sx} q_y, q_{xx} ww_{y}) \\ \n
& + ( 2 u_s^2 q_{xy}, q_{xx} ww_{y}) +2 (u_s u_{sx} q_y, q_{xy} w^2)_{x = L} \\ \label{sbs}
& + \frac{1}{2}\|u_s q_{xy} w\|_{x = L}^2 - \frac{1}{2}\|u_s q_{xy} w\|_{x = 0}^2 .
\end{align}

\noindent  The term (\ref{sbs}.8) is a favorable contribution. The cross terms, (\ref{sbs}.\{4, 5\}), are the most dangerous terms: 
\begin{align} \n
&|(\ref{sbs}.\{4, 5\})| \lesssim  \|u_s q_{xy} \cdot w\|  \| u_s  q_{xx} w_{y} \| \lesssim L^{-\frac{1}{8}} \| u_s q_{xy} w \|^2 + L^{\frac{1}{8}} \| u_s q_{xx} w_y \|^2 , \\ \n
&|(\ref{sbs}.\{1, 2, 3\})| \lesssim \|u_s q_{xy} \cdot w \|^2, \\ \n
&|(\ref{sbs}.6)| + |(\ref{sbs}.7)| \lesssim o_L(1) \|u_s q_{xy} \cdot w \|^2
\end{align}

\noindent  To estimate (\ref{sbs}.2) we have used (\ref{Poincare}) because $q|_{x = 0} = 0$. For the two boundary terms, (\ref{sbs}.\{6,7\}), we have used (\ref{pff}). 

We will move to the next Rayleigh term, which upon expanding reads: 
\begin{align} \label{newr}
-(\eps \p_{xx} \{ u_s^2 q_x \},q_{xx}w^2)  = -\eps (u_s^2 q_{xxx} + 4 u_s u_{sx} q_{xx} + 2[u_s u_{sxx} + u_{sx}^2] q_x, q_{xx}w^2).  
\end{align}

\noindent  We integrate the first term by parts in $x$: 
\begin{align*}
(\ref{newr}.1) = & (\eps u_s u_{sx} q_{xx}, q_{xx}w^2) - \frac{1}{2}\| \sqrt{\eps} u_s q_{xx} w \|_{x = L}^2 + \frac{1}{2}\| \sqrt{\eps} u_s q_{xx} w \|_{x = 0}^2 \\
\lesssim & - \| \sqrt{\eps} u_s q_{xx} w \|_{x = L}^2 + \| \sqrt{\eps}u_s q_{xx} w \|^2, 
\end{align*}

\noindent  where we appeal to (\ref{pff}). The remaining two terms in (\ref{newr}) are also directly majorized by $\| \sqrt{\eps} u_s q_{xx} w \|^2$ upon using (\ref{pff}) and the Fundamental Theorem of Calculus. 

\subsubsection*{\normalfont \textit{Step 2: $\Delta_\eps^2$ Terms}}
\begin{align} \n
(\Delta_\eps^2 v, q_{xx}w^2) \lesssim & - \| \sqrt{u_s} \{ q_{xyy}, \sqrt{\eps} q_{xxy}, \eps q_{xxx} \} w \|^2 - |q_{xy}w(\cdot, 0)|^2\\ \label{bxx.sum.visc}
&  + o_L(1) |||q|||_{w}^2 + \sqrt{L} ||||v||||_w^2 + L^{-\frac{1}{8}} \| u_s \nabla_\eps q_x w \|^2 \\ \n
& - L^2 \| q_{xx} w_y \|^2. 
\end{align}

We now treat the contributions arising from $\Delta_\eps^2 v$, starting with $\p_y^4$ \footnote{Note that all integrations by parts are justified rigorously by Lemma \ref{lemma.ratio} and its corollaries.}:
\begin{align} \n
(v_{yyyy}, q_{xx}w^2) = & - (v_{yyy}, q_{xxy}w^2) - 2(v_{yyy}, q_{xx} ww_y ) \\ \n
= & (v_{xyyy}, q_{xy}w^2) - (v_{yyy}, q_{xy}w^2)_{x=L} \\ \n
& + 2(v_{yy}, q_{xxy} ww_y) + (v_{yy}, q_{xx} (w^2)_{yy}) \\ \n
= & - (v_{xyy}, q_{xyy}w^2) - 2(v_{xyy}, q_{xy}ww_y) - (v_{xyy}, q_{xy}w^2)_{y = 0} \\ \n
& - (v_{yyy}, q_{xy}w^2)_{x = L} - 2(v_{xyy}, q_{xy} ww_y) \\ \label{lib}
& + 2(v_{yy}, q_{xy} ww_y)_{x = L} + (v_{yy}, q_{xx} (w^2)_{yy}).
\end{align}

The main terms are (\ref{lib}.1) and (\ref{lib}.3), so we begin with these. First, an expansion of: 
\begin{align*}
v_{xyy} = u_s q_{xyy} + u_{sxyy} q + u_{syy}q_x + u_{sx} q_{yy} +  2 u_{sxy} q_y + 2 u_{sy} q_{xy},
\end{align*}

\noindent  shows: 
\begin{align*}
(\ref{lib}.1)  = - ( [ u_s q_{xyy} &+ u_{sxyy}q  + u_{syy} q_x + u_{sx} q_{yy} \\
& + 2u_{sxy}q_y + 2u_{sy} q_{xy} ], q_{xyy} w^2). 
\end{align*}

First, (\ref{lib}.1.1) is a favorable contribution to the left-hand side. We estimate immediately using Poincare estimate (\ref{Poincare}), $|(\ref{lib}.1.4)| \lesssim L ||| q|||_{w}^2$. Using the Hardy inequality in (\ref{systemat1}), the fact that $q|_{y = 0} = q_x|_{y = 0} = 0$, and the interpolation inequality (\ref{interp.L}) with appropriate selections of $\alpha$: 
\begin{align*}
|(\ref{lib}.1.3)| \lesssim & \| u_{syy} \langle y \rangle \|_\infty \Big \| \frac{q_x}{y} w \Big\| \| \sqrt{u_s} q_{xyy} w \| \\
\lesssim & \{\| q_{xy} w \| + L \| \sqrt{\eps} q_{xx} w \| \} \| \sqrt{u_s} q_{xyy} w \| \\
\lesssim & L^{\frac{1}{64}} |||q|||_{w}^2 + L^{-\frac{1}{8}} \| u_s \nabla_\eps q_x w \|^2. 
\end{align*}

\noindent Let us explain the computation above, as it is will be used repeatedly. We simply apply (\ref{interp.L}) twice with different choices of $\alpha$:
\begin{align}
\begin{aligned} \label{gold}
\| q_{xy} w \| \hspace{1 mm} |||q|||_w \lesssim & \{L^{\frac{1}{64}} |||q|||_w + L^{-\frac{1}{32}} \| u_s q_{xy} w \| \} |||q|||_w \\
\lesssim & L^{\frac{1}{64}} |||q||_w^2 + L^{-\frac{1}{32}} \{ L^{-\frac{3}{32}} \|u_s q_{xy} w \|^2 + L^{\frac{3}{32}} \| \sqrt{u_s}q_{xyy} w \|^2 \} \\
\lesssim & L^{\frac{1}{64}} |||q|||_w^2 + L^{-\frac{1}{8}} \| u_s q_{xy} w \|^2.
\end{aligned}
\end{align}

For (\ref{lib}.1.2) we may first use Poincare in $x$ as $q|_{x = 0} = 0$ to majorize in the same way as above.  

Integration by parts in $y$ and use of the assumption that $|w_y| \lesssim |w|$ yields:
\begin{align*}
(\ref{lib}.1.5) = & ( 2 u_{sxyy} q_{xy}, q_y w^2) + (2 u_{sxy} q_{xy}, q_{yy} w^2) \\
& + ( 4 u_{sxy} q_{xy}, q_y ww_{y}) + (2 u_{sxy} q_y, q_{xy}w^2)_{y= 0} \\
\lesssim & \| q_{xy}, q_{yy} \cdot w \|^2 + L \| q_{xy} w \|^2  + L \| q_{xy} w \|_{y = 0}^2.
\end{align*}

\noindent  We use above that $q_{yy}$ comes with a factor of $\sqrt{L}$ according to estimate (\ref{small.guys}).

Integrate by parts in $y$:
\begin{align} \n
(\ref{lib}.1.6) = & (q_{xy}^2, u_{syy} w^2) + (q_{xy}^2, u_{sy} 2ww_{y}) +(q_{xy}^2, u_{sy} w^2)_{y = 0} \\  \n
\le & C\| \sqrt{u_s} q_{xy} \cdot w \|^2 + C\| q_{xy} \cdot \sqrt{ww_{y}} \|^2. + |\sqrt{|u_{sy}|}q_{xy} w|_{L^2(y = 0)}^2 \\ \label{duda}
\le & L^{\frac{1}{16}} ||| q |||_{w}^2 + L^{-\frac{1}{8}} \| u_s \nabla_\eps q_x w \|^2 + \|\sqrt{u_{sy}}q_{xy} w\|_{y = 0}^2
\end{align}

\noindent Above, we have used $|w_y| \lesssim |w|$ and the interpolation inequality (\ref{interp.L}). Let us emphasize the $\{y = 0\}$ boundary term from (\ref{lib}.1.6) arises with a pre-factor of $+1$, which is of bad sign. We postpone the estimation of this boundary term until (\ref{bdry.g}). 

We move to (\ref{lib}.3) for which an expansion shows: 
\begin{align*}
(\ref{lib}.3) = & - (\{u_s q_{xyy} + u_{sxyy}q  + u_{syy} q_x + u_{sx} q_{yy} \\
& + 2u_{sxy}q_y + 2u_{sy} q_{xy}\} ,q_{xy}w^2)_{y = 0} \\
\le &- (2 - C_0 L) \| \sqrt{\bar{u}_y} q_{xy}\|_{y = 0}^2, 
\end{align*}

\noindent  for some $C_0 < \infty$, independent of small $L, \eps$. Let us provide some details regarding the above estimate. For (\ref{lib}.3.1), we use (\ref{ratio.estimate}) and the fact that $|u_s| \lesssim y$ near $y = 0$ to conclude that (\ref{lib}.3.1) vanishes. Using that $q|_{y = 0} = 0$ shows that (\ref{lib}.\{2, 3\}) vanishes. Using (\ref{ratio.estimate}) together with $|u_{sx}| \lesssim y$ for $y \sim 0$ shows that (\ref{lib}.4) vanishes. This leaves only (\ref{lib}.3.5) and (\ref{lib}.3.6). The main favorable term is (\ref{lib}.3.6). For this, we have used that: 
\begin{align} \n
u_{sy}|_{y = 0} =& \bar{u}|_{y =  0} + \sum_{i = 1}^n  \sqrt{\eps}^{n+1} u^i_{eY}|_{Y = 0} + \sum_{i = 1}^n \sqrt{\eps}^n u^i_{py}|_{y = 0} \\ \label{usy.exp}
\ge & (1 - C_1 \eps) \bar{u}_y|_{y = 0},
\end{align}

\noindent for some $C_1 < \infty$ independent of $L, \eps$. Note that $\bar{u}_y|_{y = 0}$ is bounded below according to the first line of (\ref{prof.pick}), which ensures that (\ref{lib}.3.6) is, in fact, a favorable contribution. For (\ref{lib}.3.5), we use that $q|_{x = 0} = 0$ to invoke the Poincare inequality: 
\begin{align*}
|(\ref{lib}.3.5)| \le& L \| \frac{u_{sxy}}{\bar{u}_y} \|_{y = 0} \|\sqrt{\bar{u}_y}q_{xy} w \|_{y = 0} \| \sqrt{\bar{u}_y} q_{xy} w \|_{y = 0}.
\end{align*} 

\noindent  This concludes the estimate of (\ref{lib}.3). 

We apply the same calculation as in (\ref{usy.exp}) to conclude: 
\begin{align} \label{bdry.g}
(\ref{lib}.3) + (\ref{duda}.3) & \le - (2- C_0 L) \| \sqrt{\bar{u}_y} q_{xy} \|_{y = 0}^2 + (1+ C_1 \eps) \| \sqrt{\bar{u}_y} q_{xy} \|_{y = 0}^2 \\ \n
& \le - \frac{1}{2}  \| \sqrt{\bar{u}_y} q_{xy} \|_{y = 0}^2.
\end{align}

Using (\ref{BTBC}) and the Fundamental Theorem of Calculus to integrate from $x = 0$ produces the identity: 
\begin{align*}
(\ref{lib}.4) = & (v_{yyy}, \p_y \{ \frac{u_{sx}}{u_s} q \} w^2)_{x = L} \\
= & (v_{xyyy}, \p_y \{ \frac{u_{sx}}{u_s} q \} w^2) + (v_{yyy}, \p_{xy} \{ \frac{u_{sx}}{u_s} q \} w^2) \\
= & - (v_{xyy}, \p_{yy} \{ \frac{u_{sx}}{u_s} q \} w^2) - (v_{xyy}, \p_y\{ \frac{u_{sx}}{u_s} q \} 2ww_y) \\
& - (v_{xyy}, \frac{u_{sx}}{u_s} q_y w^2)_{y = 0} + (v_{yyy}, \p_{xy} \{ \frac{u_{sx}}{u_s} q \} w^2)
\end{align*}

For the first term, we distribute the $\p_{yy}$ and subsequently use (\ref{w.Hardy}), Poincare in $x$, and (\ref{small.guys}) to obtain: 
\begin{align*}
|(\ref{lib}.4.1)| = & |- (v_{xyy}, [\p_{yy}\{ \frac{u_{sx}}{u_s} \}q + 2 \p_y \{ \frac{u_{sx}}{u_s} \} q_y + \frac{u_{sx}}{u_s} q_{yy}] w^2)| \\
\lesssim & \| v_{xyy} w \| \Big[ \| \p_{yy} \{ \frac{u_{sx}}{u_s} \} y \|_\infty \| \frac{q}{y} w \| + L \| \p_y \{ \frac{u_{sx}}{u_s} \} \|_\infty \| q_{xy} w \| \\
& + \|\frac{u_{sx}}{u_s} \|_\infty \| q_{yy} w \| \Big] \\
\lesssim & \| v_{xyy} w \Big[ \| q_y w \|_2 + \| \sqrt{\eps}q w \| + L \| q_{xy} w \| + \| q_{yy} w \| \Big] \\
\lesssim & o_L(1) |||q|||_w^2. 
\end{align*}

For the second term, we again distribute the $\p_y$ and use that $|w_y| \lesssim |w|$: 
\begin{align*}
|(\ref{lib}.4.2)| \le &|(v_{xyy}, \p_y \{ \frac{u_{sx}}{u_s} \} q  2ww_y)| + |(v_{xyy}, \frac{u_{sx}}{u_s} q_y 2ww_y )| \\
\lesssim & L \| \p_y \{ \frac{u_{sx}}{u_s} \} y \|_\infty \| v_{xyy} w \| \frac{q_x}{y} w \| + L \|\frac{u_{sx}}{u_s} \|_\infty \| v_{xyy} w \| \| q_{xy} w \| \\
\lesssim & L \| \p_y \{ \frac{u_{sx}}{u_s} \} y \|_\infty \| v_{xyy} w\|\{ \|q_y w \| \\
& + L\| \sqrt{\eps} q w \|  \} + \|\frac{u_{sx}}{u_s} \|_\infty \| v_{xyy} w \| \| q_{xy} w \| \\
\lesssim & L |||q|||_w^2. 
\end{align*}

For the third term, we expand the expression for $v_{xyy}$ via: 
\begin{align*}
(\ref{lib}.4.3) = - &(u_s q_{xyy} + u_{sx} q_{yy} + 2 u_{sxy} q_{y} + 2 u_{sy} q_{xy} \\
& + u_{sxyy} q + u_{syy}q_x, \frac{u_{sx}}{u_s}q_y w^2)_{y = 0}.
\end{align*}

\noindent  (\ref{lib}.4.3.1) and (\ref{lib}.4.3.2) vanish by combining (\ref{ratio.estimate}) with $|\p_x^j u_s| \lesssim y$ for $y$ small, and (\ref{lib}.4.3.5), (\ref{lib}.4.3.6) vanish by using that $q|_{y  = 0} = q_x|_{y = 0} = 0$. This then leaves: 
\begin{align*}
|(\ref{lib}.4.3.3)| + |(\ref{lib}.4.3.4)| \lesssim L \| q_{xy} w \|_{y = 0}^2 \lesssim L |||q|||_w^2,
\end{align*}

\noindent  where we have used the Poincare inequality, which is available as $q|_{x = 0} = 0$. 

For the fourth term, we use the interpolation inequality, (\ref{interp.L}), and then Young's inequality for products to establish: 
\begin{align*}
|(\ref{lib}.4.4)| \lesssim & \| v_{yyy} w \| \| q_{xy} w \| \lesssim \| v_{yyy} \| (\delta |||q|||_w + N_\delta \| u_s \nabla_\eps q_x w\|^2) \\
\lesssim & o_L(1) |||q|||_w^2 + L^{-\frac{1}{8}} \| u_s \nabla_\eps q_x w\|^2. 
\end{align*}

We now move to (\ref{lib}.6). Again using (\ref{BTBC}) and that $v|_{x = 0} = q|_{x= 0} = 0$:  
\begin{align*}
(\ref{lib}.6) = &- 2(v_{yy}, \p_{y} \{ \frac{u_{sx}}{u_s}q \} ww_y )_{x = L} \\
\lesssim & L \| v_{xyy} w \| \| \p_{xy} \{ \frac{u_{sx}}{u_s}  q \} w \| \lesssim L |||q|||_w^2.
\end{align*}

\noindent  For $(\ref{lib}.\{2, 5\})$ we use $|w_y| \lesssim |w|$ and the interpolation inequality (\ref{interp.L}), whereas for (\ref{lib}.7) we use Poincare in $x$, (\ref{Poincare.3}), and the assumption that $|(w^2)_{yy}| \lesssim |w'|^2$:
\begin{align*}
&|(\ref{lib}.\{2, 5\})| \le L^{\frac{1}{16}} ||| q |||_{w}^2 + L^{-\frac{1}{8}} \| u_s \nabla_\eps q_x w \|^2, \\
&|(\ref{lib}.7)| \lesssim L^2 \| v_{xyy} w \| \| q_{xx}w_y \|. 
\end{align*}

\noindent  This concludes the treatment of $\p_y^4$ contributions.

We now move to contributions from $2\eps \p_{xxyy}$. We first integrate by parts in $y$, second expand the expression for $v_{xxy}$, and third perform a further $y$-integration by parts for the $2 u_{sy} q_{xx}$ contribution. This produces:
\begin{align} \n
(2\eps v_{xxyy}, q_{xx}w^2) = & (- 2\eps v_{xxy}, q_{xxy}w^2) - 4(\eps v_{xxy}, q_{xx} ww_y) \\ \n
= & (-2\eps[u_{sxxy} q + 2 u_{sxy}q_x + u_{sxx}q_y + u_s q_{xxy} + 2 u_{sx}q_{xy} \\ \n
& + 2u_{sy} q_{xx}], q_{xxy}w^2) - 4(\eps v_{xxy}, q_{xx}ww_y) \\ \n
= & -( 2\eps [ u_{sxxy} q + 2 u_{sxy} q_x + u_{sxx} q_y + u_s q_{xxy} \\ \n
&+ 2 u_{sx} q_{xy} ], q_{xxy}w^2) + ( \eps u_{syy}, q_{xx}^2 w^2) \\ \label{am}
&+ ( 2 \eps u_{sy}, q_{xx}^2 ww_{y})  - ( 4 \eps v_{xxy}, q_{xx}ww_{y})
\end{align}

Term (\ref{am}.4) contributes favorably. Terms (\ref{am}.\{1, 2\}) are estimated through the weighted Hardy's inequality (\ref{w.Hardy}), terms (\ref{am}.\{3,5\}) are estimated via Poincare's inequality (\ref{Poincare}) and Cauchy-Schwartz, and terms (\ref{am}.\{6,7,8\}) are estimated through the use of the assumption that $|w_y| \lesssim |w|$:
\begin{align*}
&|(\ref{am}.\{1, 2 \})| \le  \sqrt{\eps} \| \{u_{sxy}, u_{sxxy} \} \langle y \rangle \|_\infty \| \frac{q, q_x}{y} w \| \| \sqrt{u_s} \sqrt{\eps} q_{xxy} w \|, \\
&|(\ref{am}.\{ 3, 5 \})| \le \sqrt{\eps} \| q_{xy} w \| \| \sqrt{u_s} \sqrt{\eps} q_{xxy} w \| \\
&|(\ref{am}.\{ 6,7,8 \})| \lesssim L^{-\frac{1}{10}} \| \sqrt{\eps} q_{xx} w \|^2 + o_L(1) \| \sqrt{\eps} v_{xxy} w \|^2 \\
& \hspace{23 mm} \lesssim L^{-\frac{1}{8}} \|u_s \nabla_\eps q w \|^2 + o_L(1) |||q|||_w^2. 
\end{align*} 

We next get to the contributions from $\eps^2 v_{xxxx}$. We first integrate by parts in $x$, use that $v_{xxx}|_{x = L} = 0$, and then expand $\p_x^3 v$ in terms of $q$ to obtain: 
\begin{align} \n
( \eps^2 v_{xxxx}, q_{xx}w^2) = & - (\eps^2 [ u_{sxxx}q + 3 u_{sxx}q_x + 3 u_{sx}q_{xx} \\ \label{am1}
& + u_s q_{xxx} ], q_{xxx} w^2) - (\eps^2 v_{xxx}, q_{xx}w^2)_{x = 0}.
\end{align}

\noindent   We first estimate the first three terms with the use of (\ref{pff}) - (\ref{gp.2.5}): 
\begin{align*}
|(\ref{am1}.{1, 2, 3})| \le L \| \sqrt{u_s} \eps q_{xxx} \cdot w \|^2 + \sqrt{\eps} L \| \sqrt{\eps} q_{xx} \cdot w \|^2. 
\end{align*}

\noindent   For the boundary term, (\ref{am1}.5), we use the identity (\ref{BTBC}) to simplify and (\ref{pff}) to estimate: 
\begin{align*}
(\ref{am1}.5) = &(2 \eps^2 v_{xxx}, \frac{u_{sx}}{u_s}  q_x w^2)_{x = 0} \\
\lesssim & \| \eps^{\frac{3}{2}} v_{xxx} \cdot w \Big( \frac{u_{sx}}{u_s} \Big) \|_{x = 0} \| \sqrt{\eps} q_x \cdot w \|_{x = 0} \\
\lesssim & \| \eps^{\frac{3}{2}} v_{xxx} \cdot w ( \frac{u_{sx}}{u_s} ) \|_{x = 0} \sqrt{L} |||q|||_w \\
\lesssim & \sqrt{L} ||||v||||_w |||q|||_w.
\end{align*}

\noindent   Note we have invoked the fourth-order norm, $||||v||||_w$, due to the boundary contribution at $\{x = 0\}$. 

\subsubsection*{\normalfont \textit{Step 3: $J(v)$ Terms}}
\begin{align} \label{accept.bru}
|(J, q_{xx}w^2)| \lesssim &  o_L(1) |||q|||_{w}^2 + L^{-\frac{1}{8}} \| u_s \nabla_\eps q w \|^2 + L \| q_{xx} w_y \|^2.
\end{align}

Recalling the definition of $J$ in (\ref{defn.J.conc}), we expand $(J, q_{xx}w^2)$ via: 
\begin{align} \n
( - \eps v_{sx}v_{xy} &- v_s v_{yyy} - \eps v_s v_{yxx} \\ \label{whyso}
&+ \Delta_\eps v_s v_y - v_{sx}I_x[v_{yyy}] + \Delta_\eps v_{sx} I_x[v_y], q_{xx}w^2)
\end{align}

An integration by parts first in $x$ and then in $y$ shows: 
\begin{align*}
(\ref{whyso}.5) = &(v_{sxx} I_x[v_{yyy}], q_x w^2) + (v_{sx}v_{yyy}, q_x w^2) - (v_{sx}I_x[v_{yyy}], q w^2)_{x = L} \\
= & - (v_{sxxy} I_x[v_{yy}], q_x w^2) - (v_{sxx} I_x[v_{yy}], q_{xy} w^2) - (v_{sxx}I_x[v_{yy}], q_x 2ww_y) \\
& + (v_{sxy} v_{yy}, q w^2) + (v_{sx} v_{yy}, q_y w^2) + (v_{sx} v_{yy}, q_x 2ww_y) \\
& + (I_x[v_{yy}], \p_y\{ v_{sx} q w^2 \})_{ x= L} \\
\lesssim & L[ |||q|||_{w}^2 + \| q_{xx} \cdot w_{y} \|^2]
\end{align*}

\noindent  We have used the Hardy inequality (\ref{w.Hardy}) and Poincare in $x$, (\ref{Poincare.3}). 

The estimates for (\ref{whyso}.2) follow along the same lines. Again, integration by parts in $y$ then in $x$ and an appeal to the boundary condition (\ref{BTBC}) produces the identity:  
\begin{align*}
(\ref{whyso}.2) = &- ( v_{sxy} v_{yy}, q_x w^2) - (v_{sx} v_{yy}, q_{xy}w^2)  - ( v_{sx} v_{yy}, q_x 2ww_{y})\\
&  - ( v_{sy} v_{xyy}, q_x w^2)  - (v_s v_{xyy}, q_{xy}w^2) - (v_s v_{xyy}, q_x 2ww_{y}) \\
& + (v_{yy} w^2, \p_y \{ v_s ( \frac{u_{sx}}{u_s} ) q \})_{x = L} + (v_{yy} v_s, \frac{u_{sx}}{u_s} q 2 ww_{y})_{x = L} \\
\lesssim &o_L(1) |||q|||_{w}^2 + L^{-\frac{1}{8}} \|u_s \nabla_\eps q_x \cdot w \|^2 + L^2 \| q_{xx} \cdot w_{y} \|^2.
\end{align*}

\noindent  The above estimate relies on the Hardy type inequality (\ref{w.Hardy}) for (\ref{whyso}.3.\{1,4\}), the interpolation inequality (\ref{interp.1}) for (\ref{whyso}.2.5), and Poincare in $x$ as $v|_{x = 0} = 0$ for the boundary terms (\ref{whyso}.2.\{7,8\}). 

Next, we trivially obtain: 
\begin{align*}
&|(\ref{whyso}.1)| \lesssim \sqrt{\eps} \| v_{xy} w \| \| \sqrt{\eps} q_{xx} w \| \lesssim \sqrt{\eps}|||q|||_w^2 \\
&|(\ref{whyso}.3)| \lesssim \| \sqrt{\eps} v_{xxy} w \| \| \sqrt{\eps} u_s q_{xx} w \| \lesssim L^{-\frac{1}{8}} \| u_s \nabla_\eps q w \|^2 + o_L(1) |||q|||_w^2.
\end{align*}

For (\ref{whyso}.4), we integrate by parts in $x$ and appeal to the boundary condition (\ref{BTBC}) and $v|_{x = 0} = 0$:
\begin{align*}
(\ref{whyso}.4) = & - (\Delta_\eps v_{sx} v_y, q_x w^2) - (\Delta_\eps v_s v_{xy}, q_x w^2) - (\Delta_\eps v_s v_y, \frac{u_{sx}}{u_s}q w^2)_{x= L} \\
\le & L \| \Delta_\eps v_{sx} y \|_\infty \| v_{xy} w \| \| \frac{q_x}{y} w \| + \| \Delta_\eps v_{s} \langle y \rangle \|_\infty \| v_{xy} w \| \| \frac{q_x}{y} w \| \\
&+ L \| \frac{u_{sx}}{u_s} \langle y \rangle \|_\infty \| \Delta_\eps v_s \|_\infty \| v_{xy} w \| \frac{q_x}{y} w \| \\
\lesssim & o_L(1) |||q|||_{w}^2 + L^{-\frac{1}{8}}\| u_s \nabla_\eps q_x w \|^2. 
\end{align*}

\noindent   Above we have used the Hardy type inequality (\ref{w.Hardy}), and the interpolation inequality (\ref{interp.L}) to conclude. 

Finally, for the final term (\ref{whyso}.6) we first split the coefficient via:  
\begin{align*}
(\ref{whyso}.6) = &(\Delta_\eps v^0_{px} I_x[v_y], q_{xx}w^2) + \sum_{i = 1}^n (\sqrt{\eps}^i \Delta_\eps v^i_{px} I_x[v_y], q_{xx} w^2) \\
& + \sum_{i = 1}^n (\sqrt{\eps}^{i+2} \Delta v^i_{ex} I_x[v_y], q_{xx} w^2).
\end{align*}

The higher order contributions are easily estimated using the extra power of $\sqrt{\eps}$ by: 
\begin{align*}
|(\ref{whyso}.6.2)| + |(\ref{whyso}.6.3)| \lesssim L \| v_y w \| \| \sqrt{\eps} q_{xx} w \| \lesssim L |||q|||_w^2. 
\end{align*}

For the leading order Prandtl contribution, we integrate by parts in $x$, use that $I_x|_{x = 0} = 0$, and estimate the resulting quantity using the rapid decay of $v^0_p$:
\begin{align*}
(\ref{whyso}.6.1) = &- (\Delta_\eps v^0_{pxx} I_x[v_y], q_x w^2) - (\Delta_\eps v^0_{px} v_y, q_x w^2) \\
& + (\Delta_\eps v^0_{px} \frac{u_{sx}}{u_s} q, I_x[v_y]w^2)_{x = L} \\
\lesssim & L|||q|||_1^2
\end{align*}

\noindent This concludes the treatment of the $J(v)$ contributions. 

We estimate directly: $|(F, q_{xx}w^2)|$ is placed on the right-hand side of the desired estimate. This concludes the proof.  
\end{proof}

\begin{lemma} Let $v$ be a solution to (\ref{eqn.dif.1}). Let $w$ satisfy $|\p_y^k w| \lesssim |w|$: 
\begin{align} \label{B4.estimate}
& \| q_{yy} \cdot w \|_{y = 0}^2 + \| \sqrt{u_s} \{q_{yyy}, \sqrt{\eps} q_{xyy}, \eps q_{xxy}  \} \cdot w \|^2  \\ \n
 \lesssim & \sqrt{L} |||q|||_{w}^2  + |(F, q_{yy}w^2)|.
\end{align}
\end{lemma}
\begin{proof} We will compute the inner-product $( \text{Equation} (\ref{eqn.dif.1}), q_{yy}w^2)$. 

\subsubsection*{\normalfont \textit{Step 1: Estimate of Rayleigh Terms}}

\begin{align} \label{bstate.1}
(-\p_x R[q], q_{yy}w^2) \lesssim -\|u_s q_{yy}w\|_{x = L}^2 + L |||q|||_{w}^2.
\end{align}

First, we will expand the term:
\begin{align*}
\p_{xy} \{ u_s^2 q_y \} = u_s^2 q_{xyy} + 2 u_{s}u_{sx} q_{yy} + 2 u_s u_{sy} q_{xy} + 2[u_s u_{sxy} + u_{sx}u_{sy}] q_y, 
\end{align*}

\noindent  and upon doing so we will integrate by parts the highest order contribution, that is $(u_s^2 q_{xyy}, q_{yy}w^2)$ in $x$: 
\begin{align*}
( -\p_{xy} \{ u_s^2 q_y \}, q_{yy} w^2) = &- ( 2 u_{sx} u_{sy} q_y, q_{yy} w^2) - (2 u_s u_{sxy} q_y, q_{yy}w^2) \\
& - ( 2 u_s u_{sx} q_{yy} w^2,q_{yy}) - ( 2 u_s u_{sy} q_{xy}, q_{yy}w^2) \\
& + ( u_s u_{sx} q_{yy} w^2, q_{yy}) -(\frac{u_s^2}{2} q_{yy} w^2,q_{yy})_{x = L} \\
\lesssim &L^2 ||| q |||_{w}^2 - \|u_s q_{yy} w\|_{x = L}^2.
\end{align*}

Second, we expand the term: 
\begin{align*}
\p_{xx} \{ u_s^2 q_x \} = u_s^2 q_{xxx} + 4 u_s u_{sx} q_{xx} + [2u_s u_{sxx} + 2 u_{sx}^2] q_x. 
\end{align*}

\noindent  We subsequently use Poincare inequalities, (\ref{Poincare}):  
\begin{align*}
( - \eps \p_{xx} \{ u_s^2 q_x \}, q_{yy} w^2)  =& - ( \eps [ (u_s^2)_{xx} q_x + 2 \p_x \{u_s^2\} q_{xx} + u_s^2 q_{xxx} ], q_{yy}w^2) \\
\lesssim & L ||| q |||_{w}^2.
\end{align*}

\subsubsection*{\normalfont \textit{Step 2: Estimate of $\Delta_\eps^2$ Terms}}

\begin{align}
(\Delta_\eps^2 v, q_{yy}w^2) \lesssim - \| \sqrt{u_s}\{ q_{yyy}, \sqrt{\eps}q_{xyy}, \eps q_{xxy} \} w \|^2 - \|q_{yy}w\|_{y = 0}^2 + L |||q|||_{w}^2 
\end{align}

We begin with $\p_{yyyy}$. First, we integrate by parts in $y$, and then we expand the term $v_{yyy}$: 
\begin{align}  \n
( v_{yyyy}, q_{yy}w^2) = &- (v_{yyy}, q_{yyy} w^2) - 2(v_{yyy}, q_{yy}ww_y) - (v_{yyy}, q_{yy} w^2)_{y = 0} \\ \n
= & - ([u_s q_{yyy} +  3 u_{syy}q_y + u_{syyy}q + 3u_{sy} q_{yy}], q_{yyy}w^2) \\ \label{sbp}
& -  2(v_{yyy}, q_{yy}ww_y) - (v_{yyy}, q_{yy} w^2)_{y = 0}. 
\end{align}

We first handle the important boundary contribution from above. We integrate by parts, expand the boundary term to obtain: 
\begin{align*}
(\ref{sbp}.\{4 +6\}) = &( \frac{3}{2}u_{syy} q_{yy} w^2, q_{yy}) + (\frac{3}{2} u_{sy} q_{yy} w^2,q_{yy})_{y = 0} \\
& - (3 u_{sy} q_{yy} w^2,q_{yy})_{y = 0} + ( 3 q_{yy}^2, u_{sy} ww_{y}) \\
\le& - \frac{3}{2} \| q_{yy} \cdot w \|_{y = 0}^2 + \sqrt{L} ||| q |||_{w}^2.
\end{align*}

\noindent  Above, we have used $|w_y| \lesssim |w|$, and the estimate (\ref{small.guys}) to estimate the $q_{yy}$ term. We emphasize the importance of the precise prefactors of $- 3$ and $+\frac{3}{2}$ in the above boundary terms, which enable us to generate the required positivity.

The first term, (\ref{sbp}.1) is a favorable contribution which contributes $\| \sqrt{u_s} q_{yyy} w \|^2$. The third term is controlled by the Hardy-type inequality, (\ref{w.Hardy}):
\begin{align*}
|(\ref{sbp}.3)| \le \| u_{syy} \langle y \rangle \|_\infty \Big\| \frac{q}{y} w \Big\| \| \sqrt{u_s} q_{yyy} w \| \lesssim L |||q|||_w^2. 
\end{align*}

The second term, (\ref{sbp}.2), is controlled via an integration by parts in $y$, Poincare in $x$, (\ref{Poincare}) which is available since $q|_{x = 0} = 0$, and finally (\ref{small.guys}) to estimate the $q_{yy}$ contribution:
\begin{align*}
(\ref{sbp}.2) = & ( 3 u_{syyy} q_y + 3 u_{syy} q_{yy} , q_{yy} w^2) + 6(u_{syy} q_y, q_{yy}ww_y) \\
& + 3(u_{syy} q_y, q_{yy} w^2)_{y = 0} \\
\lesssim & \| q_{yy} w \|^2 + L \|q_{xy} \|_{y = 0} \|q_{yy}\|_{y = 0} \lesssim L |||q|||_w^2. 
\end{align*}

\noindent   This concludes the contributions of $\p_y^4$. 

We next move to $2\eps \p_{xxyy}$. We integrate by parts the following term upon using that $v_x|_{x= L} = 0$ and $q|_{x= 0}  = 0$:
\begin{align} \n 
( 2 \eps v_{xxyy}, q_{yy}w^2) = &- ( 2 \eps v_{xyy}, q_{xyy}w^2) \\ \n
= & - ( 2 \eps \p_{xyy} \{ u_s q \}, q_{xyy}w^2) \\ \n
= & - ( 2 \eps \Big[ u_{sxyy} q + 2u_{sxy} q_y + u_{sx} q_{yy} \\ \label{Nata}
& + u_{syy}q_{x} +  2u_{sy}q_{xy} +  u_{s}q_{xyy} ],  q_{xyy}w^2).
\end{align}

While (\ref{Nata}.6) is a favorable contribution, straightforward computations using the Poincare inequalities, (\ref{Poincare}), show that 
\begin{align*}
|(\ref{Nata}.1)| + |(\ref{Nata}.3)| + |(\ref{Nata}.4)| \lesssim L \sqrt{\eps} |||q|||_w^2. 
\end{align*}

We must treat (\ref{Nata}.\{2, 5\}) via integration by parts in $y$ because their coefficients do not vanish as $y \downarrow 0$. For (\ref{Nata}.2), integrate by parts in $y$, and use $|w_y| \lesssim |w|$ to obtain:
\begin{align*}
 (\ref{Nata}.2) = &( 4 \eps q_{xy}, \p_y \{ u_{sxy} q_y w^2 \}) + (4 \eps q_{xy}, u_{sxy} q_y w^2)_{y = 0} \\
  = & ( 4 \eps u_{sxyy} q_y, q_y w^2) + (4 \eps q_{xy}, u_{sxy} q_{yy} w^2) \\
  & + ( 8\eps q_{xy}, u_{sxy} q_y ww_{y})  + (4\eps q_{xy}, u_{sxy} q_y w^2)_{y = 0} \\
  \lesssim & \eps ||| v |||_{w}^2 + \eps L\| q_{xy}w \|^2 + L\eps \|q_{xy}\|_{y = 0}^2.
\end{align*}

For (\ref{Nata}.5), integration by parts in $y$ produces the expression 
\begin{align*}
(\ref{Nata}.5) =& (2 \eps q_{xy} \p_{y} \{ u_{sy} w^2 \}, q_{xy}) + (2 \eps q_{xy}, q_{xy} u_{sy}w^2)_{y = 0} \lesssim \eps |||q|||_w^2. 
\end{align*}

We next move to $\eps^2 v_{xxxx}$. For this, we integrate by parts twice in $x$, use the boundary conditions $v_{xxx}|_{x = L} = 0$ and $q_{yy}|_{x = 0} = 0$ from (\ref{BTBC}), subsequently integrate by parts in $y$, and finally expand the term $v_{xxy}$. We show this below: 
\begin{align} \n
( \eps^2 v_{xxxx}, q_{yy} w^2) = & (\eps^2 v_{xx}, q_{xxyy}w^2) - (\eps^2 v_{xx}, q_{xyy})_{x = L} \\ \n
= &  - (\eps^2 v_{xxy}, q_{xxy}w^2) - 2(\eps^2 v_{xx}, q_{xxy} ww_y ) - (\eps^2 v_{xx}, q_{xyy})_{x = L} \\ \n
= & ( - \eps^2 [ u_{sxxy} q + 2 u_{sxy} q_x + u_{sy}q_{xx} + u_s q_{xxy} \\ \n
& + u_{sxx}q_y + 2 u_{sx} q_{xy} ], q_{xxy}w^2) - ( \eps^2 v_{xx}, q_{xxy} 2 ww_{y}) \\ \label{arjit}
& + ( \eps^2 v_{xx}, \p_{yy} \{ \frac{u_{sx}}{u_s} q \} w^2)_{x = L}.
\end{align}

The term (\ref{arjit}.4) is favorable. The terms with coefficients that vanish as $y \downarrow 0$ are (\ref{arjit}.5) and (\ref{arjit}.6), and so these may be estimated directly via 
\begin{align*}
|(\ref{arjit}.5)| + |(\ref{arjit}.6)| \lesssim L \eps^{\frac{3}{2}}  ||| v |||_{w}^2. 
\end{align*}

The remaining interior terms require integration by parts in $y$: 
\begin{align*}
(\ref{arjit}.1) = &( \eps^2 q_{xx}, u_{sxxyy} q w^2) + ( \eps^2 q_{xx}, u_{sxxy} q_y w^2) \\
& + ( \eps^2 q_{xx}, u_{sxxy} q 2 ww_{y}) \\
\lesssim & \eps^{\frac{3}{2}} ||| v |||_{w}^2. 
\end{align*}

Above, we have used the weighted Hardy inequality from (\ref{w.Hardy}). Next, in a similar fashion:  
\begin{align*}
(\ref{arjit}.2) =& ( 2\eps^2 q_{xx}, u_{sxyy} q_x w^2) + (2 \eps^2 q_{xx}, u_{sxy} q_{xy}w^2) \\
& + ( 4\eps^2 q_{xx}, u_{sxy} q_x ww_{y}) \\
\lesssim & \eps^{\frac{3}{2}} ||| q |||_{w}^2.
\end{align*}

We have used (\ref{w.Hardy}) for the $q_x$ term appearing in (\ref{arjit}.2.1), (\ref{arjit}.2.3). The term (\ref{arjit}.3) can be handled analogously using that $q_{xx}|_{y = 0} = 0$:
\begin{align*}
(\ref{arjit}.3) = & \frac{\eps^2}{2}(u_{syy} q_{xx}, q_{xx}w^2) + (\eps^2u_{sy} q_{xx}, q_{xx} ww_y) \lesssim \eps \| \sqrt{\eps} q_{xx} w \|^2. 
\end{align*}

Next, we use that $|w_y| \lesssim |w|$ and split the term (\ref{arjit}.7) into $y \le 1$ and $y \ge 1$. On the $y \le 1$ piece, we use $|w| \lesssim 1$, whereas in the far-field piece we use that $|w_y| \lesssim |w|$: 
\begin{align*}
|(\ref{arjit}.7)| \lesssim & |(\eps^2 v_{xx}, q_{xxy} ww_y \chi(y) )| + |(\eps^2 v_{xx}, q_{xxy} ww_y [1-\chi(y)] )| \\
\lesssim &  \eps \| \sqrt{\eps} \frac{v_{xx}}{y} \| \| \sqrt{u_s} \sqrt{\eps} q_{xxy} \| + \eps \| \sqrt{\eps} v_{xx}  w \| \| \sqrt{u_s} \sqrt{\eps} q_{xxy} w \|.
\end{align*}

For the boundary term we distribute the $\p_{yy}$ and estimate using the Fundamental Theorem of Calculus since both $v_{xx}|_{x = 0} = q|_{x = 0} = 0$: 
\begin{align*}
|(\ref{arjit}.8)| = &(\eps^2 v_{xx}, [\p_{yy} \{ \frac{u_{sx}}{u_s} \} q + 2 \p_y \{ \frac{u_{sx}}{u_s} \} q_y + \frac{u_{sx}}{u_s} q_{yy}] w^2)_{x = L} \\
= & (\eps^2 v_{xxx}, [\p_{yy} \{ \frac{u_{sx}}{u_s} \} q + 2 \p_y \{ \frac{u_{sx}}{u_s} \} q_y + \frac{u_{sx}}{u_s} q_{yy}] w^2) \\
& + (\eps^2 v_{xx}, \p_x [\p_{yy} \{ \frac{u_{sx}}{u_s} \} q + 2 \p_y \{ \frac{u_{sx}}{u_s} \} q_y + \frac{u_{sx}}{u_s} q_{yy}] w^2) \\
\lesssim &\sqrt{\eps} \| \frac{1}{\sqrt{u_s}} \eps v_{xxx} w \| \| \sqrt{\eps} \sqrt{u_s} \{ q + q_y + q_{yy} + q_x + q_{xy} + q_{xyy} \} w \| \\
\lesssim & \sqrt{\eps} |||q|||_w^2.
\end{align*}

\noindent  Above, we have expanded: 
\begin{align*}
\| \frac{1}{\sqrt{u_s}} \eps v_{xxx} w \| =& \| \frac{1}{\sqrt{u_s}} \eps\p_{xxx} \{u_s q \} w \| \\
\le & \| \frac{1}{\sqrt{u_s}} \eps \{ u_{sxxx}q + 3 u_{sxx} q_x + 3 u_{sx} q_{xx} + u_s q_{xxx} \} \| \\
\le & |||q|||_w, 
\end{align*}

\noindent where we have used that $|\p_x^j u_s| \lesssim y$ near $\{y = 0\}$.

\subsubsection*{\normalfont \textit{Step 3: $J(v)$ Terms}}

\begin{align} \n
|(J, q_{yy}w^2)|  \lesssim L |||q|||_{w,1}^2.
\end{align}

Recalling (\ref{defn.J.conc}), we expand and estimate immediately
\begin{align*} 
&( - v_s v_{yyy} - \eps v_{sx}v_{xy}  - \eps v_s v_{xxy} \\
&+ \Delta_\eps v_s v_y - v_{sx}I_x[v_{yyy}] + I_x[v_y] ,q_{yy}w^2)\lesssim  L |||q|||_{w,1}^2. 
\end{align*}

The forcing term clearly contributes $|(F, q_{yy}w^2)|$ to the right-hand side, which concludes the proof. 
\end{proof}

\begin{lemma} \label{Lemma.bx.estimate} Let $v$ be a solution to (\ref{eqn.dif.1}). Let $w$ satisfy $|\p_y^k w| \lesssim |w|$, $|(w^2)_{yy}| \lesssim |w_y|^2$ and $|w_y| \lesssim |w|$. Then:
\begin{align} \label{EST.2nd.order}
 \| \nabla_\eps q_x \cdot u_s w \|^2 & \lesssim L^{\frac{1}{2}} \{ ||| q |||_{w}^2 +  |||| q ||||_w^2 \} + L \| q_{xx} w_y \|^2 + |(F, q_x w^2)|. 
\end{align}

\end{lemma}

\begin{proof} We compute the following inner product: $( \text{Equation } (\ref{eqn.dif.1}), q_x w(y)^2)$. 

\subsubsection*{\normalfont \textit{Step 1: Rayleigh Terms Estimates}}

\begin{align}
(- \p_x R[q], q_x w^2) \gtrsim \|u_s q_{xy}w\|^2 - L \| q_{xx} w' \|^2 - L |||q|||_{w}^2.
\end{align}

First, integrate by parts in $y$ and expand via the product rule: 
\begin{align} \n
( -\p_{xy} \{ u_s^2 q_y \}, q_x w^2) = &( \p_x \{ u_s^2 q_{y} \}, q_{xy} w^2) + ( \p_x \{u_s^2 q_y \}, q_x 2 w w_y) \\ \n
= & \|u_s q_{xy} w\|^2 + ( 2 u_s u_{sx}q_y, q_{xy}w^2) + (4u_s u_{sx} q_y, q_x ww_y) \\
& + 2(u_s^2 q_{xy}, q_x ww_y).
\end{align}

The second and third terms are majorized by $L \| u_s q_{xy}w \|^2 + L^2 \| u_s q_{xy} w \| \| q_x w_y \|$ upon using Poincare in $x$ as in (\ref{Poincare}). For the fourth, integrate by parts in $y$ to produce: 
\begin{align*}
- (q_x , q_x [2 u_s u_{sy} ww_y + u_s^2 (ww_y)_y ]) \le L \| u_{sy} y \|_\infty \|\frac{q_x}{y} w\| \| q_{xx} w' \| + L^2 \| q_{xx} w' \|^2
\end{align*}

\noindent In the above estimate, we have used Poincare in $x$, (\ref{Poincare}), Hardy in $y$, and the estimate that $|(w^2)_{yy}| \lesssim |w_y|^2$. 

The second Rayleigh contribution is as follows, upon integrating by parts in $x$ and then expanding:
\begin{align} \n
( - \eps \p_{xx} \{ u_s^2 q_{x} \}, q_x w^2) = &( \eps \p_x \{ u_s^2 q_x \}, q_{xx} w^2) - (\eps q_x, \p_x \{ u_s^2 q_x \} w^2)|_{x = 0}^{x = L} \\ \n
= & \| \sqrt{\eps} u_s q_{xx} w\|^2 + ( 2\eps u_s u_{sx} q_x, q_{xx} w^2) \\ \n
& - (\eps q_x, \p_x\{ u_s^2 q_x \} w^2)|_{x = 0}^{x= L} \\ \n
= & \| \sqrt{\eps} u_s q_{xx} w \|^2 + ( 2 \eps u_s u_{sx} q_x, q_{xx} w^2) \\ \label{ray.x}
& - (\eps u_s^2 q_x, q_{xx} w^2)|_{x = 0}^{x = L} - 2|\sqrt{\eps u_s u_{sx}} q_x w|_{x= 0}^{x = L}|^2 \\ \n 
\gtrsim & \| \sqrt{\eps} u_s q_{xx} w \|^2 -  L ||| v |||_{w}^2,
\end{align}

where we have used (\ref{Poincare}) - (\ref{Poincare.3}). The boundary terms follow from (\ref{pff}). 

\subsubsection*{\normalfont \textit{Step 2: Estimate for $\Delta_\eps^2$ terms}}

\begin{align} \label{sum.bx.lap}
|(\Delta_\eps^2 v, q_x w^2)| \lesssim \sqrt{L} [|||q|||_{w}^2 + ||||v||||_w^2] + L \| q_{xx} w_y \|^2.
\end{align}

We begin with the $\p_y^4$ contribution. A series of integration by parts in $y$ gives:
\begin{align} \n
( v_{yyyy}, q_x w^2) = & ( v_{yy}, q_{xyy} w^2) + ( v_{yy}, q_{xy} 2ww_{y}) \\ \n
& + ( v_{yy}, q_{xy} 2ww_{y}) - (v_y, q_{xy} 2 (ww_{y})_y) \\ \n
& - ( v_y, q_x 2 (ww_{y})_{yy}) +(v_{yy},q_{xy} w^2)_{y = 0} \\ \label{seeg}
& - (v_{yyy}, q_x w^2)_{y = 0}.
\end{align}

\noindent   Specifically, we have integrated by parts twice in $y$, expanded the resulting quantity, $\p_{yy} \{ q_{x}w^2 \} = q_{xyy}w^2 + 4 q_{xy} ww_y + q_x \p_{yy}\{ w^2 \}$, and further integrated by parts the final term in $y$.

We will first treat the boundary terms from (\ref{seeg}). First, $(\ref{seeg}.7) = 0$ due to the boundary condition $q_x|_{y = 0} = 0$ coupled with the asymptotic estimate (\ref{ratio.estimate}) for $v_{yyy}$. Next, an expansion shows: 
\begin{align*}
(\ref{seeg}.6) = ([ u_{syy} q + 2u_{sy}q_y + u_s q_{yy}], q_{xy} w^2)_{y = 0}. 
\end{align*}

\noindent  The first term vanishes as $q|_{y = 0} = 0$, whereas the third term vanishes according to the asymptotics in (\ref{ratio.estimate}).  The only contribution is thus the middle term for which we use that $q|_{x = 0} = 0$ to estimate $|(q_y , q_{xy})|_{y = 0}| \le L \|q_{xy}\|_{y = 0}^2$, which is an acceptable contribution to the right-hand side of (\ref{sum.bx.lap}) due to the inclusion $\|q_{xy} \|_{y = 0}$ in $|||q|||_w$.

We now turn to the bulk terms from (\ref{seeg}). An expansion shows 
\begin{align*}
(\ref{seeg}.1) = ( [ u_{syy}q + 2u_{sy}q_y + u_s q_{yy} ], q_{xyy}w^2). 
\end{align*}

\noindent   For the first term, we estimate via Hardy in $y$, (\ref{hardy.orig.w}), and Poincare in $x$: 
\begin{align*}
|(\ref{seeg}.1.1)| \le & \| u_{syy} \langle y \rangle \|_\infty \|\frac{q}{y} w \| \| \sqrt{u_s} q_{xyy} w \| \\
\lesssim & \{ \| q_y w \| + \| \sqrt{\eps}q w \| \} \| \sqrt{u_s} q_{xyy} w \| \\
\lesssim & L \{ \|q_{xy} w \| + \| \sqrt{\eps}q_x w \| \} \| \sqrt{u_s} q_{xyy} w \|. 
\end{align*}

\noindent   The middle term requires an integration by parts in $y$ via:
\begin{align*}
(\ref{seeg}.1.2) = &-  ( 2 q_{xy}, \p_y \{ u_{sy} q_y w^2  \}) - (2 q_{xy}, u_{sy} q_y w^2)_{y = 0} \\
= & - ( 2  u_{syy} q_{xy}, q_y w^2) - ( 2 q_{xy}, u_{sy} q_{yy}w^2) \\
& - ( 4 u_{sy}q_{xy}, q_y 2 ww_{y})  - (2 q_{xy}, q_y u_{sy} w^2)_{y = 0} \\
\lesssim & L \| q_{xy} \cdot w \|^2 + \| q_{xy} \cdot w \| \| q_{yy} \cdot w \| \\
& +  L \| q_{xy} \cdot \sqrt{ww_{y}} \|^2 + L \| q_{xy} \cdot w\|_{y = 0}^2 \\
\lesssim & \sqrt{L} |||q|||_{w}^2. 
\end{align*}

\noindent   Above, we have used (\ref{Poincare}) for the $q_y$ terms, the assumption that $|w_y| \lesssim |w|$, and most importantly the estimate (\ref{small.guys}) to obtain $\sqrt{L}$ control of $\| q_{yy} w \|$. The final term can be estimated via Poincare in $x$: $|(\ref{seeg}.1.3)| \lesssim L \| \sqrt{u_s} q_{xyy} w \|^2$. 

We continue with the bulk contributions from (\ref{seeg}), for which straightforward bounds using (\ref{Poincare}) and the inequalities $|w_y| \lesssim |w|$, $(w^2)_{yy} \lesssim |w_y|^2$ show: 
\begin{align*}
&|(\ref{seeg}.2, 3)| \le L \| v_{xyy} w \|  \| \sqrt{u_s} q_{xy} \cdot w \|,\\ 
&|(\ref{seeg}.4)| \le L \| v_{xy}w \| \| q_{xy} w \|, \\
& |(\ref{seeg}.5)| \le L^2 \| v_{xy} w \| \| q_{xx} w_y \|, 
\end{align*}

\noindent   all of which are acceptable contributions to the right-hand side of (\ref{sum.bx.lap}). This concludes our treatment of (\ref{seeg}). 

We move on to contributions from $\eps v_{xxyy}$. We begin with one integration by parts in $y$ and an expansion of $v_{xxy} = \p_{xxy} \{ u_s q \}$: 
\begin{align} \n 
( 2 \eps v_{xxyy}, q_x w^2) = & ( - 2\eps v_{xxy}, q_{xy}w^2) - ( 4 \eps v_{xxy},q_x ww_{y}) - (2 \eps v_{xxy}, q_x w^2)_{y = 0} \\ \n
= &- ( 2\eps [ u_{sxxy}q + u_{sy}q_{xx} + 2 u_{sxy}q_x + u_{sxx}q_y \\ \label{fp1}
& + u_s q_{xxy} + 2u_{sx} q_{xy} ], q_{xy} w^2) - ( 4\eps q_x, v_{xxy}ww_{y}).
\end{align}

\noindent   We have used (\ref{ratio.estimate}) to conclude that the $\{y = 0\}$ boundary contribution vanishes. It is straightforward to estimate using (\ref{Poincare.3}) and that $|w_y| \lesssim |w|$:
\begin{align*}
&|(\ref{fp1}.1)| +... + |(\ref{fp1}.6)| \lesssim \sqrt{\eps} ||| q |||_{w}^2 \\
&|(\ref{fp1}.7)| \lesssim L \| \sqrt{\eps}v_{xxy}w \|  \| \sqrt{\eps} q_{xx} w \|.
\end{align*}

We now move to $\p_x^4$ contributions, for which an integration by parts in $x$ and expansion gives:
\begin{align} \n
( \eps^2 v_{xxxx}, q_x w^2) = & - (\eps^2 v_{xxx}, q_{xx}w^2) - (\eps^2 v_{xxx}, q_x w^2)_{x = 0} \\ \n
 = &- (\eps^2 [u_{sxxx}q + 3 u_{sxx}q_x + 3 u_{sx} q_{xx} + u_s q_{xxx}], q_{xx}w^2)\\ \n
&  - (\eps^2 v_{xxx}, q_x w^2)_{x = 0} \\ \n
\lesssim & \sqrt{\eps} |||q|||_w^2 + \|\eps^{\frac{3}{2}}v_{xxx} \|_{x = 0} \|\sqrt{\eps} q_x w\|_{x = 0} \\
\lesssim & \sqrt{\eps}|||q|||_w^2 + \sqrt{L} ||||v||||_w |||q|||_w,
\end{align}

\noindent  where we have used estimate (\ref{pff}) for the $q_x|_{x = 0}$ boundary term.

\subsubsection*{\normalfont \textit{Step 3: $J(v)$ Terms}}

\begin{align} \n
|&(J, q_x w^2)| \lesssim  \sqrt{L} |||q|||_w  + L^2 \| q_{xx} w_y \|^2.
\end{align}

Recalling the definition of $J$ in (\ref{defn.J.conc}), we have 
\begin{align} \label{whyso.bx}
( - \eps v_{sx}v_{xy} &- v_s v_{yyy} - \eps v_s v_{yxx} + \Delta_\eps v_s v_y  \\ \n
&- v_{sx}I_x[v_{yyy}] + \Delta_\eps v_{sx} I_x[v_y], q_{x}w^2).
\end{align}

We first record the elementary inequality which will be in repeated use: 
\begin{align} \label{elementary.L}
\| I_x[f] \| \le \sqrt{L} \| I_x[f] \|_{L^\infty_x L^2_y} \le L \| f \|. 
\end{align}

We integrate by parts in $y$: 
\begin{align*}
(\ref{whyso.bx}.5) = &( I_x[v_{yy}] v_{sxy}, q_x w^2) + ( I_x[v_{yy}] v_{sx}, q_{xy} w^2) + ( I_x[v_{yy}] v_{sx}, q_x 2ww_{y}).
\end{align*}

\noindent  Using (\ref{elementary.L}), we immediately estimate (\ref{whyso.bx}.5.\{2, 3\}). The first term, (\ref{whyso.bx}.5.1), is controlled upon using that $\| v_{sxy} y \|_\infty < \infty$ and an appeal to the Hardy inequality, (\ref{w.Hardy}):  
\begin{align*}
&|(\ref{whyso.bx}.5.1)| \lesssim L \| v_{sxy} \langle y \rangle  \|_\infty \| v_{yy} w \| \| \frac{q_x}{\langle y \rangle} w \|, \\
&|(\ref{whyso.bx}.5.2)| \lesssim L \| v_{yy} w \| \| q_{xy} w \|, \\
&|(\ref{whyso.bx}.5.3)| \lesssim L^2 \| v_{yy} w \| q_{xx} w_y \|.
\end{align*}

Integration by parts in $y$ for the term (\ref{whyso.bx}.2) produces: 
\begin{align*}
(\ref{whyso.bx}.2) = (v_{yy} v_{sy}, q_x w^2) + (v_{yy}v_s, q_{xy}w^2) + 2(v_{yy} v_s, q_{x} ww_y)
\end{align*}

\noindent  From here an analogous set of estimates to (\ref{whyso.bx}.5) produces the desired estimate upon using one further Poincare inequality, $\| v_{yy} w \| \le L \| v_{xyy} w \|$, which is valid as $v|_{x = 0} = 0$. Direct Poincare inequality in $x$ using (\ref{Poincare}) yields $|(\ref{whyso.bx}.1)| + |(\ref{whyso.bx}.3)| \lesssim L |||q|||_w^2$. Terms (\ref{whyso.bx}.4) and (\ref{whyso.bx}.6) are estimated identically so we focus on (\ref{whyso.bx}.4). We estimate the $\Delta_\eps$ term using (\ref{w.Hardy}) and Poincare in $x$ as $v|_{x = 0} = 0$: 
\begin{align*}
|(\ref{whyso.bx}.4)| \le & \| \Delta_\eps v_s y \|_\infty \| v_y w \| \| \frac{q_x}{y} w \| \\
\lesssim & L \| v_{xy} w \| \{ \| q_{xy} w \| + \| \sqrt{\eps} q_x w \| \} \\
\lesssim & L |||q|||_1^2.
\end{align*}

\noindent This concludes the terms in $J$. 

We put directly the contribution $|(F, q_{x}w^2)|$ on the right-hand side of the desired estimate, which concludes the proof. 
\end{proof}

\subsection{Trace Estimates}

For the first fourth order bound, we will perform a weighted estimate for a weight $w(y)$. Let us make the following definition: 
\begin{align} \label{defn.barB}
\bar{B}(w) := \| \sqrt{\eps} u_s v_{xyy} \cdot w\|_{x = 0}^2 + \| \eps u_s v_{xxy} \cdot w\|_{x = L}^2.
\end{align}

\begin{lemma} Let $v$ be a solution to (\ref{eqn.dif.1}). Then the following estimate is valid: 
\begin{align} 
\begin{aligned} \label{beat.2}
\bar{B}(w) &+ \| \{ \sqrt{\eps} v_{xyyy}, \eps v_{xxyy}, \eps^{\frac{3}{2}}v_{xxxy} \} \sqrt{u_s}w \|^2 \\
& \lesssim \sqrt{\eps}|||q|||_1^2 + ||| q |||_{\sqrt{\eps}w}^2  +  |||q|||_{\sqrt{\eps}w}|||q|||_{w_y} + |(F, \eps u_s v_{xxyy}w^2)|. 
\end{aligned}
\end{align}

\end{lemma}
\begin{proof} We compute the inner-product $( \text{Equation (\ref{eqn.dif.1})}, \eps v_{xxyy} u_s w^2)$.

\subsubsection*{\normalfont \textit{Step 1: Rayleigh Terms}}

The main estimate in this step is: 
\begin{align} \n
(-\p_x R[q], \eps v_{xxyy}w^2) &\gtrsim \|\sqrt{\eps}u_s v_{xyy}w\|_{x = 0}^2 + \|\eps  u_s v_{xxy} w\|_{x = L}^2 \\
& - |||q|||_{\sqrt{\eps}w}^2 - \sqrt{\eps}|||q|||_1^2.
\end{align}

A series of integration by parts shows: 
\begin{align*}
( - \p_{xy} &\{ u_s v_y \}, \eps v_{xxyy} w^2 u_s) \\
= &- ( \eps [ u_{sx} v_{yy} + u_{sy} v_{xy}  + u_{sxy}v_y + u_s v_{xyy} ], v_{xxyy} w^2 u_s) \\
= & ( \eps v_{xyy}, \p_x \{ u_s u_{sx} v_{yy} \} w^2) + ( \eps v_{xxy}, \p_y \{ u_s u_{sy} v_{xy} w^2 \}) \\
& + ( \eps v_{xyy} w^2, \p_x \{ u_s u_{sxy} v_y \}) + (\eps u_s u_{sx} v_{xyy}, v_{xyy}w^2) \\
& + \frac{1}{2} \|\sqrt{\eps} v_{xyy} w u_s \|_{x = 0}^2 \\
\gtrsim &  \|\sqrt{\eps} v_{xyy} w u_s\|_{x = 0}^2 - |||q|||_{\sqrt{\eps}w}^2.
\end{align*}

Again, we expand and perform a series of integrations by parts which produces:
\begin{align} \n
- ( \eps \p_{xx} \{ u_s& v_x \}, \eps v_{xxyy} u_s w^2) \\ \n
= & - ( \eps^2 [ u_s v_{xxx} + 2 u_{sx} v_{xx} + u_{sxx}v_x ], v_{xxyy} u_s w^2) \\ \n
= & ( \eps^2 v_{xxy}, \p_y \{ u_s^2 v_{xxx} w^2 \}) + (2\eps^2 v_{xyy} w^2, \p_x \{ u_s u_{sx} v_{xx} \}) \\ \n
&  + ( \eps^2 v_{xxy}, \p_y \{ u_{sxx} u_s v_x w^2 \})\\ \n
= &(\eps^2 v_{xxy}, \p_y\{ u_s^2 w^2\} v_{xxx}) + (\eps^2 v_{xxy}, u_s^2 v_{xxxy} w^2) \\ \label{rocka}
& + 2(\eps^2 v_{xyy} w^2, \p_x\{u_s u_{sx} v_{xx} \})  + (\eps^2 v_{xxy}, \p_y \{ u_{sxx}u_s v_x w^2 \} ).
\end{align}

First, let us deal with (\ref{rocka}.1), which produces the ``main commutator" term. Using (\ref{profile.splitting}): 
\begin{align} \label{nablasplit}
|\nabla u_s| \le |\nabla u^P_s| + \sqrt{\eps}, 
\end{align}

\noindent  and so: 
\begin{align*}
|(\ref{rocka}.1)| \lesssim &\| \sqrt{\eps}v_{xxy} \sqrt{\eps}w \| \| \eps v_{xxx} w_y \| + \sqrt{\eps} \| \sqrt{\eps}v_{xxy} \| \eps v_{xxx} \| \\
& + \sqrt{\eps} \| \sqrt{\eps} v_{xxy} \sqrt{\eps}w \| \eps v_{xxx} \sqrt{\eps}w \| \\
\lesssim & |||q|||_{\sqrt{\eps}w} |||q|||_{w_y} + \sqrt{\eps} |||q|||_1^2 + \sqrt{\eps} |||q|||_{\sqrt{\eps}w}^2.
\end{align*}

The term (\ref{rocka}.2) produces a positive boundary contribution via integration by parts in $x$: 
\begin{align*}
(\ref{rocka}.2) = &\frac{1}{2} \|\eps u_s v_{xxy} w\|_{x = L}^2 - (u_s u_{sx} \eps^2 v_{xxy}, v_{xxy} w^2) \\
\gtrsim & \|\eps u_s v_{xxy}w\|_{x = L}^2 - |||q|||_{\sqrt{\eps}w}^2
\end{align*}

We estimate (\ref{rocka}.3) directly: 
\begin{align*}
&|(\ref{rocka}.3)| \lesssim \| v_{xyy} \sqrt{\eps}w \| \| \eps v_{xxx} \sqrt{\eps} w \| \lesssim |||q|||_{\sqrt{\eps}w}^2.
\end{align*} 

Finally, for (\ref{rocka}.4), we distribute the $\p_y$: 
\begin{align*}
|(\ref{rocka}.4)| =& |(\eps^2 v_{xxy}, u_{sxxy} u_s v_x w^2 + u_{sxx}u_{sy}v_x w^2 + u_{sxx}u_s v_{xy} w^2 \\
& + u_{sxx}u_s v_x 2ww_y)| \\
\lesssim & \| \sqrt{\eps}v_{xxy} \sqrt{\eps}w \| \{ \| \sqrt{\eps}v_x \sqrt{\eps}w \| + \| v_{xy} \sqrt{\eps}w \| \}
\end{align*}

We now have the lower order Rayleigh contributions. Here, the main mechanism is the pointwise inequality, (\ref{nablasplit}). We simply expand the product, integrate by parts once,  expand further the resulting expression, and estimate using this pointwise inequality: 
\begin{align*}
( \p_{xy} \{ u_{sy}v \}, \eps v_{xxyy} u_s w^2) = & ( [u_{sxyy} v + u_{sxy}v_y + u_{syy}v_x+ u_{sy}v_{xy}], \eps v_{xxyy} u_s w^2) \\
= & - ( \eps v_{xyy} w^2, \p_x \{ u_s u_{sxyy}v \}) - ( \eps v_{xyy} w^2, \p_x \{ u_s u_{sxy} v_y \}) \\
& - ( \eps v_{xxy}, \p_y \{ u_{syy} u_s v_x w^2 \}) - ( \eps v_{xxy}, \p_y \{ u_{sy} u_s v_{xy} w^2 \}) \\
= & - ( \eps u_{sx} u_{sxyy} v_{xyy}, v w^2) - ( \eps u_s u_{sxxyy} v_{xyy},v w^2) \\
& - ( \eps u_s u_{sxyy} v_x, v_{xyy}w^2) - ( \eps u_{sx} u_{sxy} v_y, v_{xyy} w^2) \\
& - ( \eps u_s u_{sxxy} v_{xyy},v_y w^2) - ( \eps u_s u_{sxy} v_{xyy}, v_{xy}w^2) \\
& - ( \eps u_s u_{syy} v_{xxy}, v_x w^2) - ( \eps u_s u_{syy} v_{xy},v_{xxy} w^2) \\
& - ( \eps u_{syy}u_{sy} v_{xxy}, v_x w^2) - ( 2 \eps u_{syy}u_s v_x, v_{xxy} ww_{y}) \\
& - ( \eps u_s u_{syy} v_{xxy},v_{xy}w^2) - ( \eps u_s u_{sy} v_{xxy},v_{xyy}w^2) \\
& - ( \eps u_{sy}^2 v_{xxy},v_{xy}w^2) - ( 2\eps u_s u_{sy} v_{xxy},v_{xy}ww_{y}) \\
\lesssim & \sqrt{\eps}|||q|||_1^2 + \sqrt{\eps}||| q |||_{\sqrt{\eps}w}^2.
\end{align*}

\noindent   Above, we have used that $|w_y| \lesssim |w|$. We have the final lower-order Rayleigh terms, for which a nearly identical argument to above is carried out: 
\begin{align*}
( \eps \p_{xx} \{ u_{sx} v\}, \eps v_{xxyy} u_s w^2) = & ( \eps^2 [u_{sxxx} v + 2 u_{sxx} v_x + u_{sx}v_{xx} ], v_{xxyy} u_s w^2) \\
= & - ( \eps^2 v_{xyy}, \p_x \{ u_{sxxx} u_s v  \} w^2) - ( 2 \eps^2 v_{xxy}, \p_y \{ u_{sxx} v_x u_s w^2 \})\\
&  - ( \eps^2 v_{xyy}, \p_x \{ u_{sx} u_s v_{xx} \} w^2) \\
= & - ( \eps^2 u_{sxxxx}u_s v_{xyy},v w^2) - ( \eps^2 u_{sxxx} u_s v_{xyy},v_x w^2) \\
& - ( \eps^2 u_{sxxx}u_{sx} v_{xyy}, v w^2) - (2\eps^2 u_s u_{sxxy} v_{xxy}, v_x w^2) \\
& - ( 2 \eps^2 u_{s} u_{sxx} v_{xy}, v_{xxy}) - (4\eps^2 u_s u_{sxx} v_x, v_{xxy} ww_{y}) \\
& - (2\eps^2 u_{sxx}u_{sy} v_{xxy}, v_x w^2) - (\eps^2 u_{sxx}u_s v_{xyy}, v_{xx} w^2) \\
& - (\eps^2 u_{sx} u_s v_{xyy}, v_{xxx} w^2) - (\eps^2 u_{sx}^2 v_{xyy}, v_{xx}w^2) \\
\lesssim & \sqrt{\eps}|||q|||_1^2 + \sqrt{\eps}||| q |||_{\sqrt{\eps}w}^2.
\end{align*}

\subsubsection*{\normalfont \textit{Step 2: Estimate of $\Delta_\eps^2$ Terms}}

This is done in (\ref{sum.vxxyy.lap}).

\subsubsection*{\normalfont \textit{Step 3: Estimate of $J(v)$ Terms}}

\begin{align} \label{kii}
|(J, \eps v_{xxyy} u_s w^2)|  \lesssim o(1) \text{LHS}(\ref{beat.2})+  |||q|||_{\sqrt{\eps}w}^2
\end{align}

Recalling the definition of $J$ from (\ref{defn.J.conc}), we expand
\begin{align} \label{beyonce}
( - v_s v_{yyy} &- \eps v_s v_{yxx} - \eps v_{sx}v_{xy} -  v_y \Delta_\eps v_s  \\ \n
&- v_{sx}I_x[v_{yyy}] +  I_x[v_y] \Delta_\eps v_{sx},\eps v_{xxyy} u_sw^2)
\end{align}

Straightforward estimates give: 
\begin{align*}
&|(\ref{beyonce}.2)| \lesssim \| \sqrt{\eps}v_{yxx} \sqrt{\eps} w \| \| \eps v_{xxyy} u_s w \| \lesssim |||q|||_{\sqrt{\eps}w} \times \text{LHS}(\ref{beat.2}), \\
&|(\ref{beyonce}.3)| \lesssim \sqrt{\eps} \| v_{xy} \sqrt{\eps} w \| \| \eps v_{xxyy} u_s w \| \lesssim \sqrt{\eps}|||q|||_{\sqrt{\eps}w} \times \text{LHS}(\ref{beat.2}), 
\end{align*}

\noindent   which upon using Young's inequality for products is clear acceptable to the right-hand side of (\ref{kii}). 

We now turn to (\ref{beyonce}.1) for which we integrate by parts in $x$, and subsequently integrate by parts the middle term in $y$ thanks to the boundary condition $v|_{y = 0} = v_y|_{y = 0} = 0$:
\begin{align*}
(\ref{beyonce}.1) = & (v_{sx}v_{yyy}, \eps v_{xyy} u_s w^2) + (v_s v_{xyyy}, \eps v_{xyy} u_s w^2) + (v_s v_{yyy}, \eps v_{xyy} u_{sx} w^2) \\
= & (v_{sx}v_{yyy}, \eps v_{xyy}w^2) - \frac{1}{2}(v_{sy} v_{xyy}, \eps v_{xyy} u_s w^2) - \frac{1}{2}(v_{s} v_{xyy}, \eps v_{xyy} u_{sy} w^2)\\
& - (v_{s} v_{xyy}, \eps v_{xyy} u_s ww_y) + (v_s v_{yyy}, \eps v_{xyy} u_{sx} w^2) \\
\lesssim & |||q|||_{\sqrt{\eps}w}^2.
\end{align*}

We next move to (\ref{beyonce}.4) for which we integrate by parts in $x$ using that $v|_{x = 0} = 0$ and $v_x|_{x = L} = 0$: 
\begin{align*}
(\ref{beyonce}.4) = & (v_{xy} \Delta_\eps v_s, \eps v_{xyy}u_s w^2) +  (v_{y} \Delta_\eps v_{sx}, \eps v_{xyy}u_s w^2) +  (v_{y} \Delta_\eps v_s, \eps v_{xyy}u_{sx} w^2) \\
\lesssim & \| \Delta_\eps v_s + \Delta_\eps v_{sx} \|_\infty |||q|||_{\sqrt{\eps}w}^2. 
\end{align*}

Next, we move to (\ref{beyonce}.5) for which we integrate by parts in $x$ and use that $I_x[f]|_{x = 0} = 0$ by definition: 
\begin{align*}
(\ref{beyonce}.5) = & (v_{sxx}I_x[v_{yyy}], \eps v_{xyy} u_s w^2) +  (v_{sx} v_{yyy}, \eps v_{xyy} u_s w^2) \\
& +  (v_{sx}I_x[v_{yyy}], \eps v_{xyy} u_{sx} w^2) \\
\lesssim & |||q|||_{\sqrt{\eps}w}^2. 
\end{align*}

Lastly, we move to (\ref{beyonce}.6), for which we again integrate by parts in $x$ and subsequently use the Poincare inequality in $x$, (\ref{Poincare}), to produce: 
\begin{align*}
(\ref{beyonce}.6) = & - (v_y \Delta_\eps v_{sx}, \eps v_{xyy}u_s w^2) - (I_x[v_y] \Delta_\eps v_{sxx}, \eps v_{xyy} u_s w^2) \\
& - (I_x[v_y] \Delta_\eps v_{sx}, \eps v_{xyy} u_{sx}w^2) \\
\lesssim & \| \p_x^j \Delta_\eps v_s \|_\infty |||q|||_{\sqrt{\eps}w}^2. 
\end{align*}

\noindent This concludes the estimation of the $J(v)$ terms. 

To conclude the proof, we simply put the forcing term, $|(F, \eps u_s v_{xxyy}w^2)|$ to the right-hand side of the desired estimate. 
\end{proof}

\begin{lemma}  Let $\zeta > 0$ be arbitrary. Let $v$ be a solution to (\ref{eqn.dif.1}), and suppose $|\p_y^k w| \lesssim |w|$. Then: 
\begin{align}\label{beat.1} 
\|\eps^\frac{3}{2} \sqrt{u_s} v_{xxx} &\cdot w\|_{x = 0}^2 +  \| \{ \eps v_{xxyy}, \eps^{\frac{3}{2}}v_{xxxy}, \eps^2 v_{xxxx} \} \cdot w \|^2 \\ \n
& \lesssim \frac{1}{\zeta}\bar{B}(1) + \bar{B}(w) + \Big(\zeta^3 + \sqrt{\eps} \Big) |||q|||_1^2 + ||| q |||_{\sqrt{\eps}w}^2\\ \n
& + |||q|||_{\sqrt{\eps}w}|||q|||_{w_y} + |(F, \eps^2 v_{xxxx}w^2)|,
\end{align}

\noindent where $\bar{B}$ has been defined in (\ref{defn.barB}).

\end{lemma}
\begin{proof} We will compute the inner-product $( \text{Equation } (\ref{eqn.dif.1}), \eps^2 v_{xxxx} w^2)$. 

\subsubsection*{\normalfont \textit{Step 1: Estimate of Rayleigh Terms}}

The main estimate of this step is: 
\begin{align} \label{4Ray1}
(-\p_x R[q], \eps^2 v_{xxxx} w^2 ) \gtrsim &\|\sqrt{u_s} \eps^{\frac{3}{2}} v_{xxx}w\|_{x = 0}^2 - ( \bar{B}(w) + \zeta^{-1} \bar{B}(1)) \\ \n
&+ (\zeta^3 + \sqrt{\eps})|||q|||_1^2 +  |||q|||_{\sqrt{\eps} w}^2  - o(1) \text{LHS(\ref{beat.1})} \\ \n
& - |||q|||_{\sqrt{\eps}w}|||q|||_{w_y}.
\end{align}

First, we will extract the positive terms: 
\begin{align} \n
( -\eps \p_x \{ u_s v_{xx} \}, \eps^2 v_{xxxx} w^2) = &(\eps^3 u_{sxx} v_{xx},v_{xxx} w^2) + \frac{3}{2} \|\eps^{\frac{3}{2}} \sqrt{ |u_{sx}|} v_{xxx} w \|^2  \\ \label{Rara.1}
& + \frac{1}{2} \|\sqrt{u_s} \eps^{\frac{3}{2}} v_{xxx} w\|_{x=0}^2 \\ \n
\gtrsim & \frac{1}{2} \|\sqrt{u_s} \eps^{\frac{3}{2}} v_{xxx} w\|_{x= 0}^2 - \| \eps v_{xxx} \cdot \sqrt{\eps}w \|^2.
\end{align}

The lower order Rayleigh term is treated as follows, using the Poincare inequality paired with $v_x|_{x = L} = 0$: 
\begin{align*}
|(\eps^3 \p_x \{ u_{sxx}v \},  v_{xxxx} w^2 ) | & \lesssim \| \eps^2 v_{xxxx} \cdot w \| \| \sqrt{\eps}v_x \cdot \sqrt{\eps}w \|  \\
& \lesssim L \times \text{LHS}(\ref{beat.1}) + L |||q|||_{\sqrt{\eps}w}^2.
\end{align*}

The next Rayleigh contributions are of the following form: 
\begin{align} \label{gggre}
- ( \p_x\{u_s v_{yy}\} \cdot \eps^2 v_{xxxx} w^2) = - ( \eps^2 [u_{sx} v_{yy} + u_s v_{xyy}],v_{xxxx} w^2).
\end{align}

For the first term from (\ref{gggre}), we integrate by parts in $x$ with no boundary contributions according to (\ref{eqn.dif.1}):
\begin{align*}
(\ref{gggre}.1) = & ( \eps^2 u_{sxx} v_{yy},v_{xxx} w^2) + (\eps^2 u_{sx}v_{xyy},v_{xxx}w^2) \\
\lesssim & \| \eps v_{xxx} \cdot \sqrt{\eps}w \| \| v_{xyy} \sqrt{\eps}w \|. 
\end{align*}

For the latter term, we require a localization. Recall the definition of (\ref{basic.cutoff}) and define: 
\begin{align*}
\chi_{\le \zeta}(y) := \chi(\frac{y}{\zeta}), \hspace{3 mm} \chi_{\zeta \le y \le 1}(y) := \chi(y) - \chi(\frac{y}{\zeta}), \hspace{3 mm} \chi_{\ge 1}(y) := 1 - \chi(y). 
\end{align*}

We then decompose:
\begin{align*}
(\ref{gggre}.2) = - ( \eps^2 u_s v_{xyy}, v_{xxxx} w^2 [ \chi_{\le \zeta}(y) + \chi_{\zeta \le y \le 1}(y) +  \chi_{\ge 1}(y)]). 
\end{align*}

Using that $u_s(0) = 0$ and $|\p_y u_s| \lesssim 1$ gives that $|u_s| \lesssim y \lesssim \zeta$ in the support of $\chi_{\le \zeta}$, and so we estimate with Young's inequality for products: 
\begin{align} \n
|( \eps^2 u_s v_{xyy}, v_{xxxx} w^2 \chi_{\le \zeta})| &\lesssim \zeta \| w \chi_{\le \zeta} \|_\infty \| v_{xyy} \| \| \eps^2 v_{xxxx} \cdot w \|  \\ \label{balloon1}
& \le o(1) \underbrace{\| \eps^2 v_{xxxx} \cdot w \|^2}_{\bigO(LHS)} + N \zeta^3 |||q|||_1^2, 
\end{align}

\noindent   for some large number $N$. Both of these contributions are acceptable to the right-hand side of (\ref{4Ray1}).

Let now $\phi = \chi_{\zeta \le y \le 1}$ or $\chi_{y \ge 1}$. We integrate by parts the term in (\ref{gggre}.2), and use that $v_{xxx}|_{x = L} = 0$ to produce only boundary contributions at $\{x = 0\}$: 
\begin{align} \n
- ( \eps^2 u_s v_{xyy}, v_{xxxx} w^2 \phi) = &(\eps^2 u_s v_{xxyy}, v_{xxx} w^2 \phi) + (\eps^2 u_{sx} v_{xyy}, v_{xxx} w^2 \phi) \\ \label{gtch1}
& + (\eps^2 u_s v_{xyy}, v_{xxx} w^2 \phi)_{x = 0}.
\end{align}

We estimate: 
\begin{align*}
&|(\ref{gtch1}.2)| \lesssim \| v_{xyy} \sqrt{\eps} w \| \| \eps^{\frac{3}{2}}v_{xxx} w \|, \\
&|(\ref{gtch1}.3)| \lesssim \frac{1}{\sqrt{\zeta}} \|\eps^{\frac{3}{2}}\sqrt{u_s}v_{xxx}\|_{x = 0} \| u_s \sqrt{\eps}v_{xyy}\|_{x = 0} \\
& \hspace{15 mm} + \|\eps^{\frac{3}{2}}\sqrt{u_s}v_{xxx} w\|_{x = 0} \| u_s \sqrt{\eps}v_{xyy}w\|_{x = 0} \\
& \hspace{12 mm} \le \frac{N}{\zeta} \bar{B}(1)^2 + o(1) \text{LHS}(\ref{beat.1}) + N \bar{B}(w)^2,
\end{align*}

\noindent   for a potentially large constant $N$. Above for (\ref{gtch1}.3), we have split into two cases: 
\begin{align}
\begin{aligned} \label{split.same}
|(\eps^2 u_s v_{xyy}&, v_{xxx}w^2 \chi_{\zeta \le y \le 1})_{x= 0}| \\
=& |(\eps^2 u_s v_{xyy}, v_{xxx} w^2 \frac{\sqrt{u_s}}{\sqrt{u_s}}\chi_{\zeta \le y \le 1})_{x = 0}| \\
\le & \frac{1}{\sqrt{\zeta}} \| \frac{\sqrt{\zeta}}{\sqrt{u_s}} \chi_{\zeta \le y \le 1} \|_\infty \| \eps^{\frac{3}{2}} \sqrt{u_s} v_{xxx} \|_{x = 0} \| u_s \sqrt{\eps}v_{xyy} \|_{x = 0},
\end{aligned}
\end{align}

\noindent  whereas in the $\phi = \chi_{\ge 1}$ case, we use that $u_s \gtrsim 1$. For the highest order term, (\ref{gtch1}.1), we integrate by parts in $y$ to get: 
\begin{align*}
(\ref{gtch1}.1) = & - (\eps^2 u_{sy} v_{xxy}, v_{xxx} w^2 \phi) - (\eps^2 u_s v_{xxy}, v_{xxxy} w^2 \phi) \\
& -2 (\eps^2 u_s v_{xxy}, v_{xxx} ww_y \phi) - (\eps^2 u_s v_{xxy}, v_{xxx} w^2 \phi_y).
\end{align*} 

First, we estimate the lower order terms: 
\begin{align*}
&|(\ref{gtch1}.1.1)| \le \sqrt{\eps} \| u^P_{sy} w^2 \|_\infty \| \sqrt{\eps}v_{xxy} \| \| \eps v_{xxx} \| \\
& \hspace{20 mm} + \sqrt{\eps} \| u^E_{sY} \|_\infty \| \sqrt{\eps}v_{xxy} \sqrt{\eps}w \| \| \eps v_{xxx} \sqrt{\eps} w \|, \\
&|(\ref{gtch1}.1.3)| \lesssim \| \sqrt{\eps}v_{xxy} \sqrt{\eps}w \| \| \eps v_{xxx} w_y \| \lesssim |||q|||_{\sqrt{\eps}w}|||q|||_{w_y}, \\
&|(\ref{gtch1}.1.4)| \lesssim \| u_s \p_y \phi w^2 \|_\infty \sqrt{\eps} \| \sqrt{\eps}v_{xxy} \| \| \eps v_{xxx} \|. 
\end{align*}

For (\ref{gtch1}.1.1), we split $u_{sy} = u^P_{sy} + \sqrt{\eps} u^E_{sY}$ according to (\ref{profile.splitting}). We highlight above that (\ref{gtch1}.2.3) is an acceptable term into the right-hand side of (\ref{4Ray1}). For the term (\ref{gtch1}.1.4), we use that the following quantity is bounded independent of $\zeta$: 
\begin{align*}
\|u_s \p_y \phi\|_\infty = & \|u_s \p_y \{ \chi_{\zeta \le y \le 1} + \chi_{y \ge 1} \}\|_\infty \\
= & \|u_s \p_y \{ \chi(y) - \chi(\frac{y}{\zeta}) + 1- \chi(y) \}\|_\infty \\
= &\| - u_s \frac{1}{\zeta} \chi'(\frac{y}{\zeta})\|_\infty \lesssim 1. 
\end{align*}

The highest order term, (\ref{gtch1}.1.2), we integrate by parts in $x$ to produce: 
\begin{align*}
(\ref{gtch1}.1.2) =& \frac{1}{2} (\eps^2 u_{sx} v_{xxy}, v_{xxy} w^2 \phi) - \frac{1}{2}\|\eps \sqrt{u_s} v_{xxy}w \sqrt{\phi}\|_{x = L}^2 \\
\lesssim & \| \sqrt{\eps}v_{xxy} \sqrt{\eps} w \|^2 + \frac{1}{\zeta} \bar{B}(1) + \bar{B}(w). 
\end{align*}

This concludes the treatment of (\ref{gtch1}). Piecing together (\ref{gtch1}) and (\ref{balloon1}), we complete the estimate of (\ref{gggre}.2). Summarizing the above estimates: 
\begin{align*}
|(\ref{gggre}.2)| \lesssim &\zeta^3 |||q|||_1^2 + \sqrt{\eps} |||q|||_1^2 + ||| q |||_{\sqrt{\eps}w}^2 + \frac{1}{\zeta} \bar{B}(1) + \bar{B}(w) \\
& + |||q|||_{\sqrt{\eps}w}|||q|||_{w_y} + o(1) \text{LHS}(\ref{beat.1}).
\end{align*}

Above, we have performed a splitting identical to (\ref{split.same}) to estimate the boundary term. For the next Rayleigh contribution, we integrate by parts in $x$ and expand: 
\begin{align} \label{mot1}
( \p_x \{ u_{syy}v \}, \eps^2 v_{xxxx} w^2) = &- ( \eps^2 v_{xxx} w^2, u_{sxxyy}v) -(2 \eps^2 v_{xxx} w^2, u_{sxyy} v_x) \\ \n
& - ( \eps^2 v_{xxx} w^2, u_{syy} v_{xx}) - (\eps^2 v_{xxx} w^2, u_{syy}v_x)_{x = 0}.
\end{align}

Upon using the decomposition (\ref{split.same}) to write: 
\begin{align} \label{use.yy}
\p_x^j u_{syy} = \p_x^j u^P_{syy} + \eps  \p_x^j u^E_{sYY}, 
\end{align}

\noindent   we estimate: 
\begin{align*}
|(\ref{mot1}.\{1,2,3\})| \le \sqrt{\eps} |||q|||_1^2 + \eps ||| q |||_{\sqrt{\eps}w}^2.
\end{align*}

Next, again using (\ref{use.yy}):
\begin{align*}
|(\ref{mot1}.4)| \lesssim &\| u^P_{syy} w \langle y \rangle \|_\infty \|\eps^{\frac{1}{4}} v_x \langle y \rangle^{-1}\|_{x = 0} \|\eps^{\frac{3}{2}}v_{xxx} w\|_{x = 0} \eps^{1/4} \\
& +\eps \|\eps^{\frac{3}{2}}v_{xxx} w\|_{x = 0} \|\sqrt{\eps}v_x \sqrt{\eps}w\|_{x = 0} \\
\lesssim & \eps^{\frac{1}{4}} |||q|||_w ||||v||||_w + \eps ||||v||||_w |||q|||_1,
\end{align*}

\subsubsection*{\normalfont \textit{Step 2: $\Delta_\eps^2$ Terms}} 

This is done in (\ref{sum.vxxxx.lap}).

\subsubsection*{\normalfont \textit{Step 3: $J(v)$ Terms}}

\begin{align} \n
|(J,  \eps^2 v_{xxxx}w^2)|  \lesssim ||| q |||_{\sqrt{\eps}w}^2 + \eps |||q|||_1^2.
\end{align}

Recalling the definition of $J$ from (\ref{defn.J.conc}), we expand 
We expand: 
\begin{align} \n
&(J, \eps^2 v_{xxxx} w^2) \\ \n
& = - ( \eps^3 v_{sx} v_{xy},v_{xxxx} w^2) - ( v_s v_{yyy}, \eps^2 v_{xxxx} w^2) \\ \n
& - ( \eps^3 v_s v_{xxy}, v_{xxxx} w^2) + (\eps^2 v_y, v_{xxxx} \Delta_\eps v_s w^2)  \\ \label{jeat}
& - (v_{sx}I_x[v_{yyy}] , \eps^2v_{xxxx}w^2) + (\Delta_\eps v_{sx}I_x[v_y], \eps^2 v_{xxxx}w^2). 
\end{align}

Next, we integrate by parts in $x$, and there are no boundary contributions at $x = 0$ due to $I_x[f]|_{x = 0} = 0$ by definition:
\begin{align*}
&(\ref{jeat}.5) =  - ( v_{sxx} I_x[v_{yyy}], \eps^2 v_{xxx} w^2) - ( v_{sx} v_{yyy}, \eps^2 v_{xxx} w^2) \\
& \hspace{14 mm} \lesssim \| v_{yyy} \sqrt{\eps}w \| \| \eps v_{xxx} \sqrt{\eps}w \| \lesssim ||| q |||_{\sqrt{\eps}w}^2.
\end{align*}

Similarly, an integration by parts in $x$ produces: 
\begin{align*}
(\ref{jeat}.6) = &-(\Delta_\eps v_{sxx} I_x[v_y] + \Delta_\eps v_{sx} v_y, \eps^2 v_{xxx}w^2)  \lesssim L \| v_{xy} \sqrt{\eps} w \| \| \eps v_{xxx} \sqrt{\eps} w \|.
\end{align*}

For (\ref{jeat}.1), we perform Young's inequality for products: 
\begin{align*}
|(\ref{jeat}.1)| \lesssim  \sqrt{\eps}  \|\eps^2 v_{xxxx} \cdot w \| \|v_{xy} \cdot \sqrt{\eps} w \|  \lesssim \delta \| \eps^2 v_{xxxx} w \|^2 + N_\delta \eps \| v_{xy} \sqrt{\eps}w \|^2. 
\end{align*}

We will now integrate by parts in $x$ to produce: 
\begin{align*}
(\ref{jeat}.2) = & ( \eps^2 v_{sx} v_{yyy}, v_{xxx} w^2) + ( \eps^2 v_s v_{xyyy}, v_{xxx} w^2).
\end{align*}

The first term can be majorized by $\| v_{yyy} \sqrt{\eps}w \| \| \eps v_{xxx} \sqrt{\eps}w \| $, which is clearly admissible. For the latter term, we integrate by parts in $y$:
\begin{align*}
(\ref{jeat}.2.2) = & - (\eps^2 v_{sy} v_{xyy}, v_{xxx}w^2) - (\eps^2 v_s v_{xyy}, v_{xxxy}w^2) - 2(\eps^2 v_s v_{xyy}, v_{xxx} ww_y).
\end{align*}

\noindent   The first and third are evidently majorized by $\| v_{xyy} \sqrt{\eps} w \| \| \eps v_{xxx} \sqrt{\eps}w \|$ upon using $|w_y| \lesssim |w|$. The middle term can be majorized $\| v_{xyy} \sqrt{\eps}w \| \| \eps^{\frac{3}{2}}v_{xxxy} w \|$ upon which we use Young's inequality for products. This concludes the bound for (\ref{jeat}.2).

Next, for (\ref{jeat}.3), an integration by parts first in $x$, using that $v_{xx}|_{x = 0} = v_{xxx}|_{x = L} = 0$, and then in $y$ for the highest order term produces: 
\begin{align*}
(\ref{jeat}.3) = & (\eps^3 v_{sx} v_{xxy}, v_{xxx}w^2) + (\eps^3 v_{s} v_{xxxy}, v_{xxx} w^2)\\
= & (\eps^3 v_{sx} v_{xxy},v_{xxx} w^2) - \frac{1}{2}(\eps^3 v_{sy}v_{xxx}, v_{xxx} w^2) - (\eps^3 v_{s}v_{xxx}, v_{xxx} ww_y) \\
\lesssim & \| \sqrt{\eps}v_{xxy} \cdot \sqrt{\eps}w \|^2 + \| \eps v_{xxx} \cdot \sqrt{\eps}w \|^2 \lesssim  ||| q |||_{\sqrt{\eps}w}^2.
\end{align*}

Finally, for (\ref{jeat}.4), we again integrate by parts in $x$ using that $v_y|_{x = 0} = v_{xxx}|_{x = L} = 0$, and use that: 
\begin{align} \label{delta.dot}
\Delta_\eps v_s = \Delta_\eps v^P_s + \eps \Delta v^E_s,
\end{align} 

\noindent   we estimate
\begin{align*}
(\ref{jeat}.4) = & - ( \eps^2 v_{xy} ,  \Delta_\eps v_s v_{xxx} w^2) - ( \eps^2 v_y ,  \Delta_\eps v_{sx} v_{xxx} w^2) \\
\lesssim & \eps \| \Delta_\eps v^P_s w^2 \|_\infty \| \eps v_{xxx} \| \| v_{xy}\| + \sqrt{\eps} \| \Delta v^E_s \|_\infty \|\eps v_{xxx} \sqrt{\eps}w \| \| v_{xy} \sqrt{\eps}w \| \\
\lesssim & \eps |||q|||_1^2 + \sqrt{\eps} |||q|||_{\sqrt{\eps}w}^2.
\end{align*}

\noindent This concludes the treatment of $J(v)$ terms. 

To conclude the proof, we simply place $|(F, \eps^2 v_{xxxx}w^2)|$ to the right-hand side of the desired estimate. 
\end{proof}


\begin{lemma}  Let $v$ be a solution (\ref{eqn.dif.1}). Then the following estimate holds: 
\begin{align*}
\| v_{yyyy} w \|_{2} \lesssim & \text{RHS of Estimates } (\ref{EST.2nd.order}), (\ref{B3.estimate}), (\ref{B4.estimate}), (\ref{beat.1}), (\ref{beat.2}) \\
& + o_L(1) \Big[ \|\{v_{sx} \cdot u^0_{yy} - u^0 \Delta_\eps v_{sx} \} w\| \Big]
\end{align*}
\end{lemma}
\begin{proof} We use the equation (\ref{eqn.dif.1}) to write the identity: 
\begin{align} \label{ucd.1}
\begin{aligned}
v_{yyyy} =& - 2 \eps v_{xxyy} - \eps^2 v_{xxxx} - \p_y \{ u_s^2 \p_y q \} - \eps \p_x \{ u_s^2 \p_x q \} \\ 
& + v_y \Delta_\eps v_s - u \Delta_\eps v_{sx} - v_{sx} \Delta_\eps u + v_s \Delta_\eps v_y.
\end{aligned}
\end{align}

We will place each term in $L^2(w)$. It is easy to see that all of the terms aside from (\ref{ucd.1}.6) and $(\ref{ucd.1}.7)[\p_{yy}]$ are controlled by the left-hand sides of estimates (\ref{EST.2nd.order}), (\ref{B3.estimate}), (\ref{B4.estimate}), (\ref{beat.1}), (\ref{beat.2}), whereas (\ref{ucd.1}.6) and $(\ref{ucd.1}.7)[\p_{yy}]$ are clearly controlled by the $u^0$ terms appearing in the desired estimate. 
\end{proof}

\section{Solution to DNS and NS} \label{Subsection.nonlin}

The aim in this section is to bring together the estimates of the prior sections.Recall our ultimate aim is the nonlinear problem defined by (\ref{sys.u0.app.unh}) and (\ref{spec.nl}). Motivated by these, we define the problem of interest in this section: 
\begin{align} \label{linearised.1}
&-\p_x R[q] + \Delta_\eps^2 v + J(v) \\ \n
& \hspace{10 mm} =  - B_{(v^0)}(\bar{v}^0) - \eps^{N_0} \mathcal{N}(\bar{u}^0, \bar{v}^0, \bar{v})  + F_{(q)}(\bar{u}^0, \bar{v}^0, \bar{v}), \\ \label{linearised.2}
&\mathcal{L} v^0 = F_{(v)}(\bar{v}) + \mathcal{Q}(\bar{u}^0, \bar{v}^0, \bar{v}) + \mathcal{H} + F^a_R.
\end{align}

\noindent Recall the definition of $F_{(q)}$ from (\ref{spec.nl}). While $\p_x F_R$, $\p_x b_{(u)}(a^\eps)$, and the $h$-dependent terms are pure forcing terms, $H[a^\eps]$ is linear. We thus take $H[a^\eps] = H[a^\eps][\bar{v}, \bar{u}^0, \bar{v}^0]$. 

We build the following linear combinations: 
\begin{align}\n
&\mathcal{B}_{X_1} := ( B_{(v^0)}, q_x + q_{xx} + q_{yy} \\ \label{defn.BX1}
& \hspace{20 mm} + \eps^{-\frac{3}{8}}\eps^2 v_{xxxx} + \eps^{-\frac{3}{8}}\eps^{-\frac{1}{8}}\eps u_s v_{xxyy}), \\ \n
&\mathcal{B}_{Y_w} := ( B_{(v^0)}, \{ \eps q_x + \eps q_{xx} + \eps q_{yy} + \eps^2 v_{xxxx} + \eps u_s v_{xxyy}\} w^2 \\ \label{defn.BYw}
& \hspace{20 mm} + \{\eps^2 v_{xxxx} + \eps^{-\frac{1}{8}} \eps u_s v_{xxyy} \}).
\end{align}

\noindent The quantities $\mathcal{N}_{X_1}, \mathcal{N}_{Y_w}$ are defined as above, with $\eps^{N_0}\mathcal{N}(\bar{u}^0, \bar{v}^0, \bar{v})$ taking the place of $B_{(v^0)}$. Similarly, the quantities $\mathcal{F}_{X_1}, \mathcal{F}_{Y_w}$ are defined as above, with $F_{(q)}$ taking the place of $B_{(v^0)}$.  As a notational point, we will sometimes need to think of $\mathcal{F}_{X_1}, \mathcal{F}_{Y_w}$ as a bilinear term. In this case, we introduce the notation $\mathcal{F}_{X_1}(F_{(q)}(\bar{v}, \bar{u}^0, \bar{v}^0), q)$ (and same with $\mathcal{F}_{Y_w}$, and $\mathcal{F}_{B}$ below). 

We also define the quantities: 
\begin{align} 
\begin{aligned} \label{mathB}
&\mathcal{B}_{B} := |( F_{(v)}(\bar{v}), q^0)| + \| F_{(v)}(\bar{v}) w_0 \|^2, \\
&\mathcal{N}_B := |( \mathcal{Q}(\bar{u}^0, \bar{v}^0, \bar{v}) + \mathcal{H}, q^0)| + \| \{\mathcal{Q}(\bar{u}^0, \bar{v}^0, \bar{v}) + \mathcal{H} \} w_0 \|^2, \\
&\mathcal{F}_B := |( F^a_R, q^0)| + \| F^a_R w_0 \|^2.  
\end{aligned}
\end{align}

\noindent One sees from the specification of $F^a_R$ in (\ref{sys.u0.app.unh}) that $F^a_R$ is a pure forcing term. 

The purpose of all of these definitions is:
\begin{lemma} Let $v$ be a solution to (\ref{linearised.1}) and $[u^0, v^0]$ a solution to (\ref{linearised.2}), and $\bold{u} \in \mathcal{X}$ as in (\ref{defn.norms.ult}). Then the following estimates are valid: 
\begin{align}
\begin{aligned} \label{bring1}
&\| v \|_{X_1}^2 \lesssim \mathcal{B}_{X_1} + \mathcal{N}_{X_1} + \mathcal{F}_{X_1}, \\
&[u^0, v^0]_B^2 \lesssim \mathcal{B}_B + \mathcal{N}_B + \mathcal{F}_B, \\
&\| v \|_{Y_{w_0}}^2 \lesssim \| v \|_{X_1}^2 + \mathcal{B}_{Y_{w_0}} + \mathcal{N}_{Y_{w_0}} + \mathcal{F}_{w_0}, 
\end{aligned}
\end{align}
\end{lemma}
\begin{proof} The $[u^0, v^0]_B$ bound follows immediately from (\ref{mainu0est}) upon selecting the forcing according to the right-hand side of (\ref{linearised.2}).

Next, we take the combination $\eps^{\frac{1}{8}} (\ref{beat.2}) + (\ref{beat.1})$ and $(\ref{B3.estimate}) + (\ref{B4.estimate}) + (\ref{EST.2nd.order})$ for $L << 1$ and $w = 1$, which produces
\begin{align*}
&||||v||||_1^2 \lesssim \eps^{\frac{3}{8}}|||q|||_1^2 + |(F, \eps^2 v_{xxxx} + \eps^{-\frac{1}{8}} \eps v_{xxyy})|, \\
&|||q|||_1^2 \lesssim o_L(1) |||q|||_1^2 + o_L(1) ||||v||||_1^2 + |(F, q_{yy} + q_{xx} + q_x)|. 
\end{align*}

\noindent From here, we conclude the $X_1$ estimate. 

Next, we take the combination $\eps^{\frac{1}{8}} (\ref{beat.2}) + (\ref{beat.1})$ and $(\ref{B3.estimate}) + (\ref{B4.estimate}) + (\ref{EST.2nd.order})$ for $L << 1$ and $w = w_0$, which produces 
\begin{align*}
&||||v||||_w^2 \lesssim \eps^{\frac{3}{8}} |||q|||_1^2 + |||q|||_{\sqrt{\eps}w}^2 + |||q|||_{w_y}^2 \\
& \hspace{20 mm} + |(F, \eps^2 v_{xxxx} w^2 + \eps v_{xxyy} u_s w^2 + \eps^{-\frac{1}{8}} \eps u_s v_{xxyy} )|, \\
&|||q|||_w^2 \lesssim o_L(1) ||||v||||_w^2 + o_L(1) \| q_{xx} \|_{w_y}^2 + |( F ,[q_{xx} + q_{yy} + q_x ] w^2)|
\end{align*}

\noindent From here, using the inequality $|w_{0y}| \lesssim 1 + \sqrt{\eps}|w_0|$, we conclude the $Y_{w_0}$ estimate.  
\end{proof}

Our aim now is to estimate each of the quantities appearing on the right-hand sides of (\ref{bring1}). We do this in a sequence of lemmas. 

\begin{lemma} \label{lemma.bb1} Let $\bold{u} \in \mathcal{X}$ as in (\ref{defn.norms.ult}). Let $C(h)$ denote a constant that is $\bigO(\| h \|_{\mathcal{C}^{M_0}(e^y)})$ for a large $M_0$. Then for $\mathcal{B}_{X_1}, \mathcal{B}_{Y_{w_0}}$ defined as in (\ref{defn.BX1}), (\ref{defn.BYw}), the following estimates are valid
\begin{align}
&|\mathcal{B}_{X_1}| \le o(1) \| v \|_{X_1}^2 + C \eps^{-\frac{1}{2}} [\bar{u}^0, \bar{v}^0]_B^2, \\
&|\mathcal{B}_{Y_{w_0}}| \le o(1) \| v \|_{Y_{w_0}}^2 + [\bar{u}^0, \bar{v}^0]_B^2. 
\end{align}
\end{lemma}
\begin{proof} Recall the specification of $B_{(v^0)}(\bar{v}^0)$ given in (\ref{portrait}). Recall also the specification of the norms (\ref{defn.norms.ult}) and (\ref{d.norm.B}). The inequality (\ref{systemat1}) will be in constant use throughout the proof of this lemma. 

\textit{Step 1: $q_{xx}$ Multiplier}

\begin{align} \label{pm1}
&|( B_{(v^0)} ,q_{xx} w^2)| \lesssim \begin{cases}  \eps^{-1/2} [ \bar{u}^0, \bar{v}^0 ]_{B}^2 + o(1) |||q|||^2 \text{ if } w = 1 \\  \eps^{-1}  [\bar{u}^0, \bar{v}^0]_B^2 + o(1) |||q|||_w^2 \text{ if } w = w_0 \end{cases}.
\end{align}

Recall the specification of $B_{(v^0)}$ given in (\ref{portrait}). We compute 
\begin{align} \label{spl.1}
(\bar{v}^0_{yyyy}, q_{xx} w^2) = & (\bar{v}^0_{yyyy}, q_{xx} w^2 \chi(y)) + (\bar{v}^0_{yyyy}, q_{xx} w^2 \{1 - \chi(y)\}).
\end{align}

\noindent For ease of notation, denote $\chi^C( y) := 1 - \chi( y)$. For the localized quantity, we estimate
\begin{align*}
(\ref{spl.1}.1) = & (\bar{v}^0_{yyyy}, q_x w^2\chi)|_{x = L} - (\bar{v}^0_{yyyy}, q_x w^2\chi)_{x = 0} \\
= & - (\bar{v}^0_{yyyy}, \frac{u_{sx}}{u_s} q w^2 \chi)|_{x = L} - (\bar{v}^0_{yyyy}, q_x w^2 \chi)_{x = 0} \\
\lesssim & \sqrt{L}  \| \bar{v}^0_{yyyy} w_0\| \| q_{xy} \| +\| \bar{v}^0_{yyyy} w_0 \| \Big( \|q_x \chi \|^{1/2} \| q_{xx} \|^{1/2}  \Big) \\
\lesssim &  \sqrt{L}  \| \bar{v}^0_{yyyy} w_0\| \| q_{xy} \| +   \eps^{-\frac{1}{4}} \| \bar{v}^0_{yyyy} w_0 \| \| \frac{q_x}{y}  \|^{1/2} \| \sqrt{\eps}q_{xx} \|^{1/2} \\
\lesssim & \eps^{-\frac{1}{4}}[\bar{u}^0, \bar{v}^0]_B |||q|||_1,
\end{align*}

\noindent where we have used $q_x^2|_{x = 0} = q_{x}^2|_{x = L} + 2I_L[q_{x}q_{xx}]$, and (\ref{defn.norms.ult}), (\ref{d.norm.B}), and (\ref{systemat1}).

For the far field quantity, we integrate by parts to produce 
\begin{align*}
(\ref{spl.1}.2) = & - (\bar{v}^0_{yyy}, q_{xxy} w^2 \chi^C )) - (\bar{v}^0_{yyy}, q_{xx} 2ww_y \chi^C) \\
& - (\bar{v}^0_{yyy}, q_{xx} w^2 (\chi^C)') \\
= & - (\bar{v}^0_{yyy}, q_{xy} w^2 \chi^C)_{x = 0} + (\bar{v}^0_{yyy}, q_{xy} w^2 \chi^C)_{x = L}  \\
& - (\bar{v}^0_{yyy}, q_{xx} 2ww_y \chi^C) - (\bar{v}^0_{yyy}, q_{xx} w^2(\chi^{C})) \\
= & - (\bar{v}^0_{yyy}, q_{xy} w^2 \chi^C)_{x = 0} - (\bar{v}^0_{yyy}, \p_y \{ \frac{u_{sx}}{u_s}q \} w^2 \chi^C)_{x = L}  \\
& - (\bar{v}^0_{yyy}, q_{xx} 2ww_y \chi^C) - (\bar{v}^0_{yyy}, q_{xx} w^2  (\chi^{C})') \\
= & - (\bar{v}^0_{yyy}, q_{xy} w^2 \chi^C)_{x = 0} - (\bar{v}^0_{yyy}, \p_y \{ \frac{u_{sx}}{u_s}q \} w^2 \chi^C)_{x = L}  \\
& - (\bar{v}^0_{yyy}, q_{xx} 2ww_y \chi^C) - (\bar{v}^0_{yyy}, q_x w^2 (\chi^C)')_{x = L} \\
& + (\bar{v}^0_{yyy}, q_x w^2 (\chi^C)')_{x = 0} \\
\lesssim &  \| \bar{v}^0_{yyy} w_0 \| \| u_s q_{xy} w \|_{x = 0} +  \sqrt{L}\| \bar{v}^0_{yyy} w_0 \| \|q_{xy} w \| \\
& + \| \bar{v}^0_{yyy} w_0 \| \|q_{xx} w_y \| +  \eps^{-1/4} \| \bar{v}^0_{yyy} w_0 \| \| q_{xy} \|^{1/2} \|\sqrt{\eps}q_{xx} \|^{1/2} \\
\lesssim & [\bar{u}^0, \bar{v}^0]_B \Big( \eps^{-\frac{1}{2}} |||q|||_{w_y} + \eps^{-\frac{1}{4}} |||q|||_w \Big).
\end{align*}

\noindent Above, we have used that $|\frac{1}{u_s}| \chi^C \lesssim \frac{1}{\bar{u}^0_p} \chi^C \lesssim 1$. We have also used the same estimates as in (\ref{spl.1}.1) for $| (\bar{v}^0_{yyy}, q_x w^2 (\chi^C)')_{x= L} - (\bar{v}^0_{yyy}, q_x w^2 (\chi^C)')_{x = 0} |$.

We next compute 
\begin{align*}
-2(\p_y \{u_s u_{sx} \bar{q}^0_y \}, q_{xx}w^2 ) =& 2( u_s u_{sx} \bar{q}^0_y, q_{xxy} w^2) + 4(u_s u_{sx} \bar{q}^0_y, q_{xx} ww_y) \\
= & - 2(\p_x \{ u_s u_{sx} \} \bar{q}^0_y, q_{xy}w^2) + 2(u_s u_{sx}\bar{q}^0_y, q_{xy}w^2)_{x = L} \\
& - 2(u_s u_{sx} \bar{q}^0_y, q_{xy}w^2)_{x = 0} + 4 (u_s u_{sx} \bar{q}^0_y, q_{xx}ww_y) \\
= & - 2(\p_x \{ u_s u_{sx} \} \bar{q}^0_y, q_{xy}w^2) - 2(u_s u_{sx}\bar{q}^0_y, \p_y \{ \frac{u_{sx}}{u_s}q \}w^2)_{x = L} \\
& - 2(u_s u_{sx} \bar{q}^0_y, q_{xy}w^2)_{x = 0} + 4 (u_s u_{sx} \bar{q}^0_y, q_{xx}ww_y) \\
\lesssim &  \sqrt{L} \| u_s \bar{q}^0_y w_0 \| \| q_{xy} w \| +  L \| \sqrt{u_s} \bar{q}^0_y w_0 \| \| q_{xy} w \| \\
& + \| u_s \bar{q}^0_y \| \| u_s q_{xy} w \|_{x = 0} +  \sqrt{L} \| u_s \bar{q}^0_y \| \| q_{xx} w_y \| \\
\lesssim &[\bar{u}^0, \bar{v}^0]_B \Big( |||q|||_w + \eps^{-1/2} |||q|||_{w_y} \Big)
\end{align*}

We compute 
\begin{align} \n
&(\p_x \{ (x+1)v_s \} \bar{v}^0_{yyy} - \p_x \{ (x+1)v_{syy} \} \bar{v}^0_y, q_{xx}w^2) \\ \n
= & - (\p_x \{ (x+1)v_{s}\} \bar{v}^0_{yy} - \p_x \{(x+1)v_{sy} \} \bar{v}^0_y, q_{xxy}w^2 + 2q_{xx}ww_y) \\ \n
= & (\p_{xx} \{ (x+1)v_s \} \bar{v}^0_{yy} - \p_{xx} \{(x+1)v_{sy} \} \bar{v}^0_y, q_{xy} w^2) - (\p_x \{(x+1)v_s \} \bar{v}^0_{yy} \\ \n
& - \p_x \{(x+1)v_{sy} \} \bar{v}^0_y, q_{xy}w^2)_{x = L} + (\p_x \{(x+1)v_s \} \bar{v}^0_{yy}  \\
\label{weget}& - \p_x \{(x+1)v_{sy} \} \bar{v}^0_y, q_{xy}w^2)_{x = 0}   - (\p_x \{ (x+1)v_{s}\} \bar{v}^0_{yy} - \p_x \{(x+1)v_{sy} \} \bar{v}^0_y,  2q_{xx}ww_y).
\end{align}

\noindent Above, we have used the identity 
\begin{align*}
\p_x \{ (x+1)v_s \} \bar{v}^0_{yyy} - \p_x \{ (x+1)v_{syy} \} \bar{v}^0_y = \p_y \{ \p_x \{ (x+1)v_s \} \bar{v}^0_{yy} - \p_x \{ (x+1)v_{sy} \} \bar{v}^0_y  \}.
\end{align*}

We estimate the first term in the $w = 1$ case: 
\begin{align*}
|(\ref{weget}.1)| \le & \|(x+1)v_{sxx} + 2 v_{sx} \|_\infty \| \bar{v}^0_{yy} \| q_{xy} \| + \| (x+1) u_{sxxx} + 2u_{sxx} \|_\infty \| \bar{v}^0_y \| \|q_{xy} \| \\
\lesssim & [\bar{u}^0, \bar{v}^0]_B |||q|||_1 
\end{align*}

We next expand 
\begin{align*}
(\ref{weget}.2) = & (\p_x \{ (x+1)v_s \} \bar{v}^0_{yy} - \p_x \{(x+1)v_{sy} \} \bar{v}^0_y, \p_y \{ \frac{u_{sx}}{u_s} q \} w^2)_{x = L} \\
= & (\p_x \{ (x+1)v_s \} \bar{v}^0_{yy} - \p_x \{ (x+1)v_{sy} \} \bar{v}^0_y, \p_y \{ \frac{u_{sx}}{u_s} \} q w^2)_{x = L} \\
& + (\p_x \{ (x+1)v_s \} \bar{v}^0_{yy} - \p_x \{ (x+1)v_{sy} \} \bar{v}^0_y,  \frac{u_{sx}}{u_s} q_y w^2)_{x = L}
\end{align*}

Thus, in the $w = 1$ case, we estimate 
\begin{align*}
&\| (x+1)v_{sx} + v_s \|_\infty \| y \p_y \{ \frac{u_{sx}}{u_s} \} + \frac{u_{sx}}{u_s} \|_\infty \| \bar{v}^0_{yy}\| \| q_{xy} \| \sqrt{L} \\
&+ \|(x+1) u_{sxx} + u_{sx} \|_\infty \| y \p_y \{ \frac{u_{sx}}{u_s} \} + \frac{u_{sx}}{u_s} \|_\infty \| \bar{v}^0_y \| \| q_{xy} \| \sqrt{L} \\
\lesssim & [\bar{u}^0, \bar{v}^0]_B |||q|||.
\end{align*}

We next continue with $w = 1$ to estimate 
\begin{align*}
|(\ref{weget}.3)| \lesssim & \| \frac{(x+1)v_{sx} + v_s}{u_s} \|_\infty \| \bar{v}^0_{yy} \| \| u_s q_{xy} \|_{x = 0} + \| \frac{v_{sy} + (x+1)v_{sxy}}{u_s}\|_\infty \| \bar{v}^0_y \| \| u_s q_{xy} \|_{x = 0} \\
\lesssim &  [\bar{u}^0, \bar{v}^0]_B |||q|||_1 
\end{align*}

This concludes the $w = 1$ case, and we move on to the $w = w_0$ case. We first record using (\ref{exp.u}), the following estimate 
\begin{align} \label{Eul.vs.1}
\| \p_x^k v_s w_0 \|_\infty \lesssim \| \p_x^k \{v^0_p + v^1_e \} \langle y \rangle \langle Y \rangle^m\|_\infty + \bigO(1) \lesssim \eps^{-\frac{1}{2}}.
\end{align}

We begin with the following, using (\ref{Eul.vs.1}):
\begin{align*}
|(\ref{weget}.1)| \lesssim & \| \p_{xx} \{ (x+1)v_s \} w_0 \|_\infty \| \bar{v}^0_{yy} \| \| q_{xy} w_0 \| + \| \p_{xx} \{ (x+1)v_{sy} \} w_0 \|_\infty \| \bar{v}^0_y \| \| q_{xy} w \| \\
\lesssim & \eps^{-1/2} \| \bar{v}^0_{yy} \| q_{xy} w_0 \| +  \|\bar{v}^0_y \| \|q_{xy} w\| \\
\lesssim &  \eps^{-1/2} [\bar{u}^0, \bar{v}^0]_B |||q|||_{w}.
\end{align*}

We move to the (\ref{weget}.2) for which 
\begin{align*}
|(\ref{weget}.2)| \lesssim & \| (x+1)v_{sx} + v_s \|_\infty \| w_0 y \p_y \{ \frac{u_{sx}}{u_s} \} + w \frac{u_{sx}}{u_s} \|_\infty \| \bar{v}^0_{yy} \| \|q_{xy} w_0 \| \sqrt{L} \\
& + \| (x+1) u_{sxx} + u_{sx}\|_\infty \| w_0 y \p_y \{ \frac{u_{sx}}{u_s} \} + w_0 \frac{u_{sx}}{u_s} \|_\infty \| \bar{v}^0_y \| \| q_{xy} \| \sqrt{L} \\
\lesssim &  \sqrt{L}  [\bar{u}^0, \bar{v}^0]_B |||q|||_w
\end{align*}

\noindent Next, again recalling (\ref{Eul.vs.1}), 
\begin{align*}
|(\ref{weget}.3)| \lesssim & \| \frac{\p_x \{ (x+1)v_s \}}{u_s} w_0 \|_\infty \| \bar{v}^0_{yy} \| \| u_s q_{xy} w_0 \|_{x = 0} \\
& + \| \frac{\p_x \{(x+1)v_{sy} \}}{u_s} w_0 \|_\infty \| \bar{v}^0_y \| \| u_s q_{xy} w_0 \|_{x = 0} \\
\lesssim & \eps^{-1/2} [\bar{u}^0, \bar{v}^0]_B |||q|||_{w_0}.
\end{align*}

Last, again using (\ref{Eul.vs.1}),
\begin{align*}
|(\ref{weget}.4)| \lesssim & \eps^{-1/2} \| \p_x \{ (x+1)v_s \} w_y \|_\infty \| \bar{v}^0_{yy} \| \sqrt{\eps}q_{xx} w \| \\
& + \eps^{-1/2} \| \p_x \{ (x+1)v_{sy} \} w_y \|_\infty \| \bar{v}^0_y \| \| \sqrt{\eps}q_{xx} w \| \\
\lesssim & \eps^{-1/2}[\bar{u}^0, \bar{v}^0]_B |||q|||_w. 
\end{align*}

Finally, upon using again (\ref{Eul.vs.1}),
\begin{align*}
|( \eps \bar{v}^0_y \p_{x} \{ (x+1)v_{sxx} \}, q_{xx}w^2)| \lesssim & \sqrt{L} \| \bar{v}^0_y \| \| \sqrt{\eps} q_{xx} w \| \\
\lesssim &  \sqrt{L}[\bar{u}^0, \bar{v}^0]_B |||q|||_w. 
\end{align*}

\noindent This concludes the $B_{(v^0)}$ terms for this multiplier.

\textit{Step 2: $q_{yy}$ Multiplier}
\begin{align}  \label{pm2}
|(B_{(v^0)} , q_{yy}w^2)|  \lesssim o(1) |||q|||_{w,1}^2 + \sqrt{L} [ \bar{u}^0, \bar{v}^0 ]_B^2 .
\end{align}

Recall again the specification of $B_{(v^0)}$ given in (\ref{portrait}). We begin with
\begin{align*}
|(\bar{v}^0_{yyyy}, q_{yy} w^2)| \le \sqrt{L} \| \bar{v}^0_{yyyy} w \| \| q_{yy} w \|.
\end{align*}

Second, 
\begin{align*}
&-2(\p_y \{ u_s u_{sx}\bar{q}^0_y \}, q_{yy}w^2) \\
= &- 2( \p_y \{ u_s u_{sx} \} \bar{q}^0_y, q_{yy}w^2) - 2 (u_s u_{sx} \bar{q}^0_{yy}, q_{yy}w^2) \\
\lesssim &  L \| u_s \bar{q}^0_y \| \| q_{yy} w \| + L \| u_s \bar{q}^0_{yy} \| \| q_{yy} w \| \\
\lesssim &  [\bar{u}^0, \bar{v}^0] |||q|||_w.
\end{align*}

Next, 
\begin{align} \n
&|(\p_x \{(x+1)v_s \} \bar{v}^0_{yyy} - \p_x \{ (x+1)v_{syy} \} \bar{v}^0_y, q_{yy}w^2 )| \\ \n
\lesssim & \sqrt{L} [  \| \bar{v}^0_{yyy} w_0 \| + \| \bar{v}^0_y \| ]  \| q_{yy} w \| \\ \label{identical.1}
\lesssim & \sqrt{L} [\bar{u}^0, \bar{v}^0]_B |||q|||_{w},
\end{align}

Finally, 
\begin{align*}
|(\eps \p_x \{(x+1)v_{sxx} \} \bar{v}^0_y, q_{yy} w^2 )| \lesssim \sqrt{L} \sqrt{\eps} \|\bar{v}^0_y \| \| q_{yy} w \|, 
\end{align*}

\noindent using the bound 
\begin{align*}
\| \eps y \{ v_{sxxx} x + v_{sxx} \} \| \lesssim \sqrt{\eps}.
\end{align*}

\textit{Step 3: $q_x$ Multiplier}

\begin{align} \n
|(B_{(v^0)}, q_x w^2)| \lesssim & L^2 \| q_{xx} w_y \|^2 \\ \label{pm3}
&+ o_L(1) |||q|||_w^2 + o_L(1)  \begin{cases}  [ \bar{u}^0, \bar{v}^0 ]_{B}^2 \text{ if } w = 1, \\  \eps^{-1} [\bar{u}^0, \bar{v}^0 ]_{B}^2 \text{ if } w = w_0  \end{cases}.
\end{align}

Recall again the specification of $B_{(v^0)}$ given in (\ref{portrait}). We begin with (letting $w$ be either $w_0$ or $1$ for this calculation)
\begin{align*}
|(\bar{v}^0_{yyyy}, q_x w^2)| =& |- (\bar{v}^0_{yyy}, q_{xy}w^2) - 2(\bar{v}^0_{yyy}, q_x ww_y)| \\
\lesssim & \sqrt{L} \| \bar{v}^0_{yyy} w_0 \| \| q_{xy} w \| + \sqrt{L} \| \bar{v}^0_{yyy} w \| \{ \sqrt{L} \| \sqrt{\eps}q_{xx} w \| + \|\frac{q_x}{y} w \| \}, 
\end{align*}

\noindent where we have used that $|w_y| \lesssim \sqrt{\eps}|w| + 1$, which is true for both choices of $w$. 

Next, 
\begin{align} \n
&(\p_x \{ (x+1)v_s \} \bar{v}^0_{yyy} - \p_x \{ (x+1)v_{syy} \} \bar{v}^0_y, q_x w^2) \\ \label{djd}
= & (\p_x \{ (x+1)v_{sy} \} \bar{v}^0_y - \p_x \{(x+1)v_s \} \bar{v}^0_{yy}, q_{xy}w^2 + 2q_x ww_y)
\end{align}

\noindent We must now distinguish the weights for $w = 1$ and $w = w_0$. In the case $w = 1$, we majorize the above quantity by 
\begin{align*}
|(\ref{djd})| \lesssim &\Big( \| v_{sy} + (x+1)v_{sxy} \|_\infty \| \bar{v}^0_y \| + \| v_s + (x+1)v_{sx} \|_\infty \| \bar{v}^0_{yy} \| \Big) \| q_{xy} \| \\
\lesssim & [\bar{u}^0, \bar{v}^0]_B |||q|||.
\end{align*}

\noindent In the case of $w = w_0$, 
\begin{align*}
|(\ref{djd})| \lesssim & [\| \{(x+1)v_{sxy} + v_{sy} \} w_y \|_\infty \| \bar{v}^0_y \| + \| \{v_s + (x+1)v_{sx} \} w_y \|_\infty \| \bar{v}^0_{yy} \| ] \\
& \times [\| q_{xy} w \| + \eps^{-1/2} L \|\sqrt{\eps} q_{xx} w \| ] \\
\lesssim & \eps^{-1/2} L [\bar{u}^0, \bar{v}^0]_B |||q|||_w
\end{align*}

Next, 
\begin{align*}
(\p_y \{ u_s u_{sx} \bar{q}^0_y \}, q_x w^2) = &  - (u_s u_{sx} \bar{q}^0_y, q_{xy}w^2) - (u_s u_{sx} \bar{q}^0_y, q_x 2ww_y).
\end{align*}

\noindent We again distinguish between the case of $w = 1$ and $w = w_0$. In the $w = 1$ case, we estimate by 
\begin{align*}
\sqrt{L}  \| \bar{q}^0_y u_s \| \| q_{xy} \| \lesssim \sqrt{L}  [\bar{u}^0, \bar{v}^0]_B |||q|||_1. 
\end{align*}

\noindent whereas in the $w = w_0$ case, we majorize by 
\begin{align*}
&\sqrt{L}  \| \bar{q}^0_y u_s \| \| q_{xy} w \| + \sqrt{L}  \| u_s \bar{q}^0_y \| \{ L\|\sqrt{\eps}q_{xx} w \| + \| \frac{q_x}{\langle y \rangle} w \|  \} \\
\lesssim & (\sqrt{L}  + L^{3/2} + \sqrt{L} )[\bar{u}^0, \bar{v}^0] |||q|||_w
\end{align*}

We move to the final term. In the case $w = 1$, 
\begin{align*}
|(\eps \bar{v}^0_y \p_x \{(x+1)v_{sxx} \}, q_x)| \lesssim \sqrt{\eps} \| \bar{q}_{xy} \| \| v^0_y \| \lesssim  \sqrt{\eps} [\bar{u}^0, \bar{v}^0]_B |||q|||_1
\end{align*}

\noindent upon using that $|\sqrt{\eps}y \p_x \{(x+1)v_{sxx} \}| \lesssim 1$. In the case $w = w_0$, 
\begin{align*}
|(\eps \bar{v}^0_y \p_x \{(x+1)v_{sxx} \}, q_x)| \lesssim \| \frac{q_x}{y} w \| \| \bar{v}^0_y \| \lesssim [\bar{u}^0, \bar{v}^0]_B |||q|||_w. 
\end{align*}

\noindent upon using that $|\eps y w_0 \p_x \{ (x+1)v_{sxx} \}| \lesssim 1$. For these profile estimates, we have used (\ref{prof.pick}).

\textit{Step 4: $\eps v_{xxyy}u_s$ Multiplier}

\begin{align} \label{pm4}
|(B_{(v^0)}, \eps u_s v_{xxyy} w^2)| \le  C \sqrt{\eps} [\bar{u}^0, \bar{v}^0]_{B}^2 + o(1) \sqrt{\eps} |||q|||_{\sqrt{\eps}w}^2.
\end{align}

Recall again the specification of $B_{(v^0)}$ given in (\ref{portrait}). We compute 
\begin{align*}
&(\bar{v}^0_{yyyy}, \eps u_s v_{xxyy} w^2) \\
= & - (\bar{v}^0_{yyyy}, \eps u_{sx} v_{xyy} w^2) - (\bar{v}^0_{yyyy}, \eps u_s v_{xyy}w^2)_{x = 0} \\
\le & \sqrt{L}  \sqrt{\eps} \| \bar{v}^0_{yyyy} w \| \| v_{xyy} \sqrt{\eps} w \| + \sqrt{\eps} \| \bar{v}^0_{yyyy} w \| \| u_s v_{xyy} \sqrt{\eps}w \|_{x = 0} \\
\lesssim & \sqrt{L} \sqrt{\eps} [\bar{u}^0, \bar{v}^0]_B |||q|||_{\sqrt{\eps}w} + \sqrt{\eps}[\bar{u}^0, \bar{v}^0]_B ||||v||||_w
\end{align*}

Next, 
\begin{align} \n
&(\p_x \{(x+1)v_s \} \bar{v}^0_{yyy} -  \p_x \{(x+1)v_{syy} \}\bar{v}^0_y, \eps u_s v_{xxyy}w^2) \\ \n
= & -(\p_x \{(x+1)v_s \} \bar{v}^0_{yyy} - \p_x \{(x+1)v_{syy} \} \bar{v}^0_y, \eps u_s v_{xyy}w^2)_{x = 0} \\ \n
& - ( \p_x \{(x+1)v_s \} \bar{v}^0_{yyy} - \p_x \{(x+1)v_{syy} \}\bar{v}^0_y ,\eps u_{sx} v_{xyy}w^2) \\ \label{emmit}
& - ( \p_{xx} \{(x+1)v_{s} \} \bar{v}^0_{yyy} - \p_{xx} \{(x+1)v_{syy} \}\bar{v}^0_y ,\eps u_{s} v_{xyy}w^2).
\end{align}

\noindent First, 
\begin{align*}
|(\ref{emmit}.1)| \lesssim & (  \sqrt{\eps} \| \sqrt{u_s} \bar{v}^0_{yyy} w_0 \|  +  \sqrt{\eps} \| \bar{v}^0_y \| ) \| \sqrt{\eps}u_s v_{xyy} w \|_{x = 0} \\
\lesssim &  \sqrt{\eps} [\bar{u}^0, \bar{v}^0]_B ||||v||||_w +  \sqrt{\eps} [\bar{u}^0, \bar{v}^0]_B ||||v||||_w. 
\end{align*}

\noindent Next, 
\begin{align*}
|(\ref{emmit}.2)| + |(\ref{emmit}.3)| \lesssim \sqrt{L} \sqrt{\eps} ( \| u_s \bar{v}^0_{yyy}w_0 \| + \| \bar{v}^0_y \|) \| v_{xyy} \sqrt{\eps}w \|
\end{align*}

\begin{align} \n
&(-2\p_y \{ u_s u_{sx} \bar{q}^0_y \}, \eps u_s v_{xxyy}w^2) \\ \n
= &(2\p_y \{ u_s u_{sx} \bar{q}^0_y \} ,\eps u_{sx} v_{xyy}w^2) + (2\p_y \{ (u_s u_{sx})_x \bar{q}^0_y \}, \eps u_s v_{xyy}w^2) \\ \n
& + 2( \p_y \{ u_s u_{sx} \bar{q}^0_y \}, \eps u_s v_{xyy}w^2)_{x = 0} \\ \n
= &(2\p_y \{ u_s u_{sx} \} \bar{q}^0_y ,\eps u_{sx} v_{xyy}w^2)  + (2 u_s u_{sx} \bar{q}^0_{yy} ,\eps u_{sx} v_{xyy}w^2) \\ \n
& + (2 \{ (u_s u_{sx})_{xy} \bar{q}^0_y , \eps u_s v_{xyy}w^2) + (2 (u_s u_{sx})_{x} \bar{q}^0_{yy} , \eps u_s v_{xyy}w^2) \\ \label{analysis}
& + 2( \p_y \{ u_s u_{sx} \} \bar{q}^0_y, \eps u_s v_{xyy}w^2)_{x = 0} + 2(  u_s u_{sx} \bar{q}^0_{yy}, \eps u_s v_{xyy}w^2)_{x = 0} 
\end{align}

\noindent We begin with the first two terms. Since $u_{sx}$ decays at $y = \infty$, 
\begin{align*}
&|(\ref{analysis}.1)| \lesssim \sqrt{\eps} \sqrt{L}  \| u_s \bar{q}^0_y \| \| v_{xyy} \sqrt{\eps}w \| , \\
&|(\ref{analysis}.2)| \lesssim \sqrt{\eps} \sqrt{L}  \| \sqrt{u_s} \bar{q}^0_{yy} \| \| v_{xyy} \sqrt{\eps}w \|,
\end{align*}

Next, we estimate the third and fourth terms
\begin{align*}
&|(\ref{analysis}.3)| + |(\ref{analysis}.4)| \lesssim \sqrt{\eps} \sqrt{L} (\| u_s \bar{q}^0_y \| + \| u_s \bar{q}^0_{yy} \|) \| v_{xyy} \sqrt{\eps} w \|, 
\end{align*}

The last two terms follow very similarly from the first two, yielding
\begin{align*}
&|(\ref{analysis}.5)| \lesssim \sqrt{\eps}  \| u_s \bar{q}^0_y \| \| \sqrt{\eps} u_s v_{xyy} w \|_{x = 0}, \\
&|(\ref{analysis}.6)| \lesssim \sqrt{\eps}  \| \sqrt{u_s} \bar{q}^0_{yy} \| \| \sqrt{\eps} u_s v_{xyy} w \|_{x = 0}
\end{align*}

We finally move to 
\begin{align*}
&(\eps \p_x \{ (x+1)v_{sxx} \} \bar{v}^0_y, \eps u_s v_{xxyy}w^2 ) \\
=& - (\eps \p_{xx} \{ (x+1)v_{sxx} \} \bar{v}^0_y, \eps u_s v_{xyy}w^2) - (\eps \p_{x} \{(x+1)v_{sxx} \} \bar{v}^0_y, \eps u_{sx} v_{xyy}w^2) \\
& - (\eps \p_x \{ (x+1)v_{sxx} \}\bar{v}^0_y, \eps u_s v_{xyy}w^2)_{x = 0} \\
\lesssim & \eps \sqrt{L}\| \bar{v}^0_y \| \| \sqrt{\eps}u_s v_{xyy}w \|_{x = 0} + \sqrt{L} \eps \| \bar{v}^0_y \| \| \sqrt{\eps}v_{xyy}w \|  \\
& + \eps \| \bar{v}^0_y \| \| \sqrt{\eps}u_s v_{xyy}w \|_{x = 0}
\end{align*}

\textit{Step 5: $\eps^2 v_{xxxx}$ Multiplier}

\begin{align} \label{pm5}
|(B_{(v^0)}, \eps^2 v_{xxxx} w^2)| \le  C \eps [\bar{u}^0, \bar{v}^0]_{B}^2 + \eps |||q|||_{\sqrt{\eps}w}^2.
\end{align}

This follows in the same manner is the previous multiplier. Putting together estimates (\ref{pm1}), (\ref{pm2}), (\ref{pm3}), (\ref{pm4}), (\ref{pm5}) according to the linear combinations in (\ref{defn.BX1}) and (\ref{defn.BYw}) gives the desired bound and completes the proof of the lemma. 
\end{proof}

\begin{lemma} \label{lemma.bb2} Let $\bold{u} \in \mathcal{X}$ as in (\ref{defn.norms.ult}). For $\mathcal{B}_B$ defined as in (\ref{mathB}), the following estimates are valid  
\begin{align}
|\mathcal{B}_{B}| \lesssim \eps \| \bar{v} \|_{Y_{w_0}}^2 + \eps^{\frac{1}{2}+ \frac{3-}{16}} \| \bar{v} \|_{X_1}^2
\end{align}
\end{lemma}
\begin{proof} We estimate each term in $F_{(v)}(\bar{v})$ which we write here for convenience: 
\begin{align} \label{Fv}
F_{(v)} := -2\eps u_s u_{sx}\bar{q}_x|_{x = 0} - 2\eps \bar{v}_{xyy}|_{x =0} - \eps^2 \bar{v}_{xxx}|_{x = 0} + \eps v_s \bar{v}_{xy}|_{x = 0}.
\end{align}

Starting with the higher order terms,
\begin{align*}
\|\eps^2 \bar{v}_{xxx} w \|_{x = 0} \le& \| \eps^2 \bar{v}_{xxx} w\{1- \chi\} \|_{x = 0} +  \| \eps^2 \bar{v}_{xxx} w  \chi  \|_{x = 0} \\
\le & \sqrt{\eps} \| \eps^{\frac{3}{2}}u_s \bar{v}_{xxx} w \|_{x = 0} + \eps^{\frac{1}{2}} \| \eps \bar{v}_{xxx} \|^{\frac{1}{2}} \| \eps^2 \bar{v}_{xxxx} \|^{\frac{1}{2}} \\
\lesssim & \sqrt{\eps} \| \bar{v} \|_{Y_w} + \eps^{\frac{1}{2}+\frac{3}{32}} \| \bar{v} \|_{X_1}.
\end{align*}

\noindent The identical argument is performed for (\ref{Fv}.2). 

For the fourth term, we expand $\bar{v}_{xy}|_{x = 0} = u_s \bar{q}_{xy}|_{x = 0} + u_{sy}\bar{q}_x|_{x = 0}$, perform a Hardy type inequality for the $\bar{q}_x$ term, and use (\ref{gp.2.5}) to obtain
\begin{align*}
\| \eps v_s \bar{v}_{xy} w \|_{x = 0} \le& \| \eps v_s u_s \bar{q}_{xy} w \|_{x = 0} + \| \eps v_s u_{sy} \bar{q}_x w \|_{x = 0} \\
\le & \sqrt{\eps} \| \bar{v} \|_{Y_w} + \eps \| v_s u_{sy} \bar{q}_x \chi \|_{x = 0} \\
\le & \sqrt{\eps} \| \bar{v} \|_{Y_w} + \eps^{\frac{3}{4}} \| \eps^{\frac{1}{4}}\frac{\bar{q}_x}{\langle y \rangle}\|_{x = 0} \\
\le & \sqrt{\eps} \| \bar{v} \|_{Y_w} + \eps^{\frac{3}{4}} \| \bar{v} \|_{X_1}.
\end{align*}

To estimate the first term from (\ref{Fv}), we split into Euler and Prandtl: 
\begin{align*}
\| \eps u_s u_{sx}\bar{q}_x w \|_{x = 0} \le& \| \eps u_s u^P_{sx} \bar{q}_x w \|_{x = 0} + \eps^{\frac{3}{2}} \| u_s u^E_{sx} \bar{q}_x w \|_{x = 0} \\
\le & \| u^P_{sx} w \langle y \rangle \|_\infty \eps \| \frac{\bar{q}_x}{\langle y \rangle} \|_{x = 0} + \sqrt{\eps} [\|\sqrt{\eps}\bar{q}_x \sqrt{\eps}w \| + \| \sqrt{\eps} \bar{q}_{xx} \sqrt{\eps}w \|] \\
\lesssim & \eps^{\frac{3}{4}} \| \bar{v} \|_{X_1} + \sqrt{\eps} \| \bar{v} \|_{Y_w}.
\end{align*}

We have thus established: 
\begin{align*}
\| F_{(v)} w_0 \| \lesssim \sqrt{\eps} \| \bar{v} \|_{Y_{w_0}} + \eps^{\frac{1}{2}+\frac{3}{32}} \| \bar{v} \|_{X_1}. 
\end{align*}

This concludes the proof. 
\end{proof}

\begin{lemma} \label{lemma.nonlinear} Let $\bold{u} \in \mathcal{X}$ as in (\ref{defn.norms.ult}). The following estimates are valid
\begin{align} \label{showin}
&|\mathcal{N}_{X_1}| \lesssim \eps^{N_0-1} [\bar{u}^0, \bar{v}^0 ]_B \| \bar{v} \|_{X_1} \| v \|_{X_1} +\eps^{N_0-\frac{3}{4}} \| \bar{v} \|_{X_1} \| \bar{v} \|_{X_1} \| v \|_{X_1}, \\ \n
& \hspace{20 mm} + \eps^{N_0-1} [\bar{u}^0, \bar{v}^0]_B^2 \| v \|_{X_1}, \\ \label{showin2}
&|\mathcal{N}_{Y_w}| \lesssim \eps^{N_0-1} [ \bar{u}^0, \bar{v}^0 ]_B \| \bar{v} \|_{Y_w} \| v \|_{Y_w} +\eps^{N_0-\frac{3}{4}} \| \bar{v} \|_{X_1} \|\bar{v} \|_{Y_w} \| v \|_{Y_w}  \\ \n
& \hspace{20 mm} + \eps^{N_0-1} [\bar{u}^0, \bar{v}^0]_B^2 \| v \|_{X_1}, \\ \label{ree}
&|\mathcal{N}_B| \lesssim \eps^{N_0-1}[\bar{u}^0, \bar{v}^0]^4 + C(h, a_2^\eps) + \eps^{N_0-1} \| \bar{v} \|_{Y_w}^2.
\end{align}
\end{lemma}
\begin{proof}   \textit{Proof of (\ref{showin}) - (\ref{showin2}):} We begin with the immediate estimates: 
\begin{align*}
|\mathcal{N}_{X_1}| \lesssim  \eps^{-\frac{1}{2}} \|\mathcal{N} \| \| v \|_{X_1} \hspace{3 mm} |\mathcal{N}_{Y_{w}}| \lesssim  \| \mathcal{N} w \| \| v \|_{Y_w}.
\end{align*}

First, recall the specification of $\mathcal{N} = Q_{11} + Q_{12} + Q_{22}$ given in (\ref{spec.nl}). We now establish the following bound:
\begin{align*}
\| \mathcal{N} \cdot w \| \lesssim \{\eps^{-1/2}[ \bar{u}^0, \bar{v}^0 ]_B + \eps^{-\frac{1}{4}} \|  \bar{v} \|_{X_1} \}||| \bar{q} |||_w. 
\end{align*}

To establish this, we go term by term through $Q_{11}$: 
\begin{align*}
&\| \bar{v}_y \Delta_\eps \bar{v} w \| \le \| \bar{v}_y \|_\infty \| \Delta_\eps \bar{v} w \| \\
&\|I_x[\bar{v}_y] \Delta_\eps \bar{v}_x w \| \le \| \bar{v}_y \|_\infty \| \Delta_\eps \bar{v}_x w \| \\
&\| \bar{v}_x I_x[\bar{v}_{yyy}] w \| \le \eps^{-\frac{1}{4}} \| \eps^{\frac{1}{4}} \bar{v}_x \|_{L^2_x L^\infty_y} \| \bar{v}_{yyy} w \| \\
&\| \bar{v} \bar{v}_{yyy} w \| \le \eps^{-\frac{1}{4}} \| \eps^{\frac{1}{4}} \bar{v} \|_{L^2_x L^\infty_y} \| \bar{v}_{yyy} w \| \\
&\| \eps \bar{v}_x \bar{v}_{xy}w \| \le \sqrt{\eps} \| \sqrt{\eps} \bar{v}_x \|_\infty \| \bar{v}_{xy} w \| \\
&\| \bar{v} \Delta_\eps \bar{v}_y w \| \le \eps^{-\frac{1}{4}} \| \eps^{\frac{1}{4}} \bar{v} \|_\infty \| \Delta_\eps \bar{v}_y w \| \\
&\| \bar{u}^0 \Delta_\eps \bar{v}_x w \| \le | \bar{u}^0 |_\infty \| \Delta_\eps \bar{v}_x w \| \lesssim [ \bar{u}^0, \bar{v}^0 ]_B |||\bar{q}|||_w, \\
&\| \bar{u}^0_{yy} \bar{v}_x w \| \le \| \bar{u}^0_{yy} \langle y \rangle \|_{L^\infty_x L^2_y} \| \bar{v}_x \langle y \rangle^{-1} w \|_{L^2_x L^\infty_y} \lesssim [\bar{u}^0, \bar{v}^0 ]_{B} |||\bar{q}|||_w.
\end{align*}

Above, we have used the following interpolation: 
\begin{align*}
\| \bar{v}_x \langle y \rangle^{-\frac{1}{2}} w \|_{L^2_x L^\infty_y} \le & \| \bar{v}_x \langle y \rangle^{-1} w \|^{\frac{1}{2}} \| \bar{v}_{xy} w\|^{\frac{1}{2}},
\end{align*}

\noindent  and the weighted Hardy's inequality, (\ref{w.Hardy}). The result follows upon remarking the following basic fact. For any function $g(x,y)$ such that $g_{x = 0 \text{ OR } x= L} = 0$ and $g|_{y = \infty} = 0$: $|g|^2 \le \| g_x \| \| g_y \| + \| g \| \| g_{xy} \|$. This immediately gives: $\| \eps^{\frac{1}{4}} v \|_\infty + \| v \langle y \rangle^{-1/2} \|_\infty + \|\nabla_\eps v \|_\infty \lesssim |||q|||_1$. A basic interpolation also gives $\| \eps^{\frac{1}{4}} \bar{v}_x \|_{L^\infty_y} \le \| \sqrt{\eps} v_x \|_{L^2_y}^{\frac{1}{2}} \|v_{xy} \|_{L^2_y}^{\frac{1}{2}}$.

We treat now the quantity $\| Q_{12} w\|$
\begin{align*}
&\| v_y v^0_{yy} w \| \le \| v_{xy} w \| \| v^0_{yy} \|_\infty \lesssim [u^0, v^0]_B |||q|||_w, \\
&\| v^0_y \Delta_\eps v w\| \le \| v^0_y \|_\infty \| \Delta_\eps v w \| \lesssim [u^0, v^0]_B |||q|||_w, \\
&\| x v^0 \Delta_\eps v_x w \| \lesssim \| v^0 \|_\infty \| \Delta_\eps v_x w \| \lesssim \eps^{-\frac{1}{2}} [u^0, v^0]_B |||q|||_w \\
& \| x v_x v^0_{yy} w\| \lesssim \eps^{-1/2} \| \sqrt{\eps}v_{xx} w \| \| v^0_{yy} \|_\infty \lesssim \eps^{-1/2} [u^0, v^0]_B |||q|||_w, \\
& \| v v^0_{yyy}w \| \le \eps^{-\frac{1}{4}} \| \eps^{\frac{1}{4}}v \|_{L^2_x L^\infty_y} \| v^0_{yyy} w \| \lesssim \eps^{-\frac{1}{4}} [u^0, v^0]_B |||q|||_w, \\
&\| v^0 \Delta_\eps v_y w \| \le \| v^0 \|_\infty \| \Delta_\eps v_y w \| \lesssim \eps^{-\frac{1}{2}} [u^0, v^0]_B |||q|||_w. 
\end{align*}

To conclude, we note that the $Q_{22}$ terms have already been treated in Lemmas \ref{Lem.double} and \ref{Lem.triple}.

\textit{Proof of (\ref{ree})} Recall the specification of $\mathcal{Q}$ from (\ref{nawa}). We begin with the multiplier of $q^0$.  First, 
\begin{align*}
\eps^{N_0} ( v^0 v^0_{yyy} - v^0_y v^0_{yy}, q^0 ) = & \eps^{N_0} (u_s q^0 \p_{yyy} \{ u_s q^0 \} - \p_y \{ u_s q^0 \} \p_{yy} \{ u_s q^0 \}, q^0) \\
= & \eps^{N_0} (u_s^2 q^0 q^0_{yyy} - u_s^2 q^0_y q^0_{yy}, q^0) + \mathcal{J}_2.
\end{align*}

\noindent Here, 
\begin{align*}
\mathcal{J}_2 := &\eps^{N_0} (u_s q^0[u_{syyy}q^0 + 3 u_{syy}q^0_y + 3 u_{sy}q^0_{yy}] \\
& - u_{sy}q^0[u_{syy}q^0 + 2u_{sy}q^0_y + u_s q^0_{yyy}] - u_s q^0_y[u_{syy}q^0 + 2 u_{sy}q^0_y], q^0)
\end{align*}

\noindent Thus, $\mathcal{J}_2$ contains harmless commutator terms which are easily seen to be size $\eps^{N_0} [[[q^0]]] [[q^0]]^2$ upon using (\ref{leh.3}), (\ref{paired}), and the rapid decay of $\p_y^k u_s$ ($k \ge 1$) which is present in each term above. We estimate
\begin{align*}
\eps^{N_0} |(u_s^2 q^0 q^0_{yyy}, q^0)| \lesssim & \eps^{N_0} [[q^0]]^2 (u_s^2 |q^0_{yyy}|, \langle y \rangle)| \\
\lesssim &\eps^{N_0-(\frac{1}{2}+)} (q^0_{yyy} \langle y \rangle Y^{\frac{1+}{2}}, \langle y \rangle^{-\frac{1+}{2}}) [[q^0]]^2 \\
\lesssim &\eps^{N_0-(\frac{1}{2}+)} \| q^0_{yyy} \langle y \rangle Y^{\frac{1+}{2}} \| [[q^0]]^2 \\
\lesssim &\eps^{N_0-(\frac{1}{2}+)} [[[q^0]]][[q^0]]^2. 
\end{align*} 

Next, recalling (\ref{sister})
\begin{align*}
\eps^{N_0}|(u_s^2 q^0_y q^0_{yy}, q^0)| \lesssim& \eps^{N_0} \| q^0 \|_\infty \| \sqrt{u_s} q^0_y \| \| u_s q^0_{yy} \|  \\
\lesssim &\eps^{N_0-\frac{1}{2}} C_{\sigma, \lambda} [[q^0]]^3. 
\end{align*}

The next nonlinear terms are 
\begin{align*}
&\eps^{N_0+1}(u^0 v_{xx}|_{x = 0} + v^0 v_{xy}|_{x = 0}, q^0) \\
\lesssim & \eps^{N_0+1} \| u^0 \|_\infty \| v_{xx}|_{x = 0} \langle y \rangle \| \| q^0_y \| + \eps^{N_0+1} \| v^0 \|_\infty \| v_{xy}|_{x = 0} w \| \| q^0_y \| \\
\lesssim & \eps^{N_0+1} [u^0, v^0]_B^2 \| a_2^\eps \langle y \rangle \| +\eps^{N_0+\frac{1}{4}} (\sqrt{\eps}\| v^0 \|_\infty) \|v_{xy} w \|^{\frac{1}{2}} \| \sqrt{\eps} v_{xxy} w \|^{\frac{1}{2}} \| q^0_y \| \\
\lesssim & \eps^{N_0+1} [u^0, v^0]_B^2 \| a_2^\eps \langle y \rangle \| +\eps^{N_0-\frac{1}{4}} [u^0, v^0]_B^2 \| v \|_{Y_w}. 
\end{align*}

To conclude, we treat the contribution of the $h$ terms: 
\begin{align*}
&|(\eps^{N_0} \{h v^0_{yy} - v^0 h_{yy} \}, q^0)|\\
 \lesssim & \eps^{N_0} \| h \langle y \rangle \|_\infty \| v^0_{yy} \| \| q^0_y \| + \eps^{N_0} \| h_{yy} y^2 \|_\infty \| v^0_y \| \| q^0_y \|  \\
 \lesssim & \eps^{N_0} \{ \| h \langle y \rangle \|_\infty + \| h_{yy} y^2 \|_\infty \} [u^0, v^0]_B^2. 
\end{align*}

Next, 
\begin{align*}
|(\mathcal{H}, q^0)| \lesssim & [[q^0]] \| \mathcal{H} \langle y \rangle^{1/2} \|_1 \\
\lesssim & \| \{h''' - v_s h'' - h \Delta_\eps u_s \} \langle y \rangle^{1/2} \|_1 [u^0, v^0]_B^2. 
\end{align*}

We now move to the contribution of $\|\mathcal{Q} w_0 \|$. We estimate the first term directly upon using (\ref{sister}):
\begin{align*}
\|\eps^{N_0}v^0 v^0_{yyy} w_0\| \le &\eps^{N_0-\frac{1}{2}} \| \sqrt{\eps} v^0 \|_\infty \| v^0_{yyy} w \|  \lesssim \eps^{N_0-\frac{1}{2}} [u^0, v^0]_B^2. 
\end{align*}

For the second nonlinearity, we have 
\begin{align*}
&\| \eps^{N_0} \p_y \{ u_s q^0\} \p_{yy} \{ u_s q^0 \} w_0 \| \\
= &\| \eps^{N_0} \Big( u_s u_{syy} q^0 q^0_y + 2 u_s u_{sy}|q^0_y|^2 + u_s^2 q^0_y q^0_{yy} + u_{sy} u_{syy} |q^0|^2 \\
& + 2 u_{sy}^2 q^0 q^0_y + u_s u_{sy} q^0 q^0_{yy} \Big) w_0 \| \\
\le & \eps^{N_0} \| \{u_s u_{syy} q^0 q^0_y + 2 u_s u_{sy} |q^0_y|^2 + u_s^2 q^0_y q^0_{yy} + u_{sy} u_{syy} |q^0|^2 \\
& + 2 u_{sy}^2 q^0 q^0_y + u_s u_{sy} q^0 q^0_{yy} \} w \|  \\
\lesssim & \eps^{N_0-1} [u^0, v^0]_B^2.
\end{align*}

\noindent Above, we have used 
\begin{align*}
\| u_s^2 q^0_y q^0_{yy} w\| \lesssim &  \| q^0_y \langle y \rangle \|_\infty \| q^0_{yy} \frac{w}{\langle y \rangle} \| \lesssim  \eps^{-1} \| \eps \langle y \rangle q^0_y \|_\infty [u^0, v^0]_B \lesssim  \eps^{-1}[u^0, v^0]_B^2. 
\end{align*}

We next move to 
\begin{align*}
&\|\eps^{N_0+1} u^0 v_{xx}|_{x = 0} w_0 \| + \eps^{N_0+1} \| v^0 v_{xy}|_{x = 0} w_0\| \\
\lesssim & \eps^{N_0+1}\Big( \| u^0 \|_\infty \| v_{xx}|_{x = 0} w \| + \| v^0 \|_\infty \| v_{xy} w \|_{x = 0} \Big) \\
\lesssim & \eps^{N_0+1} [u^0, v^0]_B^2 \| a_2^\eps w \| +\eps^{N_0+\frac{1}{4}} \| \sqrt{\eps} v^0 \|_\infty \| v_{xy} w \|^{\frac{1}{2}} \| \sqrt{\eps} v_{xxy} w \|^{\frac{1}{2}}  \\
\lesssim & \eps^{N_0+1} [u^0, v^0]_B^2 \| a_2^\eps w \| +\eps^{N_0-\frac{1}{4}} [u^0, v^0]_B^2 \| v \|_{Y_w}.   
\end{align*}

To conclude, we estimate the contributions of $h$, starting with 
\begin{align*}
\| \eps^{N_0} \{ h v^0_{yy} - v^0 h_{yy} \} w_0 \| \lesssim & \eps^{N_0} \| h w \|_\infty \| v^0_{yy} \| + \eps^{N_0} \| \frac{v^0}{y} \|_\infty \| h_{yy} y w \|_2  \\
\lesssim & \eps^{N_0} [[q^0]] [[[q^0]]] C(h) \lesssim  \eps^{N_0} [u^0, v^0]_B^4 + C(h).
\end{align*}

We next move to the $\mathcal{H}$ terms:
\begin{align*}
\| \mathcal{H}w_0 \| \le & \| [- h''' + v_s h'' - h \Delta_\eps v_s] w \| \le   C(h). 
\end{align*}
\end{proof}

The remaining terms from the right-hand sides of (\ref{bring1}) are the $\mathcal{F}$ terms, for which we estimate 
\begin{lemma} Let $\bold{u} \in \mathcal{X}$ as in (\ref{defn.norms.ult}). Assume (\ref{assume.bq.intro}) and $h \in C^\infty(e^y)$ as in (\ref{rene}). Let $n > 1 + 2N_0$ in Theorem \ref{thm.construct}. Then the forcing terms satisfy
\begin{align} \label{defn.F}
\eps^{\frac{1}{2}}|\mathcal{F}_{X_1}| + |\mathcal{F}_B| +\eps^{\frac{1}{2}} |\mathcal{F}_{Y_{w_0}}| \le o(1) + o(1) \| \bold{u} \|_{\mathcal{X}}^2 + o(1) \| \bold{\bar{u}} \|_{\mathcal{X}}^2. 
\end{align}
\end{lemma}
\begin{proof} Recalling the definition of $F_{(q)}, F^a_R$ from (\ref{sys.u0.app.unh}), (\ref{spec.nl}): 
\begin{align*}
&F_{(q)} = \p_x F_R + \p_x b_{(u)}(a^\eps) + H[a^\eps](\bar{v}, \bar{u}^0, \bar{v}^0) + \{ v_{sx} h_{yy} - h \Delta_\eps v_{sx} \}, \\
&F^a_R = F_R|_{x = 0} + b_{(u)}(a^\eps). 
\end{align*}

Examining the definition of $\mathcal{F}_{X_1}, \mathcal{F}_{Y_{w_0}}, \mathcal{F}_{B}$, we may estimate 
\begin{align*}
\eps^{\frac{1}{2}}\Big( \mathcal{F}_{X_1}(\p_x F_R, q) + \mathcal{F}_{Y_{w_0}}(\p_x F_R, q)\Big) \lesssim & \eps^{\frac{1}{2}} \| \frac{1}{\sqrt{\eps}} \p_x F_R w_0 \| [\| v \|_{Y_{w_0}} + \| v \|_{X_1}]  \\
\lesssim & o(1) + o(1) \| \bold{u} \|_{\mathcal{X}}^2, 
\end{align*}

\noindent upon recalling (\ref{thm.force.maz}). Next, 
\begin{align*}
\mathcal{F}_B(F_R|_{x = 0}, q) \le & |(F_R, q^0)| + \| F_R w_0 \|^2 \\
\lesssim &  \| F_Rw_0 \| \| q^0_y \| + \| F_R w_0 \|^2\\
\le & o(1) + o(1) [u^0, v^0]_B^2 \\
\le & o(1) + o(1) \| \bold{u} \|_{\mathcal{X}}^2. 
\end{align*}

Repeating the above estimates for the $\p_x b_{(u)}(a^\eps), b_{(u)}(a^\eps)$ terms, we obtain that these contributions to (\ref{defn.F}) are bounded by 
\begin{align*}
 &C \| \frac{1}{\sqrt{\eps}} \p_x b_{(u)}(a^\eps) \|^2 + \| b_{(u)}(a^\eps) w_0 \|^2 + o(1) \| \bold{u} \|_{\mathcal{X}}^2  \lesssim  o(1) + o(1) \| \bold{u} \|_{\mathcal{X}}^2, 
\end{align*}

\noindent upon invoking assumption (\ref{assume.bq.intro}) and consulting the definitions  (\ref{bu}).

A similar computation, consulting the definition of $H[a^\eps](\bar{u}^0, \bar{v}^0, \bar{v})$ given in (\ref{spec.nl}), produces a bound 
\begin{align*}
\| H[a^\eps](\bar{u}^0, \bar{v}^0, \bar{v}) \frac{w_0}{\sqrt{\eps}} \| \le& o(1) + o(1)  \Big( [\bar{u}^0, \bar{v}^0]_B + \|\bar{v} \|_{X_1} + \| \bar{v} \|_{Y_{w_0}} \Big),
\end{align*}

\noindent upon invoking again assumption (\ref{assume.bq.intro}). A similar estimate holds for the $h$ terms in $F_{(q)}$ using (\ref{rene}). This thus concludes the proof of (\ref{defn.F}).
\end{proof}

We are now ready to insert all of these estimates into (\ref{bring1}), which gives the following 
\begin{proposition} For $\sigma << 1$ then $L << 1$, solutions to (\ref{linearised.1}), (\ref{linearised.2}) satisfy the following set of estimates: 
\begin{align}
\begin{aligned} \label{est.imjm}
&\| v \|_{X_1}^2 \lesssim o(1) \| v \|_{X_1}^2 +  \eps^{-\frac{1}{2}} [\bar{u}^0, \bar{v}^0]_B^2 \\
& \hspace{20 mm} + \eps^{N_0 - 1} \Big( \| \bar{v} \|_{X_1}^4 + [\bar{u}^0, \bar{v}^0]_B^4 \Big) + C(h) + \mathcal{F}_{X_1}\\
&[u^0, v^0]_B^2 \lesssim \eps \| \bar{v} \|_{Y_{w_0}}^2 + \eps^{\frac{1}{2}+\frac{3-}{16}} \| \bar{v} \|_{X_1}^2 \\
& \hspace{20 mm} + \eps^{N_0 - 1}[\bar{u}^0, \bar{v}^0]_B^4 + C(h, a_2^\eps) + \mathcal{F}_{B} \\
&\| v \|_{Y_{w_0}}^2 \lesssim \| v \|_{X_1}^2  + [\bar{u}^0, \bar{v}^0]_B^2 \\
& \hspace{20 mm} + \eps^{N_0 - 1} \Big( \| \bar{v} \|_{X_1}^4 + \| \bar{v} \|_{Y_{w_0}}^4 + [\bar{u}^0, \bar{v}^0]_B^4 \Big) + C(h) + \mathcal{F}_{Y_{w_0}}.
\end{aligned}
\end{align}

\noindent Above, $C(h) = \bigO( \| h \|_{C^{M_0}(e^y)})$ for a large $M_0$. 

\end{proposition}

From here, we may immediately prove the main result: 

\begin{proof}[Proof of Theorem \ref{thm.main}] We apply a standard contraction mapping theorem to the map $\Psi$ which sends $[\bar{v}, \bar{u}^0, \bar{v}^0]$ to $[v, u^0, v^0]$ via the equations (\ref{linearised.1}) , (\ref{linearised.2}). Such a map is well-defined according to Proposition \ref{prop.L.exist} and Proposition \ref{result.v.exist}. The estimates (\ref{est.imjm}) together with the forcing estimates in  (\ref{defn.F}) produce the following inequality 
\begin{align} \label{cont.1}
\| \bold{u} \|_{\mathcal{X}}^2 \le o(1) \| \bold{u} \|_{\mathcal{X}}^2 + \| \bold{\bar{u}} \|_{\mathcal{X}}^2 + \| \bold{\bar{u}} \|_{\mathcal{X}}^4 + o(1). 
\end{align}

By repeating the above analysis for differences $\bold{u}_1 - \bold{u}_2$, and $\bold{\bar{u}}_1 - \bold{\bar{u}}_2$, (\ref{cont.1}) shows that $\Psi$ is a contraction map, and thus has a fixed point. Clearly, from (\ref{linearised.1}) and (\ref{linearised.2}), such a fixed point solves the nonlinear equations (\ref{sys.u0.app.unh}) and (\ref{spec.nl}). The homogenization procedure to derive these two systems (see Subsection \ref{subsection.rem}) ensures that this is equivalent to solving:
\begin{align*}
&\p_x \text{LHS Equation (\ref{eqn.vort.intro})} = \p_x \text{RHS Equation (\ref{eqn.vort.intro})}, \text{ and } \\
&\text{LHS Equation  (\ref{eqn.vort.intro})}|_{x = 0} = \text{RHS Equation (\ref{eqn.vort.intro})}|_{x = 0}.
\end{align*}

\noindent Thus, such a fixed point solves (\ref{eqn.vort.intro}) itself. 
\end{proof}

\newpage

\appendix

\part*{Appendix}

\section{Derivation of Equations} \label{appendix.derive}

We will assume the expansions: 
\begin{align}
&U^\eps = \tilde{u}^n_s + \eps^{N_0} u, \hspace{3 mm} V^\eps = \tilde{v}^n_s + \eps^{N_0} v, \hspace{3 mm} P^\eps = \tilde{P}^n_s + \eps^{N_0} P.
\end{align}

\noindent We will denote the partial expansions: 
\begin{align}
&u_s^i = \sum_{j = 0}^i \sqrt{\eps}^j u^j_e + \sum_{j = 0}^{i-1} \sqrt{\eps}^j u^j_p, \hspace{5 mm} \tilde{u}_s^i = u_s^i + \sqrt{\eps}^i u^i_p, \\
&v_s^i = \sum_{j = 1}^i \sqrt{\eps}^{j-1} v^j_e + \sum_{j = 0}^{i-1} \sqrt{\eps}^j v^j_p, \hspace{5 mm} \tilde{v}_s^i = v_s^i + \sqrt{\eps}^i v^i_p, \\
&P^i_s = \sum_{j = 0}^i \sqrt{\eps}^j P^j_e, \hspace{5 mm}  \tilde{P}_s^i = P_s^i + \sqrt{\eps}^i \Big\{ P^i_p + \sqrt{\eps} P^{i,a}_p \Big\}.
\end{align}

\noindent  We will also define $u^{E,i}_s = \sum_{j = 0}^i \sqrt{\eps}^j u^j_e$ to be the ``Euler" components of the partial sum. Similar notation will be used for $u^{P,i}_s, v^{E,i}_s, v^{P,i}_s$. The following will also be convenient: 
\begin{align}
\begin{aligned} \label{profile.splitting}
&u_s^E := \sum_{i =0}^n \sqrt{\eps}^i u^i_e, \hspace{3 mm} v_s^E := \sum_{i = 1}^n \sqrt{\eps}^{i-1} v^i_e, \\
&u_s^P := \sum_{i = 0}^n \sqrt{\eps}^i u^i_p, \hspace{3 mm} v_s^P := \sum_{i = 0}^n \sqrt{\eps}^i v^i_p, \\
&u_s = u_s^P + u_s^E, \hspace{3 mm} v_s = v_s^P + v_s^E. 
\end{aligned}
\end{align} 

\noindent  The $P^{i,a}_p$ terms are ``auxiliary Pressures" in the same sense as those introduced in \cite{GN} and \cite{Iyer} and are for convenience. We will also introduce the notation: 
\begin{align} \label{bar.defs}
\bar{u}^i_p := u^i_p - u^i_p|_{y = 0}, \hspace{5 mm} \bar{v}^i_p := v^i_p - v^i_p(x,0), \hspace{5 mm} \bar{v}^i_e = v^i_e - v^i_e|_{Y = 0}.
\end{align}

\subsection{$i = 0$}

We first record the properties of the leading order $(i = 0)$ layers. For the outer Euler flow, we will take a shear flow, $[u^0_e(Y), 0, 0]$. The derivatives of $u^0_e$ decay rapidly in $Y$ and that is bounded below, $|u^0_e| \gtrsim 1$. 

For the leading order Prandtl boundary layer, the equations are: 
\begin{align}
\left.
\begin{aligned} \label{Pr.leading}
&\bar{u}^0_p u^0_{px} + \bar{v}^0_p u^0_{py} - u^0_{pyy} + P^0_{px} = 0, \\
&u^0_{px} + v^0_{py} = 0, \hspace{3 mm} P^0_{py} = 0, \hspace{3 mm} u^0_p|_{x = 0} = U^0_P, \hspace{3 mm} u^0_p|_{y = 0} = - u^0_e|_{Y = 0}.
\end{aligned}
\right\}
\end{align}

It is convenient to state results in terms of the quantity $\bar{u}^0_p$, whose initial data is simply $\bar{U}^0_P := u^0_e(0) + U^0_P$. Our starting point is the following result of Oleinik in \cite{Oleinik}, P. 21, Theorem 2.1.1:
\begin{theorem}[Oleinik]   \label{thm.Oleinik} Assume boundary data is prescribed satisfying $U^0_P \in C^\infty$ and exponentially decaying $|\p_y^j \{\bar{U}^0_P - u^0_e(0)\}|$ for $j \ge 0$ satisfying: 
\begin{align} 
\begin{aligned} \label{OL.1}
& \bar{U}^0_P > 0 \text{ for } y > 0, \hspace{3 mm} \p_y \bar{U}^0_P(0) > 0, \hspace{3 mm} \p_y^2 \bar{U}^0_P \sim y^2 \text{ near } y = 0
\end{aligned}
\end{align}

\noindent  Then for some $L > 0$, there exists a solution, $[\bar{u}^0_p, \bar{v}^0_p]$ to (\ref{Pr.leading}) satisfying, for some $y_0, m_0 > 0$, 
\begin{align} \label{coe.2}
&\sup_{x \in (0,L)} \sup_{y \in (0, y_0)} |\bar{u}^0_p, \bar{v}^0_p, \p_y \bar{u}^0_p, \p_{yy}\bar{u}^0_p, \p_x \bar{u}^0_p| \lesssim 1, \\ \label{coe.1}
&\sup_{x \in (0,L)} \sup_{y \in (0, y_0)} \p_y \bar{u}^0_p > m_0 > 0. 
\end{align}
\end{theorem}

\noindent  By evaluating the system (\ref{Pr.leading}) and $\partial_y$ of (\ref{Pr.leading}) at $\{y = 0\}$ we conclude: 
\begin{align*}
\bar{u}^0_{pyy}|_{y = 0} = \bar{u}^0_{pyyy}|_{y = 0} = 0. 
\end{align*}

\subsection{$1 \le i \le n-1$}

We now list the equations to be satisfied by the $i$'th layers, starting with the $i$'th Euler layer:
\begin{align} \label{des.eul.1}
\left.
\begin{aligned}
&u^0_e \p_x u^i_e + \p_Y u^0_e v^i_e + \p_x P^i_e =: f^i_{E,1}, \\
&u^0_e \p_x v^i_e + \p_Y P^i_e  =: f^i_{E,2}, \\
&\p_x u^i_e + \p_Y v^i_e = 0, \\
&v^i_e|_{Y = 0} = - v^0_p|_{y = 0}, \hspace{5 mm} v^i_e|_{x = 0, L} = V_{E, \{0, L\}}^i \hspace{5 mm} u^i_e|_{x = 0} = U^i_{E}.
\end{aligned}
\right\}
\end{align}

For the $i$'th Prandtl layer:
\begin{align} \label{des.pr.1}
\left.
\begin{aligned}
&\bar{u} \p_x u^i_p + u^i_p \p_x \bar{u} + \p_y \bar{u} [v^i_p - v^i_p|_{y = 0}] + \bar{v} \p_y u^i_p + \p_x P^i_p - \p_{yy} u^i_p := f^{(i)}, \\  
& \p_x u^i_p + \p_y v^i_p = 0,  \hspace{5 mm} \p_y P^i_p = 0\\  
& u^i_p|_{y = 0} = -u^i_e|_{y = 0}, \hspace{5 mm} [u^i_p, v^i_p]_{y \rightarrow \infty} = 0, \hspace{5 mm} v^i_p|_{x = 0} = \text{prescribed initial data}. 
\end{aligned}
\right\}
\end{align}

The relevant definitions of the above forcing terms are given below. Note that as a matter of convention, summations that end with a negative number are empty sums.
\begin{definition}[Forcing Terms] \label{def.forcing}  
\begin{align*}
&-f^i_{E,1} := u^{i-1}_{ex} \sum_{j = 1}^{i-2} \sqrt{\eps}^{j-1} \{u^j_e + u^j_p(x,\infty) + u^{i-1}_e \sum_{j = 1}^{i-2} \sqrt{\eps}^{j-1} u^j_{ex} \\
& \hspace{15 mm} + \sqrt{\eps}^{i-2}[ \{u^{i-1}_{e} + u^{i-1}_p(x,\infty) \} u^{i-1}_{ex} + v^{i-1}_e u^{i-1}_{eY}] \\
& \hspace{15 mm} + u^{i-1}_{eY} \sum_{j = 1}^{i-2} \sqrt{\eps}^{j-1} v^j_e + v^{i-1}_e \sum_{j = 1}^{i-2} \sqrt{\eps}^{j-1} u^j_{eY} - \sqrt{\eps} \Delta u^{i-1}_e - g^{u,i}_{ext, e} \\
&-f^i_{E,2} := v^{i-1}_{eY} \sum_{j = 1}^{i-2} \sqrt{\eps}^{j-1} v^j_e + v_e^{i-1} \sum_{j = 1}^{i-2} \sqrt{\eps}^{j-1} v^j_{eY} + \sqrt{\eps}^{i-2}[v^{i-1}_e v^{i-1}_{eY} + u^{i-1}_e v^{i-1}_{ex}] \\
& \hspace{15 mm} + \{u_e^{i-1} + u^{i-1}_p(x,\infty)\} \sum_{j =1}^{i-2} \sqrt{\eps}^{j-1} v^j_{ex} + v^{i-1}_{ex} \sum_{j = 1}^{i-2} \sqrt{\eps}^{j-1} \{u^j_e + u^j_p(x,\infty) \}\\
& \hspace{15 mm} - \sqrt{\eps} \Delta v^{i-1}_e - g^{v, i}_{ext, e}, \\
&-f^{(i)} := \sqrt{\eps} u^{i-1}_{pxx} + \eps^{-\frac{1}{2}} \{ v^i_e - v^i_e(x,0) \} u^0_{py} + \eps^{-\frac{1}{2}} \{ u^0_e - u^0_e(0) \} u^{i-1}_{px} + \eps^{-\frac{1}{2}} \{ u^{P, i-1}_{sx} \\
& \hspace{15 mm} - \bar{u}^0_{sx} \} u^{i-1}_p +  \eps^{-\frac{1}{2}} \{ u^{E, i-1}_{sx}  - \bar{u}^0_{sx} \} \{u^{i-1}_p - u^{i-1}_p(x,\infty) \} + \eps^{-\frac{1}{2}} v^{i-1}_p \{ \bar{u}^{i-1}_{sy} \\
& \hspace{15 mm} - u^0_{py} \} + u^{i-1}_{px} \sum_{j =1 }^{i-1} \sqrt{\eps}^{j-1}(u^j_e + u^j_p)  + \eps^{-\frac{1}{2}} (v_s^{i-1} - v_s^1) u^{i-1}_{py} + \eps^{-\frac{1}{2}} (v^1_e \\
& \hspace{15 mm} - v^1_e(x,0)) u^{i-1}_{py} + \sqrt{\eps} u^i_{eY} \sum_{j = 0}^{i-1} \sqrt{\eps}^j v^j_p + v^i_e \sum_{j = 1}^{i-1} \sqrt{\eps}^{j-1} u^j_{py} + u^i_{ex} \sum_{j = 0}^{i-1} \sqrt{\eps}^j \{u^j_p \\
& \hspace{15 mm} - u^j_p(x,\infty) \} + u^i_e \sum_{j = 0}^{i-1} \sqrt{\eps}^j u^j_{px} + \int_y^\infty \p_x \{ \sqrt{\eps}^2 u^i_e \sum_{j = 0}^{i-1} \sqrt{\eps}^j v^j_{px} + \sqrt{\eps} v^i_{ex} \\
& \hspace{15 mm} \times \sum_{j = 0}^{i-1} \sqrt{\eps}^j \{u^j_p - u^j_p(x,\infty)\} + \sqrt{\eps}^2 v^i_{eY} \sum_{j = 0}^{i-1} \sqrt{\eps}^j v^j_p + \sqrt{\eps} v^i_e \sum_{j = 0}^{i-1} \sqrt{\eps}^j v^j_{py} \\
& \hspace{15 mm} + \sqrt{\eps} v^{i-1}_s v^{i-1}_{py} + \sqrt{\eps} v_{sy}^{i-1} v^{i-1}_p  + \sqrt{\eps} v^{E,i-1}_{sx} \{u^{i-1}_p - u^{i-1}_p(x,\infty)\} \\
& \hspace{15 mm} +  \sqrt{\eps} v^{P,i-1}_{sx} u^{i-1}_{p} + \sqrt{\eps} u_s^{i-1} v^{i-1}_{px} + \sqrt{\eps} \Delta_\eps v^{i-1}_p + \sqrt{\eps}^i \{ u^{i-1}_p v^{i-1}_{px} + v^{i-1}_p v^{i-1}_{py} \} \} \ud z \\
& \hspace{15 mm} - g^{u, i}_{ext, p} + \int_y^\infty \p_x \{ \sqrt{\eps}^2 g^{v, i}_{ext, p} \} \ud z. 
\end{align*}
\end{definition}

For $i = 1$ only, we make the following modifications. The aim is to retain only the required order $\sqrt{\eps}$ terms into $f^{(1)}$. $f^{(2)}$ will then be adjusted by including the superfluous terms. Moreover, $f^{(1)}$ will contain the important $g^{u, 1}_{ext, p}$ external forcing term. Specifically, define: 
\begin{align} 
\begin{aligned} \label{defn.f1.special}
f^{(1)} := &g^{u, 1}_{ext, p} - u^0_p u^1_{ex}|_{Y = 0} - u^0_{px} u^1_e|_{Y = 0} \\
& - \bar{u}^0_{eY}(0) y u^0_{px} - v^0_p u^0_{eY} - v^1_{eY}(0) y u^0_{py}.
\end{aligned}
\end{align} 

\subsection{$i = n$}

For the final Prandtl layer, we must enforce the boundary condition $v^n_p|_{y = 0} = 0$. Define the quantities $[u_p, v_p, P_p]$ to solve
\begin{align} \label{des.pr.1}
\left.
\begin{aligned}
&\bar{u} \p_x u_p + u_p \p_x \bar{u} + \p_y \bar{u} v_p + \bar{v} \p_y u_p + \p_x P_p - \p_{yy} u_p := f^{(n)}, \\  
& \p_x u_p + \p_y v_p = 0,  \hspace{5 mm} \p_y P^i_p = 0\\  
& [u_p, v_p]|_{y = 0} = [-u^n_e, 0]|_{y = 0}, \hspace{5 mm} u_p|_{y \rightarrow \infty} = 0 \hspace{5 mm} v_p|_{x = 0} = V_P^n. 
\end{aligned}
\right\}
\end{align}

Note the change in boundary condition of $v_{p}|_{y = 0} = 0$ which contrasts the $i = 1,..,n-1$ case. This implies that $v_p = \int_0^y u_{px} \ud y'$. For this reason, we must cut-off the Prandtl layers: 
\begin{align*}
&u^n_p := \chi(\sqrt{\eps}y) u_p + \sqrt{\eps} \chi'(\sqrt{\eps}y) \int_0^y u_p(x, y') \ud y', \\
&v^n_p := \chi(\sqrt{\eps}y) v_p. 
\end{align*}

Here $\mathcal{E}^n$ is the error contributed by the cut-off: 
\begin{align*}
\mathcal{E}^{(n)} &:= \bar{u} \p_x u^{n}_{p} + u^n_p \p_x \bar{u}  +\bar{v} \p_y u^n_{p} + v^n_p \p_y \bar{u}  - u^n_{pyy} - f^{(n)}. 
\end{align*}

Computing explicitly: 
\begin{align} \n
\mathcal{E}^{(n)} := &(1-\chi) f^{(n)} + \bar{u} \sqrt{\eps} \chi'(\sqrt{\eps}y) v_p(x,y) + \bar{u}_{x} \sqrt{\eps} \chi' \int_0^y u_p \\  \n
& + \bar{v} \sqrt{\eps} \chi' u_p + \eps \bar{v} \chi'' \int_0^y u_p + \sqrt{\eps} \chi' u_p \\ \label{dan.1}
& + \eps^{\frac{3}{2}} \chi''' \int_0^y u_p + 2\eps \chi'' u_p + \sqrt{\eps} \chi' u_{py}.
\end{align}

We will now define the contributions into the next order, which will serve as the forcing for the remainder term: 
\begin{align} \n
&\underbar{f}^{(n+1)} := \sqrt{\eps}^n \Big[ \eps u^n_{pxx} + v^n_p\{ \bar{u}^n_{sy} - u^0_{py} \} + \{u^0_e - u^0_e(0) \} u^n_{px} \\ \n
& \hspace{15 mm} + u^n_{px} \sum_{j = 1}^n  \sqrt{\eps}^j (u^j_e + u^j_p) + \{ u^n_{sx} - \bar{u}^0_{sx} \} u^n_p + (v^n_s - v^1_s) u^n_{py} \\ \n
& \hspace{15 mm} + \{ v^1_e - v^1_e(x,0) \} u^n_{py} \Big] + \sqrt{\eps}^n \mathcal{E}^{(n)} + \sqrt{\eps}^{n+2} \Delta u^n_e \\ \label{underbar.f}
& \hspace{15 mm} + \sqrt{\eps}^n u^n_{ex} \sum_{j = 1}^{n-1} \sqrt{\eps}^j u^j_e + \sqrt{\eps}^n u^n_e \sum_{j = 1}^{n-1} \sqrt{\eps}^j u^j_{ex} + \sqrt{\eps}^{2n} [ u^n_e u^n_{ex} \\ \n
& \hspace{15 mm} + v^n_e u^n_{eY}] + \sqrt{\eps}^{n+1} u^n_{eY} \sum_{j= 1}^{n-1} \sqrt{\eps}^{j-1} v^j_e + \sqrt{\eps}^{n-1}v^n_e \sum_{j = 1}^{n-1} \sqrt{\eps}^{j+1} u^j_{eY} ..
\end{align}

\begin{align} \n
& \underbar{g}^{(n+1)} := \sqrt{\eps}^n \Big[ v_s^n \p_y v^n_p + \p_y v_s^n v^n_p + \p_x v^n_s u^n_p + u^n_s \p_x v^n_p - \Delta_\eps v^n_p \\ \n
& \hspace{15 mm}  + \sqrt{\eps}^n \Big( u^n_p \p_x v^n_p + v^n_p \p_y v^n_p \Big) \Big] + (\sqrt{\eps})^n \p_Y v^n_e \sum_{j = 1}^{n-1} (\sqrt{\eps})^{j-1} v^j_e \\ \label{underbar.g}
& \hspace{15 mm}+ \sqrt{\eps}^{n-1} v^n_e \sum_{j = 1}^{i-1} \sqrt{\eps}^j \p_Y v^j_e + \sqrt{\eps}^{2n-1} [v^n_e v^n_{eY} + u^n_e \p_x v^n_e]  \\ \n
& \hspace{15 mm} + \sqrt{\eps}^n u^n_e \sum_{j =1}^{n-1} (\sqrt{\eps})^{j-1}\p_x v^j_e + \sqrt{\eps}^{n-1} \p_x v^n_e \sum_{j = 0}^{n-1} \sqrt{\eps}^j u^j_e + \sqrt{\eps}^{n+1} \Delta v^n_e.
\end{align}

\subsection{Remainder System} \label{subsection.rem}

 A straightforward linearization yields:
\begin{align} \label{rem.sys.1}
\left.
\begin{aligned} 
&-\Delta_\eps u^{(\eps)} + S_u + \p_x P^{(\eps)} = \eps^{-N_0} \underbar{f}^{(n+1)} - \eps^{N_0} \{u^{(\eps)} u^{(\eps)}_x + v^{(\eps)} u^{(\eps)}_y \} \\
&-\Delta_\eps v^{(\eps)} + S_v + \frac{\p_y}{\eps}P^{(\eps)} =\eps^{-N_0} \underbar{g}^{(n+1)} - \eps^{N_0}\{ u^{(\eps)} v^{(\eps)}_x + v^{(\eps)} v^{(\eps)}_y \} \\
&\p_x u^{(\eps)} + \p_y v^{(\eps)} = 0.
\end{aligned}
\right\}
\end{align}

Denote: 
\begin{align}
u_s := \tilde{u}^n_s, \hspace{5 mm} v_s := \tilde{v}^n_s.
\end{align}

Here we have defined: 
\begin{align} \label{Su}
&S_u = u_s \p_x u^{(\eps)} + u_{sx}u^{(\eps)} + v_s \p_y u^{(\eps)} + u_{sy}v^{(\eps)}, \\ \label{Sv}
&S_v = u_s \p_x v^{(\eps)} + v_{sx}u^{(\eps)} + v_s \p_y v^{(\eps)} + v_{sy}v^{(\eps)}.
\end{align}

Let us discuss now the boundary conditions. We take 
\begin{align*}
&u^\eps|_{x = 0} := u^0 (\text{unknown}), \\
&v^\eps|_{x = 0} := v^0 (\text{unknown}), \\
&v^\eps|_{y = 0} =v^\eps_y|_{y = 0} = 0, \\
&v^\eps_x|_{x = L} := a_1^\eps(y), v^\eps_{xx}|_{x = 0} := a_2^\eps(y), v^\eps_{xxx}|_{x = L} := a_3^\eps(y). 
\end{align*}

Going to the vorticity formulation of (\ref{rem.sys.1}) yields the system (\ref{eqn.vort.intro}), with 
\begin{align} \label{forcingdefn}
&F_R := \eps^{-N_0} ( \p_y \underbar{f}^{(n+1)} - \eps \p_x \underbar{g}^{(n+1)}).
\end{align}

In Section \ref{Section.1}, our main object of analysis with the vorticity equation evaluated at the $\{x = 0\}$ boundary, $(\ref{eqn.vort.intro})|_{x = 0}$, which reads: 
\begin{align}
\begin{aligned} \label{sys.u0.app.unh}
&\mathcal{L} v^0 = F_{(v)} + F^a_R + \mathcal{Q}(u^0, v^0, v) + \mathcal{H}, \\
&\mathcal{L} v^0 := v^0_{yyyy} - \{ u_s v^0_{yy} - u_{syy}v^0 \} - \{ v_s v^0_{yyy} - v^0_y v_{syy} \} \\
& \hspace{20 mm} + \eps u_{sxx} v^0 + \eps v_{sxx} v^0_y , \\
&\mathcal{Q}(u^0, v^0, v) := \eps^{N_0} \Big[v^0_y v^0_{yy} - v^0 v^0_{yyy} + \eps u^0 v_{xx}|_{x = 0} + \eps v^0 v_{xy}|_{x = 0} \\
& \hspace{20 mm} +   h v^0_{yy} - v^0 h_{yy}  \Big] \\
& \mathcal{H} := [- h_{yyy} + v_s h_{yy} - h \Delta_\eps u_s], \\
&F_{(v)}(v) := \eps u_s v_{xx}|_{x = 0} - 2\eps v_{xyy}|_{x = 0} - \eps^2 v_{xxx}|_{x = 0} + \eps v_s v_{xy}|_{x = 0}, \\ 
&F^a_R := F_R|_{x = 0} + \eps u_s a^\eps_{xx}|_{x  = 0} - 2 \eps a^\eps_{xyy}|_{x = 0} - \eps^2 a^\eps_{xxx}|_{x = 0} + \eps v_s a^\eps_{xy}|_{x = 0} \\ 
& \hspace{5 mm} :=  F_R|_{x = 0} + b_{(u)}(a^\eps)|_{x = 0}. 
\end{aligned}
\end{align}

We homogenize the $v^\eps$ via (\ref{homfin}). Define the quotients: 
\begin{align*}
q^\eps := \frac{v^\eps}{u_s}, \hspace{3 mm} \tilde{q} := \frac{\tilde{v}}{u_s}, \hspace{3 mm} q := \frac{v}{u_s}, \hspace{3 mm} q^0 := \frac{v^0}{u_s|_{x = 0}}. 
\end{align*}

The $\p_x$ of vorticity equation (DNS) satisfied by $[u^\eps, v^\eps]$ is as follows 
\begin{align}
\begin{aligned}  \label{spec.nl.unh}
&-\p_x R[q^{(\eps)}] + \Delta_\eps^2 v^{(\eps)} + \p_x \{ v_s \Delta_\eps u^{(\eps)} - u^{(\eps)} \Delta_\eps v_s \} \\
& \hspace{10 mm} = \eps^{N_0}\p_x \{ v^\eps \Delta_\eps u^\eps - u^\eps \Delta_\eps v^\eps \} + \p_x F_R, \\
&v^{(\eps)}|_{x = 0} = v^0,  v^{(\eps)}_{xx}|_{x = 0} = a^{(\eps)}_2, v^{(\eps)}_x|_{x = L} = a^{(\eps)}_1, v^{(\eps)}_{xxx}|_{x = L} = a^{(\eps)}_3, \\
&v^{(\eps)}|_{y = 0} = v^{(\eps)}_y|_{y = 0} = 0, 
\end{aligned}
\end{align}

We now homogenize equation (\ref{spec.nl.unh}) by writing it in terms of $[u, v]$. First, the linear contributions are given in terms of the following 
\begin{align} \label{bu}
&b_{(u)}(\tilde{v}) = - R[\tilde{v}] + I_x[\tilde{v}_{yyyy}] + 2\eps \tilde{v}_{xyy} + \eps^2 \tilde{v}_{xxx} - \eps \tilde{v}_{xy} \\ \n
& \hspace{10 mm} + v_s I_x[\tilde{v}_{yyy}] - \Delta_\eps v_s I_x[\tilde{v}_y],
\end{align}

We now arrive at the nonlinearity. For this, we will use  (\ref{homfin}) to write 
\begin{align*}
\p_x \{ v^\eps \Delta_\eps u^\eps - u^\eps \Delta_\eps v^\eps\} = & \eps^{N_0} (Q_{11} + Q_{12} + Q_{13} + Q_{22} + Q_{23} + Q_{33}),
\end{align*}

\noindent where the quadratic terms are 
\begin{align*}
Q_{11} := & v_y \Delta_\eps v - u \Delta_\eps v_x + v_x \Delta_\eps u - v \Delta_\eps v_y, \\
Q_{12} := & v_y v^0_{yy} + v^0_y \Delta_\eps v - x v^0 \Delta_\eps v_x - x v_x v^0_{yy} - v v^0_{yyy} - v^0 \Delta_\eps v_y, \\
Q_{22} := & v^0_y v^0_{yy} - v^0 v^0_{yyy}, 
\end{align*}

\noindent and the linear terms are 
\begin{align*}
Q_{13} := & v_y \Delta_\eps a^\eps + a^\eps_y \Delta_\eps v - u \Delta_\eps a^\eps_x + I_x[a^\eps] \Delta_\eps v_x - v \Delta_\eps a^\eps_y - a^\eps \Delta_\eps v_y  \\
Q_{23} :=& v^0_y \Delta_\eps a^\eps + a^\eps_y v^0_{yy} - x v^0 \Delta_\eps a^\eps_x - v^0 \Delta_\eps a^\eps_y - a^\eps v^0_{yyy}
\end{align*}

\noindent and the forcing term is 
\begin{align*}
Q_{33} := a^\eps_y \Delta_\eps a^\eps + I_x[a^\eps] \Delta_\eps a^\eps_x - a^\eps \Delta_\eps I_x[a^\eps] - a^\eps \Delta_\eps a^\eps_y.
\end{align*}

The last step is to use the identity (recalling (\ref{bu}), (\ref{rene}), and (\ref{portrait})):
\begin{align}
\p_x b_{(u)}(v^0) + \{v_{sx}u^0_{yy} - u^0 \Delta_\eps v_{sx} \} = B_{v^0} + \{ v_{sx} h_{yy} - h \Delta_\eps v_{sx} \},
\end{align}

Piecing together the preceding, we arrive at the homogenized system 
\begin{align}
\begin{aligned} \label{spec.nl}
&-\p_x R[q] + \Delta_\eps^2 v + J(v)  + B_{v^0}  = \eps^{N_0} \mathcal{N}  + F_{(q)} , \\
&\mathcal{N} := Q_{11} + Q_{12} + Q_{22}, \\
&F_{(q)} :=  \p_x F_R + \p_x b_{(u)}(a^\eps) + H[a^\eps](v, u^0, u^0) + \{v_{sx} h_{yy} - h \Delta_\eps v_{sx} \}, \\
&H[a^\eps](v, u^0, v^0) := Q_{13}(u,v) + Q_{23}(v^0)+ Q_{33}(a^\eps),
\end{aligned}
\end{align}

\noindent where we have defined $J, B_{v^0}$ in (\ref{defn.J.conc}) and (\ref{portrait}).

\section{Prandtl Layers} \label{appendix.prandtl}

\subsection{Formulation of D-Prandtl System}

In this subsection, we will analyze the linearized Prandtl equations, (\ref{des.pr.1}). We will rename the unknowns $u_p = u^i_p$ and $v_p = \bar{v}^i_p$, and the linearized quantities $\bar{u} = \bar{u}^0_p, \bar{v} = \bar{v}^0_p$. The equation then reads: 
\begin{align}
\begin{aligned} \label{Orig.Orig.Pr}
&\bar{u} \p_x u_p + u_p \p_x \bar{u} + \bar{v} \p_y u_p + v_p \p_y \bar{u} - \p_{yy} u_p = f, \\
&u_p|_{y = 0} = - \phi(x,0), v_p|_{y = 0} = 0, v_p|_{x = 0} = \bar{v}^i_p|_{x = 0}(y). 
\end{aligned}
\end{align}

\noindent We homogenize the system so that $u|_{y = 0} = 0$ via: 
\begin{align} \label{antiPsi}
u = u_p - u_p(x,0) \psi(y), \hspace{3 mm} v = v_p + \phi_x(x) I_\psi(y), \hspace{3 mm} I_\psi(y) := \int_y^\infty \psi(\theta) \ud \theta. 
\end{align}

\noindent  Here, we select $\psi$ to be a $C^\infty$ function satisfying the following: 
\begin{align} \label{psi.spec}
\psi(0) = 1, \hspace{3 mm} \int_0^\infty \psi = 0, \hspace{3 mm} \psi \text{ decays as } y \uparrow \infty. 
\end{align}

\noindent  The unknowns $[u,v]$ satisfy the system: 
\begin{align}
\begin{aligned} \label{origPrLay}
&\bar{u} \p_x u + u \p_x \bar{u} + \bar{v} \p_y u + v \p_y \bar{u} - \p_{yy} u = f + G = : g_1, \\
&u_x + v_y = 0, \\
&u|_{y = 0} = 0, \hspace{3 mm} v|_{y = 0} = 0, \hspace{3 mm} v|_{x = 0} = \bar{v}^i_p|_{x = 0} - \phi_x(0) I_\psi(y) =: \bar{V}_0(y). \\
&-G = \bar{u} \psi \phi_x + \bar{u}_{x} \psi \phi + \bar{v} \psi' \phi + \bar{u}_{y} \phi_x I_\psi - \psi'' \phi.  
\end{aligned}
\end{align}

\noindent  By applying $\p_y$, we obtain the system: 
\begin{align}  \label{eval.2}
- \bar{u} v_{yy} + v \bar{u}_{yy} - u \bar{v}_{yy} + \bar{v} u_{yy} - u_{yyy} = \p_y g_1,
\end{align}

\noindent  and in the $q$ formulation: 
\begin{align}
\begin{aligned} \label{origPrLay.beta}
&- \p_{xy} \{ \bar{u}^2 q_y \} + \p_y^4 v + \Lambda + U   = \p_{xy} g_1, \\
&q|_{y = 0} = 0, \hspace{3 mm} q|_{x = 0} = \frac{1}{\bar{u}}|_{x= 0}(y) \bar{V}_0(y) := f_0(y). 
\end{aligned}
\end{align}

\noindent We have defined: 
\begin{align*}
&\Lambda := \bar{v}_{xyy}I_x[v_y] + \bar{v}_{yy} v_y - \bar{v}_{x}I_x[v_{yyy}] - \bar{v} v_{yyy}, \\
&U := - \bar{v}_{xyy}u^0 + \bar{v}_{x}u^0_{yy}.
\end{align*}

\noindent We record here the identity: 
\begin{align}  \label{beta.a}
\p_{xy} \{ \bar{u}^2 q_y \} = & 2 \bar{u}_{x} \bar{u}_{y} q_y + 2 \bar{u} \bar{u}_{xy} q_y + 2 \bar{u} \bar{u}_{y} q_{xy} + 2 \bar{u}_{} \bar{u}_{x} q_{yy} + \bar{u}^2 q_{xyy}.
\end{align}

We will approximate the system (\ref{origPrLay}) by introducing the parameter $\theta > 0$. First, define the profile: 
\begin{align*}
\bar{u}_{}^{(\theta)} := \bar{u} + \theta, \bar{v}^{(\theta)} = \bar{v}. 
\end{align*}

\noindent  It is clear that $\p_x \bar{u}^{(\theta)} + \p_y \bar{v}^{(\theta)} = 0$. Define now the solution $[u^{(\theta)}, v^{(\theta)}]$ to the following system:
\begin{align}
\begin{aligned} \label{sysb}
&\bar{u}_{}^{(\theta)} u^{(\theta)}_x + \bar{u}^{(\theta)}_x u^{(\theta)} + \bar{v}^{(\theta)} u^{(\theta)}_y + \bar{u}^{(\theta)}_{y} v^{(\theta)} - \p_{yy} u^{(\theta)} = g_1, \\
&\p_x u^{(\theta)} + \p_y v^{(\theta)} = 0, \\
&u^{(\theta)}|_{y = 0} = \theta, v^{(\theta)}|_{y = 0} = 0, v^{(\theta)}|_{x = 0} = \bar{V}_0(y). 
\end{aligned}
\end{align}

We may also define the corresponding quotient: 
\begin{align} \label{qb}
q^{(\theta)} = \frac{v^{(\theta)}}{\bar{u}^{(\theta)}}, 
\end{align}

\noindent  which satisfies the following
\begin{align}
\begin{aligned} \label{sysb.beta}
&- \p_{xy} \{ |\bar{u}^{(\theta)}|^2 \p_y q^{(\theta)}\} + \p_y^4 v^{(\theta)} + \Lambda_\theta[v^{(\theta)}] + U_\theta[u^{0, (\theta)}] = \p_{xy}g_1, \\
&q^{(\theta)}|_{y = 0} = \p_y q^{(\theta)}|_{y = 0} = 0, \\
&q^{(\theta)}|_{x = 0} = f_0^\theta(y) := \frac{1}{\bar{u}^\theta}|_{x = 0}(y) \bar{V}_0(y),
\end{aligned}
\end{align}

\noindent  where $u^{0, (\theta)} := u^{(\theta)}|_{x = 0}$, and $\Lambda_\theta, U_\theta$ are:
\begin{align*}
&\Lambda_\theta := \bar{v}^{(\theta)}_{xyy}I_x[v^{(\theta)}_y] + \bar{v}^{(\theta)}_{yy} v^{(\theta)}_y  - \bar{v}^{(\theta)}_x I_x[v^{(\theta)}_{yyy}] - \bar{v}^{(\theta)} v^{(\theta)}_{yyy}, \\
&U_\theta := - \bar{v}^{\theta}_{xyy} u^{0,\theta} + \bar{v}^{(\theta)}_x u^{0,\theta}_{yy}.
\end{align*}

\begin{remark}[Notation]   We will drop the superscript $(\theta)$ from here on out. It will be understood that we are dealing with the system in (\ref{sysb.beta}) for $\theta > 0$, and all estimates stated will be independent of the parameter $\theta$. 
\end{remark}

Our aim now is to derive compatibility conditions for the initial data. By computing $\p_x$ of (\ref{origPrLay}) and evaluating at $y = 0$, we obtain the condition: 
\begin{align*}
v_{yyy}|_{y = 0} = \p_x g_1|_{y = 0} \text{ on } (0,L). 
\end{align*}

\noindent  We therefore assume the compatibility condition: 
\begin{align} \label{compatibility.1}
v_{yyy}|_{x = 0, y = 0} (= \bar{v}^i_{pyyy}|_{x = 0}(0)) = \p_x g_1|_{x = 0, y = 0}.  
\end{align}

\noindent  Note that all compatibility conditions are placed on $\bar{v}^i_p|_{x = 0}$. This is because these compatibility conditions occur at $y = 0$, for which $\bar{v}^i_p|_{x = 0} = \bar{V}_0$ (recall the definition of $\bar{V}_0$ in (\ref{origPrLay})). We also require the second-order compatibility which can be obtained as follows. Taking $\p_x$ of (\ref{eval.2}): 
\begin{align*}
- \p_x \{ - \bar{u} v_{yy}+ v \bar{u}_{yy} + \bar{v} u_{yy} - u \bar{v}_{yy} \} + \p_y^4 v = \p_{xy}g_1. 
\end{align*}

\noindent  Evaluating at $y = 0$ gives the identity: 
\begin{align*}
\p_y^4 v|_{y = 0} = \p_{xy} g_1|_{y = 0} \text{ on } (0,L).
\end{align*}

\noindent  We thus assume the compatibility at $x = 0, y = 0$: 
\begin{align} \label{compatibility.2}
\p_y^4 v|_{x = 0}(y = 0) (= (\bar{v}^i_p|_{x = 0})''''(0)) = \p_{xy}g_1|_{y =0}(x = 0). 
\end{align}

\noindent Starting from the $q$ formulation in (\ref{origPrLay.beta}), we will further distribute on the Rayleigh term: 
\begin{align*}
- \p_y \{ \bar{u}^2 q_{xy} \} - \p_y \{2 \bar{u} \bar{u}_{x} q_y \} + \p_y \p_x \{ \bar{v} u_y - u\bar{v}_{y} \} + \p_y \p_y^3 v = \p_y \p_x g_1. 
\end{align*}

\noindent  We now compute at $\{x = 0\}$:
\begin{align} \n
\bar{u}^2 q_{xy} = &- \int_y^\infty \p_y\{ \bar{u}^2 q_{xy} \} \ud y' \\  \n
= & \int_y^\infty \p_y \Big\{ \p_x g_1 - \p_y^3 v + 2 \bar{u} \bar{u}_{x} q_y - \bar{v}_{x} u^0_y + \bar{v}  v_{yy} - v_y \bar{v}_{y} + u^0 \bar{v}_{xy}  \Big\} \ud y' \\ \label{bahumbug}
= & - \{ \p_x g_1 - \p_y^3 v + 2 \bar{u} \bar{u}_{x} q_y - \bar{v}_{x}u^0_y + \bar{v} v_{yy} - v_y \bar{v}_{y} + u^0 \bar{v}_{xy} \}. 
\end{align}

\noindent  It is clear that all quantities are vanishing at $y = 0$. We thus have that: 
\begin{align*}
\bar{u} q_{xy}|_{x = 0} \in L^2. 
\end{align*}

\noindent  A computation of $\p_y$ shows: 
\begin{align*}
&\p_{xy} g_1+ \p_y^4 v + \p_y \{ 2 \bar{u}_{x}\bar{u} q_y - \bar{v}_{x} u^0_y + \bar{v} v_{yy} - v_y \bar{v}_{y} + u^0 \bar{v}_{yy} \}|_{y = 0} \\
& = \p_{xy}g_1(0,0) + \p_y^4v|_{x = 0}(y  = 0) = 0. 
\end{align*}

\noindent  Thus $q_{xy}$ itself is in $L^2$. Using this we may easily bootstrap to higher order in $\p_y$ compatibility conditions for $v^0$ which we refrain from writing. These conditions in turn assure that: 
\begin{lemma} \label{lemma.compat.1}   Assume the compatibility conditions on $V_0$ given in (\ref{compatibility.1}) and (\ref{compatibility.2}). Assume also higher-order compatibility conditions on $V_0$ at $y = 0$ which we do not explicitly specify. Assume exponential decay on $\p_y^k V_0$ for $k \ge 1$. Then there exist functions $f_k(y) \in L^2_w(\mathbb{R}_+) \cap C^\infty(\mathbb{R}_+)$ for exponential weight $w$ such that 
\begin{align} \label{whois}
\p_x^k q_y|_{x = 0} = f_k(y) \in L^2_w(\mathbb{R}_+) \text{ for } k \ge 1.
\end{align} 

\noindent  Moreover, $f_k$ depend only on the given profile $V_0$ and the forcing term $g_1$. 
\end{lemma}

Our task now is to establish criteria on the initial data, $\bar{v}^i_p|_{x = 0}$ so that $u^i_p|_{x = 0}$ can be bounded. We evaluate the velocity equation (\ref{Orig.Orig.Pr}) at $x = 0$ to obtain the equation: 
\begin{align} \label{wknd}
&L^1_{v_{\parallel}} u^0 = f - r(y), \hspace{5 mm} u^0(0) = - u^{i}_e|_{x = 0}(0). 
\end{align}

\noindent To invert this for $u^0$, we assume: 
\begin{align} \label{integral.cond}
u_{\parallel y}|_{x = 0}(0) u^{i}_e|_{x = 0}(0) - \int_0^\infty u_{\parallel} e^{-\int_1^y v_{\parallel}} \{f(y) - r(y) \} \ud y  = 0,
\end{align}

\noindent where $r(y) := \bar{v}^i_p u_{\parallel y} - u_{\parallel} \bar{v}^i_{py}$.

\begin{lemma}   Assume the integral condition, (\ref{integral.cond}) is satisfied by the initial data $\bar{v}^i_p|_{x = 0}$. Then the solution $u^0$ to (\ref{wknd}) exists and satisfies: 
\begin{align} \label{base1}
&|\p_y^k u^0 e^{My}|_\infty \le C_{K,M} (v^0, g_1) \text{ for }k \ge 1, \\
&u^0(0) = - u^i_e|_{x = 0}(0) \text{ and }  \lim_{y \uparrow \infty} u^0 = 0. 
\end{align}
\end{lemma}
\begin{proof} First, we compute the Wronskian of $u_{\parallel}$ and $\tilde{u}_s$: 
\begin{align*}
W = u_{\parallel} \tilde{u}_{sy} - \tilde{u}_s u_{\parallel y} = u_{\parallel}(1)^2 \exp \Big[ \int_1^y v_{\parallel} \Big].
\end{align*}

Next, we express the solution to (\ref{wknd}) in the following manner: 
\begin{align} \n
u^0 = &- u^i_e|_{x = 0}(0) \frac{\tilde{u}_s}{\tilde{u}_s(0)} + c_1 u_{\parallel} - \frac{1}{u_s(1)^2} \tilde{u}_s \int_0^y u_{\parallel} e^{- \int_1^z v_{\parallel}} \{f(z) - r(z) \} \ud z \\ \label{yanks}
& + \frac{1}{u_s(1)^2} u_s \int_0^y \tilde{u}_s e^{-\int_1^z v_{\parallel}} \{ f(z) - r(z) \} \ud z. 
\end{align}

We now compute:
\begin{align*}
\tilde{u}_s(0) = - \frac{u_s(1)^2}{u_{\parallel y}(0)} e^{\int_1^0 v_{\parallel}}.
\end{align*}

Using this, we now evaluate at $y = \infty$ and observe that the terms with a $\tilde{u}_s$ prefactor vanish according to the integral condition, (\ref{integral.cond}).
\begin{align*}
\frac{u_{\parallel y}(0)}{u_s(1)^2} u^i_e|_{x = 0}(0) e^{-\int_1^0 v_{\parallel}} - \frac{1}{u_s(1)^2} \int_0^\infty u_{\parallel} e^{- \int_1^y v_{\parallel}}\{f - r(y) \} \ud y = 0. 
\end{align*} 

This proves that $u^0$ as defined in (\ref{yanks}) is bounded as $y \uparrow \infty$. We next notice that the derivative of $\int_0^y u_{\parallel} e^{-\int_1^\infty v_{\parallel}} \{ f - r(z) \}$ is the integrand itself, which decays fast enough to eliminate $\tilde{u}_s$ at $\infty$. Therefore we also see that $\p_y^k u^0$ for $k \ge 1$ decays rapidly. 

Finally, we need to ensure that $u^0 \rightarrow 0$ as $y \uparrow \infty$. It is clear that $L^1_{v_{\parallel}} u^0 = 0$, and so we are free to modify $u^0$ by factors of $u_\parallel$. Thus we modify (\ref{yanks}) by subtracting off a factor of $c u_\parallel$, for $c$ appropriately selected so as to ensure $u^0(\infty) = 0$. 
\end{proof}

\begin{remark}  Note that the conditions (\ref{integral.cond}) and compatibility conditions at $\{y = 0\}$ together define a large class of data. For instance, we could select $\frac{\bar{v}^i_p}{u_{\parallel}}|_{x = 0}$ to be an increasing function and specify the higher order derivatives, $\p_y^k \bar{v}^i_p(0)$ as in (\ref{compatibility.1}) and (\ref{compatibility.2}).
\end{remark}

Summarizing the above, 
\begin{lemma}   Assume smooth data, $\bar{v}^i_p|_{x = 0}$, are prescribed that satisfies the compatibility conditions (\ref{compatibility.1}), (\ref{compatibility.2}), as well as higher order compatibility conditions. Assume also that $\bar{v}^i_p|_{x = 0}$ satisfies the integral condition (\ref{integral.cond}). Let $q = \frac{v}{\bar{u}^0_p}$ solve (\ref{origPrLay.beta}) and $u^0$ be constructed from $v$ via (\ref{wknd}). Then $[u = u^0 - \int_0^x v_y, v]$ solve (\ref{origPrLay}). Further, let $[\bar{u}^i_p, \bar{v}^i_p]$ be reconstructed from $[u, v]$ using (\ref{antiPsi}). Then $[\bar{u}^i_p, \bar{v}^i_p]$ are solutions to (\ref{Orig.Orig.Pr}).
\end{lemma}


\subsection{Linearized Prandtl Estimates}

Let $\chi$ denote the cut-off function from (\ref{basic.cutoff}). Fix $w = e^{Ny}$ for some large $N$. Denote by $q^{(k)} := \p_x^k q$. We will now define several norms: 
\begin{align} 
\begin{aligned} \label{norm.X.layer}
&\| q \|_X :=  \sup_{0 \le x_0 \le L} \Big[ \|\bar{u} q_{xy} \|_{x = x_0} + \|q_{yyy} w \chi \|_{x= x_0} \Big]  + \| \sqrt{\bar{u}} q_{xyy} w \| + \| v_{yyyy} w \|, \\ 
&\| q \|_{\mathcal{E}} := \sup_{0 \le x_0 \le L} \| \bar{u} q_{xy} \|_{x= x_0} + \| \sqrt{\bar{u}} q_{xyy} \| \\
&\| q \|_{\mathcal{H}} := \sup_{0 \le x_0 \le L} \| q_{yyy} w\{1 - \chi \}\|_{x = x_0} + \| v_{yyyy} w \{1 - \chi\} \| + \| q_{xyy} w \{1 - \chi\} \| \\
& \| q \|_{X_k} := \| q^{(k)} \|_X,  \hspace{3 mm} \| q \|_{\mathcal{E}_k} := \| q^{(k)} \|_{\mathcal{E}}, \hspace{3 mm} \| q \|_{\mathcal{H}_k} := \| q^{(k)} \|_{\mathcal{H}}, \\
& \| q \|_{X_{\langle k \rangle}} = \sum_{i = 0}^k \| q \|_{X_{i}}.
\end{aligned}
\end{align}

\begin{remark}[Notation]   The notation $p_{k}$ will denote an inhomogeneous polynomial of one variable of unspecified power in the quantity $\| \bar{q} \|_{X_k}$. Similarly. $p_{\langle k \rangle}$ will denote such a polynomial in the quantity $\| \bar{q} \|_{X_{\langle k \rangle}}$. In general, we will suppress those constants which depend on $\| \bar{q} \|_{X_{\langle k \rangle}}$, and only display those which depend on $\| \bar{q} \|_{X_{\langle k + 1 \rangle}}$.
\end{remark}

\begin{lemma}   The following inequalities are valid: 
\begin{align*}
&\| q^{(k)}_{xy} \| + \| \{q^{(k)}_y, q^{(k)}_{yy} \} w \| + \| q^{(k)} \langle y \rangle^{-1} \| \le o_L(1)(1 + \| q \|_{X_k}) \\
& \|v^{(k)} \langle y \rangle^{-1} \| + \| \{v^{(k)}_{y}, v^{(k)}_{yy}, v^{(k)}_{yyy}, v^{(k)}_{xy} \} w \| \le o_L(1) (1 + \| q \|_{X_k}) \\
&\| q^{(k)}_{xy} w \| + \| v^{(k)}_{xyy} w \| \lesssim 1 + \| q \|_{X_k} \\
&\| q^{(k)} \|_\infty + \| v^{(k)}_y \|_\infty + \| q^{(k)}_y \|_{\infty, \ge 1} \lesssim 1 + o_L(1) p(\| q \|_{X_{\langle k \rangle}}) \\
&\| v^{(k)}_{yy}, v^{(k)}_{yyy} \|_\infty \le 1 + o_L(1) \| q \|_{X_{\langle k + 1 \rangle}}.
\end{align*}
\end{lemma}
\begin{proof} The first step is to obtain control over $\| q^{(k)}_{xy} \|$ via interpolation. 
\begin{align*}
\| q^{(k)}_{xy} \{1 - \chi(\frac{y}{\delta}) \} \| \le \delta^{-1} \| \bar{u} q^{(k)}_{xy} \| \le L \delta^{-1} \sup_x |\bar{u} q^{(k)}_{xy}| \le L \delta^{-1} \| q \|_{X_k}. 
\end{align*}

\noindent Near the $\{y = 0\}$ boundary, one interpolates: 
\begin{align*}
|( \chi(\frac{y}{\delta}) \p_y \{ y \}, |q^{(k)}_{xy}|^2)| \lesssim &\| \chi y q^{(k)}_{xyy} \|^2 + ( \frac{y}{\delta} \chi'(\frac{y}{\delta}), |q^{(k)}_{xy}|^2) \\
\lesssim & \delta \| q^{(k)} \|_{X}^2 + L^2 \delta^{-2} \| q^{(k)} \|_X^2. 
\end{align*}

\noindent  Optimizing $\sqrt{\delta} + L \delta^{-1}$, one obtains $\delta = L^{2/3}$. Thus, $\| q^{(k)}_{xy} \| \lesssim L^{1/3} \| q \|_{X_k}$. From here, a basic Poincare inequality gives: 
\begin{align*}
\| q^{(k)}_y \| = \| q^{(k)}_y|_{x = 0} + \int_0^x q^{(k)}_{xy} \| \lesssim \sqrt{L} |q^{(k)}_y|_{x = 0}\| + L \| q^{(k)}_{xy} \|
\end{align*}

\noindent  From here, Hardy inequality gives immediately $\| q^{(k)} \langle y \rangle^{-1} \| \le \| q^{(k)}_y \|$. 

The next step is to establish the uniform bound via straightforward Sobolev embedding: 
\begin{align*}
|q^{(k)}|^2 \lesssim \sup_x |q^{(k)}_y \langle y \rangle\|^2 \lesssim |q^{(k)}_y|_{x = 0} \langle y \rangle\|^2 + L \| q^{(k)}_{xy} \langle y \rangle \|^2 \lesssim 1 + o_L(1) \| q \|_{X_k}. 
\end{align*}

A Hardy computation gives:  
\begin{align*}
\| q^{(k)}_{xy} w \{1 - \chi\} \| \lesssim& \| q^{(k)}_{xy} \|_{2,loc} + \| q^{(k)}_{xyy} w \{1 - \chi\} \| \lesssim \| q \|_{X_k}.
\end{align*}

We record the following expansions which follow from the product rule upon recalling that $v = \bar{u} q$:
\begin{align}
\begin{aligned} \label{expressionsk}
&|v^{(k)}| \lesssim \sum_{j = 0}^k |\bar{u}^j q^{(k-j)}| \\
&|v^{(k)}_y| \lesssim \sum_{j = 0}^k |\bar{u}^j_{y} q^{(k-j)}| + |\bar{u}^j q^{(k-j)}_y| \\
&|v^{(k)}_{yy}| \lesssim \sum_{j = 0}^k |\bar{u}^j_{yy} q^{(k-j)}| + |\bar{u}^j_{y} q^{(k-j)}_{y}| + |\bar{u}^j q^{(k-j)}_{yy}| \\
&|v^{(k)}_{yyy}| \lesssim \sum_{j = 0}^k |\bar{u}^j_{yyy} q^{(k-j)}| + |\bar{u}^j_{yy} q^{(k-j)}_y| + |\bar{u}^j_{y} q^{(k-j)}_{yy}|+ |\bar{u}^j q^{(k-j)}_{yyy}| \\
&|v^{(k)}_{xy}| \lesssim \sum_{j = 0}^k |\bar{v}^{j}_{yy} q^{(k-j)}| + |\bar{u}^j_{y} q^{(k-j)}_x| + |\bar{v}^j_{y} q^{(k-j)}_y| + |\bar{u}^j q^{(k-j)}_{xy}| \\
&|v^{(k)}_{xyy}| \lesssim \sum_{j = 0}^k |\bar{v}^j_{yyy} q^{(k-j)}| + |\bar{u}^j_{yy} q^{(k-j)}_x| + |\bar{v}^j_{yy} q^{(k-j)}_y| + |\bar{u}^j_{y} q^{(k-j)}_{xy}| \\
& \hspace{10 mm} + |\bar{v}^j_{y} q^{(k-j)}_{yy}| + |\bar{u}^j q^{(k-j)}_{xyy}|.
\end{aligned}
\end{align}

We will restrict to $k = 0$ for the remainder of the proof, as the argument works for general $k$ in a straightforward way. From (\ref{expressionsk}), $\| v_y\|$ follows obviously. Next, 
\begin{align} \label{broad}
\| v_{yy} \| \lesssim & \| \bar{u}_{yy} q \| + \| \bar{u}_{y} q_y \| + \| \bar{u} q_{yy} \| \lesssim  \sqrt{L} + o_L(1) \| q \|_X. 
\end{align}

\noindent  From here, $\| v_{yyy} \|_{loc}$ can be interpolated in the following way: 
\begin{align*}
&(v_{yyy}, v_{yyy} \chi(\frac{y}{\delta})) = ( \p_y \{ y \} \chi(\frac{y}{\delta}), |v_{yyy}|^2) \\
= & - ( y \chi(\frac{y}{\delta}) v_{yyy}, \p_x^{k-1} v_{yyyy}) - ( y \delta^{-1} \chi'(\frac{y}{\delta}), |v_{yyy}|^2) \\
\lesssim &\delta^2 \| v_{yyyy} \|^2 + \| \psi_\delta v_{yyy} \|^2. 
\end{align*}

\noindent  For the far-field component, we may majorize via: 
\begin{align*}
|(v_{yyy}, v_{yyy} \{1 - \chi(\frac{y}{\delta}) \} )| \lesssim \| \psi_\delta v_{yyy} \|^2.
\end{align*}

\noindent  Here $\psi_\delta = 1 - \chi(\frac{10 y}{\delta})$, the key point being that both $\{1 - \chi(\frac{y}{\delta})\}$ and $\chi'(\frac{y}{\delta})$ are supported in the region where $\psi_\delta = 1$. To estimate this term, we may integrate by parts: 
\begin{align*}
( \psi v_{yyy}, v_{yyy}) = &- ( \psi v_{yy}, v_{yyyy})  - ( \delta^{-1} \psi' v_{yy}, v_{yyy}) \\
= & - ( \psi v_{yy}, v_{yyyy}) + ( \delta^{-2} \frac{\psi''}{2}, |v_{yy}|^2) \\
\lesssim & \delta^2 \|  v_{yyyy} \|^2 + N_\delta \| v_{yy} \|^2. 
\end{align*}

\noindent  Thus, 
\begin{align} \label{interp.3}
\| v_{yyy} \| \le & \delta \| v_{yyyy} \| + N_\delta \| v_{yy} \|.
\end{align}

\noindent  We combine the above with (\ref{broad}) to select $\delta = L^{0+}$ to achieve control over $\|\p_y^j v^{(k)}\|$ for $j = 1,2,3$. From here, we can obtain: 
\begin{align*}
\| q_{yy} \|_{loc} \simeq \| \int_0^y v_{yyy} \ud y' \|_{loc} \le o(1) \| v_{yyy} \|. 
\end{align*}

\noindent  Away from the $\{y = 0\}$ boundary, we estimate trivially: 
\begin{align*}
\| q_{yy}\{1 - \chi(\frac{y}{\delta}) \} w \| &\lesssim \| \bar{u} q_{yy} w \{1 - \chi(\frac{y}{\delta}) \} \|_{2} \\
& \lesssim \sqrt{L} \|\bar{u} q_{yy} w\|_{x = 0} + L \| \sqrt{\bar{u}} q_{xyy} w \|. 
\end{align*}

\noindent  From here, obtaining $\| q_y w \|$ follows from Hardy. We now turn our attention to the weighted estimates for $v_y, v_{yy}, v_{xy}, v_{xyy}$, which follow from (\ref{expressionsk}), whereas for $v_{yyy}$, we use the Prandtl equation to produce the identity: 
\begin{align*}
v_{yyy} = & \bar{u}_{yyy} q + 3 \bar{u}_{yy} q_y + 3 \bar{u}_{y} q_{yy} + \bar{u} q_{yyy} \\
= &  \Big( - \bar{u} \bar{v}_{yy} + \bar{v} \bar{u}_{yy} \Big) q +  3 \bar{u}_{yy} q_y + 3 \bar{u}_{y} q_{yy} + \bar{u} q_{yyy}
\end{align*}

The uniform estimates subsequently follow from straightforward Sobolev embeddings.
\end{proof}

\begin{lemma}[$\p_x^k$ Energy Estimate]   Assume $q$ solves (\ref{origPrLay.beta}) 
\begin{align}
\begin{aligned} \label{Ek}
\| q \|_{\mathcal{E}_k}^2 \lesssim & |\bar{u} \p_x^k q_{xy}|_{x = 0}\|^2  + o_L(1)p_{\langle k + 1 \rangle} (1 + \| q \|_{X_{\langle k \rangle}}^2 ) \\
& + o_L(1) C(u^0) + o_L(1) \| \p_x^{k+1} \p_{xy}g_1 \langle y \rangle \|^2.
\end{aligned}
\end{align}
\end{lemma}
\begin{proof} For this estimate, we work with the $k+1$ times $x$-differentiated version of (\ref{sysb.beta}), which we record below: 
\begin{align}
- \p_x^{k+1} \p_y \{ \bar{u}^2 q_y \} + v^{k+1}_{yyyy} + \p_x^{k+1} J_b(v) + \p_x^{k+1} U_b = \p_x^{k+2} \p_y g_1. 
\end{align} 

We will take inner product of the above equation against $\p_x^{k+1} q$ and then integrate over $x \in [0, x_0]$, where $0 < x_0 < L$. We start with the Rayleigh term: 
\begin{align} \n
- \int_0^{x_0}  ( &\p_x^{k+1} \p_{xy} \{ \bar{u}^2 q_y \}, \p_x^{k+1} q ) =  \int_0^{x_0} (\p_x^{k+2} \{ \bar{u}^2 q_y \}, \p_x^{k+1} q_y ) \\ \label{rara.1}
= & \sum_{j = 0}^{k+2} \sum_{j_1 + j_2 = j} \int_0^{x_0} ( \p_x^{j_1} \bar{u} \p_x^{j_2} \bar{u} \p_x^{k+2 - j} q_y, \p_x^{k+1} q_y ). 
\end{align}

\noindent  The $j = 0$ case corresponds to: 
\begin{align*}
\int_0^{x_0} \frac{\p_x}{2} (\bar{u}^2 \p_x^{k+1} q_y, \p_x^{k+1} q_y ) - \int_0^{x_0} (\bar{u} \bar{u}_{x} \p_x^{k+1} q_y, \p_x^{k+1} q_y ). 
\end{align*}

\noindent  By splitting into $1 \le j \le \frac{k+2}{2}$ and $\frac{k+2}{2} \le j \le k+2$ cases, the remaining terms in (\ref{rara.1}) can be majorized by: 
\begin{align*}
|(\ref{rara.1})[j \neq 0]| \lesssim & \Big\| \frac{\p_x^{\langle \frac{k+1}{2} \rangle} \bar{u}}{\bar{u}} \Big\|_\infty^2 \| \bar{u} \p_x^{\langle k + 1 \rangle} q_y \|^2 \\
& + \Big\| \frac{\p_x^{\langle \frac{k+2}{2} \rangle}\bar{u}}{\bar{u}} \Big\|_\infty \| \p_x^{\langle \frac{k+2}{2} \rangle} q_y \| \| \p_x^{k+2} \bar{u} \|_\infty \| \bar{u} \p_x^{k+1} q_y \| \\
\lesssim & o_L(1) p(\| q_s \|_{X_{\langle k \rangle}}) \| q \|_{X_{\langle k \rangle}}^2.
\end{align*}

\noindent  Summarizing: 
\begin{align*}
 \sup_{x_0 \le L} \int_0^{x_0} &(\p_x \p_x^k \p_{xy} \{ \bar{u}^2 q_y \} , \p_x \p_x^k q ) \\
 &\gtrsim  \sup |\bar{u} \p_x^k q_{xy}\|^2 - C(q_0) - o_L(1) p(\| q_s \|_{X_{\langle k \rangle}}) \| q \|_{X_{\langle k \rangle}}^2. 
\end{align*}

We now move to $\p_y^4$ term: 
\begin{align} \n
&\int_0^{x_0} (\p_x^{k+1} v_{yyyy} , \p_x^{k+1} q )  =- \int_0^{x_0} ( \p_x^{k+1} v_{yyy}, \p_x^{k+1} q_y ) \\ \n
& = \int_0^{x_0} ( \p_x^{k+1} v_{yy}, \p_x^{k+1} q_{yy} ) + \int_0^{x_0} \p_x^{k+1} v_{yy} \p_x^{k+1} q_y(0) \\ \label{sub1}
& = \int_0^{x_0} (\p_x^{k+1} [ \bar{u} q_{yy} + 2 \bar{u}_{y} q_y + \bar{u}_{yy} q ], \p_x^{k+1} q_{yy} ).
\end{align}

\noindent  We have above used that $q_y|_{y = 0} = 0$, which is due to the fact that $b > 0$. First, we will focus on (\ref{sub1}.1), which is: 
\begin{align*}
(\ref{sub1}.1) \gtrsim &\int_0^{x_0} ( \bar{u} \p_x^{k+1} q_{yy}, \p_x^{k+1} q_{yy} ) \\
& - \Big\| \frac{\p_x^{\langle k+1 \rangle} \bar{u}}{\bar{u}} \Big\|_\infty \| \sqrt{\bar{u}} \p_x^{\langle k \rangle} q_{yy} \| \| \sqrt{\bar{u}} \p_x^{k } q_{xyy} \| \\
\gtrsim & \| \sqrt{\bar{u}} \p_x^k q_{xyy} \|^2 - ( 1 + o_L(1) p_{k+1} ) \| q \|_{X_{k}} (1+ o_L(1) \| q \|_{X_{\langle k \rangle}}). 
\end{align*}

Next, for (\ref{sub1}.2): 
\begin{align*}
(\ref{sub1}.2) = &\int_0^{x_0} ( 2 \bar{u}_{y} \p_x^{k+1} q_y, \p_x^{k+1} q_{yy} ) + \sum_{1 \le j \le \frac{k+1}{2}} \int_0^{x_0} ( \p_x^j \bar{u}_{y} \p_x^{k+1 - j} q_y, \p_x^{k+1} q_{yy} ) \\
& + \sum_{\frac{k+1}{2} \le j \le k+1} \int_0^{x_0} ( \p_x^j \bar{u}_{y} \p_x^{k+1-j} q_y, \p_x^{k+1} q_{yy} ). 
\end{align*}

  First, integration by parts gives: 
\begin{align*}
(\ref{sub1}.2.1) = & - \int_0^{x_0} ( \bar{u}_{yy} \p_x^{k+1} q_y, \p_x^{k+1} q_y ) - \int_0^{x_0} \bar{u}_{y}(0) |\p_x^{k+1} q_y(0)|^2 \\
= & \bigO( \| \bar{u}_{yy} \|_\infty \| \p_x^{k} q_{xy} \|^2  ) - \int_0^{x_0} \bar{u}_{y}(0) |\p_x^{k+1} q_y(0)|^2 \\
= & o_L(1) \| q \|_{X_k}^2.
\end{align*}

\noindent  We again note that $q_y(0) = 0$ due to $b > 0$, so that the boundary term vanishes. We treat: 
\begin{align*}
(\ref{sub1}.2.2) = &- \int_0^{x_0} ( \p_x^{k+1} q_y, \p_x^j \bar{u}_{yy} \p_x^{k+1-j} q_y ) - \int_0^{x_0} ( \p_x^{k+1} q_y, \p_x^j \bar{u}_{y} \p_x^{k+1 -j} q_{yy} ) \\
& - \int_0^{x_0} \p_x^j \bar{u}_{y} \p_x^{k-j+1} q_y \p_x^{k+1} q_y(0) \\
=  &- \int_0^{x_0} ( \p_x^{k+1} q_y, \p_x^j \bar{u}_{yy} \p_x^{k+1-j} q_y ) - \int_0^{x_0} (  \p_x^{k+1} q_y, \p_x^j \bar{u}_{y} \p_x^{k+1 -j} q_{yy} )\\
\lesssim & \| \p_x^{\langle \frac{k+1}{2}\rangle} \bar{u}_{yy} \|_\infty \| \p_x^{\langle k \rangle} q_y \| \| \p_x^{k+1} q_y \| \\
& + \| \p_x^{\langle \frac{k+1}{2} \rangle} \bar{u}_{y} \|_\infty \| \p_x^{\langle k \rangle} q_{yy} \| \| \p_x^{k+1} q_{y} \| \\
\lesssim & o_L(1) p_{\langle k \rangle} ( \| q \|_{X_{\langle k \rangle}}^2 + 1 ).
\end{align*}

  A similar integration by parts produces 
\begin{align*}
(\ref{sub1}.2.3) \lesssim & \| \p_x^{\langle \frac{k+1}{2} \rangle} q_y \| \| \p_x^{\langle k + 1 \rangle} \bar{u}_{yy} \|_\infty \| \p_x^{k} q_{xy} \| \\
& + \| \p_x^{\langle \frac{k+1}{2} \rangle} q_{yy} \| \| \p_x^{\langle k+1 \rangle} \bar{u}_{y} \|_\infty \| \p_x^{k} q_{xy} \|  \\
\lesssim  & o_L(1) p_{\langle k + 1 \rangle} ( \| q \|_{X_{\langle k \rangle}}^2 + 1 ). 
\end{align*}

Next, we move to: 
\begin{align*}
(\ref{sub1}.3) = & \sum_{j = 0}^{k+1} \int_0^{x_0} ( \p_x^j \bar{u}_{yy} \p_x^{k+1 - j} q, \p_x^{k+1} q_{yy}).
\end{align*}

 We must split the above term into several cases. First, let us handle the $j = 0$ case for which (\ref{coe.2}) gives us the required bound: 
\begin{align*}
|(\ref{sub1}.3)[j=0]| \lesssim& \| \frac{1}{\bar{u}} \bar{u}_{yy} \langle y \rangle \|_\infty \| \p_x^{k} q_{xy} \| \| \sqrt{\bar{u}} \p_x^k q_{xyy} \| \\
\lesssim &o_L(1) \| q \|_{X_{k}}^2. 
\end{align*}

We now handle the case of $1 \le j \le k/2$, which requires a localization using $\chi$ as defined in (\ref{basic.cutoff}):
\begin{align*}
|(\ref{sub1}.3)[1-\chi]| \lesssim &\| \p_x^{\frac{k}{2}} \bar{v}_{yyy} \|_{L^\infty_x L^2_y} \| \p_x^{\langle k \rangle} q \|_{L^2_x L^\infty_y} \| q \|_{X_k} \\
\lesssim & o_L(1) p_{\langle k \rangle}  \| q \|_{X_{\langle k \rangle}}^2.
\end{align*}

  For the localized component, we integrate by parts in $y$: 
\begin{align} \n
(\ref{sub1}.3)[\chi] = & \sum_{j = 0}^{k+1} \int_{0}^{x_0} ( \p_x^k q_y \chi', \p_x^j \bar{u}_{yy} \p_x^{k+1-j}q ) +  \sum_{j = 0}^{k+1} \int_0^{x_0} ( \p_x^k q_y \chi, \p_x^j \bar{u}_{yyy} \p_x^{k+1-j} q ) \\ \label{fetty}
& +  \sum_{j = 0}^{k+1} \int_0^{x_0} \langle \p_x^k q_y \chi, \p_x^j \bar{u}_{yy} \p_x^{k+1-j} q_y \rangle \\ \n
\lesssim & \| \p_x^{k-1} q_{xy} \| \| \p_x^{\frac{k}{2}} \bar{v}_{yyy} \| \| \p_x^{\langle k \rangle} q \|_{L^\infty_{loc}}  + \| \p_x^{k-1} q_{xy} \| \| \p_x^{\frac{k}{2}} \bar{v}_{yyyy} \| \| \p_x^{\langle k \rangle} q \|_{L^\infty_{loc}} \\ \n
& + \| \p_x^{k-1} q_{xy} \| \| \p_x^{\frac{k}{2}} \bar{v}_{yyy} \|_\infty \| \p_x^{\langle k \rangle} q_{xy} \| \\  \n
\lesssim & o_L(1)p_{\langle k \rangle} \| q \|_{X_{\langle k \rangle}}^2.  
\end{align}

We now treat the case in which $k/2 \le j \le k$, which still requires localization 
\begin{align*}
|(\ref{sub1}.3)[\chi_{\ge 1}]| \lesssim  &\| \p_x^k q_{xyy} \sqrt{\bar{u}} \| \| \p_x^{\frac{k}{2}} q \|_{L^2_x L^\infty_y} \|\p_x^{\langle k - 1 \rangle} \bar{v}_{yyy} \|_{L^\infty_x L^2_y} \\
\lesssim & o_L(1) p_{\langle k \rangle} \| q \|_{X_{\langle k \rangle}}^2.
\end{align*}

Finally, we deal with the case when $y \lesssim 1$ for $k/2 \le j \le k$, which again requires integration by parts in $y$ as in (\ref{fetty}): 
\begin{align*}
|(\ref{fetty}[j \ge k/2])| \lesssim &\| \p_x^{k-1} q_{xy} \| \| \p_x^{\langle k - 1 \rangle} \bar{v}_{yyy} \| \| \p_x^{\frac{k}{2}} q \|_{L^\infty_{loc}} \\
& + \| \p_x^{k-1} q_{xy} \| \| \p_x^{\langle k - 1 \rangle} \bar{v}_{yyyy} \| \| \p_x^{\frac{k}{2}} q \|_{L^\infty_{loc}} \\
& + \| \p_x^{k-1} q_{xy} \| \| \p_x^{\langle k - 1 \rangle}\bar{v}_{yyy} \|_\infty \| \p_x^{\langle k/2 \rangle} q_{xy} \| \\
\lesssim & o_L(1)p_{\langle k \rangle}(1 + \| q \|_{X_{\langle k \rangle}}^2). 
\end{align*}

We now move to the $\Lambda$ terms:
\begin{align*}
\int_0^{x_0}( & \p_x^{k+1} \{ \bar{v}_{xyy}I_x[v_y] \}, \p_x^{k+1} q) \\
 \lesssim &\| \p_x^{\langle \frac{k}{2} \rangle} \bar{v}_{xyy} \langle y \rangle \|_\infty \| \p_x^{\langle k+1 \rangle} I_x[v_y] \| \| \p_x^{k+1} q_y \| \\
& + \| \p_x^{\langle k+1 \rangle} \bar{v}_{xyy}  \| \| \p_x^{\langle \frac{k+1}{2} \rangle} I_x[v_y] \langle y \rangle \|_\infty \| \p_x^{k+1} q_y \|. 
\end{align*}

\begin{align*}
\int_0^{x_0} ( & \p_x^j \bar{v}_{yy} \p_x^{k+1-j} v_y, \p_x^{k+1} q  ) \\
\lesssim & \| \p_x^{\langle \frac{k+1}{2} \rangle} \bar{v}_{yy} \langle y \rangle \|_\infty \| \p_x^{\langle k+1 \rangle} v_y \| \| \p_x^{k+1} q_y \| \\
& + \| \p_x^{\langle k + 1 \rangle} \bar{v}_{yy} \| \| \p_x^{\langle \frac{k+1}{2} \rangle} v_y \langle y \rangle \|_\infty \| \p_x^{k+1} q_y \|. 
\end{align*}

\begin{align*}
\int_0^{x_0} ( \p_x^j &\bar{v}_{x} \p_x^{k+1-j} I_x[\p_y^3 v], \p_x^{k+1} q ) \\
 = &- \int_0^{x_0} ( \p_x^j \bar{v}_{xy} \p_x^{k+1-j} I_x[\p_y^2 v], \p_x^{k+1} q ) \\
 & - \int_0^{x_0} ( \p_x^j \bar{v}_{x} \p_x^{k+1-j} I_x[\p_y^2 v], \p_x^{k+1} q_y ) \\
\lesssim & \| \p_x^{\langle \frac{k+1}{2} \rangle} \bar{v}_{xy} \langle y \rangle \|_\infty \| \p_x^{\langle k+1 \rangle} I_x[v_{yy}] \| \| \p_x^{k+1} q_y \| \\
& + \| \p_x^{\langle \frac{k+1}{2}\rangle} \bar{v}_{x} \|_\infty \| \p_x^{\langle k + 1 \rangle} I_x[v_{yy}] \| \| \p_x^{k+1} q_{y} \| \\
& + \| \p_x^{\langle \frac{k+1}{2} \rangle} I_x[v_{yy}] \langle y \rangle \|_\infty \| \p_x^{\langle k + 1 \rangle} \bar{v}_{xy} \| \| \p_x^{k+1} q_y\| \\
& + \| \p_x^{\langle k +1 \rangle} \bar{v}_{xy} \| \| \p_x^{\langle \frac{k+1}{2} \rangle} I_x[v_{yy}] \langle y \rangle \|_\infty \| \p_x^{k+1} q_{y} \|.
\end{align*}

\begin{align*}
\int_0^{x_0} &( \p_x^j \bar{v} \p_x^{k+1 - j} v_{yyy}, \p_x^{k+1} q ) \\
= & - \int_0^{x_0} ( \p_x^{k+1 - j} v_{yy} \p_x^j \bar{v}_{y}, \p_x^{k+1} q ) - \int_0^{x_0} ( \p_x^{k+1 - j} v_{yy} \p_x^j \bar{v}, \p_x^{k+1} q_y )  \\
 \lesssim &\| \p_x^{\langle \frac{k+1}{2} \rangle} \bar{v}_{y} \langle y \rangle \|_\infty \| \p_x^{\langle k + 1 \rangle} v_{yy} \| \| \p_x^{k+1} q_y \| \\
& + \| \p_x^{\langle \frac{k+1}{2} \rangle} \bar{v}_{} \|_\infty \| \p_x^{\langle k + 1 \rangle} v_{yy} \| \| \p_x^{k+1} q_{y} \| \\
& + \| \p_x^{\langle \frac{k+1}{2} \rangle} v_{yy} \langle y \rangle \|_\infty \| \p_x^{\langle k+1 \rangle} \bar{v}_{y} \| \| \p_x^{k+1} q_y \| \\
& + \| \p_x^{\langle \frac{k+1}{2} \rangle} \bar{v} \|_\infty \| \p_x^{\langle k + 1 \rangle} v_{yy} \| \| \p_x^{k} q_{xy} \|.
\end{align*}

\noindent  Summarizing the $\Lambda$ contributions: 
\begin{align*}
&\sup_{x_0 \le L} \int_0^{x_0} (\p_x \p_x^k \Lambda, \p_x \p_x^k q )  \lesssim o_L(1)p_{\langle k + 1 \rangle} (1 + \| q \|_{X_{\langle k \rangle}}^2 ).
\end{align*}

Finally, we have the $u^0$ contributions: 
\begin{align*}
& \int_0^{x_0} ( \p_x^{k+1} U, \p_x^{k} q_x )  = \int_0^{x_0} ( u^0 \p_x^{k+1} \bar{v}_{xyy} - u^0_{yy} \p_x^{k+1} \bar{v}_{x}, \p_x^{k} q_x ) \\
& \hspace{10 mm} \lesssim o_L(1) \| u^0, u^0_{yy} \cdot \langle y \rangle \|_\infty \| \p_x^{k+1} \bar{v}_{xyy} \| \| \p_x^{k} q_{xy} \| \\
& \hspace{10 mm} \lesssim o_L(1) p_{\langle k + 1 \rangle}. 
\end{align*}
\end{proof}

\begin{lemma}[$\p_x^k$ $\p_y^4$ Estimate]   Assume $v$ is a solution to (\ref{origPrLay}). Then the following estimate holds: 
\begin{align}  \label{solid.py4.cof}
\begin{aligned}
\| \p_x^k v_{yyyy} \|_{2,loc} \lesssim & \| q \|_{\mathcal{E}_{\langle k \rangle}} + o_L(1) p_{\langle k + 1 \rangle} (1 + \| q \|_{X_{\langle k \rangle}} ) \\
& + o_L(1) C(u^0) + \| \p_{xy} \p_x^k g_1 \|_{2,loc}.
\end{aligned}
\end{align}
\end{lemma}

\begin{proof} We apply $\p_x^k$ to the equation to obtain the following pointwise inequality:
\begin{align}
\begin{aligned} \label{sasak}
|v^{(k)}_{yyyy}| \lesssim &|\bar{u}^{j_1} \bar{u}^{j_2} q^{(k-j)}_{xyy}| + |\bar{u}_{x}^{j_1} \bar{u}_{y}^{j_2} q_y^{k-j}| + |\bar{u}^{j_1} \bar{u}_{xy}^{j_2} q_y^{k-j}| + |\bar{u}^{j_1} \bar{u}_{y}^{j_2} q_{xy}^{k-j}| \\
& + |\bar{u}^{j_1} \bar{u}_{x}^{j_2} q^{(k-j)}_{yy}| + |\bar{v}^j_{xyy} I_x[v^{(k-j)}_y]| + |\bar{v}^j_{yy}v^{(k-j)}_y| + |v^{j}_{sx} I_x[v^{(k-j)}_{yyy}]| \\
& + |\bar{v}^j v^{(k-j)}_{yyy} | + |v^{(k)}_{sxyy} u^0| + |v^{(k)}_{sx} u^0_{yy}| + |\p_{xy}g_1^k|.
\end{aligned}
\end{align}

Placing the terms on the right-hand side above in $L^2_{loc}$ gives the desired result: 
\begin{align*}
&\| (\ref{sasak}.1)\| \lesssim \| \bar{v}_{y}^{\langle k - 1 \rangle}, \bar{v}_{yy}^{\langle k - 1 \rangle} \|_\infty^2 \| \bar{u} q_{xyy}^{\langle k \rangle} \|, \\
&\| (\ref{sasak}.2) \| \lesssim \| \bar{v}_{y}^{\langle k - 1 \rangle} \|_\infty \| \bar{v}_{yy}^{\langle k - 1 \rangle} \|_\infty \| q^{\langle k \rangle}_y \| \\
&\| (\ref{sasak}.3) \| \lesssim \| \bar{v}_{y}^{\langle k - 1 \rangle} \|_\infty \| \bar{v}_{yy}^{\langle k - 1 \rangle} \|_\infty \| q^{\langle k \rangle}_y \| \\
&\| (\ref{sasak}.4) \| \lesssim \| \bar{v}^{\langle k - 1 \rangle}_{y} \|_\infty \| \bar{v}^{\langle k - 1 \rangle}_{yy} \|_\infty \| q^{\langle k \rangle}_{xy} \| \\
&\| (\ref{sasak}.5) \| \lesssim \| \bar{v}_{y}^{\langle k - 1 \rangle} \|_\infty \| \bar{v}^{\langle k \rangle}_{y} \|_\infty \| q^{\langle k \rangle}_{yy} \| \\
&\| (\ref{sasak}.6, 7) \| \lesssim \| \bar{v}^{\langle k + 1 \rangle}_{yy} \| \| v_y^{\langle k \rangle} \|_\infty \\
&\| (\ref{sasak}.8, 9) \| \lesssim \|\bar{v}^{\langle k + 1 \rangle} \|_\infty \| v^{\langle k \rangle}_{yyy} \| \\
&\| (\ref{sasak}.10, 11) \| \lesssim \| u^0, u^0_{yy} \langle y \rangle^2 \|_\infty \| \bar{v}^{k+1}_{yy} \|
\end{align*}
\end{proof}

We now move to a $\| \cdot \|_{\mathcal{H}_k}$ estimate, for which we first recall the definition in (\ref{norm.X.layer}).
\begin{lemma}[Weighted $\p_x^k H^4$]    Assume $q$ solves (\ref{origPrLay.beta}). Then the following estimate is valid:
\begin{align}
\begin{aligned}
\| q \|_{\mathcal{H}_k}^2 \lesssim & \| \p_x^k \p_{xy} g_1 \cdot w \chi \|^2 + C_k(q_0) + o_L(1) C(u^0) \\ \label{weight.H4.cof}
& + o_L(1) p_{\langle k + 1 \rangle}\Big(1+ \| q \|_{X_{\langle k \rangle}}^2 \Big).
\end{aligned}
\end{align}
\end{lemma}
\begin{proof} We take $\p_x^k$ of equation (\ref{sysb.beta}), which produces: 
\begin{align}
\begin{aligned} \label{produce}
&- \p_{xy} \{ \bar{u}^2 q^{(k)}_y \} + v^{(k)}_{yyyy} + \p_{xy} \{ \sum_{j = 1}^k \sum_{j_1 + j_2 = j} c_{j_1, j_2, j} \p_x^{j_1}\bar{u} \p_x^{j_2} \bar{u}  q^{(k-j)}_y \} \\
& + \p_x^k \Lambda(v) + \p_x^k U = \p_x^k \p_{xy}g_1 
\end{aligned}
\end{align}

We start with the ``main terms", (\ref{produce}.1) and (\ref{produce}.2). We fix $x = x_0$, square the equation, take $L^2(x = x_0)$, and expand to produce the identity:
\begin{align} \n
&\|[\p_{xy} \{\bar{u}^2 q^{(k)}_y \} - v^{(k)}_{yyyy}] \cdot w \{1 - \chi\}\|_{x = x_0}^2 \\ \n
&= \| v^{(k)}_{yyyy} \{1 - \chi \} w \|_{x = x_0}^2 +  \| [ \bar{u}_{}^2 q^{(k)}_{xyy} + 2 \bar{u} \bar{u}_{x} q^{(k)}_{yy} + 2 \bar{u}_{y} \bar{u}_{x} q^{(k)}_y \\ \n
& \hspace{5 mm} + 2 \bar{u} \bar{u}_{xy} q^{(k)}_y + 2 \bar{u} \bar{u}_{y} q^{(k)}_{xy}  ] w \{1- \chi \} \|_{x = x_0}^2 - ( 2 v^{(k)}_{yyyy}, [ \bar{u}^2 q^{(k)}_{xyy} + 2 \bar{u} \bar{u}_{x} q^{(k)}_{yy} \\ \label{riseshinek}
& \hspace{5 mm} + 2 \bar{u}_{y} \bar{u}_{x} q^{(k)}_y + 2 \bar{u} \bar{u}_{xy} q^{(k)}_y + 2 \bar{u} \bar{u}_{y} q^{(k)}_{xy} ] w^2 \{1 - \chi\}^2)_{x = x_0} \\ \n
& \gtrsim  \| v^{(k)}_{yyyy} \{1 - \chi \} w \|_{x = x_0}^2 +  \| [ \bar{u}^2 q^{(k)}_{xyy}  \{1 - \chi \} w \|_{x = x_0}^2 - \sum_{l = 3}^{11} |(\ref{riseshinek}.l)|
\end{align}

\noindent  All terms are estimated in a straightforward manner except for (\ref{riseshinek}.7), so we begin with: 
\begin{align*}
&|(\ref{riseshinek}.3)| \lesssim \| \bar{u}_{x}  \|_\infty^2 \|q^{(k)}_{yy} \cdot w \|_{x = x_0}^2 \\
&|(\ref{riseshinek}.4)| \lesssim \| \bar{u}_{y} \|_\infty^2 \| \bar{v}_{y} \|_\infty^2 \|q^{(k)}_y w \|_{x = x_0}^2 \\
&|(\ref{riseshinek}.5)| \lesssim \| \bar{u} \bar{v}_{yy} \|_\infty^2 \|q^{(k)}_y w \|_{x = x_0}^2,  \\
&|(\ref{riseshinek}.6)| \lesssim \| \bar{u} \bar{u}_{y} w \|_\infty^2 \| q^{(k)}_{xy} \|_{x = x_0}^2, \\
& (\ref{riseshinek}.8)| \lesssim \| \bar{u}_{x} \|_\infty | v^{(k)}_{yyyy} w \{1 -\chi\} \|_{x = x_0} \| q^{(k)}_{yy} w \{1 -\chi\}\|_{x = x_0} \\
&(\ref{riseshinek}.9)| \lesssim \| \bar{u}_{y} \bar{v}_{y}  \|_\infty |q^{(k)}_y w \{1 -\chi\}\|_{x = x_0} \|v^{(k)}_{yyyy} w \{1 -\chi\}\|_{x = x_0}, \\
&(\ref{riseshinek}.10)| \lesssim \| \bar{u} \bar{v}_{yy} \|_\infty |q^{(k)}_y w  \{1 -\chi\}\|_{x = x_0} \|v^{(k)}_{yyyy}w \{1 -\chi\}\|_{x = x_0} \\
&(\ref{riseshinek}.11)| \lesssim \| \bar{u} \bar{u}_{y} w \|_\infty \| q^{(k)}_{xy} \|_{x = x_0} \|v^{(k)}_{yyyy} w \{1 -\chi\}\|_{x = x_0}.
\end{align*}

\noindent  Upon integrating in $x$, we may summarize the above estimates via: 
\begin{align*}
&|(\ref{riseshinek}.3)| + ... +  |(\ref{riseshinek}.6)| + |(\ref{riseshinek}.8)| +... + |(\ref{riseshinek}.11)| \\
&\lesssim o_L(1) p(\| \bar{q} \|_{X_{\langle 1 \rangle}}) (1 + \| q \|_{X_k} ). 
\end{align*}

We thus move to (\ref{riseshinek}.7) for which we integrate by parts once in $y$, expand: 
\begin{align*}
v^{(k)}_{yyy} := & \p_x^k \{ \bar{u} q \}_{yyy} = \p_x^k \{ \bar{u}_{yyy}q + 3 \bar{u}_{yy} q_y + 3 \bar{u}_y q_{yy} + \bar{u} q_{yyy} \} \\
= & \sum_{j = 0}^k c_j \p_x^j \bar{u}_{yyy} \p_x^{k-j} q + 3 c_j \p_x^j \bar{u}_{yy} \p_x^{k-j} q_y + 3 c_j \p_x^j \bar{u}_y \p_x^{k-j} q_{yy} + c_j \p_x^j \bar{u} \p_x^{k-j} q_{yyy}. 
\end{align*}

\noindent First, upon integrating by parts once in $y$ (ignoring commutator terms, which are dealt with in (\ref{riseshinekk})), let us highlight the main positive contribution from the last term above, for $j = 0$:
\begin{align*}
 &2(\bar{u} \p_x^k q_{yyy}, \bar{u}^2 q^{(k)}_{xyyy} \{1 - \chi \}^2 w^2)_{x = x_0} = (2 \bar{u}^3 q^{(k)}_{yyy}, q^{(k)}_{xyyy} \{1 - \chi\}^2 w^2)_{x = x_0} \\
 = & \p_x \| |\bar{u}|^{\frac{3}{2}} q^{(k)}_{yyy} \{1 - \chi \} w \|_{x =x_0}^2 - 3( \bar{u}^2 \bar{u}_x q^{(k)}_{yyy}, q^{(k)}_{yyy} \{1 - \chi \}^2 w^2)_{x = x_0}.
\end{align*} 

\noindent  Hence:
\begin{align} \n 
- 2( v^{(k)}_{yyyy}, &\bar{u}^2 q^{(k)}_{xyy} w^2 \{1 - \chi\}^2)_{x = x_0} \\ \n
= &\p_x \| q^{(k)}_{yyy} |\bar{u}|^{\frac{3}{2}} w \{1 - \chi\} \|_{x = x_0}^2 -  3(|q^{(k)}_{yyy}|^2,  \bar{u}^2 \bar{u}_{x} w^2 \{1 - \chi\}^2)_{x = x_0} \\ \n
& - ( q_{xyy}^k, \p_y \{ [ \bar{u}_{yyy}^j q^{(k-j)} + 3 \bar{u}_{yy}^j q_y^{k-j} + 3 \bar{u}_{y}^j q_{yy}^{k-j} ] \\  \n
& \times \bar{u}^2 w^2 \{1 - \chi\}^2 \} + ( 2 v^{(k)}_{yyy}, q^{(k)}_{xyy} \p_y \{ \bar{u}^2 w^2 \{1 - \chi\}^2 \})_{x = x_0} \\ \label{riseshinekk}
& + \sum_{j = 1}^k c_j (\p_x^j \bar{u} \p_x^{k-j} q_{yyy}, \bar{u}^2 q^{(k)}_{xyyy} w^2 \{1 - \chi \}^2)_{x = x_0}
\end{align}

First, we estimate: 
\begin{align*}
|(\ref{riseshinekk}.2)| \lesssim \| |\bar{u}|^{\frac{3}{2}} q^{(k)}_{yyy} w\{1- \chi\} \|_{x = x_0}^2
\end{align*}

Next, 
\begin{align*}
(\ref{riseshinekk}.3) = & - ( q^{(k)}_{xyy} \bar{u}^j_{yyyy},  q^{(k-j)} \bar{u}^2 w^2 \{1 - \chi\}^2)_{x = x_0} - ( q^{(k)}_{xyy} \bar{u}^j_{yyy}, q^{(k-j)} \bar{u} \bar{u}_{y} w^2 \{1 - \chi\}^2)_{x = x_0} \\
& - ( q^{(k)}_{xyy} \bar{u}^j_{yyy}, q^{(k-j)}_y \bar{u}^2 w^2 \{1 - \chi\}^2)_{x = x_0} - ( q^{(k)}_{xyy} \bar{u}^j_{yyy}, q^{(k-j)} \bar{u}^2 2 ww_{y} \{1 - \chi\}^2)_{x = x_0} \\
& - ( q^{(k)}_{xyy} \bar{u}^j_{yyy}, q^{(k-j)} \bar{u}^2 w^2 \{1 - \chi\} \chi')_{x = x_0}. 
\end{align*}

\noindent  We will estimate each term above with the help of the Prandtl identities, which follow from (\ref{Pr.leading}), for $\bar{u}$:
\begin{align}
\begin{aligned} \label{pink}
&|\bar{u}^j_{yyy}| \lesssim |\p_x^j \{ \bar{u} \bar{v}_{yy} + \bar{v} \bar{u}_{yy} \}|, \\
&|\bar{u}^j_{yyyy}| \lesssim |\p_x^j \{ \bar{u}_{y} \bar{v}_{yy} + \bar{u} \bar{v}_{yyy} + \bar{v}_{y} \bar{u}_{yy} + \bar{u} \bar{v} \bar{v}_{yy} + \bar{v}^2 \bar{u}_{yy} \}|
\end{aligned}
\end{align}

\noindent  Inserting this expansion into (\ref{riseshinekk}.3.1) gives: 
\begin{align*}
|(\ref{riseshinekk}.3.1.1)| \lesssim & | (q^{(k)}_{xyy}, q^{\langle k  \rangle} \bar{u}^2 w^2 \{1 - \chi\}^2 \bar{u}^{\langle k \rangle}_{y} \bar{v}^{\langle k \rangle}_{yy})_{x = x_0}| \\
\lesssim & \|q^{(k)}_{xyy} \{1 - \chi\} w\|_{x = x_0} \|q^{\langle k  \rangle}\|_\infty \|\bar{v}^{\langle k -1 \rangle}_{yy} \|_\infty \|\bar{v}^{\langle k \rangle}_{yy} w  \{1 - \chi\}\|_{x = x_0}, \\
|(\ref{riseshinekk}.3.1.2)| \lesssim & |( q^{(k)}_{xyy}, q^{\langle k  \rangle} \bar{u}^2 w^2 \{1 - \chi\}^2 \p_x^{\langle k \rangle} \bar{u} \p_x^{\langle k \rangle} \bar{v}_{yyy})| \\
\lesssim & \|q^{(k)}_{xyy} \{1 - \chi\} w\|_{x = x_0} \|q^{\langle k  \rangle}\|_\infty \|\p_x^{\langle k -1  \rangle} \bar{v}_{y} \|_\infty \|\p_x^{\langle k  \rangle} \bar{v}_{yyy} \{1 - \chi\} w\|_{x = x_0}, \\
|(\ref{riseshinekk}.3.1.3)| \lesssim &|( q^{(k)}_{xyy}, q^{\langle k  \rangle} \p_x^{\langle k \rangle} \bar{v}_{y} \p_x^{\langle k \rangle} \bar{u}_{yy} \{1 - \chi\}^2 w^2 \bar{u}^2)| \\
\lesssim & \|q^{(k)}_{xyy} \{1 - \chi\} w\|_{x = x_0} \|q^{\langle k  \rangle}\|_\infty \|\p_x^{\langle k - 1 \rangle} \bar{v}_{yyy} \langle y \rangle^2\|_\infty |\p_x^{\langle k \rangle} \bar{v}_{y}  w \{1 - \chi\}\|_{x = x_0}, \\
(\ref{riseshinekk}.3.1.4) \lesssim &|( q^{(k)}_{xyy}, q^{\langle k  \rangle} \bar{u}^2 w^2 \{1 - \chi\}^2 \p_x^{\langle k - 1 \rangle} \bar{v}_{y} \p_x^{\langle k \rangle} \bar{v} \p_x^{\langle k \rangle} \bar{v}_{yy})| \\
\lesssim & \|q^{(k)}_{xyy} \{1 - \chi\} w\|_{x = x_0} \|q^{\langle k  \rangle}\|_\infty \|\p_x^{\langle k \rangle}\bar{v}\|_\infty \|\p_x^{\langle k - 1 \rangle} \bar{v}_{y}\|_\infty \|\p_x^{\langle k \rangle} \bar{v}_{yy} ww \{1 - \chi\}\|_{x = x_0}, \\
|(\ref{riseshinekk}.3.1.5)| \lesssim & |( q^{(k)}_{xyy}, q^{\langle k \rangle} \p_x^{\langle k \rangle} \bar{v} \p_x^{\langle k \rangle}\bar{v} \p_x^{\langle k - 1 \rangle} \bar{v}_{yyy} \bar{u}^2 \{1 - \chi\}^2 w^2)| \\
\lesssim & \|q^{(k)}_{xyy} \{1 - \chi\} w\|_{x = x_0} \|q^{\langle k \rangle}\|_\infty \|\p_x^{\langle k \rangle} \bar{v}\|_\infty^2 \|\p_x^{\langle k - 1 \rangle}\bar{v}_{yyy} \{1 - \chi\} w \|_{x = x_0} 
\end{align*}

We now move to: 
\begin{align*}
|(\ref{riseshinekk}.3.2)| \lesssim & | (q^{(k)}_{xyy}, q^{\langle k \rangle} w^2 \bar{u} \bar{u}_{y} \{1 - \chi\}^2 [ \bar{u}^{\langle k \rangle} \bar{v}^{\langle k \rangle}_{yy} + v^{\langle k \rangle}_s \bar{u}^{\langle k \rangle}_{yy}] )| \\
\lesssim & \|q^{(k)}_{xyy} w \{1 - \chi\}\|_{x= x_0} \|q^{\langle k  \rangle}\|_\infty \Big[ \|\bar{v}^{\langle k - 1 \rangle}_{y}\|_\infty \|\bar{v}^{\langle k \rangle}_{yy} w \{1 - \chi\}\|_{x = x_0}  \\
& + \|\bar{v}^{\langle k \rangle}\|_\infty \|\bar{u}^{\langle k \rangle}_{yy} w \{1 - \chi\}\| \Big] \|_{x = x_0} \| \bar{u}_{y} \langle y \rangle \|_\infty \\
|(\ref{riseshinekk}.3.3)| \lesssim & |( q^{(k)}_{xyy}, q^{\langle k  \rangle}_y w^2 \bar{u}^2 \{1 - \chi\}^2 [ \bar{u}_{}^{\langle k \rangle} \bar{v}^{\langle k \rangle}_{yy} + \bar{v}^{\langle k \rangle} \bar{u}^{\langle k \rangle}_{yy} ] )| \\
\lesssim & \|q^{(k)}_{xyy} w \{1 - \chi\}\|_{x =x_0} \|q^{\langle k \rangle}_y w  \{1 - \chi\}\|_{x = x_0} \Big[ \| \bar{u}^{\langle k \rangle} \bar{v}^{\langle k \rangle}_{yy} \ \|_\infty + \| \bar{v}^{\langle k \rangle} \bar{u}^{\langle k \rangle}_{yy}  \|_\infty \Big] \\
|(\ref{riseshinekk}.3.4)| \lesssim & |(q^{(k)}_{xyy}, q^{\langle k \rangle} w^2 \{1 - \chi\}^2 \bar{u}^2 [\bar{u}^{\langle k \rangle} \bar{v}^{\langle k \rangle}_{yy} + \bar{v}^{\langle k \rangle} \bar{u}^{\langle k \rangle}_{yy} ])| \\
\lesssim & \|q^{(k)}_{xyy} w \{1 - \chi\}\|_{x = x_0} \|q^{\langle k  \rangle}\|_\infty \Big[ \|\bar{v}^{\langle k \rangle}_{yy} w \{1 - \chi\}\|_{x = x_0} \|\bar{v}^{\langle k - 1 \rangle}_{y} \|_\infty + \|\bar{v}^{\langle k \rangle}\|_\infty \times \\
&  \|\bar{v}^{\langle k - 1 \rangle}_{yyy} w \{1 - \chi\}\|_{x = x_0}  \Big] \\
|(\ref{riseshinekk}.3.5)| \lesssim & |( q^{(k)}_{xyy}, q^{\langle k \rangle} \{1 - \chi\} \{1 - \chi\}' [ \bar{u}^{\langle k \rangle} \bar{v}^{\langle k \rangle}_{yy} + \bar{v}^{\langle k \rangle} \bar{u}^{\langle k \rangle}_{yy} ])| \\
\lesssim & \|q^{(k)}_{xyy}\|_{loc} \|q^{\langle k  \rangle}\|_{\infty, loc} \Big[ \|\bar{v}_{y}^{\langle k - 1\rangle}\|_{\infty, loc} \|\bar{v}^{\langle k \rangle}_{yy}\|_{loc} \\
& + \|\bar{v}^{\langle k \rangle}\|_{loc} \|\bar{v}^{\langle k - 1 \rangle}_{yyy}\|_{\infty, loc} \Big].
\end{align*}
 
We now move to: 
\begin{align*}
|(\ref{riseshinekk}.4)| \lesssim & |( q^{(k)}_{xyy} \bar{u}^j_{yyy}, q^{(k-j)}_y \bar{u}^2 w^2 \{1 - \chi\}^2)| + |( q^{(k)}_{xyy} u^j_{syy}, q^{(k-j)}_{yy} \bar{u}^2 w^2 \{1 - \chi\}^2)| \\
& + |( q^{(k)}_{xyy}, \bar{u}^j_{yy} q^{(k-j)}_y \bar{u} \bar{u}_{y} w^2 \{1 - \chi\}^2)| + |(q^{(k)}_{xyy}, \bar{u}^j_{yy} q^{(k-j)}_y \bar{u}^2 ww_{y} \{1 - \chi\}^2)| \\
&  + |( q^{(k)}_{xyy}, \bar{u}^j_{yy} q^{(k-j)}_y \bar{u}^2 w^2 \{1 - \chi\} \{1 - \chi\}' )|
\end{align*}

\noindent  We proceed to estimate each of these terms, with the use of the identities (\ref{pink}):
\begin{align*}
|(\ref{riseshinekk}.4.1)| \lesssim & |( q^{(k)}_{xyy}, q^{\langle k \rangle}_y \bar{u}^2 w^2 \{1 - \chi\}^2 [ \bar{u}^{\langle k \rangle} \bar{v}^{\langle k \rangle}_{yy} + \bar{v}^{\langle k \rangle} \bar{u}^{\langle k \rangle}_{yy} ] )| \\
\lesssim & \|q^{(k)}_{xyy} w \{1 - \chi\}\|_{x = x_0} |q^{\langle k  \rangle}_y w\{1 - \chi\}\|_{x = x_0} \Big[ \| \bar{v}_{y}^{\langle k - 1 \rangle}\|_\infty \| \bar{v}^{\langle k - 1 \rangle}_{yy}  \|_\infty \\
& + \| \bar{v}^{\langle k \rangle} \|_\infty  \| \bar{v}_{yy}^{\langle k- 1\rangle} \|_\infty \Big] \\
|(\ref{riseshinekk}.4.2)| \lesssim & \|q^{(k)}_{xyy} w \{1 - \chi\}\|_{x = x_0} |q^{\langle k  \rangle} \{1 - \chi\} w \langle y \rangle^{-1}\|_{x = x_0} \| \bar{v}_{yyy}^{\langle k - 1 \rangle} \langle y \rangle \|_\infty \\
|(\ref{riseshinekk}.4.3)| \lesssim & \|q^{(k)}_{xyy} w \{1 - \chi\}\|_{x = x_0} |q^{\langle k \rangle} \{1 - \chi\} w \langle y \rangle^{-2}\|_{x = x_0} \| \bar{v}_{yyy}^{\langle k - 1 \rangle} \|_\infty \| \bar{u}_{y} \langle y \rangle^2 \|_\infty \\
|(\ref{riseshinekk}.4.4)| \lesssim & \|q^{(k)}_{xyy} w \{1 - \chi\}\|_{x = x_0} \| \bar{v}_{yyy}^{\langle k - 1 \rangle} \|_\infty |q_y^{\langle k  \rangle} w \{1 - \chi\}\|_{x = x_0} \\
|(\ref{riseshinekk}.4.5)| \lesssim & \|q^{(k)}_{xyy}\|_{x = x_0, loc} \|q_y^{\langle k  \rangle}\|_{x = x_0,loc} \| \bar{v}_{yyy}^{\langle k - 1 \rangle} \|_{\infty, loc}. 
\end{align*}

We now move to: 
\begin{align*}
(\ref{riseshinekk}.5) \lesssim & ( q^{(k)}_{xyy}, \bar{u}^j_{yy} q_{yy}^{k-j} \bar{u}^2 w^2 \{1 - \chi\}^2) + ( q^{(k)}_{xyy}, \bar{u}_{y}^j q_{yyy}^{k-j} \bar{u}^2 w^2 \{1 - \chi\}^2) \\
& + |( q^{(k)}_{xyy}, \bar{u}_{y}^j q_{yy}^{k-j} \bar{u} \bar{u}_{y} w^2 \{1 - \chi\}^2) + (q^{(k)}_{xyy}, \bar{u}_{y}^j q_{yy}^{k-j} \bar{u}^2 ww_{y} \{1 - \chi\}^2) \\
& + ( q^{(k)}_{xyy}, \bar{u}_{y}^j q_{yy}^{k-j} \bar{u}^2 w^2 \{1 - \chi\} \chi') \\
\lesssim & \|q^{(k)}_{xyy} w \{1 - \chi\}\| |q_{yy}^{\langle k \rangle} w \{1 - \chi\} \|_{x = x_0} \|\bar{v}^{\langle k - 1 \rangle}_{yy}  \|_\infty \\
& + \|q^{(k)}_{xyy} w \{1 - \chi\}\| |q_{yyy}^{\langle k \rangle} w \{1 - \chi\}\|_{x = x_0} \| \bar{v}^{\langle k- 1 \rangle}_{yy} \|_\infty \\
& + \|q^{(k)}_{xyy} w \{1 - \chi\}\|_{x = x_0} \|q_{yy}^{\langle k  \rangle} w \{1 - \chi\}\|_{x = x_0} \| \bar{v}_{yy}^{\langle k - 1 \rangle} \|_\infty \| \bar{u}_{y} \|_\infty \\
& + \|q^{(k)}_{xyy} w \{1 - \chi\}\|_{x = x_0} \|q_{yy}^{\langle k  \rangle} w  \{1 - \chi\}\|_{x = x_0} \| \bar{v}_{yy}^{\langle k - 1 \rangle}  \|_\infty  \\
& + \|q^{(k)}_{xyy}\|_{x = x_0,loc} \|q^{\langle k  \rangle} \|_{x = x_0,loc} \| \bar{v}_{yy}^{\langle k - 1 \rangle} \|_{\infty, loc}.
\end{align*}

Next, 
\begin{align*}
|(\ref{riseshinekk}.6)| \lesssim & |( v^{(k)}_{yyy}, q^{(k)}_{xyy} \{ \bar{u} \bar{u}_{y} w^2 \{1 - \chi\}^2 + \bar{u}^2 ww_{y} \{1 - \chi\}^2 + \bar{u}^2 w^2 \{1 - \chi\} \chi' \})| \\
\lesssim & \|q^{(k)}_{xyy} w \{1 - \chi\}\|_{x = x_0} \|v^{(k)}_{yyy} w \{1 - \chi\} \|_{x = x_0} \| \bar{u}_{y}\|_\infty \\
& + \|q^{(k)}_{xyy} w \{1 - \chi\}\|_{x = x_0} \|v^{(k)}_{yyy} w \{1 - \chi\}\|_{x = x_0} \\
& + \|q^{(k)}_{xyy} w|_{x = x_0,loc} \|v^{(k)}_{yyy} \{1 - \chi\}|_{x = x_0,loc}.
\end{align*}

To conclude, we have 
\begin{align*}
(\ref{riseshinekk}.7) = & - \sum_{j =1}^k (q^{(k)}_{xyy}, \p_y \{ \bar{u}^2 \p_x^j \bar{u} \p_x^{k-j} q_{yyy} w^2 \{1 - \chi \}^2  \})_{x = x_0} \\
\lesssim & \| q^{(k)}_{xyy}  w \{1 - \chi \}\|_{x = x_0} \| \bar{v}_{yy}^{\langle k - 1 \rangle} \|_\infty \| q^{\langle k - 1 \rangle}_{yyyy} \{1 - \chi \} w \|_{x = x_0},
\end{align*}

\noindent all of which are acceptable contributions due to the cut-off $\{1 - \chi \}$. 

  This now concludes our treatment (\ref{riseshinekk}) and consequently (\ref{riseshinek}). We now move to the remaining terms from (\ref{produce}), starting with the Rayleigh commutator term, (\ref{produce}.3):
\begin{align*}
&\|  \p_{xy} \{ \p_x^{j_1}\bar{u} \p_x^{j_2} \bar{u} \p_x^{k-j} q_y \} w \{1 - \chi\}(x) \|_{x = x_0}^2 \\
&\lesssim \| \Big[ \p_x^{\langle k + 1 \rangle}\bar{u}_{y} \p_x^{\langle k \rangle}\bar{u} \p_x^{\langle k -1  \rangle} q_y + \p_x^{\langle k +1 \rangle} \bar{u} \p_x^{\langle k \rangle}\bar{u}_{y} \p_x^{\langle k - 1 \rangle} q_y \\
& \hspace{7 mm} + \p_x^{\langle k \rangle} \bar{u} \p_x^{\langle k \rangle} \bar{u} \p_x^{\langle k -1 \rangle} q_{yy}  + \p_x^{\langle k \rangle} \bar{u} \p_x^{\langle k \rangle} \bar{u}_{y} \p_x^{\langle k \rangle} q_y \\
& \hspace{7 mm}  + |\p_x^{\langle k \rangle} \bar{u}|^2 \p_x^{\langle k \rangle} q_{yy} \Big]  w \{1 - \chi\}(x) \|_{x = x_0}^2 \\
&\lesssim  \|\p_x^{\langle k - 1 \rangle} \bar{v}_{y} \|_\infty \|\p_x^{\langle k \rangle} \bar{v}_{yy} w \{1 - \chi\} \|_\infty \|\p_x^{\langle k - 1 \rangle} q_y w \{1 - \chi\} \|_{x = x_0}^2 \\
& \hspace{7 mm} + \| \p_x^{\langle k - 1 \rangle} \bar{v}_{yy}  \|_\infty^2 \| \p_x^{\langle k \rangle} \bar{v}_{y} \|_\infty^2 \|\p_x^{\langle k - 1 \rangle} q_y w \{1 - \chi\} \|_{x = x_0}^2 \\
& \hspace{7 mm} + \| \p_x^{\langle k - 1 \rangle} \bar{v}_{y} \|_\infty^2 \|\p_x^{\langle k - 1 \rangle} q_{yy} w \{1 - \chi\}\|_{x = x_0}^2  \\
& \hspace{7 mm} + \| \p_x^{\langle k-1 \rangle} \bar{v}_{y} \|_\infty \| \p_x^{\langle k - 1 \rangle} \bar{v}_{yy} \|_\infty^2 \|\p_x^{\langle k \rangle} q_y w \{1 - \chi\} \|_{x = x_0}^2 \\
& \hspace{7 mm} + \| \p_x^{\langle k - 1 \rangle} \bar{v}_{y} \|_\infty^2 |\p_x^{\langle k \rangle} q_{yy} w \{1 - \chi\} \|_{x = x_0}^2 
\end{align*}

We now move to the $\Lambda$ terms: 
\begin{align*}
\p_x^k \Lambda = \sum_{j = 0}^k  \bar{v}^j_{xyy} I_x[v^{(k-j)}_y] + \bar{v}^j_{yy} v^{(k-j)}_{y} - \bar{v}^j_{x} I_x[v^{(k-j)}_{yyy}] - \bar{v}^j v^{(k-j)}_{yyy}
\end{align*}

We estimate directly:
\begin{align*}
&\| \bar{v}^j_{xyy} I_x[v^{(k-j)}_y]  w \{1 - \chi\}(x) \|_{x = x_0}^2 \lesssim \|\bar{v}^{\langle k \rangle}_{xyy} w  \{1 - \chi\}(x)\|_{x = x_0}^2 \| v_y^{\langle k  \rangle}   \|_\infty^2 \\
&\| \bar{v}^j_{yy} v^{(k-j)}_y  w \{1 - \chi\}(x) \|_{x = x_0}^2 \lesssim \|\bar{v}^{\langle k \rangle}_{xyy} w \{1 - \chi\}(x) \|_{x = x_0}^2 \| v_y^{\langle k  \rangle}   \|_\infty^2 \\
&\|  |\bar{v}^j_{x} I_x[v^{(k-j)}_{yyy} \{1 - \chi\} w (x) \|_{x = x_0}^2 \lesssim \| \bar{v}^{\langle k +1 \rangle} \|_\infty \| v^{\langle k \rangle}_{yyy} w \{1 - \chi\} (x) \|_{x = x_0}^2 \\
& \| \bar{v}^j |v^{(k-j)}_{yyy} \{1 - \chi\} w(x) \|_{x = x_0}^2  \lesssim \| \bar{v}^{\langle k +1 \rangle} \|_\infty \| v^{\langle k \rangle}_{yyy} w \{1 - \chi\} (x) \|_{x = x_0}^2.
\end{align*}

Upon integrating in $x$, the above terms are majorized by $o_L(1) p(\| \bar{q} \|_{X_{\langle k+ 1 \rangle}}) (1 + \| q \|_{X_k} )$. We now move to the $U(u^0)$ terms: 
\begin{align*}
\int |\p_x^k U(u^0)|^2 w^2 \{1 - \chi\}^2 \le & \int \Big[ |\bar{v}^k_{xyy}|^2 |u^0|^2 + |\bar{v}^k_{x}|^2 |u^0_{yy}|^2 \Big] w^2 \{1 - \chi\}^2 \\
\le & \|u^0 \|_\infty^2 \|\bar{v}^k_{xyy} w \{1 - \chi\}\|_{x = x_0}^2 + \|u^0_{yy} w \{1 - \chi\}\|^2 \|\bar{v}^k_{x}\|_\infty^2. 
\end{align*}

Integrating, the above is majorized by $C(u^0) o_L(1)  p(\| \bar{q} \|_{X_{\langle k +1 \rangle}})$. Similarly, the $g$ contributions are clearly estimated via $|\p_{xy} g_1^{k} \{1 - \chi\} w \|^2$.
\end{proof}

\begin{proposition}   For $k \ge 0$, and let $q$ solve (\ref{origPrLay.beta}). Then:
\begin{align} \label{iprob}
\| q \|_{X_{ k }} \lesssim C(q_0) + \| \p_x^k \p_{xy}g_1 w \|^2 + o_L(1) \| \p_x^k \p_x \p_{xy}g_1 \langle y \rangle \|^2 + o_L(1) C(u^0). 
\end{align}
\end{proposition}
\begin{proof} We add together (\ref{Ek}), a small multiple of (\ref{solid.py4.cof}) and (\ref{weight.H4.cof}). On the left-hand side, this produces $\sup [ |q^{(k)}_{yyy}w \{1-\chi\}|^2 + |\bar{u} q^{(k)}_{xy}|^2] + \| v^{(k)}_{yyyy} w \{1-\chi\} \|^2 + \| v^{(k)}_{yyyy} \|_{loc} + \| q^{(k)}_{xyy} w \{1-\chi\} \|^2 + \| q^{(k)}_{xyy} \sqrt{\bar{u}} \|^2$, which can clearly be combined to majorize $\| q^{(k)} \|_X$. On the right-hand side $\| \p_{xy}\p_x^k g_1 w \{1-\chi\} \|^2 + C(q_0) + o_L(1) C(u^0) + o_L(1) p(\| q_s \|_{X_{\langle k+ 1 \rangle}})(C(q_0) + \| q^{(k)} \|_X^2) + o(1) \| q^{(k)} \|_{\mathcal{E}} + o_L(1) \| \p_{xxy}\p_x^k g_1 \langle y \rangle \|^2 + |\bar{u} q_{xy}^k(0,\cdot)|^2$. Of these, the $o(1) \| q^{(k)} \|_{\mathcal{E}}$ term is absorbed to the left-hand side. The $o_L(1) p(\| \bar{q} \|_{X_{\langle k+1 \rangle}})\| q \|_X^2$ term is also absorbed to the left-hand side. Finally, the initial value $|\bar{u} q^{(k)}_{xy}(0,\cdot)|^2$ is obtained through (\ref{whois}).
\end{proof}

We can upgrade to higher $y$ regularity by using the equation. In this direction, we establish the following lemma: 
\begin{lemma} Let $q$ solve (\ref{origPrLay.beta}). Then the following inequality is valid: 
\begin{align}
\| \p_y^5 v \| + \| \p_y^6 v \| \lesssim \| q \|_{X_{\langle 1 \rangle}} + C(u^0). 
\end{align}
\end{lemma}
\begin{proof} We begin with the following identity 
\begin{align} \n
\bar{u} v^{(1)}_{yyyy} = &\bar{u} \p_{yy} \p_{xy} \{ \bar{u} q_y \} \\ \n
= & \p_{yy} \{ \bar{u} \p_{xy} \{ \bar{u} q_y \} \} - \bar{u}_{yy} \p_{xy} \{ \bar{u} q_y \} - 2 \bar{u}_y \p_{xyy} \{ \bar{u} q_y \} \\ \n
= & \p_{yy} \p_{xy} \{ \bar{u}^2 q_y \} - \p_{yy} \{ \bar{u}_{xy} \bar{u} q_y \} - \p_{yy} \{ \bar{u}_x \p_y \{ \bar{u} q_y \} \} \\ \n
& - \p_{yy} \{ \bar{u}_y \p_x \{ \bar{u} q_y \} \} - \bar{u}_{yy} \p_{xy} \{ \bar{u} q_y \} - 2 \bar{u}_y \p_{xyy} \{ \bar{u} q_y \}  \\ \n
= & \p_{yy} \Big\{ v_{yyyy} + \Lambda(v) + U(u^0) - \p_{xy}g_1 \Big\} - \p_{yy} \{ \bar{u}_{xy} \bar{u} q_y \} \\ \n
& - \p_{yy} \{ \bar{u}_x \p_y \{ \bar{u} q_y \} \} - \p_{yy} \{ \bar{u}_y \p_x \{ \bar{u} q_y \} \}  - \bar{u}_{yy} \p_{xy} \{ \bar{u} q_y \} \\ \label{sist}
& - 2 \bar{u}_y \p_{xyy} \{ \bar{u} q_y \}. 
\end{align}

\noindent We rearrange the above to solve for $\p_y^6 v$. We thus estimate each of the other terms in (\ref{sist}). We clearly have 
\begin{align*}
\| \bar{u} v^{(1)}_{yyyy} \| + \| \bar{u} q_{yyy} \| + \| q^{(1)}_{yy} \| + \| q_{yyy}^{(1)} \| + \| \bar{u} q_{xy} \| \lesssim \| q \|_{X_{\langle 1 \rangle}}. 
\end{align*}

\noindent This accounts for all of the $q$ terms from (\ref{sist}),and since $u^0, g_1$ are arbitrarily regular, it remains to estimate $\p_{yy}\Lambda(v)$. An examination of the terms in $\Lambda(v)$ shows that we must estimate the latter two, higher order terms, as the former two will be controlled by $\| q \|_{X}$. 
\begin{align*} 
\|\p_{yy} \Lambda(v) \|  = & \| \p_{yy} \{ \bar{v}_{xyy} I_x[v_y] + \bar{v}_{yy} v_y - \bar{v}_x I_x[v_{yyy}] - \bar{v} v_{yyy} \} \| \\
\lesssim & \| q \|_{X} + \| \bar{v}_x I_x[\p_{y}^5 v] \| + \| \bar{v} \p_y^5 v \|  \\
\lesssim & \| q \|_X + o(1) \| \p_y^6 v \|. 
\end{align*}

\noindent Above, we have used the integration by parts inequality 
\begin{align} \n
\| \bar{v} \p_y^5 v \|^2 = & ( \bar{v} \p_y^5 v, \bar{v} \p_y^5 v) = - 2(\bar{v}_y \p_y^5 v, \bar{v} \p_y^4 v) - (\bar{v} \p_y^6 v, \bar{v} \p_y^4 v) \\ \n
\lesssim & \| \bar{v}_y \|_\infty \| \bar{v} \p_y^5 v \|  \p_y^4 v \| + \| \bar{v} \|_\infty^2 \| \p_y^6 v \| \| \p_y^4 v \| \\ \label{fifth}
\lesssim & o(1) \| \bar{v} \p_y^5 v \|^2 + o(1) \| \p_y^6 v \|^2 + \| \p_y^4 v \|^2. 
\end{align}

Summarizing, we have thus obtained 
\begin{align*}
\| \p_y^6 v \| \lesssim \| q \|_{X_{\langle 1 \rangle}} + o(1) \| \p_y^6 v \| + C(u^0),
\end{align*}

\noindent which proves the lemma upon pairing with (\ref{fifth}).
\end{proof}

It is clear that we can upgrade to higher $y$ regularity by iterating the above.

\subsection{Passing to the limit}

We must now pass to the limit as $\theta \downarrow 0$. 
\begin{proposition}   Let the forcing $g_1$ satisfy: 
\begin{align*}
\| \p_x^k \p_{xy} g_1 w \| + \| ( \langle y \rangle \p_x) \p_x^k \p_{xy}g_1 \| < \infty. 
\end{align*}

\noindent for all $k = 0,...,k_0$. Let the following initial data be prescribed:
\begin{align*}
q|_{x = 0} := f_0(y),
\end{align*}

\noindent  satisfying suitable compatibility conditions as detailed in (\ref{whois}) and the integral condition, (\ref{integral.cond}). Then there exists a unique solution to the equation (\ref{origPrLay.beta}) satisfying the estimate: 
\begin{align*}
\| q \|_{X_k} \lesssim 1 \text{ for } k \le k_0 - 3.
\end{align*}

\noindent  Moreover $q$ achieves the initial data $f_0$ at $\{x = 0\}$. The pair $u = u^0 - \int_0^x v_y, v = \bar{u} q$ satisfy the original Prandtl equation (\ref{origPrLay}).
\end{proposition}
\begin{proof} First, for each $\theta > 0$, a standard Galerkin method produces global solutions to (\ref{sysb}), according to \cite{GN}. Second, we will denote the following notation: Let $f_0$ be the prescribed value for $q|_{x = 0}$. Take $f_0^{(\theta)} = f_0$ for all $\theta > 0$. Define now:
\begin{align*}
f_k^{(\theta)} := \p_x^k q^{(\theta)}|_{x = 0} \text{ for } b \ge 0 \text{ for } k = 1,...,k_0.
\end{align*}

\noindent  For $\theta > 0$, and $k = 1,...,k_0$, the functions $f_k^{(\theta)}$ are obtained by evaluating the equation (\ref{sysb.beta}) at $x = 0$, as the equation continues to hold up to the initial hypersurface, $\{x = 0\}$. The compatibility conditions used to produce (\ref{whois}) ensure that $|f_k^{(\theta)}|_{2}$ are uniformly bounded in $\theta$ for each $k$. Therefore, the constant $C(q_0)$ in (\ref{iprob}) is uniform in $b$, and we may obtain the following, uniform in $b$, estimate $\| q^{(\theta)}\|_{X_{\langle k_0 \rangle}} \lesssim 1$. 

It thus remains to pass to the limit in $X_{\langle k_0 \rangle}$ for some large $k_0$. It is more convenient to work with $L^2_{xy}$ norms, so we will define the following norm:
\begin{align*}
&\| q \|_{\tilde{X}} := \| \bar{u} q_{xy} \| + \| q_{yyy} w \{1-\chi\} \| + \| \sqrt{\bar{u}} q_{xyy} w \| + \| v_{yyyy} w \|, \\
&\| q \|_{\tilde{X}_k} := \| \p_x^k q \|_{\tilde{X}}.
\end{align*}

\noindent  Instead of passing to weak limits, we now want to translate the uniform bound above into strong convergence of lower order norms (up to passing to a further subsequence) using compactness. To do so, we will define the lower order norms: 
\begin{align*}
\| q \|_{Y_{l}} := \sum_{j = 1}^3 \| \p_x^l \p_y^{j} v \tilde{w} \|_{L^2} + \| \p_x^l v \langle y \rangle^{-1} \|_{L^2} \text{ for } \tilde{w} = e^{(N- )y}. 
\end{align*}

Due to the disparity of the weight $e^{(N-) y}$ in $Y_l$ norms with $w = e^{Ny}$ in the $X, \tilde{X}$ norms and the presence of $\p_y^4$ in $\tilde{X}$, standard compactness arguments show $\tilde{X}_{k_0} \hookrightarrow \hookrightarrow Y_{k_0 - 2}$. Thus, we have the strong convergence of a further $b$-subsequence: 
\begin{align} \label{strongYconv}
q^{(\theta)} \rightarrow q \text{ strongly in } Y_{k_0 - 2}. 
\end{align}

\noindent  Standard Sobolev embedding shows $Y_{k_0  - 2} \subset L^\infty$ for $k_0$ sufficiently large. This then implies uniform convergence of $\p_x^{\langle k_0 - 3 \rangle} \p_y^{\langle 2 \rangle} q^{(\theta)} \rightarrow \p_x^{\langle k_0 - 3 \rangle} \p_y^{\langle 2 \rangle} q$ in $\Omega$. 

We must check two things. First, the equation (\ref{origPrLay.beta}) is satisfied pointwise by $q$. To see this, testing (\ref{sysb.beta}) we obtain on the left-hand side: 
\begin{align}
\begin{aligned} \label{base2}
&( \p_x\{ \bar{u}^{(\theta)} q^{(\theta)}_y \}, \phi_y) + ( v^{(\theta)}_{yy}, \phi_{yy}) + ( \bar{v}_{xyy}^{(\theta)},I_x[v^{(\theta)}_y]) + ( \bar{v}_{yy}^{(\theta)}, v^{(\theta)}_y) \\
&+ ( I_x[v^{(\theta)}_{yy}], \p_y \{ \bar{v}_{x}^{(\theta)} \phi \}) + ( v^{(\theta)}_{yy}, \p_y \{ \bar{v}_{} \phi \}) - (v^{(\theta)}_{sxyy},u^{0,\theta}) + ( \bar{v}_{x}^{(\theta)}, u^{0,\theta}_{yy}). 
\end{aligned}
\end{align}

\noindent  It is obvious that we can pass to the limit in the first six integrals using the convergence in (\ref{strongYconv}). For the $u^{0,\theta}$ terms, we recall the equation (\ref{wknd}): 
\begin{align*}
L_{\bar{v}^{(\theta)}} u^{0,\theta} = \p_y g_1|_{x = 0} + \bar{u}^{\theta} \bar{V}^0_{yy} - \bar{u}_{yy}^{(\theta)} \bar{V}^0
\end{align*}

\noindent  This gives estimates $\| \p_y^K u^{0,\theta} \langle y \rangle^M \|_\infty \lesssim 1$ uniformly in $\theta$ by (\ref{base1}). Next, by considering differences, we may write: 
\begin{align*}
L_{\bar{v}_{}^{(\theta)}} \Big( u^{0,\theta} - u^{0} \Big) = \theta \bar{V}^0_{yy} - L_{\bar{v}_{}^{(\theta)} - \bar{v}_{}^{(0)} } u^{0}.
\end{align*}

\noindent  The right-hand side, when placed in $W^{k,\infty}(\langle y \rangle^M)$ is $o(\theta)$. Thus, again by (\ref{base1}) we obtain that $|\p_y^K \{ u^{0,\theta} - u^{0}\} \langle y \rangle^M| \lesssim o(\theta)$. From here, it is clear that we can pass to the limit as $\theta \rightarrow 0$ in the final two integrals of (\ref{base2}).

Second by Sobolev embedding: $\|\p_x^k q|_{x = 0} - f_k^{(\theta)}\|_\infty \rightarrow 0$. This implies for $k = 0$ that $q|_{x = 0}$ is the prescribed value of $f_0$. For values of $k \le k_0 - 3$, this implies that the equation (\ref{origPrLay.beta}) holds up to the initial hypersurface $\{x = 0\}$  $\p_x^k q|_{x = 0}$ can be computed by evaluating the equation at $x = 0$.  
\end{proof}

\subsection{Final Layer}

Recall $w_0 = \langle y \rangle\langle Y \rangle^m$ and the definition of $F_R, \p_x F_R$ from (\ref{forcingdefn}). 

\begin{lemma}   \label{propn.forcing} 
\begin{align}
\| F_R|_{x = 0} w_0 \| + \| \p_x F_R \frac{w_0}{\sqrt{\eps}} \|  \lesssim \sqrt{\eps}^{n-1-2N_0}. 
\end{align}
\end{lemma}
\begin{proof} We recall the specification of the forcing terms $\underbar{f}^{(n+1)}$ and $\underbar{g}^{(n+1)}$ given in (\ref{underbar.f}) - (\ref{underbar.g}). We are interested in the quantities $\| \p_y \underbar{f}^{(n+1)} - \eps \p_x \underbar{g}^{(n+1)} (0,\cdot)\|, \| \p_y \underbar{f}^{(n+1)} - \eps \p_x \underbar{g}^{(n+1)} w_0 (0,\cdot)\|$, and $\| \p_{xy} \underbar{f}^{(n+1)} - \eps \p_{xx} \underbar{g}^{(n+1)} \|, \| \p_{xy} \underbar{f}^{(n+1)} - \eps \p_{xx} \underbar{g}^{(n+1)} w_0\|$ and subsequently scale by $\eps^{-N_0}$. It is straightforward to see it suffices to compute the latter two quantities, as the $\{x = 0\}$ boundary terms follow in an identical manner. 

To do this, we will first extract the model behavior of many terms from $\underbar{f}^{(n+1)}$ and $\underbar{g}^{(n+1)}$. Define three functions: $\varphi$ which satisfies rapid decay: $|\nabla^K \varphi| \lesssim e^{-My}$ for arbitrary large $M$ and any $K \ge 0$, $\omega(y)$ which is supported in $y \in [0, \frac{1}{\sqrt{\eps}})$, and whose derivatives can be written as $\varphi$. Third, a function $\phi(Y)$ is a rapidly decaying function of $Y$. We then see: 
\begin{align*}
&(\ref{underbar.f}.\{1,...,7\}) = \sqrt{\eps}^n [\sqrt{\eps} \varphi(y) + \sqrt{\eps} \phi(Y) \omega(y)], \\
&(\ref{underbar.f}.8) = \sqrt{\eps}^{n+1} \omega(y) \phi(Y), \\
&(\ref{underbar.f}.\{9,...,14\}) = \sqrt{\eps}^{n+1} \phi(Y)
\end{align*}

We estimate immediately upon paying $\sqrt{\eps}^{-\frac{1}{2}}$ when changing from $L^2$ in $y$ coordinates to $L^2$ in $Y$ coordinates: 
\begin{align*}
&\| \p_y \{ (\ref{underbar.f}) \} (1+w_0) \| \lesssim \sqrt{\eps}^{n}.
\end{align*}

We now move to $\underbar{g}^{(n+1)}$. The first seven terms are supported in $Y \lesssim 1$ due to the presence of $u^n_p, v^n_p$ and therefore unaffected by the weight $\langle Y \rangle^m$. Upon changing variables to $L^2_Y$ and losing a factor of $\eps^{\frac{1}{4}}$ from the Jacobian, these terms are easily majorized by: 
\begin{align*}
\| \eps \p_x (\ref{underbar.g}).\{1, ..., 7\} \cdot w_0 \| \lesssim \sqrt{\eps}^{n+\frac{1}{2}} 
\end{align*}

The Euler terms in (\ref{underbar.g}) are of the form $\sqrt{\eps}^{n}\phi(Y)$ after applying $\p_x^2$, and therefore are majorized by: 
\begin{align*}
\| \eps \p_x (\ref{underbar.g}.\{8, ..., 13 \}) \| \lesssim \| \eps \sqrt{\eps}^n \phi(Y) \langle y \rangle \langle Y \rangle^m  \| \lesssim \sqrt{\eps}^n.
\end{align*}
\end{proof}

\section{Euler Layers} \label{appendix.Euler}

\subsection{Elliptic Estimates}

Our starting point is the system (\ref{des.eul.1}). Going to vorticity yields the system we will analyze:  
\begin{align} 
\begin{aligned} \label{eqn.vort}
&u^0_e \Delta v^i_e + u^0_{eYY} v^i_e = F^{(i)} := \p_Y f^i_{E,1} - \p_x f^i_{E,2}, \\
&v^i_e|_{Y = 0} = - v^{i-1}_p|_{y = 0}, \hspace{3 mm} v^i_e|_{x = 0,L} = V^i_{E, \{0, L\}}, \hspace{3 mm} u^i_e|_{x = 0} = U^i_{E,0}.
\end{aligned}
\end{align}

\noindent The data for $u^i_e|_{x = 0}$ is required because $u^i_e = u^i_e|_{x = 0} - \int_0^x v^i_{eY}$ will be recovered through the divergence free condition upon constructing $v^i_e$. 

We will quantify the decay rates as $Y \uparrow \infty$ for the quantities $V^i_{E,{0, L}}$ and $F^{(i)}$. 
\begin{definition} \label{def.wm1}   In the case of $i = 1$, define $w_{m_1} = Y^{m_1}$ if $v^1_{e}|_{x = 0} \sim Y^{-m_1}$ or $w_{m_1} = e^{m_1 Y}$ if $v^1_e|_{x = 0} \sim e^{-m_1 Y}$ as $Y \uparrow \infty$. This now fixes whether or not $w_{m}$ will refer to polynomial or exponential growth rates. For other layers, we will assume: 
\begin{align}
\begin{aligned}  \label{decay.rates.euler}
&V^i_{E,\{0, L\}} \sim w_{m_i}^{-1} \text{ for } m_i >> m_1 \\
&F^{(i)} \sim w_{l_i}^{-1} \text{ for some } l_i >> 0. 
\end{aligned}
\end{align}

Finally, let $m_i' := \min \{ m_i, l_i \}$. 
\end{definition}

Define: 
\begin{align}
S(x,Y) = (1 - \frac{x}{L}) \frac{V_{i,0}(Y)}{v^{i-1}_p(0,0)} v^{i-1}_p(x,0) + \frac{x}{L}\frac{V_{i,L}(Y)}{v^{i-1}_p(L,0)}v^{i-1}_p(x,0),
\end{align}

\noindent and consider the new unknown:
\begin{align*}
\bar{v} := v^i_e - S, 
\end{align*}

\noindent which satisfies the Dirichlet problem: 
\begin{align}
-u^0_e \Delta \bar{v} + u^0_{eYY} \bar{v} = F^{(i)} + \Delta S, \hspace{5 mm} \bar{v}|_{\p \Omega} = 0.
\end{align}

From here, we have for any $m < m_i' - n_0$ for some fixed $n_0$, perhaps large, 
\begin{align}
||v \cdot w_m||_{H^1} \lesssim 1.
\end{align}

To go to higher-order estimates, we must invoke that the data are well-prepared in the following sense: taking two $\p_Y^2$ to the system yields:
\begin{align} \label{corn.1}
&\p_Y^2 v^i_e(0,Y) = \p_Y^2 V_{i,0}(Y), \\ \label{corn.2}
&\p_Y^2 v^i_e(L,Y) = \p_Y^2 V_{i,L}(Y), \\ \label{corn.3}
&\p_Y^2 v^i_e(x,0) = \frac{1}{u^0_e(0)} \Big\{ v^{i-1}_{pxx}(x,0) + u^0_{eYY}(0) v^{i-1}_p(x,0) + F^{(i)}(x,0) \Big\}.
\end{align}

Our assumption on the data, which are compatibility conditions, ensure: 
\begin{align}
&\p_Y^2 V_{i,0}(0) = \frac{1}{u^0_e(0)} \Big\{ v^{i-1}_{pxx}(0,0) + u^0_{eYY}(0) v^{i-1}_p(0,0) + F^{(i)}(0,0) \Big\}, \\
&\p_Y^2 V_{i,0}(L) = \frac{1}{u^0_e(0)} \Big\{ v^{i-1}_{pxx}(L,0) + u^0_{eYY}(0) v^{i-1}_p(L,0) + F^{(i)}(L,0) \Big\}.
\end{align}

It is natural at this point to introduce the following definition:
\begin{definition}[Well-Prepared Boundary Data]    \label{def.well.prp} Consider the corner $(0,0)$. There exists a value of $\Big(\p_Y^2 v^i_e|_{Y = 0}\Big)|_{x = 0}$ which is obtained by evaluating (\ref{corn.3}) at $x = 0$. There exists a value of $\Big(\p_Y^2 v^i_e|_{x = 0} \Big)|_{Y = 0}$ which is obtained by evaluating (\ref{corn.1}) at $Y = 0$. These two values should coincide. The analogous statement should also hold for the corner $(L,0)$. In this case, we say that the boundary data are ``well-prepared to order 2". The data are ``well-prepared to order $2k$" if we can repeat the procedure for $\p_Y^{2k}$.
\end{definition}

We thus have the following system:
\begin{align} \n
-u^0_e \Delta v^1_{eYY} + u^0_{eYY} v^1_{eYY} &+ \p_Y^4 u^0_e v^1_e + 2 \p_Y^3 u^0_e v^1_{eY} \\ \label{eul.sys.1} &- 2 u^0_{eY} \Delta v^1_{eY} - u^0_{eYY} \Delta v^i_e = \p_{YY} F^{(i)}.
\end{align}

We can define another homogenization in the same way:
\begin{align}
S_{(2)}(x,Y) = (1 - \frac{x}{L}) \frac{V''_{i,0}(Y)}{\p_Y^2 v^i_{e}(x,0)} \p_Y^2 v^{i-1}_p(x,0) + \frac{x}{L}\frac{V''_{i,L}(Y)}{\p_Y^2 v^{i-1}_p(L,0)}\p_Y^2 v^{i-1}_p(x,0),
\end{align}

\noindent which is smooth and rapidly decaying by the assumption that the data are well-prepared. Let us consider the system for $\bar{v} := v^1_{eYY} - B_{(2)}$. The first step is to rewrite:
\begin{align} \n
v^i_{exx} &= -v^i_{eYY} + \frac{u^0_{eYY}}{u^0_e} v^i_e + F^i, \\
& = - \bar{v} + S_{(2)} + F^i.
\end{align}

We can now rewrite the system (\ref{eul.sys.1}) in terms of $\bar{v}$:
\begin{align} \n
-u^0_e \Delta \bar{v} &+ u^0_{eYY} \bar{v} + \p_Y^4 u^0_e v^1_e + 2\p_Y^3 u^0_e v^1_{eY} \\ \n
& -2u^0_{eY} [\p_Y\{ \bar{v} + S_2 \} + \p_Y \{ \bar{v} + S_2 + F^i \}] \\ 
& - u^0_{eYY} \Big[ F^i \Big] = u^0_e \Delta S_2 + u^0_{eYY} S_2 + \p_{YY}F^i.
\end{align}

Obtaining estimates for $\bar{v}$ yields for any $m < m_i' - n_0$: 
\begin{align}
||\bar{v} \cdot w_m||_{H^1} \lesssim 1.
\end{align}

Translating to the original unknown gives:
\begin{align} \label{H3eul.est.1}
||v^i_{eYY}, v^1_{eYYx}, v^i_{eYYY} \cdot w_M||_{L^2} \lesssim 1.
\end{align}

Using the equation and Hardy in $Y$, we can obtain: 
\begin{align} \label{H3eul.est}
||v^i_{exx}, v^i_{exxx}, v^1_{exY}, v^i_{exxY} \cdot w_M||_{L^2} \lesssim 1.
\end{align}

Thus, we have the full $H^3$ estimate. $u^1_e$ can be recovered through the divergence free condition:
\begin{align} \label{u1.rec}
u^i_e(x,Y) := u^i_e(0,Y) - \int_0^x \p_Y v^i_e(x', Y) \ud x'.
\end{align}

The compatibility conditions can be assumed to arbitrary order by iterating this process, and thus we can obtain:
\begin{proposition}   \label{L.e.constr} There exists a unique solution $v^i_e$ satisfying (\ref{eqn.vort}). With $u^i_e$ defined through (\ref{u1.rec}), the tuple $[u^i_e, v^i_e]$ satisfy the system (\ref{des.eul.1}). For any $k \ge 0$ and $M \le m_i' - n_0$ for some fixed value $n_0 > 0$: 
\begin{align}
||\{u^i_e, v^i_e \} w_M ||_{H^k} \le C_{k,M}.
\end{align}
\end{proposition}
\begin{proof}
The existence follows from Lax-Milgram, whereas the estimates follow from continuing the procedure resulting in (\ref{H3eul.est.1}) - (\ref{H3eul.est}).
\end{proof}

\begin{corollary} \label{Cor.cheap.quotient}   Assume $m_i >> m_1$ for $i = 2,...,n$. Then: 
\begin{align} \label{cheap.quotient}
\| \{ u^i_e, v^i_e \} w_{\frac{m_1}{2}} \|_{H^k} \lesssim 1. 
\end{align}
\end{corollary}
\begin{proof} This follows from two points. First, for the $i = 1$ case, the forcing is absent and therefore the parameter $l_1$ can be taken arbitrarily large. In particular this implies that $m_1' = m_1$. Second, a subsequent application of the above proposition shows that the $i$-th layer quantities decay like $m_1 - n_0$. An examination of the forcing terms $f^i_{E,1}, f^i_{E,2}$ shows that these quantities decay as $w_{m_1 - n_0}^{-1}$. Thus, for $i \ge 2$, we can take the parameter $l_i = m_1 - n_0 = m_i'$. Therefore, if $m_1$ is sufficiently large, $\frac{m_1}{2} << m_1 - 10 n_0$.
\end{proof}

Recall the definition of $m_i$ from Definition \ref{def.wm1}. The main estimate here is: 
\begin{lemma} \label{Quotient.Lemma}   Let $v^i_e$ be a solution to (\ref{eqn.vort}). For any $m' < m_i$, $\|v^i_{ex}(0,\cdot) w_{m'} \|_\infty \lesssim 1$.
\end{lemma}

\begin{proof} We first homogenize $v^i_e$ by introducing $\bar{v}_e := v^i_e - b$. Recall the definition of $\chi$ in (\ref{basic.cutoff}). We will localize using $1 -\chi(\frac{Y}{N})$ for some large, fixed $N > 1$. A direct computation produces the following: 
\begin{align*}
\Delta ( \{ 1 - \chi(\frac{Y}{N}) \} \bar{v}_e) =& \{ 1 - \chi(\frac{Y}{N}) \} \frac{u^0_{eYY}}{u^0_e} v^1_e - \{ 1 - \chi(\frac{Y}{N}) \} \Delta b \\
& + 2 \p_Y \{ 1 - \chi(\frac{Y}{N}) \} \bar{v}_{eY} + \p_{YY} \{ 1 - \chi(\frac{Y}{N}) \} \bar{v}_e := R.
\end{align*}

Let $w = w_{m'}^{-1}$. Now we define the quotient $q^\delta = \frac{\{ 1 - \chi(\frac{Y}{N}) \} \bar{v}_e}{w(Y) + \delta}$, which satisfies: 
\begin{align} \label{difrn}
&\underbrace{\Delta q^\delta + 2 \frac{w_Y}{w + \delta} q^\delta_Y}_{:= \mathcal{T}_\delta} + \frac{w_{YY}}{w+\delta} q^\delta = \frac{R}{w+\delta}.
\end{align}

\noindent  \textit{Case 1: $w_{m_i}$ are polynomials in $Y$}

The following inequalities hold, independent of small $\delta$: 
\begin{align} \label{meeshka.1}
&|\frac{R}{w+\delta}| \lesssim 1, \\ \label{meeshka.2}
&|\frac{w_{YY}}{w+\delta}\{ 1 - \chi(\frac{Y}{N}) \}| \le o(1). 
\end{align}

\noindent The inequality, (\ref{meeshka.2}), holds because $|w_{YY}| \lesssim Y^{-2} |w|$ for polynomial decay, so by taking $N$ large, we can majorize the desired quantity by $o(1)$.  To apply the maximum principle to $q^\delta$, we introduce the following barrier, for $m$ large and fixed and for $f = f(x)$ satisfying $f''(x) < - 1$: 
\begin{align*}
q_-^\delta := q^\delta - f(mx) \Big[ \sup |\frac{R}{w+\delta}| + \sup|\frac{w_{YY}}{w} q^\delta| \Big], \\
q_+^\delta := q^\delta - f(mx) \Big[ \sup |\frac{R}{w+\delta}| + \sup|\frac{w_{YY}}{w} q^\delta| \Big]
\end{align*}

Immediate computations gives $\mathcal{T}_\delta[q_-^\delta] \ge 0$ and $\mathcal{T}_\delta[g_+^\delta] \le 0$. Applying the maximum principle to both $q_-^\delta, q_+^\delta$ gives: 
\begin{align*}
\| q^\delta \|_\infty \lesssim \sup|\frac{R}{w+\delta}| + \sup |\frac{w_{YY}}{w}q^\delta|. 
\end{align*}

Applying (\ref{meeshka.1}) and (\ref{meeshka.2}) gives: 
\begin{align*}
\| q^\delta \|_\infty \lesssim 1, 
\end{align*}

\noindent uniformly in $\delta > 0$. Due to the cutoff $\{1 - \chi(\frac{Y}{N})\}$, all quantities are supported away from $Y = 0$, we may differentiate the equation, (\ref{difrn}), in $Y$ to obtain the new system: 
\begin{align*}
\Delta q^\delta_Y + 2 \frac{w_Y}{w+\delta} q^\delta_{YY} + [\frac{w_{YY}}{w+\delta} + 2 \p_Y\{ \frac{w_Y}{w+\delta} \}] q^\delta_Y = \p_Y\{ \frac{R}{w+\delta} \} - \p_Y\{ \frac{w_{YY}}{w+\delta} \} q^\delta.
\end{align*}

Clearly, we may repeat the above argument for the unknown $q^\delta_Y$. Bootstrapping further to $q^{\delta}_{YY}$ and using the equation, we establish: 
\begin{align*}
\|q^\delta_{xx} \|_\infty \lesssim 1. 
\end{align*}

For each fixed $Y$, $q^\delta_x(x_\ast, y) = 0$ for some $x_\ast = x_\ast(Y) \in [0,L]$ since $q^\delta(0,Y) = q^\delta(L,Y) = 0$. Thus, using the Fundamental Theorem of Calculus $\| q_x^\delta \|_\infty \lesssim 1$. Finally, we use the pointwise in $Y$ equality: 
\begin{align*}
|q^\delta_x - q_x| = \delta | \frac{\chi \bar{v}_{ex}}{w(w+\delta)} | \rightarrow 0 \text{ as } \delta \downarrow 0, \text{ pw in } Y.
\end{align*} 

Thus, for each fixed $Y$, there exists a $\delta_\ast = \delta_\ast(Y) > 0$ such that for $0 < \delta < \delta_\ast(Y)$, $|q^\delta_x - q_x| \le 1/2$. Thus, $|q_x(Y)| \lesssim 1$. This is true for all $Y$. Thus, $\| q_x \|_\infty \lesssim 1$.

\noindent \textit{Case 2: $w_{m_i}$ are exponential in $Y$}

In this case, we start with (\ref{difrn}), and perform $H^k$ energy estimates. We replace (\ref{meeshka.1}) and (\ref{meeshka.2}) with: 
\begin{align}
\| \frac{R}{w+\delta} \langle Y \rangle^M \| < \infty \text{ for large } M. 
\end{align}

From here, straightforward energy estimates show $\| q^\delta \|_{H^k} \lesssim 1$ for any $k$. This is achieved by repeatedly differentiating in $Y$ and using that the cutoff $\{1 - \chi(\frac{Y}{N}) \}$ localizes away from the boundary $\{Y = 0\}$. We thus conclude $\| q^\delta_{xx} \|_{\infty} \lesssim \| q^\delta \|_{H^4} \lesssim 1$ using Sobolev embedding. The proof then concludes as in the polynomial case. 
\end{proof}

The following proposition summarizes the profile constructions from the Appendix: 
\begin{theorem} \label{thm.construct} Assume the shear flow $u^0_e(Y) \in C^\infty$, whose derivatives decay rapidly. Assume (\ref{OL.1}) regarding $\bar{u}^0_p|_{x = 0}$, and the conditions
\begin{align}  \label{compatibility.1.fin}
& \bar{v}^i_{pyyy}|_{x = 0}(0) = \p_x g_1|_{x = 0, y = 0}, \\ \label{compatibility.2.fin}
&\bar{v}^i_p|_{x = 0}''''(0) = \p_{xy}g_1|_{y =0}(x = 0), \\ \label{integral.cond}
&\bar{u}^0_{p y}|_{x = 0}(0) u^{i}_e|_{x = 0}(0) - \int_0^\infty \bar{u}^0_{p} e^{-\int_1^y \bar{v}^0_{p}} \{f^{(i)}(y) - r^{(i)}(y) \} \ud y  = 0,
\end{align}

\noindent where $r^{(i)}(y) := \bar{v}^i_p \bar{u}^0_{p y} - \bar{u}^0_{p} \bar{v}^i_{py}$. We assume also standard higher order versions of the parabolic compatibility conditions (\ref{compatibility.1.fin}), (\ref{compatibility.2.fin}). Let $v^i_e|_{x = 0}, v^i_e|_{x = L}, u^i_e|_{x = 0}$  be prescribed smooth and rapidly decaying Euler data. We assume on the data standard elliptic  compatibility conditions at the corners $(0,0)$ and $(L,0)$ obtained by evaluating the equation at the corners. In addition, assume 
\begin{align}
&v^1_e|_{x = 0} \sim Y^{-m_1} \text{ or } e^{-m_1 Y} \text{ for some } 0 < m_1 < \infty,\\
& \|\p_Y^k \{ v^i_e|_{x = 0} - v^i_e|_{x = L} \} \langle Y \rangle^M\|_\infty \lesssim L
\end{align}

\noindent Then all profiles in $[u_s, v_s]$ exist and are smooth on $\Omega$. The following estimates hold: 
\begin{align}
\begin{aligned} \label{prof.pick}
&\bar{u}^0_p > 0,\bar{u}^0_{py}|_{y = 0} > 0, \bar{u}^0_{p yy}|_{y = 0} = \bar{u}^0_{p yyy}|_{y = 0} = 0 \\
&\| \nabla^K \{ u^0_p, v^0_p\} e^{My} \|_\infty \lesssim 1 \text{ for any } K \ge 0, \\
&\|u^i_p \|_\infty + \| \nabla^K  u^i_p  e^{My} \|_\infty + \| \nabla^J v^i_p e^{My} \|_\infty \lesssim 1 \text{ for any } K \ge 1, M \ge 0, \\
&\| \nabla^K \{u^1_e, v^1_e\}  w_{m_1} \|_\infty \lesssim 1 \text{ for some fixed } m_1 > 1  \\
&\| \nabla^K \{u^i_e, v^i_e\} w_{m_i}\|_\infty \lesssim 1 \text{ for some fixed } m_i > 1,
\end{aligned}
\end{align}

\noindent where $w_{m_i} \sim e^{m_i Y}$ or $(1+Y)^{m_i}$. 

In addition the following estimate on the remainder forcing holds: 
\begin{align} \label{thm.force.maz}
\| F_R|_{x = 0} w_0 \| + \| \p_x F_R \frac{w_0}{\sqrt{\eps}} \|  \lesssim \sqrt{\eps}^{n-1-2N_0},
\end{align}
\noindent where $F_R$ has been defined in (\ref{forcingdefn}).
\end{theorem}

\noindent \textbf{Acknowledgements:} This research is supported in part by NSF grants DMS-1611695, DMS-1810868, Chinese
NSF grant 10828103, BICMR, as well as a Simon Fellowship. Sameer Iyer was also supported in part by NSF grant DMS 1802940.

\def\bibindent{3.5em}

\end{document}